\documentclass[%
	a4paper,
	twoside,
	parskip=half,
	headsepline,
	12pt,
	bibliography=totoc, 
	chapterprefix
]{scrbook}

\usepackage[english]{babel}
\usepackage[utf8]{inputenc}
\usepackage{makeidx}
\usepackage[T1]{fontenc}

\usepackage{mathrsfs}
\usepackage{amssymb}
\usepackage{amsfonts}
\usepackage{amsthm}
\usepackage{amsmath}
\usepackage{shadethm}
\usepackage{mathdots}
\usepackage{latexsym}

\usepackage{graphicx}
\usepackage{color}
\usepackage{tikz}
\usetikzlibrary{arrows}
\usetikzlibrary{arrows.meta}
\usepackage{epsfig}
\usepackage{booktabs}
\usepackage{pdfpages}

\usetikzlibrary{decorations.text}

\usepackage{units}
\usepackage{scrpage2}    
\usepackage[hyphens]{url}
\usepackage{framed}
\usepackage{enumerate}
\usepackage{tocstyle}
\usepackage{yfonts}
\usepackage{array}
\usepackage{dsfont}
\usepackage{setspace} 
\usepackage{hyperref}
\usepackage{pdfpages}
\usepackage{float}
\usepackage{floatflt}
\usepackage{wrapfig} 
\usepackage{chngcntr}

\usepackage{algorithm}
\usepackage{algpseudocode}

\counterwithout{equation}{chapter}

\usetikzlibrary{snakes}

\usetocstyle{classic} 

\makeindex

\begin{document}

\setkomafont{sectioning}{\normalcolor\bfseries}
\theoremstyle{definition}
\newtheorem{definition}{Definition}[chapter]
\newtheorem{remark}[definition]{Remark}
\newtheorem{observation}[definition]{Observation}
\newtheorem{example}[definition]{Example}
\theoremstyle{plain}
\newtheorem{theorem}[definition]{Theorem}
\newtheorem{lemma}[definition]{Lemma} 
\newtheorem{corollary}[definition]{Corollary}
\newtheorem{conjecture}[definition]{Conjecture}
\newtheorem{proposition}[definition]{Proposition}

\setkomafont{chapter}{\centering\LARGE\scshape}

\newcommand{\Hrule}{\rule{\linewidth}{0.7mm}}

\newcommand{\R}{\mathbb R}
\newcommand{\N}{\mathbb N}
\newcommand{\Z}{\mathbb Z}
\newcommand{\Q}{\mathbb Q}
\newcommand{\C}{\mathbb C}
\newcommand{\set}[1]{\left\{ #1 \right\}}
\newcommand{\surj}{\textrm{Surj}}
\newcommand{\condset}[2]{\left\{ #1 \left|~ #2 \right.\right\}}

\newcommand{\B}{\cal{B}}
\newcommand{{\D}}{\cal D}
\newcommand{\F}{{\cal F}}
\newcommand{\K}{{\cal K}}
\newcommand{\G}{{\cal G}}
\newcommand{\A}{\mathfrak{A}}

\newcommand{\define}{\colon\!\!\!=}
\newcommand{\bi}[2]{{ #1 \choose #2}}	
\newcommand{\cupdot}{\mathbin{\dot{\cup}}}
\newcommand{\dcup}{\mathop{\dot{\cup}}}	
\newcommand{\bigdcup}{\mathop{\dot{\bigcup}}}	
\newcommand{\bigcupdot}{\mathop{\ooalign{$\bigcup$\cr\kern+3pt $\cdot$}}}
\newcommand{\lb}{\left(}
\newcommand{\rb}{\right)}
\newcommand{\vektor}[1]{ \begin{bmatrix} #1 \end{bmatrix} }
\newcommand{\comp}[1]{\overline{#1}}
\newcommand{\Prob}[1]{P\!\lb #1\rb}
\newcommand{\Varprob}[1]{\tilde{P}\!\lb #1\rb}
\newcommand{\aufr}[1]{\left\lceil #1\right\rceil} 
\newcommand{\abr}[1]{\left\lfloor #1\right\rfloor}
\newcommand{\fkt}[2]{#1\!\lb #2\rb}
\newcommand{\fkti}[2]{#1\lb #2\rb}
\newcommand{\opfkt}[2]{\operatorname{#1}\!\lb #2\rb}
\newcommand{\fkte}[3]{#1^{#2}\!\lb #3\rb}
\newcommand{\stueckfkt}[4]{\left\{\begin{array}{ll} #1, & #2 \\#3, & #4\end{array}\right.} 
\newcommand{\stueckfktd}[6]{\left\{\begin{array}{ll} #1, & #2 \\#3, & #4\\#5, &#6\end{array}\right.} 
\newcommand{\ggT}[1]{\operatorname{ggT}\!\lb #1\rb}
\newcommand{\erw}[1]{\operatorname{E}\!\lb #1\rb} 
\newcommand{\varerw}[1]{\operatorname{E}\!\left[ #1\right]}
\newcommand{\abs}[1]{\left|#1\right|} 
\newcommand{\varabs}[1]{|#1|} 
\newcommand{\ord}[1]{\operatorname{o}\lb #1\rb} 
\newcommand{\Ord}[1]{\mathcal{O}\!\lb #1\rb} 
\newcommand{\sprod}[2]{<\!#1,#2\!>}

\newcommand{\ffac}[2]{#1^{\underline{#2}}}		
\newcommand{\sfac}[2]{#1^{\overline{#2}}}		
\newcommand{\bell}{\textrm{Bell}}		
\newcommand{\pbin}[2]{\left[{#1 \atop #2}\right]}	
\newcommand{\ramsey}[2]{R\!\lb #1,#2\rb} 
\newcommand{\rang}[1]{\operatorname{Rang}\!\lb #1\rb}

\newcommand{\V}[1]{V\!\lb #1\rb}
\newcommand{\E}[1]{E\!\lb #1\rb}
\newcommand{\nb}[1] {N\!\lb #1\rb}
\newcommand{\knb}[2]{N_{#1}\!\lb #2\rb} 
\newcommand{\nbclosed}[1] {N\!\left[ #1\right]}
\newcommand{\cng}[2] {N_{\overline{#1}}\!\lb #2\rb}
\newcommand{\induz}[2]{#1 \!\left[ #2 \right]}
\newcommand{\dist}[2]{\operatorname{dist}\!\lb #1,#2\rb}
\newcommand{\distg}[3]{\fkt{\operatorname{dist}_{#1}}{#2, #3}}
\newcommand{\distlg}[3]{\fkt{\operatorname{dist}_{\operatorname{L}\lb #1\rb}}{#2, #3}}
\newcommand{\gir}{\textrm{girth}}	
\newcommand{\X}[1]{{\chi}\!\lb #1\rb}
\newcommand{\Xq}[1]{\overline{\chi}\!\lb #1\rb}
\newcommand{\Xqr}[2]{\overline{\chi_{#2}}\!\lb #1\rb}
\newcommand{\strong}[1]{\chi_{_s}\!\left( #1 \right)}
\newcommand{\strongi}[1]{\chi_{_s}'\!\left( #1 \right)}
\newcommand{\stronk}[2]{\chi_{#1}\!\left( #2 \right) }
\newcommand{\stronki}[2]{\chi_{#1}'\!\left( #2 \right) }
\newcommand{\stern}[1]{\induz{\star}{#1}}
\newcommand{\lineg}[1]{\fkt{\operatorname{L}}{ #1 }}
\newcommand{\linesq}[1]{\fkt{\operatorname{L}}{ #1 }^2}
\newcommand{\im}[1]{\fkt{\operatorname{im}}{#1}}
\newcommand{\am}[1]{\operatorname{am}\!\lb #1 \rb}
\newcommand{\amc}[1]{\fkt{\operatorname{amc}}{#1}}
\newcommand{\kcov}[2]{\fkt{\overline{\chi_{#1}}}{#2}}
\newcommand{\ams}{\operatorname{am}}
\newcommand{\ims}{\operatorname{im}}
\newcommand{\amcs}{\operatorname{amc}}
\newcommand{\iso}[1]{\mathring{I}\!\!\lb #1\rb}
\newcommand{\kcl}[2]{\fkt{\omega_{#1}}{#2}}
\newcommand{\kst}[2]{\fkt{\alpha_{#1}}{#2}}
\newcommand{\kam}[2]{\operatorname{am_{#1}}\!\lb #2 \rb}
\newcommand{\kamc}[2]{\fkt{\operatorname{amc_{#1}}}{#2}}
\newcommand{\kmat}[2]{\fkt{\nu_{#1}}{#2}}
\newcommand{\tw}[1]{\fkt{\operatorname{tw}}{#1}}
\newcommand{\vc}[1]{\fkt{\theta}{#1}}

\newcommand{\ksec}[2]{\left[ #1 \right]_{ #2 }}
\newcommand{\dual}[1]{#1^{^*}}
\newcommand{\rank}[1]{\fkt{\operatorname{rang}}{#1}}
\newcommand{\chord}[1]{\fkt{\operatorname{chord}}{#1}}
\newcommand{\chordi}[1]{\fkt{\operatorname{chord}'}{#1}}

\newcommand{\DTIME}[1]{\fkt{\operatorname{DTIME}}{#1}}
\newcommand{\NTIME}[1]{\fkt{\operatorname{NTIME}}{#1}}
\newcommand{\npcomp}{\mathcal{NP}\text{-complete}}
\newcommand{\nphard}{\mathcal{NP}\text{-hard}}
\newcommand{\FPT}{\mathcal{FPT}}
\newcommand{\XP}{\mathcal{XP}}

\tikzset{
    position/.style args={#1:#2 from #3}{
        at=(#3.#1), anchor=#1+180, shift=(#1:#2)
    }
}

\pagenumbering{roman}


\thispagestyle{empty}

\begin{center}
\large \selectfont \mdseries \textsc{Fakultät für Mathematik, Informatik und Naturwissenschaften der Rheinisch-Westfälischen Technischen Hochschule Aachen}\\[9ex]

\Large
Master Thesis\\[8ex]

\huge{\textbf{The Strong Colors of Flowers}}\\
\Large{\textbf{The Structure of Graphs with chordal Squares}}\\[6ex]

\Large{Sebastian Wiederrecht}\\[7ex]
at Lehrstuhl II für Mathematik\\[2ex]
11.03.2016\\[9ex]
\begin{tabular}{r l}
	Referee: &Prof.\ Dr.\ Eberhard Triesch \\
	Second Referee: &Dr.\ Robert Scheidweiler
\end{tabular}

\end{center}

\newpage

\newpage
\thispagestyle{empty}
\mbox{}

\newpage
\setcounter{page}{1}
	\thispagestyle{plain}
\storeareas\stdSettings
\areaset{12cm}{\textheight}
\setlength{\topmargin}{0cm}
\setlength{\footskip}{0cm}

\section*{Abstract}
A proper {\em vertex coloring} of a graph is a mapping of its vertices on a set of colors, such that two adjacent vertices are not mapped to the same color. This constraint may be interpreted in terms of the distance between to vertices and so a more general coloring concept can be defined: The strong coloring of a graph. So a {\em $k$-strong coloring} is a coloring where two vertices may not have the same color if their distance to each other is at most $k$. The $2$-strong coloring of the line graph is known as the {\em strong edge coloring}.\\
Coloring the $k$th {\em power} $G^k$ of a graph $G$ is the same as finding a $k$-strong coloring of $G$ itself. In order to obtain a graph class on which the $2$-strong coloring becomes efficiently solvable we are looking for a structure that produces induced cycles in the square of $G$, so that by excluding this structure we obtain a graph class with chordal squares, where a {\em chordal graph} is a graph without any induced cycles of length $\geq 4$. Such a structure is called a {\em flower.}\\
Another structure will be found and explained, which is responsible for flowers to appear in the line graph of $G$: The {\em sprouts}. With this graphs with chordal line graph squares are described as well.\\
Some attempts in generalizing those structures to obtain perfect graph squares are being made and the general concept of chordal graph powers, i.e. the existence of a smallest power for which a graph becomes chordal, the {\em power of chordality} is introduced in order to solve some coloring related $\mathcal{NP}$-hard problems on graphs with parameterized algorithms. Some connections to the famous parameter {\em treewidth} arise alongside with some deeper connections between edge and vertex coloring.
	
\clearpage\stdSettings

\newpage
\thispagestyle{empty}
\mbox{}

	\newpage
	\thispagestyle{plain}

\setlength{\topmargin}{0cm}
\setlength{\footskip}{0cm}

\section*{Acknowledgements}
This thesis was the greatest challenge I had to face in my academical life so far and its completion would not have been possible without the help and support of several people on whom I could and have heavily relied.\\
Before everyone else Robert Scheidweiler has to be mentioned not only as the second referee of this thesis but as my mentor in these early steps into the world of research. As I move on I definitely will miss our sometimes hour-long conversations not only about mathematics but about life and chance. Thank you Robert, for your support, your ideas and even your criticism. Without you this thesis would not exist.\\
In various forms I received help and support from friends and especially one stands out before everyone else, Robert Löw. There are times in which I actually question if I were able to achieve my bachelor's degree in the first place if it was not for Robert. He is a dear and wise friend who always seemed to have time to listen to the problems I had with my research and he patiently listened going on about various topics concerning this thesis even if I could not come up with a specific question to ask. Together with Leon Eikelmann, Nora Lüpkes and Pascal Vallet, he endured the tiring act of proof reading. I wish to thank all four of them for dedicating their time on this task.\\
At last my parents deserve some thankful words for both their moral and financial support during my whole time at the RWTH. They tried to listen and understand what I am doing and always believed I myself knew what that was, even if I myself was not entirely sure of it.  

\newpage
\thispagestyle{empty}
\mbox{}

\newpage
\thispagestyle{plain}

\setlength{\topmargin}{0cm}
\setlength{\footskip}{0cm}
~\\
~\\
~\\
~\\
~\\
~\\
~\\
~\\
~\\
~\\
\begin{quotation}
	\noindent Complexity is the prodigy of the world. Simplicity is the sensation of the universe. Behind complexity, there is always simplicity to be revealed. Inside simplicity, there is always complexity to be discovered.
\end{quotation}
\vspace{-4mm}
\begin{flushright}
\em	Yu Gang, Chinese University of Hong Kong
\end{flushright}

\clearpage\stdSettings

\newpage
\thispagestyle{empty}
\mbox{}

\tableofcontents

\pagenumbering{arabic}

\chapter{Introduction}


Dividing a given set into subsets is a fundamental procedure in mathematics and often the subsets are required to satisfy some prescribed requirements. The arguably most popular division of a set into subsets in graph theory is the coloring of the vertices of a graph such that no two adjacent vertices have the same color.\\
With its roots in the numerous attempts to solve the famous Four Color Problem the theory of vertex coloring has become one of the most studied field of graph theory. And due to the sheer number of unsolved problems, open questions and known - and unknown - applications, it still is an exciting part of modern science.\\
A number of graph coloring problems have their roots in a practical problem in communications known as the {\em Channel Assignment Problem}. In this problem there are transmitters like antennas or satellite dishes located in some geographic region. Those transmitters may interfere with each other which is a very common problem. There is a sheer number ob different possible reasons for this interference such as their distance to each other, the time of day, month or year, the terrain on which the transmitter is built, its power or the existence of power lines in the area.\\
Such a problem can be modeled by a graph whose vertices represent the transmitters and two transmitters are joined by an edge if and only if they interfere with each other. The goal is to assign frequencies or channels to the transmitters in a manner that prevents the signal of being disrupted by the interference. Obviously, as this is a real world problem, there are some restricted resources like a limited number of available channels or a limited budget to buy such frequencies. This gives some natural optimality criteria in which the channels should be assigned.\\
This problem, in numerous variations, has been studied in economical and military backgrounds, sometimes just for the sake of its beauty, and, since the interpretation of channels as colors gives rise to graph coloring problems, it clearly stands in the tradition of the Four Color Problem, which itself rose from a practical problem with restricted resources.\\
{\em Distance Coloring}, the assignment of colors to the vertices of a graph in such a manner that vertices within a certain distance to each other receive different colors is one of the many variations of the Channel Assignment Problem and has some very exciting, and sometimes even surprising, implications in the theory of graphs that may go far beyond the original coloring problem.\\
In this thesis we will investigate some of these implications with a strong focus on the algorithmic approachability. In terms of coloring problems this means the concepts of chordality and perfection which are important parts of the foundation of algorithmic graph theory. This will lead us, on various occasions, away from the original coloring problem and deeper into the structural theory of graphs, but eventually those detours will prove fruitful.

\chapter{Definitions and Graph Coloring}


In this chapter we state some basic definitions in graph theory that are involved in this thesis. The concepts of coloring the vertices and edges of a graph are introduced alongside with their generalizations and concepts of perfection in this context. Furthermore we give illustrating examples and prove some basic results.

\section{Basic Terminology and Definitions}

\begin{definition}[Graph]A {\em graph} $G=\lb V,E\rb$ consists of two sets: the {\em vertex set} and the {\em edge set}, denoted by $\V{G}$ and $\E{G}$, or $V$ and $E$ when no ambiguities arise, the elements of $V$ are called {\em vertices} and the elements of $E$, which are $2$-element subsets of $V$, are called {\em edges} respectively. Each edge consists of two different vertices called {\em endpoints}. An edge $e$ with endpoints $u$ and $v$ is denoted by $e=uv$.\\
The cardinality of the vertex set $V$ is called the {\em order} of $G$, denoted by $\abs{V}$; the cardinality of the edge set $E$ is denoted by $\abs{E}$ and called the {\em size} of $G$.\\
The graph which has only one vertex and no edges is called the {\em trivial graph}. The graph $G$ is finite if and only if $\abs{V}$ is a finite number.
\end{definition}

\begin{definition}[Simple Graph]
An {\em undirected graph} is one in which edges have no orientation. A {\em loop} is an edge connected at both ends to the same vertex. Multiple edges are two, three or more edges connecting the same pair of vertices. A {\em simple graph} is an undirected graph that has no loops and no multiple edges.
\end{definition}

In this thesis, we only consider finite and simple graphs.

\begin{definition}[Neighborhood]
Two vertices $u$ and $v$ are said to be {\em adjacent} in $G$ if $uv\in\E{G}$. Adjacent vertices are {\em neighbors} of each other. The set of all neighbors of any vertex $w$ in some Graph $G$ is denoted by $\fkt{N_G}{w}$, or $\fkt{N}{w}$ when no ambiguities arise.\\
The {\em closed neighborhood} of $w$ is given by $\induz{N}{w}=\fkt{N}{w}\cup\set{w}$.\\
Two edges $e$ and $e'$ are said to be adjacent if they share a common endpoint and the {\em edge neighborhood} of $e$ is denoted by $\fkt{N}{e}=\condset{e'\in\E{G}}{e\cap e'\neq\emptyset}$.\\
A vertex $v$ is incident to an edge $e$ - and vice versa - if $v$ is an endpoint of $e$.
\end{definition}

\begin{definition}[Degree]
The {\em degree} of a vertex $v$ in a graph $G$ is the number of edges incident to $v$, denoted by $\fkt{\deg_G}{v}$, or $\fkt{\deg}{v}$ when no ambiguities arise. If $G$ is a simple graph it holds $\fkt{\deg_G}{v}=\abs{\fkt{N}{v}}$. The {\em maximum degree} of a graph G is given by $\fkt{\Delta}{G}=\max_{v\in\V{G}} \fkt{\deg_G}{v}$ and the {\em minimum degree} of $G$ is given by $\fkt{\delta}{G}=\min_{v\in\V{G}} \fkt{deg_G}{v}$.
\end{definition}

\begin{definition}[Paths and Cycles]
Let $G$ be a graph. For any $u, v \in\V{G}$, a {\em $u$-$v$ walk} of $G$ is a finite alternating sequence of vertices and edges of $G$ which begins with vertex $u$ and ends with vertex $v$, such that each edge’s endpoints are the preceeding and following vertices in the sequence. A walk is said to be {\em closed} if the sequence starts and ends at the same vertex.\\
A {\em $u$-$v$ path} is a $u$-$v$ walk with no repeated edges and vertices, denoted by $P_{u,v}$.\\
The {\em length} of the path, denoted by $\abs{P_{u,v}}$, is defined to be the number of edges in the sequence. A path is called an {\em empty path} if its length is zero, i.e., the sequence has one vertex and no edges.\\
A {\em cycle} is a closed walk with no repetitions of vertices and edges allowed, other than the repetition of the starting and ending vertex. A cycle of order, or {\em length}, $n$ is denoted by $C_n$, $n$ is a positive integer and $n\geq3$ for any cycle.\\
The {\em girth} $\fkt{\operatorname{girth}}{G}$ of a graph is the length of a longest cycle in $G$.
\end{definition}

\begin{definition}[Distance]
The {\em distance} between vertices $u$ and $v$ in graph $G$, denoted by $\distg{G}{u}{v}$, is the length of a shortest path between $u$ and $v$. If the graph $G$ is clear from the context, we denote the distance between vertices $u$ and $v$ by $\dist{u}{v}$.
\end{definition} 

Since only simple graphs are considered in this thesis an edge is uniquely determined by its end vertices. Therefore we can use a sequence of vertices to represent a walk, path or cycle.\\
In Figure \ref{fig2.1} we find that the paths $P_1=a,b,c$, $P_2=a,b,e,d,c$, $P_3=a,f,e,d,c$ and $P_4=a,f,e,b,c$ are all possible $a$-$c$ paths. Then $\dist{a}{c}=\min\set{ P_1,P_2,P_3,P_4 }=2$. 

\begin{figure}[!h]\label{fig2.1}
\begin{center}
\begin{tikzpicture}

\node (a) [draw,circle,fill,inner sep=1.5pt] {};
\node (b) [draw,circle,fill,inner sep=1.5pt,position=0:1.3cm from a] {};
\node (c) [draw,circle,fill,inner sep=1.5pt,position=0:2.6cm from b] {};
\node (d) [draw,circle,fill,inner sep=1.5pt,position=90:1.3cm from c] {};
\node (e) [draw,circle,fill,inner sep=1.5pt,position=180:1.3cm from d] {};
\node (f) [draw,circle,fill,inner sep=1.5pt,position=180:2.6cm from e] {};

\node (l_a) [position=225:0.05cm from a] {$a$};
\node (l_b) [position=270:0.05cm from b] {$b$};
\node (l_c) [position=315:0.05cm from c] {$c$};
\node (l_d) [position=45:0.05cm from d] {$d$};
\node (l_e) [position=90:0.05cm from e] {$e$};
\node (l_f) [position=135:0.05cm from f] {$f$};

\path
(a) edge (b)
(b) edge (c)
(c) edge (d)
(d) edge (e)
(e) edge (f)
(f) edge (a)
(b) edge (e)
;

\end{tikzpicture}
\end{center}
\caption{An example for paths and distance}
\end{figure}
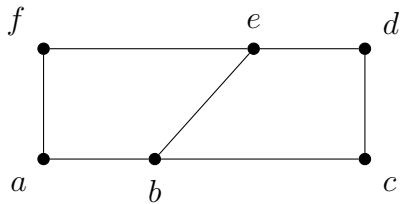

\begin{definition}[$k$-Neighborhood]
For some vertex $v\in\V{G}$ the {\em $k$-Neighborhood} of $v$ is the set $\fkt{N_k}{v}=\condset{w\in\V{G}\setminus{v}}{\distg{G}{v}{w}\leq k}$. The {\em closed $k$-Neighborhood} is given by $\induz{N_k}{v}=\set{v}\cup\fkt{N_k}{v}$.
\end{definition}

\begin{definition}[Connectivity]
In a graph $G$, two vertices $u$ and $v$ are called {\em connected} if $G$ contains a path from $u$ to $v$. Otherwise, they are called {\em disconnected}. A graph is said to be {\em connected} if every pair of vertices in the graph is connected.\\
Connectivity defines an equivalence relation on the set of vertices $V$, the equivalence classes are called {\em connectivity components} or just {\em components} of $G$.
\end{definition}

\begin{definition}[Complete Graph]
A {\em complete graph} $K_n$, where $n$ is a positive integer, is a graph of order $n$ such that every pair of distinct vertices in $K_n$ are connected by a unique edge.\\
The {\em trivial} graph is $K_1$ .
\end{definition} 

\begin{definition}[Subgraph]
A {\em subgraph} of a graph $G=\lb V,E\rb$ is a graph $H=\lb V_H,E_H\rb$ with $V_H\subseteq V$ and $E_H\subseteq E$, denoted by $H\subseteq G$.\\
If $H$ is a subgraph of $G$, then $G$ is a {\em supergraph} of $H$. Furthermore, the subgraph $H$ is called an {\em induced subgraph} of $G$ on $V_H$ if $uv\in E$ implies $uv\in E_H$ for any $u,v\in V_H$.\\
Let $U\subseteq V$, the induced subgraph of $G$ on the vertex set $U$ is denoted by $\induz{G}{U}$.\\
For some set $W\subseteq\E{G}$ of edges the {\em edge induced} subgraph $\induz{G}{W}$ is the graph induced by the vertex set $\bigcup_{e\in W}w$.
\end{definition}

\begin{definition}[$\mathcal{F}$-free]
We say that $G$ {\em contains} a graph $F$, if $F$ is isomorphic to an induced subgraph of $G$. A graph is {\em $F$-free} if it does not contain $F$, and for a family of graphs $\mathcal{F}$, $G$ is {\em $\mathcal{F}$-free} if $G$ is $F$-free for every $F\in\mathcal{F}$.
\end{definition}

\begin{definition}[Clique and Stable Set]
The set $C\subseteq V$ is called a {\em clique} if and only if the induced subgraph $\induz{G}{C}$ is a complete graph. We denote the maximum cardinality of such a subset $C$ of $V$ by $\fkt{\omega}{G}$.\\
The set $I\subseteq V$ is called {\em independent} or {\em stable} if and only if the induced subgraph $\induz{G}{I}$ does not contain any edges. We denote the maximum cardinality of such a subset $I$ of $V$ by $\fkt{\alpha}{G}$. 
\end{definition}

\begin{figure}[!h]
\begin{center}
\begin{tikzpicture}

\node (anchor) [] {};

\node (center) [inner sep=1.5pt] {};

\node (label) [inner sep=1.5pt,position=135:2.5cm from center] {$G$};

\node (u1) [inner sep=1.5pt,position=90:1.8cm from center,draw,circle,fill] {};
\node (u2) [inner sep=1.5pt,position=180:1.8cm from center,draw,circle,fill] {};
\node (u3) [inner sep=1.5pt,position=270:1.8cm from center,draw,circle,fill] {};
\node (u4) [inner sep=1.5pt,position=0:1.8cm from center,draw,circle,fill] {};

\node (v1) [inner sep=1.5pt,position=135:0.9cm from center,draw,circle,fill] {};
\node (v2) [inner sep=1.5pt,position=225:0.9cm from center,draw,circle,fill] {};
\node (v3) [inner sep=1.5pt,position=315:0.9cm from center,draw,circle,fill] {};
\node (v4) [inner sep=1.5pt,position=45:0.9cm from center,draw,circle,fill] {};

\node (lu1) [position=90:0.07cm from u1] {$u_1$};
\node (lu2) [position=180:0.07cm from u2] {$u_2$};
\node (lu3) [position=270:0.07cm from u3] {$u_3$};
\node (lu4) [position=0:0.07cm from u4] {$u_4$};

\node (lv1) [position=135:0.07cm from v1] {$v_1$};
\node (lv2) [position=225:0.07cm from v2] {$v_2$};
\node (lv3) [position=315:0.07cm from v3] {$v_3$};
\node (lv4) [position=45:0.07cm from v4] {$v_4$};

\path
(u1) edge (v1)
	 edge (v4)
(u2) edge (v1)
	 edge (v2)
(u3) edge (v2)
	 edge (v3)
(u4) edge (v3)
	 edge (v4)
;

\path 
(v1) edge (v2)
(v2) edge (v3)
(v3) edge (v4)
(v4) edge (v1)
;


\node (center) [inner sep=1.5pt,position=0:5cm from anchor] {};

\node (label) [inner sep=1.5pt,position=135:2cm from center] {$H_1$};

\node (u1) [inner sep=1.5pt,position=90:1.8cm from center,draw,circle,fill] {};

\node (v1) [inner sep=1.5pt,position=135:0.9cm from center,draw,circle,fill] {};
\node (v2) [inner sep=1.5pt,position=225:0.9cm from center,draw,circle,fill] {};
\node (v3) [inner sep=1.5pt,position=315:0.9cm from center,draw,circle,fill] {};
\node (v4) [inner sep=1.5pt,position=45:0.9cm from center,draw,circle,fill] {};

\node (lu1) [position=90:0.07cm from u1] {$u_1$};

\node (lv1) [position=135:0.07cm from v1] {$v_1$};
\node (lv2) [position=225:0.07cm from v2] {$v_2$};
\node (lv3) [position=315:0.07cm from v3] {$v_3$};
\node (lv4) [position=45:0.07cm from v4] {$v_4$};

\path
(u1) edge (v1)
	 edge (v4)
;

\path 
(v1) edge (v2)
(v2) edge (v3)
(v3) edge (v4)
(v4) edge (v1)
;


\node (center) [inner sep=1.5pt,position=0:10cm from anchor] {};

\node (label) [inner sep=1.5pt,position=135:2cm from center] {$H_2$};

\node (u1) [inner sep=1.5pt,position=90:1.8cm from center,draw,circle,fill] {};

\node (v1) [inner sep=1.5pt,position=135:0.9cm from center,draw,circle,fill] {};
\node (v2) [inner sep=1.5pt,position=225:0.9cm from center,draw,circle,fill] {};
\node (v3) [inner sep=1.5pt,position=315:0.9cm from center,draw,circle,fill] {};
\node (v4) [inner sep=1.5pt,position=45:0.9cm from center,draw,circle,fill] {};

\node (lu1) [position=90:0.07cm from u1] {$u_1$};

\node (lv1) [position=135:0.07cm from v1] {$v_1$};
\node (lv2) [position=225:0.07cm from v2] {$v_2$};
\node (lv3) [position=315:0.07cm from v3] {$v_3$};
\node (lv4) [position=45:0.07cm from v4] {$v_4$};

\path
(u1) edge (v1)
	 edge (v4)
;

\path 
(v1) edge (v2)
(v2) edge (v3)
(v3) edge (v4)
;


\node (anchor2) [position=280:4.5cm from anchor] {};

\node (center) [inner sep=1.5pt,position=0:0cm from anchor2] {};

\node (label) [inner sep=1.5pt,position=90:1cm from center] {$C$};

\node (u4) [inner sep=1.5pt,position=0:1.8cm from center,draw,circle,fill] {};

\node (v3) [inner sep=1.5pt,position=315:0.9cm from center,draw,circle,fill] {};
\node (v4) [inner sep=1.5pt,position=45:0.9cm from center,draw,circle,fill] {};

\node (lu4) [position=0:0.07cm from u4] {$u_4$};

\node (lv3) [position=315:0.07cm from v3] {$v_3$};
\node (lv4) [position=45:0.07cm from v4] {$v_4$};

\path
(u4) edge (v3)
	 edge (v4)
;

\path 
(v3) edge (v4)
;


\node (center) [inner sep=1.5pt,position=0:7cm from anchor2] {};

\node (label) [inner sep=1.5pt,position=135:2.5cm from center] {$I$};

\node (u1) [inner sep=1.5pt,position=90:1.8cm from center,draw,circle,fill] {};
\node (u2) [inner sep=1.5pt,position=180:1.8cm from center,draw,circle,fill] {};

\node (v3) [inner sep=1.5pt,position=315:0.9cm from center,draw,circle,fill] {};

\node (lu1) [position=90:0.07cm from u1] {$u_1$};
\node (lu2) [position=180:0.07cm from u2] {$u_2$};
\node (lv3) [position=315:0.07cm from v3] {$v_3$};

\end{tikzpicture}
\end{center}
\caption{An example for induced subgraphs, cliques and stable sets.}
\label{fig2.2}
\end{figure}
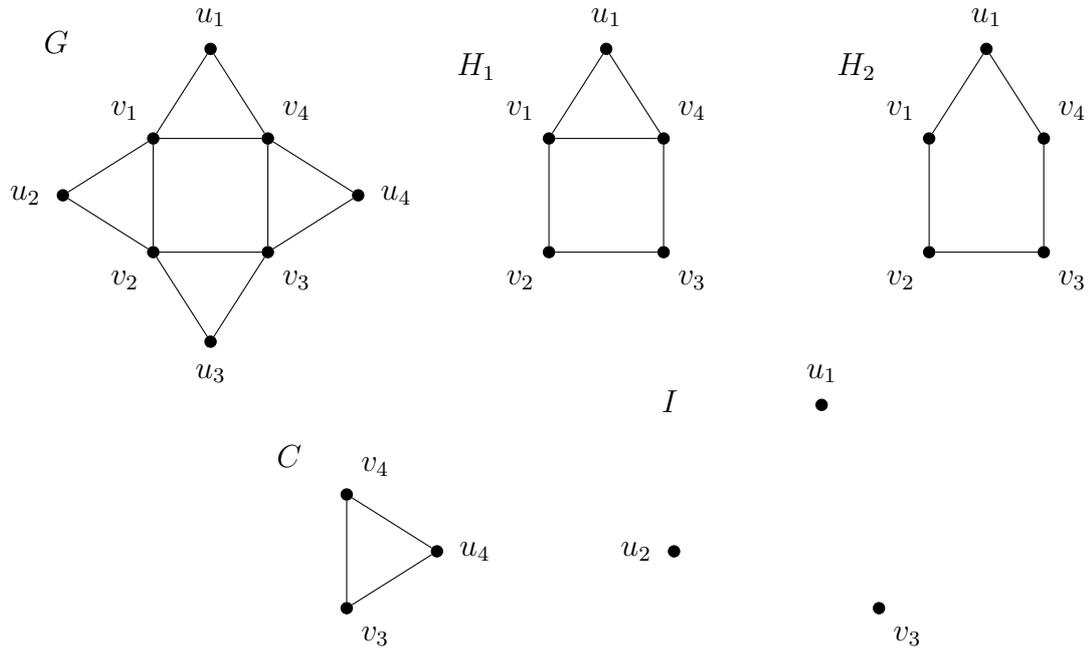

Consider the graph $G$ in Figure \ref{fig2.2}, the graphs $H_1$ and $H_2$ are both subgraphs of $G$. $H_1$ contains all edges that exist in $G$ between the vertices $U=\set{u_1,v_1,v_2,v_3,v_4}$, therefore it holds $\induz{G}{U}\cong H_1$ and $H_1$ is an induced subgraph of $G$. The graph $H_2$ on the other hand lacks the edge $v_1v_4$ and therefore is not an induced subgraph of $G$.\\
The graph $C$ too is an induced subgraph of $G$ and in addition $C\cong K_3$, so $C$ is a clique in $G$, i.e. $C$ is a maximum clique of $G$ and $\fkt{\omega}{G}=3$.\\
At last $I$ is an induced subgraph of $G$, in particular one that does not contain edges which makes it a stable set of vertices. Although $I$ is maximal in the manner that no other vertex of $G$ could be included without producing an edge, it is no maximum stable set. We have $\fkt{\alpha}{G}=4$ and the only maximum stable set of $G$ is given by $\set{u_1,u_2,u_3,u_4}$.  


\begin{definition}[Basic Operations]
For some subset of vertices $S\subset\V{G}$ of $G$. $G-S$ denotes the induced subgraph $\induz{G}{\V{G}\setminus S}$ of $G$, which results from $G$ by deleting $S$. If $\abs{S}=1$ with $S=\set{v}$ we write $G-v$.\\
For deleting edges we have to differentiate between two possibilities, deletion and subtraction. Let $W\subset\E{G}$, deleting the edges means the removal of $W$ from $G$, denoted by $G\setminus W$, which results in the graph $\lb \V{G},\E{G}\setminus W \rb$, subtraction, denoted by $G-W$, means the induced subgraph obtained by the deletion of the endpoints of all edges in $W$, which results in the graph $G-\lb\bigcup_{e\in W}e\rb$.
\end{definition}

\begin{definition}[Line Graph]
The {\em line graph} $\lineg{G}$ of a graph $G$ is the graph defined on the edge set
$\E{G}$ (each vertex in $\lineg{G}$ represents a unique edge in $G$) such that two vertices in
$\lineg{G}$ are defined to be adjacent if their corresponding edges in $G$ share a common
endpoint. The distance between two edges in G is defined to be the distance between
their corresponding vertices in $\lineg{G}$. For two edges $e, e'\in\E{G}$ we write $\distg{G}{e}{e'}=\distlg{G}{e}{e'}$.
\end{definition}

\begin{definition}[Complement]
The {\em complement} of a graph $G=\lb V,E\rb$ is given by $\overline{G}=\lb V,\overline{E}\rb$, where $\overline{E}$ denotes the edges $\binom{V}{2}\setminus E$ and $\binom{V}{2}$ is the set of all subsets of $V$ with cardinality $2$. 
\end{definition}

\begin{definition}[Tree]
A {\em tree} $T$ is a simple connected graph that has no cycles. Any vertex in a tree with degree one is called a {\em leaf}.
\end{definition}

\begin{definition}[Star]
A {\em star} is a tree with exactly $n-1$ leaves, where $n$ is the order of the tree.\\
A {\em doublestar} is a tree with exactly $n-2$ leaves, where $n$ is the order of the tree. If the degrees of the two non-leaf vertices of a double star are $x$ and $y$, where $\fkt{\deg}{x}\leq\fkt{\deg}{y}$, we denote such a double star by $D_{x,y}$.
\end{definition}

\begin{definition}[Chordal Graph]
Let $C=\lb v_1,\dots,v_n\rb$ be a cycle, the graph induced by the vertices of $C$ is denoted by $\induz{G}{C}$. If $\induz{G}{C}\cong C$ the cycle is called {\em chordless}, otherwise there exists some edge $v_iv_j$ in $G$ with $j\neq i\pm1\lb\!\!\!\!\mod n\rb$, this edge is called a {\em chord}.\\
If $G$ does not contain a chordless cycle of length $\geq 3$ it is called {\em chordal}. 
\end{definition}

\begin{definition}[Planar Graph]
A {\em planar graph} is a graph that can be embedded in the plane, i.e., it can be drawn on the plane in such a way that its edges intersect only at their endpoints.\\
A {\em plane graph} is a drawing of a planar graph on the plane such that no edges cross each other. A plane graph divides the plane into areas surrounded by edges, called {\em faces} (or {\em regions}). The area outside the plane graph is the {\em unbounded face} of the graph.
\end{definition}

\section{Vertex Coloring}

\begin{definition}[Proper Vertex Coloring]
For a graph $G$, a {\em vertex coloring} is a mapping $F$ of the vertices of $G$ to some set of colors, which usually are represented by integers: $F\colon\V{G}\rightarrow\set{1,\dots,c}$. If $F$ uses exactly $c$ different colors, we call $F$ a {\em $c$-coloring} of $G$.\\
The {\em color classes} $\fkt{F^{-1}}{i}=F_i$ give a partition of the vertex set of $G$.\\
A {\em proper vertex coloring} is a vertex coloring $F$ with $xy\in\E{G}~\Rightarrow\fkt{F}{x}\neq\fkt{F}{y}$. In this thesis a proper vertex coloring well simply be referred to as a {\em coloring}.\\
The {\em chromatic number} $\X{G}$ of $G$ is the smallest integer $c$ for which a proper $c$-coloring of $G$ exists.
\end{definition}

The neighborhood constraint of a proper vertex coloring can be rephrased in terms of the distance between vertices.
\begin{align*}
\distg{G}{x}{y}\leq 1~\Rightarrow \fkt{F}{x}\neq\fkt{F}{y}
\end{align*}
A natural generalization of proper vertex coloring seems to be the extension of the distance in which two vertices are not allowed to have the same color.\\
The idea of coloring objects with some constraint to their distance was first introduced by Eggleton, Erd\H{o}s and Skilton in their work on the chromatic number of distance graphs (\cite{eggleton1985realline}) and then was extended to general graphs by Sharp (\cite{Sharp2007distance}), who calls colorings of this type {\em distance colorings}. For better compatibility in terms of terminology with the main chapters of this thesis we will refer to such colorings as {\em strong} colorings.

\begin{definition}[$k$-strong Vertex Coloring]
For a graph $G$ and some positive integer $k\in\N$, a {\em $k$-strong vertex coloring} or simply {\em $k$-strong coloring} is a vertex coloring $F\colon\V{G}\rightarrow\set{1,\dots,c}$ with $\distg{G}{x}{y}\leq k~\Rightarrow \fkt{F}{x}\neq\fkt{F}{y}$ for all $x,y\in\V{G}$, $x\neq y$.\\
The {\em $k$-strong chromatic number} $\stronk{k}{G}$ of $G$ is the smallest integer $c$ for which a $k$-strong $c$-coloring of $G$ exists.
\end{definition}
 
While for all graphs $G$ the simple connection between the clique number of $G$ and its chromatic number $\fkt{\omega}{G}\leq \X{G}$ holds, we can derive a similar relationship for all $k$ by generalizing the concept of cliques in the same manner as we did with colorings.

\begin{definition}[$k$-strong Clique]
The set $C\subseteq \V{G}$ is called a {\em $k$-strong clique} for some positive integer $k\in\N$ if and only if for each pair $x,y\in C$ it holds $\distg{G}{x}{y}\leq k$.\\
The {\em $k$-strong clique number} $\kcl{k}{G}$ is the maximum cardinality of such a subset $C$ of $\V{G}$.
\end{definition}

\begin{definition}[$k$-strong Neighborhood]~\\
The set $\knb{k}{v}=\condset{w\in\V{G}\setminus\set{v}}{\distg{G}{v}{w}\leq k}$ is called the {\em $k$-strong neighborhood} or {\em $k$-neighborhood} of $v$.
\end{definition}

There are some basic observations hat can be made for those parameters and all $k\in\N$.
\begin{enumerate}[i)]
\item $\kcl{k-1}{G}\leq\kcl{k}{G}$

\item $\stronk{k-1}{G}\leq\stronk{k}{G}$

\item $\kcl{k}{G}\leq\stronk{k}{G}$
\end{enumerate}

The case $k=1$ clearly resembles the standard $k$-coloring problem. 

A well known upper bound to the chromatic number is given in dependency of the maximum degree $\fkt{\Delta}{G}$. A generalized version for $k$-strong colorings will be investigated in the next chapter.

\begin{theorem}\label{thm2.1}
For any graph $G$ it holds
\begin{align*}
\stronk{1}{G}\leq\fkt{\Delta}{G}+1.
\end{align*}
\end{theorem}

A special and very important class of graphs is the class of so called perfect graphs which was introduced first by Claude Berge in 1963. Not only are those graphs very important due to their structural properties, but also due to some computational results to which we will come back later.\\
The definition of perfection which we give in this thesis is not the same as the one that was introduced by Berge, although it is equivalent due to the so called weak perfect graph theorem.

\begin{definition}[Perfection]
A graph $G$ is called {\em perfect} if for all induced subgraphs $H\subseteq G$ it holds
\begin{align*}
\fkt{\omega}{H}=\X{H}.
\end{align*}
\end{definition}

In terms of $k$-strong colorings we get the generalized definition.

\begin{definition}[$k$-strong Perfection]
A graph $G$ is called {\em $k$-strong perfect} or {\em $k$-perfect} if for all induced subgraphs $H\subseteq G$ holds
\begin{align*}
\kcl{k}{H}=\stronk{k}{H}.
\end{align*}
\end{definition}

For the case $k=1$, the general perfection, several very important characterizations are known, especially the so called weak and the strong perfect graph theorem.

\begin{theorem}[Weak Perfect Graph Theorem, Lovász. 1972 \cite{Lovasz1987perfect}]\label{thm2.2}~\\
A graph $G$ is $1$-perfect if and only if $\overline{G}$ is $1$-perfect.
\end{theorem}

\begin{theorem}[Strong Perfect Graph Theorem, Chudnovsky, Robertson, Seymour, Thomas 2002 \cite{Chudnovsky2006strongperfect}]\label{thm2.3}~\\
A graph $G$ is $1$-perfect if and only if neither $G$ nor $\overline{G}$ contains an induced circle of odd length $\geq 5$.
\end{theorem}

What exactly does a proper vertex coloring of $G$ correspond to in the complement? For some proper coloring $F$ a color class $F_i$ induces a stable set in $G$, therefore its induced graph in $\overline{G}$ is complete. We cover the vertices of $\overline{G}$ with cliques and get the following.

\begin{definition}[Clique Cover]
For a graph $G$, a {\em clique cover} is a vertex $c$-coloring $C$ with $x,y\in C_i~\Rightarrow xy\in\E{G}$, where $C_i$ is a color class of $C$.\\
The {\em clique cover number} $\Xq{G}$ of $G$ is the smallest integer $c$ for which a $c$-clique cover of $G$ exists.
\end{definition}

Again we can generalize the concept of clique covers in terms of the distance between vertices and get a strong version.

\begin{definition}[$k$-strong Clique Cover]
For a graph $G$ and some positive integer $k\in\N$, a {\em $k$-strong clique cover} is a vertex $c$-coloring $C$ with $x,y\in C_i~\Rightarrow \distg{G}{x}{y}\leq k$\\
The {\em clique cover number} $\kcov{k}{G}$ of $G$ is the smallest integer $c$ for which a $k$-strong $c$-clique cover of $G$ exists.
\end{definition}

\begin{definition}[$k$-strong Stable Set]
The set $I\subseteq\V{G}$ is called a {\em $k$-strong stable set} for some positive integer $k\in\N$ if and only if for each pair $x,y\in I$, $x\neq y,$ it holds $\distg{G}{x}{y}>k$.\\
The {$k$-strong stability number} $\kst{k}{G}$ is the maximum cardinality of such a subset $I$ of $\V{G}$. 
\end{definition}

Some basic observations similar to those regarding cliques and colorings can be made for all $k\in\N$. The difference between the two concepts is that, while a $k$-strong clique gets bigger and bigger with increasing $k$ and therefore more colors are needed, a $k$-strong stable set gets smaller with increasing $k$ and therefore fewer $k$-strong cliques are required to cover the whole graph.
\begin{enumerate}[i)]
\item $\kst{k}{G}\leq\kst{k-1}{G}$

\item $\kcov{k}{G}\leq\kcov{k-1}{G}$

\item $\kst{k}{G} \leq\kcov{k}{G}$
\end{enumerate}

A well known result regarding a family of perfect graphs which will be studied very intense in this thesis is the following.

\begin{theorem}\label{thm2.4}
Every chordal graph is perfect.
\end{theorem}

\section{Edge Coloring}

The coloring of edges is somewhat of a special case of vertex coloring. The standard definition of a proper edge coloring induces a coloring of the vertices of the line graph.

\begin{definition}[Proper Edge Coloring]
For a graph $G$, an {\em edge coloring} is a mapping $F$ of the edges of $G$ to some set of colors: $F\colon\E{G}\rightarrow\set{1,\dots,c}$.\\
A {\em proper edge coloring} is an edge coloring $F$ with $e\cap e'\neq\emptyset~\Rightarrow\fkt{F}{e}\neq\fkt{F}{e'}$ for all $e,e'\in\E{G}$.\\
The {\em chromatic index} $\fkt{\chi'}{G}$ is the smallest integer $c$ for which a proper $c$-edge coloring of $G$ exists.
\end{definition}

A pair of edges $e$ and $e'$ that share a common endpoint and therefore fulfill $e\cap e'\neq\emptyset$ correspond to two adjacent vertices in the line graph. Therefore a proper edge coloring of some graph $G$ corresponds to a proper vertex coloring of $\lineg{G}$,
\begin{align*}
\fkt{\chi'}{G}=\X{\lineg{G}}.
\end{align*}
With that we can give a definition of the generalized $k$-strong edge coloring by using the induced distance in the line graph.

\begin{definition}[$k$-strong Edge Coloring]
For a graph $G$ and some positive integer $k\in\N$, a {\em $k$-strong edge coloring} is an edge coloring $F\colon\E{G}\rightarrow\set{1,\dots,c}$ with $\distlg{G}{e}{e'}\leq k~\Rightarrow\fkt{F}{e}\neq\fkt{F}{e'}$ for all $e,e'\in\E{G}$, $e\neq e'$.\\
The {\em $k$-strong chromatic index} $\stronki{k}{G}$ of $G$ is the smallest number $c$ for which a $k$-strong edge $c$-coloring of $G$ exists.
\end{definition}

A proper edge coloring partitions the set of edges into subsets where no two edges share a common endpoint. Such a set of edges is called a matching.

\begin{definition}[Matching]
The set $M\subseteq \E{G}$ is called a {\em matching} if and only if for all pairs of edges $e,e'\in M$ it holds $e\cap e'\neq\emptyset$.\\
We denote the maximum cardinality of such a subset $M$ of $E$ by $\fkt{\nu}{G}$.
\end{definition}

\begin{definition}[$k$-strong Matching]
The set $M\subseteq \E{G}$ is called a {\em $k$-strong matching} for some positive integer $k\in\N$ if and only if for all pairs of edges $e,e'\in M$ it holds $\distlg{G}{e}{e'}>k$.\\
We denote the maximum cardinality of such a subset $M$ of $E$ by $\kmat{k}{G}$.
\end{definition}

Obviously a star represents a structure which is somewhat similar to a clique in the way that no more than one edge of a star can be part of the same matching. It is due to this property that those structures are called anti-matchings.

\begin{definition}[$k$-strong Anti-Matching]
The set $A\subseteq\E{G}$ is called a {\em $k$-strong anti-matching} for some positive integer $k\in\N$ if and only if for all pairs of edges $e,e'\in A$ it holds $\distlg{G}{e}{e'}\leq k$.\\
We denote the maximum cardinality of such a subset $A$ of $E$ by $\kam{k}{G}$.
\end{definition}

The relation between $k$-strong anti-matchings and $k$-strong cliques becomes obvious by translating the definition of $k$-strong anti-matchings to the language of line graphs. We get the following relations.
\begin{enumerate}[i)]
\item $\kam{k}{G}=\kcl{k}{\lineg{G}}$

\item $\kam{k}{G}\leq\stronki{k}{G}$
\end{enumerate}
With another translation we obtain the last remaining dual concept in terms of taking the line graph of a graph: The anti-matching cover of a graph.

\begin{definition}[$k$-strong Anti-Matching Cover]
For a graph $G$ and some positive integer $k\in\N$, a {\em $k$-strong anti-matching cover} is an edge coloring $C$ with $e,e'\in C_i~\Rightarrow \distlg{G}{e}{e'}\leq k$.\\
The {\em $k$-strong anti-matching cover number} $\kamc{k}{G}$ is the smallest integer $c$ for which a $k$-strong $c$-anti-matching cover of $G$ exists. 
\end{definition}

Again we can derive the known inequality
\begin{align*}
\kmat{k}{G}=\kst{k}{\lineg{G}}\leq\kcov{k}{\lineg{G}}=\kamc{k}{G}.
\end{align*}

In addition it is possible to show some basic relation between $k$-strong anti-matchings and cliques in the same graph.

\begin{lemma}\label{lemma2.1}
Let $k$ be a positive integer and $A$ a $k$-strong anti-matching of a graph $G$. Then $C=\bigcup_{e\in A}e$ is a $\lb k+1\rb$-strong clique in $G$.
\end{lemma}

\begin{proof}
Let $A\subseteq\E{G}$ be a $k$-strong anti-matching in $G$, then it suffices to show $\distg{G}{v}{w}\leq k+1$ for all $v,w\in C=\bigcup_{e\in A}e$.\\
For each pair of edges $e,e'\in A$ it holds $\distlg{G}{e}{e'}\leq k$ and with that a shortest $e$-$e'$-path in $\lineg{G}$ contains at most $k$ edges and $k-1$ other vertices, corresponding to another $k-1$ edges in $G$.\\
In general, any path of length $l$ in the line graph, connecting two vertices $a$ and $b$ of the line graph corresponds to a path in the original graph connecting the two edges corresponding to the vertices $a$ and $b$. This path, in $G$, connects the two endpoints of $a$ and $b$ that have the smallest distance to each other and thus is by $1$ shorter than the path in the line graph.\\ 
Therefore there exists a pair of vertices $x,x'$ with $x\in e$ and $x'\in e'$ satisfying $\distg{G}{x}{x'}\leq k-1$. Furthermore for $e=xy$ and $e'=x'y'$ we get $\distg{G}{y}{y'}\leq k+1$.
\end{proof}

\begin{lemma}\label{lemma2.2}
Let $k$ be a positive integer and $C$ a $k$-strong clique of a graph $G$. Then $\E{\induz{G}{C}}$ is a $\lb k+1\rb$-strong anti-matching in $G$.
\end{lemma}

\begin{proof}
Let $e,e'\in\E{\induz{G}{C}}$ with $e=xy$ and $e'=x'y'$, then it holds w.l.o.g. $\distg{G}{x}{x'}\leq k$ and therefore there exists a path from $x$ to $x'$ in $G$ using at most $k$ edges. This path corresponds to a path from $e$ to $e'$ in $\lineg{G}$ using at most $k$ vertices and therefore at most $k+1$ edges. Hence $\distlg{G}{e}{e'}\leq k+1$.
\end{proof}

One could conjecture the existence of some tighter bound, i.e. the correspondence of a $k$-strong clique to a $k$-strong anti-matching and the other way around. The following examples will show that generally no such thing exists.

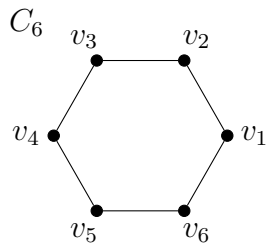
\begin{example}~

\begin{figure}[!h]
\begin{center}
\begin{tikzpicture}

\node (ascenter) [inner sep=1.5pt] {};

\node (aslabel) [inner sep=1.5pt,position=135:1.7cm from ascenter] {$C_6$};

\node (as1) [inner sep=1.5pt,position=0:1cm from ascenter,draw,circle,fill] {};
\node (as2) [inner sep=1.5pt,position=60:1cm from ascenter,draw,circle,fill] {};
\node (as3) [inner sep=1.5pt,position=120:1cm from ascenter,draw,circle,fill] {};
\node (as4) [inner sep=1.5pt,position=180:1cm from ascenter,draw,circle,fill] {};
\node (as5) [inner sep=1.5pt,position=240:1cm from ascenter,draw,circle,fill] {};
\node (as6) [inner sep=1.5pt,position=300:1cm from ascenter,draw,circle,fill] {};

\node (label12) [inner sep=1.5pt,position=0:1.2cm from ascenter] {$v_1$};
\node (label23) [inner sep=1.5pt,position=60:1.2cm from ascenter] {$v_2$};
\node (label34) [inner sep=1.5pt,position=120:1.2cm from ascenter] {$v_3$};
\node (label45) [inner sep=1.5pt,position=180:1.2cm from ascenter] {$v_4$};
\node (label56) [inner sep=1.5pt,position=240:1.2cm from ascenter] {$v_5$};
\node (label61) [inner sep=1.5pt,position=300:1.2cm from ascenter] {$v_6$};

\path
(as1) edge (as2)
(as2) edge (as3)
(as3) edge (as4)
(as4) edge (as5)
(as5) edge (as6)
(as1) edge (as6);

\end{tikzpicture}
\end{center}
\end{figure} 

The edges in $A=\set{\set{v_1,v_2}, \set{v_3,v_4}, \set{v_5,v_6}}$ form a maximum $2$-strong anti-matching in $C_6$, simultaneously $\V{C_6}=\bigcup_{e\in A}e$ forms a $3$-strong clique. This is the exact result which follows from Lemma \ref{lemma2.1} and in fact $\E{C_6}$ forms a $3$-strong anti-matching.

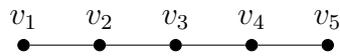
\begin{figure}[!h]
\begin{center}
\begin{tikzpicture}

\node (ascenter) [inner sep=1.5pt] {};

\node (aslabel) [inner sep=1.5pt,position=135:1.2cm from ascenter] {$P_4$};

\node (as1) [inner sep=1.5pt,position=0:0cm from ascenter,draw,circle,fill] {};
\node (as2) [inner sep=1.5pt,position=0:1cm from ascenter,draw,circle,fill] {};
\node (as3) [inner sep=1.5pt,position=0:2cm from ascenter,draw,circle,fill] {};
\node (as4) [inner sep=1.5pt,position=0:3cm from ascenter,draw,circle,fill] {};
\node (as5) [inner sep=1.5pt,position=0:4cm from ascenter,draw,circle,fill] {};

\node (label1) [inner sep=1.5pt,position=90:0.07cm from as1] {$v_1$};
\node (label2) [inner sep=1.5pt,position=90:0.07cm from as2] {$v_2$};
\node (label3) [inner sep=1.5pt,position=90:0.07cm from as3] {$v_3$};
\node (label4) [inner sep=1.5pt,position=90:0.07cm from as4] {$v_4$};
\node (label5) [inner sep=1.5pt,position=90:0.07cm from as5] {$v_5$};

\node (shadow) [inner sep=1.5pt,position=45:1.2cm from as5] {};
\node (shadoww) [inner sep=1.5pt,position=0:0.5cm from shadow] {};

\path
(as1) edge (as2)
(as2) edge (as3)
(as3) edge (as4)
(as4) edge (as5);

\end{tikzpicture}
\end{center}
\end{figure} 

In a path of length $4$ the two outer vertices are in distance exactly $4$, hence $P_4$ is a $4$-strong clique. The line graph again is a path, but shortened by $1$, therefore $P_4$ is a $3$-strong anti-matching.\\
The graph $P_4$ shows that we cannot improve Lemma \ref{lemma2.1}.\\
The other way around, a $k$-strong clique which also is a $\lb k+1\rb$-strong anti-matching is hardly that obvious.

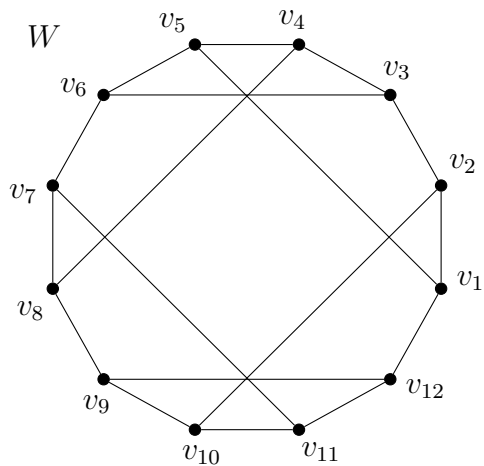
\begin{figure}[!h]
\begin{center}
\begin{tikzpicture}

\node (ascenter) [inner sep=1.5pt] {};

\node (aslabel) [inner sep=1.5pt,position=135:3.4cm from ascenter] {$W$};

\node (as1) [inner sep=1.5pt,position=345:2.5cm from ascenter,draw,circle,fill] {};
\node (as2) [inner sep=1.5pt,position=15:2.5cm from ascenter,draw,circle,fill] {};
\node (as3) [inner sep=1.5pt,position=45:2.5cm from ascenter,draw,circle,fill] {};
\node (as4) [inner sep=1.5pt,position=75:2.5cm from ascenter,draw,circle,fill] {};
\node (as5) [inner sep=1.5pt,position=105:2.5cm from ascenter,draw,circle,fill] {};
\node (as6) [inner sep=1.5pt,position=135:2.5cm from ascenter,draw,circle,fill] {};
\node (as7) [inner sep=1.5pt,position=165:2.5cm from ascenter,draw,circle,fill] {};
\node (as8) [inner sep=1.5pt,position=195:2.5cm from ascenter,draw,circle,fill] {};
\node (as9) [inner sep=1.5pt,position=225:2.5cm from ascenter,draw,circle,fill] {};
\node (as10) [inner sep=1.5pt,position=255:2.5cm from ascenter,draw,circle,fill] {};
\node (as11) [inner sep=1.5pt,position=285:2.5cm from ascenter,draw,circle,fill] {};
\node (as12) [inner sep=1.5pt,position=315:2.5cm from ascenter,draw,circle,fill] {};

\node (label1) [inner sep=1.5pt,position=15:0.07cm from as1] {$v_1$};
\node (label2) [inner sep=1.5pt,position=45:0.07cm from as2] {$v_2$};
\node (label3) [inner sep=1.5pt,position=75:0.07cm from as3] {$v_3$};
\node (label4) [inner sep=1.5pt,position=105:0.07cm from as4] {$v_4$};
\node (label5) [inner sep=1.5pt,position=135:0.07cm from as5] {$v_5$};
\node (label6) [inner sep=1.5pt,position=165:0.07cm from as6] {$v_6$};
\node (label7) [inner sep=1.5pt,position=195:0.07cm from as7] {$v_7$};
\node (label8) [inner sep=1.5pt,position=225:0.07cm from as8] {$v_8$};
\node (label9) [inner sep=1.5pt,position=255:0.07cm from as9] {$v_9$};
\node (label10) [inner sep=1.5pt,position=285:0.07cm from as10] {$v_{10}$};
\node (label11) [inner sep=1.5pt,position=315:0.07cm from as11] {$v_{11}$};
\node (label12) [inner sep=1.5pt,position=345:0.07cm from as12] {$v_{12}$};

\path
(as1) edge (as2)
(as2) edge (as3)
(as3) edge (as4)
(as4) edge (as5)
(as5) edge (as6)
(as6) edge (as7)
(as7) edge (as8)
(as8) edge (as9)
(as9) edge (as10)
(as10) edge (as11)
(as11) edge (as12)
(as12) edge (as1)
;

\path
(as1) edge (as5)
(as2) edge (as10)
(as3) edge (as6)
(as4) edge (as8)
(as7) edge (as11)
(as9) edge (as12)
;

\end{tikzpicture}
\end{center}
\end{figure} 

In the graph $W$ all pairs of vertices $x,y\in\V{W_4}$ satisfy $\distg{W_4}{x}{y}\leq 3$, hence $W$ forms a $3$-strong clique. On the other hand the edges $v_1v_2$ and $v_7v_8$ are within a distance of $4$, therefore $W$ forms a $4$-strong anti-matching and hence Lemma \ref{lemma2.2} too cannot be improved.
\end{example}
Several other extensions of the ordinary chromatic number have been proposed and studied in the past and some of them, or better, the perspective proposed by them has some interesting impact on the relation between strong vertex and strong edge coloring.\\
One of those concepts is the {\em acyclic coloring} (see \cite{borodin1979acyclic}, \cite{alon1996acyclic} or \cite{hou2015acyclic} for some details) which is the minimum number of colors in a proper vertex coloring such that the vertices of any two colors induce an acyclic graph. This concept has been transported to the realm of edge colorings, where an {\em acyclic edge coloring} is a proper edge coloring for which the edges of any cycle use at least three different colors (see \cite{alon2001acyclic}, \cite{alon2002algorithmic} and \cite{venkateswarlu2016acyclic} for details). This particular parameter seems to even have some important application in chemistry and might be interesting simply due to this fact (see \cite{rajasingh2015new}).\\
If we investigate the $2$-strong vertex coloring or the $2$-strong edge coloring, we soon stumble upon some very basic principals in terms of lower bounds. For a $2$-strong vertex coloring the vertices of a path of length $2$ may not have the same color, thus any cycle with at least three vertices has at least three different colors assigned to its vertices. Similar observations can be made for $2$-strong edge colorings and so it becomes clear that $\stronk{2}{G}$ and $\stronki{2}{G}$ pose upper bounds on the {\em acyclic chromatic number} and the {\em acyclic chromatic index}.\\
Another interesting concept of vertex colorings is the {\em star coloring}, which is a stronger version of the acyclic coloring (see \cite{vince1988star}, \cite{abbott1993star} and \cite{venkatesan2015star} for details). Here now it is required that the vertices of any two color classes of a star coloring induce a so called {\em star forest}.\\
Again we can draw some relations to the $2$-strong coloring, which poses as an upper bound on the {\em star chromatic number} as well, since, as we will see next, any two color classes of a proper $2$-strong coloring induce a $2$-strong matching, which is a special case of a star forest.
\begin{lemma}\label{lemma2.3}
Let $G$ be a graph and $F$ a $2$-strong coloring of $G$, then the vertices of any two color classes $F_i$ and $F_j$ of $F$ induce a $2$-strong matching of $G$.	
\end{lemma}
\begin{proof}
Suppose $F_i$ and $F_j$ do not induce a $2$-strong matching of $G$, then there must exist a path of length $2$ in $\induz{G}{F_i\cup F_j}$, say $P$. Let $P=xyz$, then no two of the three vertices may have the same color, thus we need at least three colors for a proper $2$-strong coloring of $\induz{G}{F_i\cup F_j}$, a contradiction.	
\end{proof}
So, if we assign the tuple of the two different colors under the given $2$-strong coloring $F$ as a color to every edge, we obtain a proper $2$-strong edge coloring of the same graph.
\begin{corollary}\label{cor2.1}
Let $G$ be a graph, the following inequality holds:
\begin{align*}
\stronki{2}{G}\leq\binom{\stronk{2}{G}}{2}.
\end{align*}	
\end{corollary}
A somewhat similar, but a bit weaker, result is true if we swap vertex and edge coloring and assign any pair of two colors of edges incident to a vertex in $G$ as its color. Note that in this case the coloring of the vertex is not clear, since a vertex may be incident to more than two edges. But since the following holds for any two color classes, any pair of such colors may be chosen.
\begin{lemma}\label{lemma2.4}
Let $G$ be a graph and $F$ a $2$ strong edge coloring of $G$, then the vertices of degree $2$ in the induced subgraph of the vertices belonging to the edges of any two color classes $F_i$ and $F_j$ of $F$ form a stable set.	
\end{lemma}
\begin{proof}
Suppose in the graph induced by the vertices of $F_i$ and $F_j$, call it $G_{ij}$, are two vertices $x$ and $y$ of degree $2$ with $\distg{G}{x}{y}=1$. If they are adjacent, they are adjacent in $G_{ij}$ too and with both of them having a degree of two we can find a path of length $3$ in $G_{ij}$. Now there are at least three colors necessary to color the edges of a path of length $3$ in a $2$-strong edge coloring, thus this is a contradiction to $F$ being a proper $2$-strong edge coloring. 
\end{proof}
\begin{corollary}\label{cor2.2}
Let $G$ be a graph, the following inequality holds:
\begin{align*}
\fkt{\chi}{G}\leq\binom{\stronki{2}{G}}{2}.
\end{align*}
\end{corollary}
\vspace{-1.5mm}
The following example shows that in general the first bound is tight. For the second bound we give an example with a big gap, that shows that the actual error made by the procedure proposed above can be much smaller. We are confident that those bounds can easily be extended to general $k$-strong colorings.
\begin{example}~
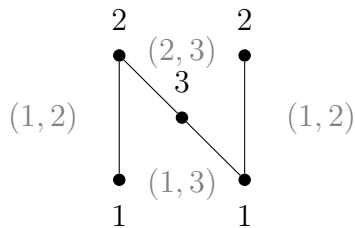
\begin{figure}[H]
	\begin{center}
		\begin{tikzpicture}
		
		\node (ascenter) [inner sep=1.5pt] {};
		
		\node (v1) [inner sep=1.5pt,position=225:1cm from ascenter,draw,circle,fill] {};
		\node (v2) [inner sep=1.5pt,position=135:1cm from ascenter,draw,circle,fill] {};
		\node (v3) [inner sep=1.5pt,draw,circle,fill] {};
		\node (v4) [inner sep=1.5pt,position=315:1cm from ascenter,draw,circle,fill] {};
		\node (v5) [inner sep=1.5pt,position=45:1cm from ascenter,draw,circle,fill] {};
		
		\node (e1) [inner sep=0pt,position=180:1cm from ascenter] {};
		\node (e2) [inner sep=0pt,position=135:5mm from ascenter] {};
		\node (e3) [inner sep=0pt,position=315:5mm from ascenter] {};
		\node (e4) [inner sep=0pt,position=0:1cm from ascenter] {};
				
		\node (lv1) [inner sep=1.5pt,position=270:2mm from v1] {$1$};
		\node (lv2) [inner sep=1.5pt,position=90:2mm from v2] {$2$};
		\node (lv3) [inner sep=1.5pt,position=90:2mm from v3] {$3$};
		\node (lv4) [inner sep=1.5pt,position=270:2mm from v4] {$1$};
		\node (lv5) [inner sep=1.5pt,position=90:2mm from v5] {$2$};

		\node (le1) [inner sep=1.5pt,position=180:2mm from e1] {\color{gray} $\lb 1,2\rb$};
		\node (le2) [inner sep=1.5pt,position=45:2mm from e2] {\color{gray} $\lb 2,3\rb$};
		\node (le3) [inner sep=1.5pt,position=225:2mm from e3] {\color{gray} $\lb 1,3\rb$};
		\node (le4) [inner sep=1.5pt,position=0:2mm from e4] {\color{gray} $\lb 1,2\rb$};

		\path
		(v1) edge (v2)
		(v2) edge (v3)
		(v3) edge (v4)
		(v4) edge (v5);
		
		\end{tikzpicture}
	\end{center}
	\caption{A $2$-strong edge coloring constructed from a $2$-strong coloring.}
	\label{fig2.3}
\end{figure}
\vspace{-3mm}
For this graph $G$ we have $\stronk{2}{G}=3$ and $\stronki{2}{G}=3=\binom{3}{2}$ and thus the bound of Corollary \autoref{cor2.2} is tight for $G$.	
\end{example}
\begin{example}~
	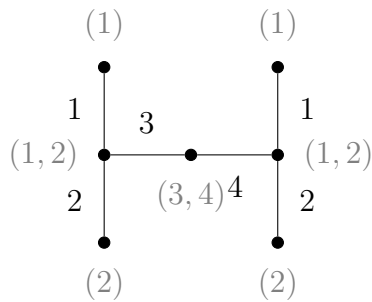
\begin{figure}[H]
		\begin{center}
			\begin{tikzpicture}
			
			\node (ascenter) [inner sep=1.5pt] {};
			
			\node (v2) [inner sep=1.5pt,position=180:1cm from ascenter,draw,circle,fill] {};
			\node (v1) [inner sep=1.5pt,position=90:1cm from v2,draw,circle,fill] {};
			\node (v3) [inner sep=1.5pt,position=270:1cm from v2,draw,circle,fill] {};
			\node (v4) [inner sep=1.5pt,draw,circle,fill] {};
			\node (v6) [inner sep=1.5pt,position=0:1cm from ascenter,draw,circle,fill] {};
			\node (v5) [inner sep=1.5pt,position=90:1cm from v6,draw,circle,fill] {};
			\node (v7) [inner sep=1.5pt,position=270:1cm from v6,draw,circle,fill] {};
			
			\node (e1) [inner sep=0pt,position=90:5mm from v2] {};
			\node (e2) [inner sep=0pt,position=270:5mm from v2] {};
			\node (e3) [inner sep=0pt,position=180:5mm from ascenter] {};
			\node (e4) [inner sep=0pt,position=0:5mm from ascenter] {};
			\node (e5) [inner sep=0pt,position=90:5mm from v6] {};
			\node (e6) [inner sep=0pt,position=270:5mm from v6] {};
			
			\node (lv1) [inner sep=1.5pt,position=90:2mm from v1] {\color{gray} $\lb 1\rb$};
			\node (lv2) [inner sep=1.5pt,position=180:2mm from v2] {\color{gray} $\lb 1,2\rb$};
			\node (lv3) [inner sep=1.5pt,position=270:2mm from v3] {\color{gray} $\lb 2\rb$};
			\node (lv4) [inner sep=1.5pt,position=270:2mm from v4] {\color{gray} $\lb 3,4\rb$};
			\node (lv5) [inner sep=1.5pt,position=90:2mm from v5] {\color{gray} $\lb 1\rb$};
			\node (lv6) [inner sep=1.5pt,position=0:2mm from v6] {\color{gray} $\lb 1,2\rb$};
			\node (lv7) [inner sep=1.5pt,position=270:2mm from v7] {\color{gray} $\lb 2\rb$};
						
			\node (le1) [inner sep=1.5pt,position=180:2mm from e1] {$1$};
			\node (le2) [inner sep=1.5pt,position=180:2mm from e2] {$2$};
			\node (le3) [inner sep=1.5pt,position=90:2mm from e3] {$3$};
			\node (le4) [inner sep=1.5pt,position=270:2mm from e4] {$4$};
			\node (le5) [inner sep=1.5pt,position=0:2mm from e5] {$1$};
			\node (le6) [inner sep=1.5pt,position=0:2mm from e6] {$2$};
						
			\path
			(v1) edge (v2)
			(v2) edge (v4)
			(v3) edge (v2)
			(v4) edge (v6)
			(v5) edge (v6)
			(v7) edge (v6);
			
			\end{tikzpicture}
		\end{center}
		\caption{A ordinary vertex coloring constructed from a $2$-strong edge coloring.}
		\label{fig2.4}
	\end{figure}
	\vspace{-3mm}
For this graph $G$ we have $\stronki{2}{G}=4$ and $\fkt{\chi}{G}=2$, but our procedure assigns $4$ different colors, while the bound by Corollary \autoref{cor2.2} gives us a value of $\binom{4}{2}=6$.\\
This example also shows that the procedure used in the proof of Lemma \autoref{lemma2.4} by far is not optimal and why we did not obtain a bound on the $2$-strong chromatic number. But it suggests that we could obtain such a bound from the same structure, if we changed it a little, since changing one of the two $\lb 1,2\rb$ labels to one of its alternatives (resp. $\lb 1,3\rb$ or $\lb 1,4\rb$) would result in a proper $2$-strong vertex coloring in $5$ colors where just $4$ are needed.	
\end{example}
As we were able to transfer the concepts of coloring the vertices of a graph to the edges while preserving clique-like structures and stable sets, we can give an analog definition of perfection in terms of edge coloring. The concept of perfection of graphs regarding strong edge colorings was first introduced by Chen and Chang (\cite{ChenChang2014perfectionstrongedge}).
\begin{definition}[$k$-strong am-Perfection]
A graph $G$ is called {\em $k$-strong am-perfect} if for all edge induced subgraphs $H\subseteq G$ holds
\begin{align*}
\kam{k}{H}=\stronki{k}{H}.
\end{align*} 
\end{definition}

\chapter{Powers of Graphs}


In this chapter we will give a brief summary of results regarding upper bounds on the chromatic number of graph powers, we then will have a quick look into some special graph classes which are closed under the taking of powers and some share the attribute of being chordal, which leads us to the characterization of chordal graphs closed under the taking of powers by Laskar and Shier.\\
We then proceed to general powers of chordal graphs, where we will see some very exciting results in terms of creating cycles through taking the power of graphs. This leads to a characterization of graphs whose squares become chordal and even a first attempt to find graphs whose squares are perfect.\\
\\
Distance colorings, or in our notation strong colorings, were also studied under another name. Instead of just defining some specific coloring with restrictions in the distance between vertices one could add more edges to the graph connecting those vertices within a given distance to each other. This produces a new graph whose regular chromatic number is the same as the $k$-strong chromatic number of the original graph.\\
This is known as taking the $k$-th power of a graph and has been studied very intense not only in terms of colorings.

\begin{definition}[Power of Graphs]\label{def3.1}
The {\em $k$-th power} $G^k$ of a graph $G=\lb V,E\rb$ is given by the vertex set $\V{G^k}=V$ and the edge set $\E{G^k}=\condset{vw}{v,w\in V,~\distg{G}{v}{w}\leq k}$. 
\end{definition}

As mentioned above the following relations hold for the $k$-th power of a graph $G$:
\begin{enumerate}[i)]
\item $\kcl{k}{G}=\fkt{\omega}{G^k}$

\item $\kst{k}{G}=\fkt{\alpha}{G^k}$

\item $\stronk{k}{G}=\fkt{\chi}{G^k}$

\item $\kcov{k}{G}=\fkt{\overline{\chi}}{G^k}$ 
\end{enumerate}

There have been some attempts on general bounds of the chromatic number of graph powers. Especially the case $k=2$, the so called square of a graph, has experienced a lot of attention. 

\section{Upper Bounds on the Chromatic Number of Graph Powers}

We will begin here with a very basic observation in terms of the maximum degree $\Delta$ and then explore more sophisticated bounds that make use of the girth, or of properties of special graph classes. A generalized version of Theorem \autoref{thm2.1} can be obtained via a simple induction.

\begin{lemma}\label{lemma3.1a}
For any graph $G$ and positive integer $k\in\N$ it holds
\begin{align*}
\stronk{k}{G}\leq \fkt{\Delta}{G}^k+1.
\end{align*}
\end{lemma}

\begin{proof}
With $\stronk{k}{G}=\fkt{\chi}{G^k}$ we can make use of Theorem \autoref{thm2.1} and reduce the proof to showing $\fkt{\Delta}{G^k}\leq \fkt{\Delta}{G}^k$. Now let $G$ be a graph with maximum degree $\Delta$ and $v\in\V{G}$ with $\fkt{\deg_G}{v}=\Delta$. By definition we have $\fkt{\deg_G}{x}\leq \Delta$ for all $x\in\knb{i}{v}$ with $i\in\set{1,\dots,k}$. We get
\begin{align*}
\fkt{\deg_{G^k}}{v}\leq& \fkt{\Delta}{G}+\fkt{\Delta}{G}\lb \fkt{\Delta}{G}-1\rb+\fkt{\Delta}{G}\lb \fkt{\Delta}{G}-1\rb^2+\dots+\fkt{\Delta}{G}\lb \fkt{\Delta}{G}-1\rb^{k-1}\\
\leq& \fkt{\Delta}{G}^k,
\end{align*}
by proving the second inequality by induction over $k$.
\end{proof}

The natural upper bound is even better if not compressed into $\fkt{\Delta}{G}^k$. For graphs with $\fkt{\Delta}{G}>1$ the following transformation can be applied.
\begin{align*}
\fkt{\Delta}{G}\sum_{i=1}^{k}\lb \fkt{\Delta}{G}-1\rb^{i-1}=\frac{\fkt{\Delta}{G}-\fkt{\Delta}{G}^{k+1}}{1-\fkt{\Delta}{G}}
\end{align*}
 
\begin{theorem}\label{thm3.1}
For any graph $G$ with $\fkt{\Delta}{G}>1$ and any positive integer $k\in\N$ it holds
\begin{align*}
\stronk{k}{G}\leq\frac{\fkt{\Delta}{G}-\fkt{\Delta}{G}^{k+1}}{1-\fkt{\Delta}{G}}+1.
\end{align*}
\end{theorem} 

Alon and Mohar were able to improve this for $k=2$ on graphs with bounded girth. The cases with $3\leq \fkt{\operatorname{girth}}{G}\leq 6$ and $\fkt{\operatorname{girth}}{G}\geq 7$ are considered separate.

\begin{theorem}[Alon, Mohar 2002 \cite{AlonMohar2002graphpowers}]\label{thm3.2}
Let $G$ be a graph with $3\geq \fkt{\operatorname{girth}}{G}\leq 6$, then it holds
\begin{align*}
\stronk{2}{G}\leq\lb 1+\ord{1}\rb\fkt{\Delta}{G}^2.
\end{align*}
\end{theorem}

\begin{theorem}[Alon, Mohar 2002 \cite{AlonMohar2002graphpowers}]\label{thm3.3}
Let $G$ be a graph with $\fkt{\operatorname{girth}}{G}\geq 7$, then it holds
\begin{align*}
\stronk{2}{G}\leq\fkt{\Theta}{\frac{\fkt{\Delta}{G}^2}{\log \fkt{\Delta}{G}}}.
\end{align*}
\end{theorem}

To be precise they were able to show a much more interesting fact for this upper bound of the chromatic number of squared graphs. With $\fkt{f_2}{\Delta,g}$ they defined some function for the maximum possible value for $\stronk{2}{G}$ for all graphs with maximum degree $\Delta$ and girth $g$ and proved the following result by using the probabilistic method. 

\begin{theorem}[Alon, Mohar 2002 \cite{AlonMohar2002graphpowers}]\label{thm3.4}
\begin{enumerate}[i)]
\item There exists a function $\fkt{\epsilon}{\Delta}$ that tends to $0$ as $\Delta$ tends to infinity such that for all $g\leq 6$
\begin{align*}
\lb 1-\fkt{\epsilon}{\Delta}\rb\Delta^2\leq\fkt{f_2}{\Delta,g}\leq\Delta^2+1.
\end{align*}

\item There are absolute positive constants $c_1$, $c_2$ such that for every $\Delta\geq 2$ and every $g\geq7$
\begin{align*}
c_1\frac{\Delta^2}{\log \Delta}\leq\fkt{f_2}{\Delta,g}\leq c_2\frac{\Delta^2}{\log \Delta}.
\end{align*}
\end{enumerate}
\end{theorem}

They were able to generalize those results to higher powers by extending their random constructions used in the proof of Theorem \autoref{thm3.4}, for $k\geq 3$ they obtained the following bounds for the generalized upper bound function $\fkt{f_k}{\Delta,g}$.

\begin{theorem}[Alon, Mohar 2002 \cite{AlonMohar2002graphpowers}]
There exists an absolute constant $c>0$ such that for all integers $k\geq 1$, $\Delta\geq 2$ and $g\geq 3k+1$
\begin{align*}
\fkt{f_k}{\Delta,g}\leq \frac{c}{k}\frac{\Delta^k}{\log \Delta}.
\end{align*}
For every integer $k\geq 1$ there exists a positive number $b_k$ such that for every $\Delta\geq 2$ and $g\geq 3$
\begin{align*}
\fkt{f_k}{\Delta,g}\geq b_k\frac{\Delta^k}{\log \Delta}.
\end{align*}
\end{theorem}

Coloring the vertices of planar graphs has a long tradition in chromatic graph theory. In 1977 Wegner (see \cite{Wegner1977diameterandcoloring}) first investigated the chromatic number of squared planar graphs. He was able to prove the simple bound $\stronk{2}{G}\leq 8$ for every planar graph $G$ with maximum degree $\fkt{\Delta}{G}=3$. He also conjectured that this bound could be improved to $7$, which was confirmed by Thomassen (see \cite{Thomassen2001tuttecycles}) in 2001. Wegner gave a conjecture for bounds of the chromatic number of squared planar graphs with higher maximum degrees too, along with examples that prove these bounds to be tight if his conjecture is correct.
\begin{conjecture}[Wegner. 1977 \cite{Wegner1977diameterandcoloring}]\label{conj3.1}
Let $G$ be a planar graph. Then
\begin{align*}
\stronk{2}{G}\leq\stueckfkt{\fkt{\Delta}{G}+5}{\text{if}~ 4\leq\fkt{\Delta}{G}\leq 7,}{\abr{\frac{3\fkt{\Delta}{G}}{2}}+1}{\text{if}~\fkt{\Delta}{G}\geq 8.}
\end{align*}
\end{conjecture}

While there were many attempts to prove Wegner's conjecture it remains open. So far the best bounds regarding this specific problem were obtained by Molloy and Salavatipour (see \cite{MolloySalavatipour2005squareplanar}) in 2005, they gave both a general and an asymptotic bound.

\begin{theorem}[Molloy, Salavatipour. 2005 \cite{MolloySalavatipour2005squareplanar}]
Let $G$ be a planar graph. Then
\begin{align*}
\stronk{2}{G}\leq \aufr{\frac{5}{3}\fkt{\Delta}{G}}+78.
\end{align*}
\end{theorem}

\begin{theorem}[Molloy, Salavatipour. 2005 \cite{MolloySalavatipour2005squareplanar}]
Let $G$ be a planar graph with $\fkt{\Delta}{G}\geq 241$. Then
\begin{align*}
\stronk{2}{G}\leq \aufr{\frac{5}{3}\fkt{\Delta}{G}}+25.
\end{align*}
\end{theorem}

By forbidding the complete graph on $4$ vertices as a minor - note that the $K_5$ already is forbidden in planar graphs and with $K_{3,3}$ also containing the $K_4$ as a minor those graphs are a true subset of planar graphs - Lih, Wang and Zhu (see \cite{LihWangZhu2003squareminorfree}) were able to further improve these bounds for this subclass.

\begin{theorem}[Lih, Wang, Zhu. 2003 \cite{LihWangZhu2003squareminorfree}]\label{thm3.5}
Let $G$ be a $K_4$-minor free graph. Then
\begin{align*}
\stronk{2}{G}=\stueckfkt{\fkt{\Delta}{G}+3}{\text{if}~2\leq\fkt{\Delta}{G}\leq 3,}{\abr{\frac{3\fkt{\Delta}{G}}{2}}+1}{\text{if}~\fkt{\Delta}{G}\geq 4.}
\end{align*}
\end{theorem}

\section{Graph Classes Closed Under Taking Powers}

Another approach to the problem of chromatic numbers of graph powers is the question which families of graphs are closed under the taking of powers. 

\begin{definition}
Let $\mathscr{F}$ be a family of graphs, $\mathscr{F}$ is called {\em closed under taking powers} if $G^{k}\in\mathscr{F}$ implies $G^{k+1}\in\mathscr{F}$ for all $k\in\N$. 
\end{definition}

The connection between such a family and the chromatic number of graph powers lies in the the fact that a great number of such families, which are known to be closed under taking powers, that are subclasses of chordal graphs and therefore perfect.\\
The first known result in this type of studies in graph powers was obtained by Lubiw (see \cite{lubiw1982stronglychordal}). His family is one of very important chordal graphs, the so called strongly chordal graphs which can be characterized in terms of induced subgraphs.\\
Those subgraphs are a very important special case of a graph class which we will study in the next section.

\begin{definition}[Sun]
A {\em sun} is a graph $S_n=\lb U\cup W,E\rb$ with $U=\set{u_1,\dots,u_n}$ and $W=\set{w_1,\dots,w_n}$ such that $W$ induces a complete graph and $U$ is a stable set and furthermore $u_iv_j\in E$ if and only if $j=i$ or $j=i+1\lb\!\!\!\mod n \rb$.
\end{definition} 

\begin{figure}[H]
\begin{center}
\begin{tikzpicture}

\node (center) [inner sep=1.5pt] {};

\node (label) [inner sep=1.5pt,position=135:1.7cm from center] {};

\node (u1) [inner sep=1.5pt,position=90:1.6cm from center,draw,circle,fill] {};
\node (u2) [inner sep=1.5pt,position=162:1.6cm from center,draw,circle,fill] {};
\node (u3) [inner sep=1.5pt,position=234:1.6cm from center,draw,circle,fill] {};
\node (u4) [inner sep=1.5pt,position=306:1.6cm from center,draw,circle,fill] {};
\node (u5) [inner sep=1.5pt,position=18:1.6cm from center,draw,circle,fill] {};

\node (v1) [inner sep=1.5pt,position=126:0.8cm from center,draw,circle] {};
\node (v2) [inner sep=1.5pt,position=198:0.8cm from center,draw,circle] {};
\node (v3) [inner sep=1.5pt,position=270:0.8cm from center,draw,circle] {};
\node (v4) [inner sep=1.5pt,position=342:0.8cm from center,draw,circle] {};
\node (v5) [inner sep=1.5pt,position=54:0.8cm from center,draw,circle] {};

\node (lu1) [position=90:0.07cm from u1] {$u_1$};
\node (lu2) [position=162:0.07cm from u2] {$u_2$};
\node (lu3) [position=234:0.07cm from u3] {$u_3$};
\node (lu4) [position=306:0.07cm from u4] {$u_4$};
\node (lu5) [position=18:0.07cm from u5] {$u_5$};

\node (lv1) [position=126:0.07cm from v1] {$w_1$};
\node (lv2) [position=198:0.07cm from v2] {$w_2$};
\node (lv3) [position=270:0.07cm from v3] {$w_3$};
\node (lv4) [position=342:0.07cm from v4] {$w_4$};
\node (lv5) [position=54:0.07cm from v5] {$w_5$};

\path
(u1) edge (v1)
	 edge (v5)
(u2) edge (v1)
	 edge (v2)
(u3) edge (v2)
	 edge (v3)
(u4) edge (v3)
	 edge (v4)
(u5) edge (v4)
	 edge (v5)
;

\path[bend right] 
(v1) edge (v2)
(v2) edge (v3)
(v3) edge (v4)
(v4) edge (v5)
(v5) edge (v1)
;

\path
(v1) edge (v3)
(v3) edge (v5)
(v5) edge (v2)
(v2) edge (v4)
(v4) edge (v1)
;

\end{tikzpicture}
\caption{The sun with $n=5$.}
\label{fig3.1}
\end{center}
\end{figure}
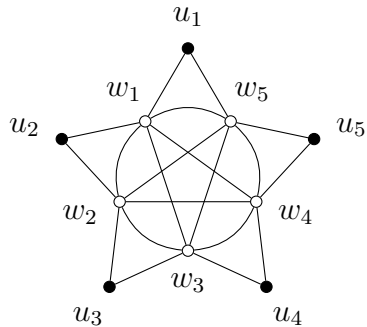 

\begin{theorem}[Farber. 1983 \cite{Farber1983stronglychordal}]\label{thm3.6}
A chordal graph is strongly chordal if and only if it does not contain a sun.
\end{theorem}

Note that Farber called those forbidden subgraphs not suns but trampolines, later on the term sun became more popular.

\begin{theorem}[Lubiw. 1982 \cite{lubiw1982stronglychordal}]\label{thm3.7}
If $G^k$ is a strongly chordal graph, so is $G^{k+1}$ for all $k\in\N$.
\end{theorem}

There are many known characterizations of strongly chordal graphs and two of them are in terms of so called totally balanced matrices.
Totally-balanced matrices were used by Lovász (see \cite{lovasz2014combinatorial}) and are those matrices not containing the incidence matrix of a cycle of length at least three. Thus a hypergraph is totally balanced if and only if its incidence matrix is totally balanced. Hence a graph $G$ is totally balanced if and only if $G$ is a forest.\\
Although strongly chordal graphs in general may contain cycles there is a strong relation between them and totally balanced matrices.

\begin{definition}[Neighborhood Matrix]
The {\em neighborhood matrix} $\fkt{M}{G}$ of a graph $G$ on $n$ vertices $v_1,\dots,v_n$ is the $n\times n$ matrix with entry $\lb i,j\rb$ equals $1$ if and only if $v_i\in\nbclosed{v_j}$ and $0$ otherwise.
\end{definition}

\begin{theorem}[Lubiw. 1982 \cite{lubiw1982stronglychordal}]\label{thm3.8}
A graph $G$ is strongly chordal if and only if $\fkt{M}{G}$ is totally balanced.
\end{theorem}

\begin{definition}[Clique Matrix] 
The {\em clique matrix} $\fkt{C}{G}$ of a graph $G$ on $n$ vertices $v_1,\dots,v_n$ with maximal Cliques $C_1,\dots,C_q$ is the $q\times n$ matrix with entry $\lb i,j\rb$ equals $1$ if and only if $v_j\in C_i$ and $0$ otherwise.
\end{definition}

\begin{theorem}[Lubiw. 1982 \cite{lubiw1982stronglychordal}]\label{thm3.9}
A graph $G$ is strongly chordal if and only if $\fkt{C}{G}$ is totally balanced.
\end{theorem}

The structure of strongly chordal graphs allows a lot of problems, that remain hard to solve on general chordal graphs, to be solved efficiently. In chapter 5 we will revisit this class of graphs.

\begin{definition}[Intersection Graph]
The {\em intersection graph} of a family $\mathscr{I}$ of sets is the graph who has a vertex for every set in $\mathscr{I}$ and two distinct vertices are adjacent if and only if their corresponding sets intersect.
\end{definition}

The concept of intersection graphs plays an important role in graph theory, especially in the study of different classes of chordal graphs. Given certain structures to the sets in $\mathscr{I}$ the corresponding intersection graph often holds very strong structural properties.\\
One of the most excessively studied family of intersection graphs are the interval graphs i.e. the intersection graphs of intervals on the real line. The closure of interval graphs under the taking of powers was first shown by Raychaudhuri (see \cite{raychaudhuri1987intervallpowers}) but Chen and Chang (see \cite{ChenChang2001families}) gave a much simpler proof in terms of a characterization by Ramalingam and Rangan .

\begin{theorem}[Ramalingam, Rangan. 1988 \cite{ramalingam1988unified}]\label{thm3.10}
A graph $G$ is an interval graph if and only if it has an {\em interval ordering}, which is an ordering of $\V{G}$ into $\left[ v_1,\dots,v_n \right]$ such that
\begin{align*}
i<l<j~\text{and}~v_iv_j\in\E{G}~\Rightarrow~ v_lv_j\in\E{G}.
\end{align*}
\end{theorem}

\begin{theorem}[Raychaudhuri. 1987 \cite{raychaudhuri1987intervallpowers}]\label{thm3.11}
If $G^k$ is an interval graph, so is $G^{k+1}$ for all $k\in\N$.
\end{theorem}

\begin{proof}
Let $G^k$ be an interval graph and $\sigma=\left[v_1,\dots,v_n\right]$ an interval ordering of $G^k$. We want to show that $\sigma$ is an interval ordering for $G^{k+1}$ as well.\\
Now suppose $i<l<j$ and $v_iv_j\in\E{G^{k+1}}$, which implies $\distg{G}{v_i}{v_j}\leq k+1$. If $\distg{G}{v_i}{v_j}\leq k$, then we get $v_iv_j\in\E{G^k}$ and with $\sigma$ being an interval ordering of $G^k$ $v_lv_j\in\E{G^k}\subseteq\E{G^{k+1}}$.\\
Now suppose $\distg{G}{v_i}{v_j}=k+1$. Let $P$ be a shortest $v_i,v_j$-path in $G$ and let $v_a$ be the vertex adjacent to $v_j$ on $P$. Then, $\distg{G}{v_i}{v_a}=k$ and $\distg{G}{v_a}{v_j}=1$ hence $v_iv_a$ and $v_a,v_j$ are edges in $G^k$. If $i<l<a$, then $v_lv_a$ exists in $G^k$ by $\sigma$ and so $\distg{G}{v_j}{v_j}\leq\distg{G}{v_l}{v_a}+\distg{G}{v_a}{v_j}\leq k+1$ holds and $v_j$ is adjacent to $v_l$ in $G^{k+1}$. If $a<l<j$, then $v_lv_j$ exists in $G^k$ and therefore in $G^{k+1}$, in any case $v_l$ and $v_j$ are adjacent and $\sigma$ is an interval ordering of $G^{k+1}$, which, by Theorem \autoref{thm3.11}, makes $G^{k+1}$ an interval graph. 
\end{proof}

\begin{remark}\label{rem3.1}
Strongly chordal graphs and interval graphs are chordal.
\end{remark}

An even larger class of graphs closed under the taking of powers and including interval graphs is given by the so called cocomparability graphs. Just like interval graphs we can give a special ordering $\sigma$ of the vertices of such a graph in order to not only characterize them, but to prove the closure.

\begin{definition}[Comparability Graph] 
A {\em comparability graph} is the underlying graph of an acyclic digraph, which can be viewed as a partially ordered set (poset). I.e. a graph $G$ is a comparability graph if and only if it has a {\em transitive ordering} $\sigma=\left[ v_1,\dots,v_n\right]$ of $\V{G}$ such that
\begin{align*}
i<l<j~\text{and}~v_iv_j\in\E{G}~\Rightarrow~ v_lv_j\in\E{G}.
\end{align*}
\end{definition}

\begin{definition}[Cocomparability Graph]
A {\em cocomparability graph} is the complement of a comparability graph, i.e. it has a {\em cocomparability ordering} $\sigma=\left[ v_1,\dots,v_n\right]$ of $\V{G}$ such that
\begin{align*}
i<l<j~\text{and}~v_iv_j\in\E{G}~\Rightarrow~v_iv_l\in\E{G}~\text{or}~v_lv_j\in\E{G}.
\end{align*}
\end{definition}

Again, while originally proven by Flotow in 1995 (see \cite{flotow1995trapezoid}) Chen and Chang (see \cite{ChenChang2001families} gave a much simpler and more elegant proof.

\begin{theorem}[Flotow. 1995 \cite{flotow1995trapezoid}]
If $G^k$ is a cocomparability graph, so is $G^{k+1}$ for all $k\in\N$.
\end{theorem}

\begin{proof}
Similar to the proof for interval graphs we take some cocomparability graph $G^k$ for some $k\in\N$ and a corresponding cocomparability ordering $\sigma=\left[ v_1,\dots,v_n\right]$ and show that $\sigma$ is a cocomparability ordering of $G^{k+1}$ as well.\\
Suppose $i<l<j$ and $v_iv_j\in\E{G^{k+1}}$, hence $\distg{G}{v_i}{v_j}\leq k+1$. If $\distg{G}{v_i}{v_j}\leq k$ holds either $v_iv_l$ or $v_lv_j$ exist in $G^k$ with $\sigma$ being a cocomparability ordering and therefore at least one of those two edges exists in $G^{k+1}$.\\
So we suppose $\distg{G}{v_i}{v_j}=k+1$ and chose some vertex $v_a$ on a shortest $v_i,v_j$-path in $G$ with $\distg{G}{v_i}{v_a}=k$ and $\distg{G}{v_a}{v_j}=1$. With that $v_a$ is adjacent to $v_i$ and $v_j$ in $G^k$. If $i<l<a$, then either $\distg{G}{v_i}{v_l}\leq k$ or $\distg{G}{v_l}{v_j}\leq k$ is implied by the existence of either one of the corresponding edges in $G^k$. If $v_iv_l$ exists in $G^k$ it does so in $G^{k+1}$ too and we are done, so suppose $v_l,v_j\in\E{G^k}$. We get $\distg{G}{v_l}{v_j}\leq\distg{G}{v_l}{v_a}+\distg{G}{v_a}{v_j}\leq k+1$ and so the edge $v_lv_j$ exists in $G^{k+1}$. Thus we reach the case $a<l<j$ and either the edge $v_av_l$ or $v_lv_j$ exists in $G^k$, which implies the distance conditions $\distg{G}{v_a}{v_l}\leq k$ or $\distg{G}{v_l}{v_j}\leq k$. With $\distg{G}{v_l}{v_j}\leq k$ the edge $v_lv_j$ exists in both $G^k$ and $G^{k+1}$. And for $\distg{G}{v_a}{v_l}\leq k$ we get $\distg{G}{v_l}{v_j}\leq \distg{G}{v_a}{v_l}+\distg{G}{v_a}{v_j}\leq k+1$ and therefore the existence of $v_lv_j$ in $G^{k+1}$. Hence $\sigma$ is a cocomparability order of $G^{k+1}$ and therefore $G^{k+1}$ is a cocomparability graph.  
\end{proof}

\begin{remark}\label{rem3.2}
Cocomparability graphs are perfect.
\end{remark}

\subsection{Chordal Graphs}

In general chordal graphs are not closed under the taking of powers. Nevertheless it is possible to show a very similar property. In 1980 Laskar and Shier (see \cite{LaskarShier1980chordal}) showed that while $G^3$ and $G^5$ are chordal if $G$ itself is chordal, but in general the same does not hold for $G^2$.\\
We have already seen an example of such a graph. The sun $S_5$ in \autoref{fig3.1}. But such graphs do not necessarily need to be suns. Chordal graphs that look a lot like suns, but are none, have this property too. 

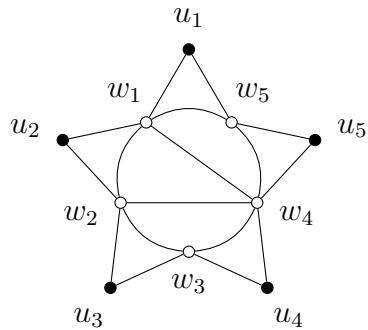
\begin{figure}[H]
\begin{center}
\begin{tikzpicture}

\node (center) [inner sep=1.5pt] {};

\node (label) [inner sep=1.5pt,position=135:1.7cm from center] {};

\node (u1) [inner sep=1.5pt,position=90:1.6cm from center,draw,circle,fill] {};
\node (u2) [inner sep=1.5pt,position=162:1.6cm from center,draw,circle,fill] {};
\node (u3) [inner sep=1.5pt,position=234:1.6cm from center,draw,circle,fill] {};
\node (u4) [inner sep=1.5pt,position=306:1.6cm from center,draw,circle,fill] {};
\node (u5) [inner sep=1.5pt,position=18:1.6cm from center,draw,circle,fill] {};

\node (v1) [inner sep=1.5pt,position=126:0.8cm from center,draw,circle] {};
\node (v2) [inner sep=1.5pt,position=198:0.8cm from center,draw,circle] {};
\node (v3) [inner sep=1.5pt,position=270:0.8cm from center,draw,circle] {};
\node (v4) [inner sep=1.5pt,position=342:0.8cm from center,draw,circle] {};
\node (v5) [inner sep=1.5pt,position=54:0.8cm from center,draw,circle] {};

\node (lu1) [position=90:0.07cm from u1] {$u_1$};
\node (lu2) [position=162:0.07cm from u2] {$u_2$};
\node (lu3) [position=234:0.07cm from u3] {$u_3$};
\node (lu4) [position=306:0.07cm from u4] {$u_4$};
\node (lu5) [position=18:0.07cm from u5] {$u_5$};

\node (lv1) [position=126:0.07cm from v1] {$w_1$};
\node (lv2) [position=198:0.07cm from v2] {$w_2$};
\node (lv3) [position=270:0.07cm from v3] {$w_3$};
\node (lv4) [position=342:0.07cm from v4] {$w_4$};
\node (lv5) [position=54:0.07cm from v5] {$w_5$};

\path
(u1) edge (v1)
	 edge (v5)
(u2) edge (v1)
	 edge (v2)
(u3) edge (v2)
	 edge (v3)
(u4) edge (v3)
	 edge (v4)
(u5) edge (v4)
	 edge (v5)
;

\path[bend right] 
(v1) edge (v2)
(v2) edge (v3)
(v3) edge (v4)
(v4) edge (v5)
(v5) edge (v1)
;

\path
(v2) edge (v4)
(v4) edge (v1)
;

\end{tikzpicture}
\caption{Another example of a chordal graph whose square is not chordal.}
\label{fig3.2}
\end{center}
\end{figure}

This led Laskar and Shier to the conjecture that every odd power of a chordal graph is chordal itself. Duchet (see \cite{Duchet1984classical}) proved an even stronger result which led to a number of so called Duchet-type results on graph powers.\\
Duchet's proof makes use of very basic tools such as walks. For chordal graphs the following lemma will be very useful.

\begin{lemma}[Duchet. 1984 \cite{Duchet1984classical}]\label{lemma3.1}
Every closed walk $v_1,\dots,v_n,v_1$ with $n\geq 2$ contains a chord $v_iv_{i+1}$ or a repetition $v_i=v_{i+1}$.
\end{lemma}

In order to prove Duchet's Theorem we need an additional definition.

\begin{definition}[Graph Modulo Subset]
Let $G$ be a graph and $V_1,\dots,V_m\subseteq\V{G}$ subsets of $\V{G}$. The graph $\fkt{G}{V_1,\dots,V_m}$ has the vertex set $\set{V_1,\dots,V_m}$ and $V_i$ and $V_j$ are adjacent if and only if $\lb\bigcup_{v\in V_i}\nb{v}\rb\cap V_j\neq\emptyset$.
\end{definition}

\begin{lemma}[Duchet. 1984 \cite{Duchet1984classical}]\label{lemma3.2}
Let $G$ be a chordal graph, if $V_1,\dots,V_m$ are connected subsets of $\V{G}$, $\fkt{G}{V_1,\dots,V_m}$ is also chordal.
\end{lemma}

\begin{proof}
Let $C=C_1\dots C_qC_1$ be a cycle with length $q\geq4$ in $\fkt{G}{V_1,\dots,V_m}$. We call a closed walk $W=w_1,\dots,w_p$ in $G$ a $C$-walk if and only if there is a decomposition of $W$ into $q$ subwalks $W_i$ where $W_i$ is a nonempty walk in $\induz{G}{C_i}$. $W_i$ is called the $i$-th component of the $C$-decomposition $W_1,\dots,W_q$.\\
Now we consider a $C$-walk $W$ with minimal length $p$, with vertices $v_1,\dots,v_p$ and $W_1,\dots,W_q$ the $C$-decomposition of $W$. Lemma \autoref{lemma3.1} implies the existence of two vertices $v_t$ and $v_{t+2}$ such that
\begin{align*}
v_t=v_{t+2}~\text{or}~v_t~\text{is adjacent to}~v_{t+2}.
\end{align*}
The minimality of $W$ implies furthermore that $v_t$ and $v_{t+2}$ are in different components, otherwise $W$ could be shortened by cutting out the vertex $v_{t+1}$ either by using the edge $v_tv_{t+2}$ or by just not using the edge $v_tv_{t+1}$ in $W$. Let $W_{\alpha}$ and $W_{\beta}$ be those different components with 
\begin{align*}
\abs{\alpha-\beta}\neq 1,~\lb \alpha,\beta\rb\neq\lb q,1\rb~\text{and}~\lb \alpha,\beta\rb\neq\lb 1,q\rb.
\end{align*}
Therefore $W_{\alpha}$ and $W_{\beta}$ are linked in $\fkt{G}{V_1,\dots,V_m}$, which results in a chord in $C$, thus $\fkt{G}{V_1,\dots,V_m}$ is chordal.
\end{proof}

\begin{theorem}[Duchet's Theorem, Duchet. 1984 \cite{Duchet1984classical}]\label{thm3.12}
Let $G$ be a graph and $k\in\N$. If $G^k$ is chordal, so is $G^{k+2}$.
\end{theorem}

\begin{proof}
Let $\V{G}=\set{v_1,\dots,v_n}$, we define $V_i\define\knb{k}{v_i}\cup\set{v_i}$ for all $i\in\set{1,\dots,n}$. Then $\fkt{G^k}{V_1,\dots,V_n}$ is isomorphic to $G^{k+2}$ and with Lemma \autoref{lemma3.2} we obtain the chordality of $G^{k+2}$.
\end{proof}

Although Laskar and Shier were not able to prove their own conjecture about odd powers of chordal graphs, they were able to use Duchet's Theorem to successfully characterize all chordal graphs whose squares are chordal and therefore described the family of chordal graphs which is closed under taking powers.\\
We begin by expanding the concept of suns, which enabled us to describe strongly chordal graphs in terms of forbidden subgraphs.

\begin{definition}[Sunflower]
A {\em sunflower} of size $n$ is a graph $S=\lb U\cup W, E \rb$ with $U=\set{u_1,\dots,u_n}$ and $W=\set{w_1,\dots,w_n}$ such that $W$ induces a chordal graph and $U$ is a stable set and furthermore $u_iw_j\in E$ if and only if $j=i$ or $j=i+1\lb \!\!\!\mod n\rb$.\\
The family of all sunflowers of size $n$ is denoted by $\mathcal{S}_n$. 
\end{definition}

\begin{definition}[Suspended Sunflower]
A sunflower $S\in\mathcal{S}_n$ contained in some graph $G$ is called a {\em suspended sunflower} in $G$ if there exists a vertex $v\notin\V{S}$, such that $v$ is adjacent to at least one pair of vertices $u_i$ and $u_j$ with $j\neq i\pm1\lb \!\!\!\mod n\rb$.
\end{definition}

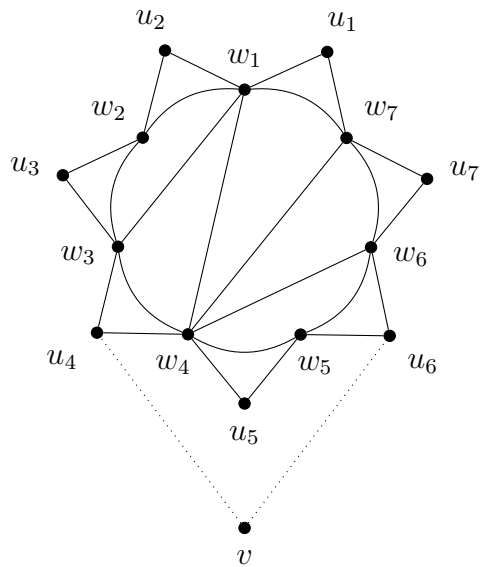
\begin{figure}[H]
\begin{center}
\begin{tikzpicture}

\node (center) [inner sep=0pt] {};

\node (1) [draw,circle,inner sep=1.5pt,fill,position=90:1.6cm from center] {};
\node (l1) [position=90:0.07cm from 1] {$w_1$};

\node (2) [draw,circle,inner sep=1.5pt,fill,position=141.4:1.6cm from center] {};
\node (l2) [position=141.4:0.07cm from 2] {$w_2$};

\node (3) [draw,circle,inner sep=1.5pt,fill,position=192.8:1.6cm from center] {};
\node (l3) [position=192:0.07cm from 3] {$w_3$};

\node (4) [draw,circle,inner sep=1.5pt,fill,position=244.2:1.6cm from center] {};
\node (l4) [position=244.2:0.07cm from 4] {$w_4$};

\node (5) [draw,circle,inner sep=1.5pt,fill,position=295.6:1.6cm from center] {};
\node (l5) [position=295.6:0.07cm from 5] {$w_5$};

\node (6) [draw,circle,inner sep=1.5pt,fill,position=347:1.6cm from center] {};
\node (l6) [position=347:0.07cm from 6] {$w_6$};

\node (7) [draw,circle,inner sep=1.5pt,fill,position=38.5:1.6cm from center] {};
\node (l7) [position=38.5:0.07cm from 7] {$w_7$};

\node (u1) [draw,circle,inner sep=1.5pt,fill,position=63.8:2.35cm from center] {};
\node (lu1) [position=63.8:0.07cm from u1] {$u_1$};

\node (u2) [draw,circle,inner sep=1.5pt,fill,position=115.2:2.35cm from center] {};
\node (lu2) [position=115.2:0.07cm from u2] {$u_2$};

\node (u3) [draw,circle,inner sep=1.5pt,fill,position=166.6:2.35cm from center] {};
\node (lu3) [position=166.6:0.07cm from u3] {$u_3$};

\node (u4) [draw,circle,inner sep=1.5pt,fill,position=218:2.35cm from center] {};
\node (lu4) [position=218:0.07cm from u4] {$u_4$};

\node (u5) [draw,circle,inner sep=1.5pt,fill,position=270:2.35cm from center] {};
\node (lu5) [position=270:0.07cm from u5] {$u_5$};

\node (u6) [draw,circle,inner sep=1.5pt,fill,position=320.8:2.35cm from center] {};
\node (lu6) [position=320.8:0.07cm from u6] {$u_6$};

\node (u7) [draw,circle,inner sep=1.5pt,fill,position=12.3:2.35cm from center] {};
\node (lu7) [position=12.3:0.07cm from u7] {$u_7$};

\path
(1) edge [bend right,bend angle=5] (2)
(2) edge [bend right,bend angle=5] (3)
(3) edge [bend right,bend angle=5] (4)
(4) edge [bend right,bend angle=5] (5)
(5) edge [bend right,bend angle=5] (6)
(6) edge [bend right,bend angle=5] (7)
(7) edge [bend right,bend angle=5] (1)
(u1) edge (1)
	 edge (7)
(u2) edge (1)
	 edge (2)
(u3) edge (2)
	 edge (3)
(u4) edge (3)
	 edge (4)
(u5) edge (4)
	 edge (5)
(u6) edge (5)
	 edge (6)
(u7) edge (6)
	 edge (7)
(3) edge (1)
(4) edge (1)
(4) edge (7)
(4) edge (6)	 
;

\node (w) [draw,circle,inner sep=1.5pt,fill,position=270:4cm from center] {};
\node (lw) [position=270:0.07 from w] {$v$};

\path
(w) edge [dotted] (u4)
(w) edge [dotted] (u6);

\end{tikzpicture}
\caption{A suspended sunflower $S\in\mathcal{S}_7$.}\label{fig3.3}
\end{center}
\end{figure}

Another very basic lemma on chordal graphs will simplify the proof.

\begin{lemma}\label{lemma3.5}
Let $C$ be a cycle of a chordal graph $G$. Then for each edge $uv\in C$, there exists a $w\in C$, such that $wuv$ is a triangle, that is the edges $uv$, $vw$ and $uw$ exist in $G$.
\end{lemma} 

\begin{lemma}[Laskar, Shier. 1980 \cite{LaskarShier1980chordal}]\label{lemma3.3}
Let $G$ be a chordal graph with $G^2$ not being chordal and let $C$ be an induced cycle of length $\geq 4$ in $G^2$. Then no edge of $C$ is an edge of $G$.
\end{lemma}

\begin{lemma}[Laskar, Shier. 1983 \cite{LaskarShier1983powercenterchordal}]\label{lemma3.4}
Let $G$ be a chordal graph with $G^2$ not being chordal, then $G$ has at least one sunflower $S\in\mathcal{S}_n$, $n\geq4$, which is not suspended.
\end{lemma}

\begin{proof}
Let $G$ be a chordal graph, $G^2$ not chordal and $C_n=\lb u_1\dots u_n\rb$, $n\geq4$, a chordless cycle in $G^2$. By Lemma \autoref{lemma3.3} no edge of $C_n$ is an edge of $G$, but every edge of the form $u_iu_{i+1\lb\!\!\! \mod n\rb}$ exists in $G^2$, thus $\distg{G}{u_i}{u_{i+1\lb \!\!\!\mod n\rb}}=2$ and there must exist some vertex $w_i$ extending the edge to a path $u_iw_iu_{i+1\lb \!\!\!\mod n\rb}$ of length $2$ in $G$. This results in a cycle $C_{2n}''=\lb u_1w_1u_2\dots w_{n-1}u_nv_n\rb$ in $G$. All the $w_i$ are distinct and disjoint from the $u_i$, otherwise $C_n$ would have a chord in $G^2$.\\
Obviously no $u_i$ is adjacent to another $u_j$, hence the $u_i$ form a stable set in $G$. Also $w_i$ is adjacent in $G$ to $u_i$ and $u_{i+1\lb \!\!\!\mod n\rb}$ and to no other $u$, otherwise $C$ would not be chordless in $G^2$.\\
Now consider the edge $u_iw_i$, by Lemma \autoref{lemma3.5} there must exist a vertex in $C_{2n}''$ adjacent in $G$ to $u_i$ and $w_i$. The only vertex that can be adjacent to both other vertices without producing a chord in $C_n$ is $w_{i-1\lb \!\!\!\mod n\rb}$. Hence $w_iw_j\in\E{G}$ for all $i\in{1,\dots,n}$ and $j=i\pm1\lb \!\!\!\mod n\rb$ and with that $Z=\lb w_1\dots w_n\rb$ is a cycle in $G$. Since $G$ is chordal $Z$ must have chords, i.e. $\induz{G}{Z}$ is chordal. Thus $W=\set{w_1,\dots,w_n}$ and $U=\set{u_1,\dots,u_n}$ together form a sunflower $S\in\mathcal{S}_n$, $n\geq4$, in $G$.\\
In addition, since $C_n$ is an induced cycle in $G^2$, no pair $u_j,u_k$ of vertices with $j\neq k\pm1\lb \!\!\!\mod n\rb$ can be adjacent in $G^2$. Hence $\distg{G}{u_j}{u_k}\geq3$. Thus the sunflower $S$ is not suspended.
\end{proof}

\begin{lemma}[Laskar, Shier. 1983 \cite{LaskarShier1983powercenterchordal}]\label{lemma3.6}
If $G$ is a chordal graph and $G^2$ is chordal too, $G$ does not contain a sunflower $S\in\mathcal{S}_n$, $n\geq4$, which is not suspended.
\end{lemma}

\begin{proof}
Suppose $G$ and $G^2$ are chordal and $G$ contains a not suspended sunflower $S\in\mathcal{S}_n$, $n\geq4$. Then clearly $U$ forms a cycle in $G^2$ and with $S$ being unsuspended $\distg{G}{u_i}{u_j}\geq3$ holds for all $i\in\set{1,\dots,n}$ and $j\neq i\pm1\lb \!\!\!\mod n\rb$. Therefore $\induz{G^2}{U}\cong C_n$ and $G^2$ is not chordal.
\end{proof}

By combining Lemma \autoref{lemma3.4} and Lemma \autoref{lemma3.6} we obtain the following result.

\begin{theorem}[Laskar, Shier. 1983 \cite{LaskarShier1983powercenterchordal}]\label{thm3.13}
Let $G$ be a chordal graph. Then $G^2$ is chordal if and only if $G$ does not contain an unsuspended sunflower $S\in\mathcal{S}_n$ with $n\geq4$.
\end{theorem}

And by taking Dutchet's Theorem into account we finally reach our family of chordal graphs that is closed under the taking of powers.

\begin{theorem}[Laskar, Shier. 1983 \cite{LaskarShier1983powercenterchordal}]\label{thm3.14}
Let $G$ be a chordal graph without unsuspended sunflowers $S\in\mathcal{S}_n$ with $n\geq4$, then $G^k$ is chordal and does not contain an unsuspended sunflower $S\in\mathcal{S}_n$ with $n\geq4$ for all $k\in\N$.
\end{theorem}

Some corollaries for more graph classes, that have separately been shown to stay chordal under any power directly arise from this result.

\begin{corollary}\label{cor3.1}
If $G$ is a tree, then $G^k$ is chordal for all $k\in\N$.
\end{corollary}

\begin{definition}[Block Graph]
A graph $G$ is called a {\em block graph} if each $2$-connected component ({\em block}) of $G$ induces a complete subgraph (i.e. a block graph is chordal).
\end{definition}

\begin{corollary}[Jamison. \cite{jamisonpowers}]\label{cor3.2}
If $G$ is a block graph, then $G^k$ is chordal for all $k\in\N$.
\end{corollary}

\section{Chordal Powers and Powers of Chordal Graphs}

In this section we will investigate the relation of the chordality of a graph and its powers even further. The section is divided into two subsections, in the first we give an answer to the question which graphs become chordal when squared and in the second we state some fundamental properties of powers of chordal graphs.\\
An important tool for obtaining those results are several additional expansions of the suns, or sunflowers, we have seen before. Those seem to be the key when it comes to circles obtained by taking the square of a graph.

\subsection{Chordal Squares}

We begin with a small lemma that is a generalization of Lemma \autoref{lemma3.3}.

\begin{lemma}\label{lemma3.7}
Let $G$ be a graph and $k\geq2$ an integer. If $C$ is an induced cycle in $G^k$, then $G^r$ cannot contain two consecutive edges of $C$ for all $r\leq\abr{\frac{k}{2}}$.
\end{lemma}

\begin{proof}
Let $C_n=\lb v_1e_1\dots v_ne_n\rb$ be an induced cycle in $G^k$. Suppose there is an $1\leq r\leq\abr{\frac{k}{2}}$ and an $i\in\set{1,\dots,n}$ with $\set{e_i,e_{i+1\lb \!\!\!\mod n\rb}}\subseteq\E{G^r}$.\\
Then $\distg{G}{v_{i+1\lb \!\!\!\mod n\rb}}{v_{i+2\lb \!\!\!\mod n\rb}}\leq r\leq\abr{\frac{k}{2}}$ and $\distg{G}{v_{i}}{v_{i+1\lb \!\!\!\mod n\rb}}\leq r\leq\abr{\frac{k}{2}}$ and therefore $\distg{G}{v_i}{v_{i+2\lb \!\!\!\mod n\rb}}\leq k$. Hence the edge $v_iv_{i+2\lb \!\!\!\mod n\rb}$ exists in $G^k$, which is a chord in $C_n$ and a contradiction.   
\end{proof}

\begin{corollary}\label{cor3.3}
Let $G$ be a graph and $k\geq 2$ an integer. If $C_n$ is an induced cycle in $G^k$, $G^r$ contains at most $\abr{\frac{n}{2}}$ edges of $C_n$ for all $r\leq\abr{\frac{k}{2}}$.
\end{corollary}

The first generalization of sunflowers we will attempt is a mild one. In order to properly describe the growth of cycles in powers of general graphs we need a concept of general sunflowers, which means that the flower does not necessarily have to be a chordal graph.

\begin{definition}[General Sunflower]
A {\em general sunflower} of size $n$ is a graph $S=\lb U\cup W, E\rb$ with $U=\set{u_1,\dots,u_n}$ and $W=\set{w_1,\dots,w_n}$ such that the following properties hold.
\begin{enumerate}[i)]

\item  $u_iu_j\notin E$ for $j\neq i\pm1\lb \!\!\!\mod n\rb$

\item $u_iu_{i\pm 1\lb \!\!\!\mod n\rb}\in E\Rightarrow u_iu_{i\mp1\lb \!\!\!\mod n\rb}\in E$

\item $u_iw_j\in E$ if and only if $j=i$ or $j=i+1\lb \!\!\!\mod n\rb$

\end{enumerate}
The family of all general sunflowers of size $n$ is denoted by $\mathcal{F}_n$.\\
If $V$ does induce a chordless cycle we call $S$ a {\em non-chordal sunflower}. And as before $S$ is called {\em suspended} if $S$ is contained in some graph $G$ and there exists some vertex $v$ adjacent to $u_i$ and $u_j$ with $j\neq i\pm1\lb \!\!\!\mod n\rb$.  	
\end{definition}

\begin{figure}[H]
	\begin{center}
		\begin{tikzpicture}
		
		\node (center) [inner sep=1.5pt] {};
		
		\node (label) [inner sep=1.5pt,position=135:1.7cm from center] {};
		
		\node (u1) [inner sep=1.5pt,position=90:1.8cm from center,draw,circle,fill] {};
		\node (u2) [inner sep=1.5pt,position=180:1.8cm from center,draw,circle,fill] {};
		\node (u3) [inner sep=1.5pt,position=270:1.8cm from center,draw,circle,fill] {};
		\node (u4) [inner sep=1.5pt,position=0:1.8cm from center,draw,circle,fill] {};
		
		\node (v1) [inner sep=1.5pt,position=135:0.9cm from center,draw,circle] {};
		\node (v2) [inner sep=1.5pt,position=225:0.9cm from center,draw,circle] {};
		\node (v3) [inner sep=1.5pt,position=315:0.9cm from center,draw,circle] {};
		\node (v4) [inner sep=1.5pt,position=45:0.9cm from center,draw,circle] {};

		\node (lu1) [position=90:0.07cm from u1] {$u_1$};
		\node (lu2) [position=180:0.07cm from u2] {$u_2$};
		\node (lu3) [position=270:0.07cm from u3] {$u_3$};
		\node (lu4) [position=0:0.07cm from u4] {$u_4$};
		
		\node (lv1) [position=135:0.07cm from v1] {$w_1$};
		\node (lv2) [position=225:0.07cm from v2] {$w_2$};
		\node (lv3) [position=315:0.07cm from v3] {$w_3$};
		\node (lv4) [position=45:0.07cm from v4] {$w_4$};

		\path
		(u1) edge (v1)
		edge (v4)
		(u2) edge (v1)
		edge (v2)
		(u3) edge (v2)
		edge (v3)
		(u4) edge (v3)
		edge (v4)
		;
		
		\path [bend right]
		(v1) edge (v2)
		(v2) edge (v3)
		(v3) edge (v4)
		(v4) edge (v1)
		;

		\end{tikzpicture}
		\caption{A general sunflower $S\in\mathscr{S}_4$.}
		\label{fig3.4}
	\end{center}
\end{figure}
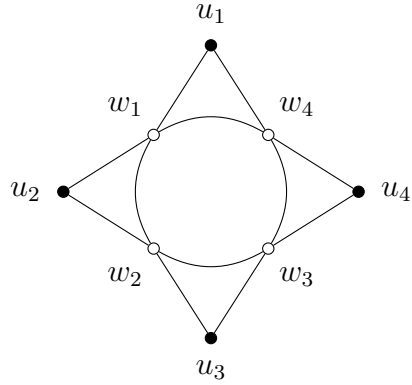

The first attempt in finding induced subgraphs that are responsible for $G^2$ not being chordal was done by Balakrishnan and Paulraja (see \cite{balakrishnan1981graphsa} and \cite{balakrishnan1981graphsb}). They found parts of the general sunflowers, the path $P_5+a$, which is a path on five vertices $v_1,\dots,v_5$ where the edge $v_2v_4$ is added. An example for such a path would be the path $u_1w_1u_2w_2u_3$ in the general sunflower in \autoref{fig3.4}, which also contains the edge $w_1w_2$.

\begin{theorem}[Balakrishnan, Paulraja. 1981 \cite{balakrishnan1981graphsa} and \cite{balakrishnan1981graphsb}]\label{thm3.15}
Let $G$ be a graph. If $G$ does not contain a $K_{1,3}$, $P_5+a$ nor a $C_n$ with $n\geq6$, then $G^2$ is chordal.	
\end{theorem}

Flotow (see \cite{flotow1997graphs}) stated that there cannot be a family of forbidden subgraphs of $G$ necessary for the chordality of $G^m$ for $m\geq 2$. He made this point by supposing the existence of the family $\lb G_i\rb_{i\in I}$ of such forbidden subgraphs and then constructing the graph $G$ simply by taking all those $G_i$ and joining all of their vertices to an additional vertex $v$. With this construction he ensured $G^m$ to be complete and therefore chordal for all $m\geq 2$ despite $G$ containing all forbidden subgraphs.\\
Hence we will not be able to find such a family, but the concept of suspended sunflowers might still work for us. In Flotow's construction all sunflowers would be suspended due to the additional vertex $v$.\\
Flotow also generalized  Balakrishnan's and Paulraja's result to arbitrary powers. By doing so the following result came up.

\begin{theorem}[Flotow. 1997 \cite{flotow1997graphs}]\label{thm3.16}
Let $G$ be a graph. If $G$ does not contain a $K_{1,3}$ nor a $C_n$ with $n\geq 4$, then $G^2$ is chordal.	
\end{theorem}

There seems to be some kind of gap between graphs containing induced cycles of length $4$, $5$ and $\geq6$. We will now show that to fill this gap there is one particular family of general sunflowers that need to be taken into account.

\begin{lemma}\label{lemma3.8}
Let $G$ be a graph that does not contain a $K_{1,3}$ nor a $C_n$ with $n\geq5$. If $G^2$ is not chordal $G$ contains a non-chordal sunflower $S\in\mathcal{F}_4$ which is not suspended.
\end{lemma}

\begin{proof}
With $G^2$ not being chordal there exists an induced cycle $C=\lb u_1\dots u_{\abs{C}}\rb$ with $\abs{C}\geq 4$ in $G^2$.\\
The proof is divided into two major cases, each of them again divided into subcases as follows.
\begin{enumerate}
\item[] \begin{enumerate}
	\item [Case 1] $\abs{C}\geq 5$
		\begin{enumerate}
		\item[] \begin{enumerate}
				
			\item [Subcase 1.1] No edge of $C$ exists in $G$.
			\item [Subcase 1.2] At least one edge of $C$ already exists in $G$.
		\end{enumerate}
		\end{enumerate}

	\item [Case 2] $C=\lb u_1u_2u_3u_4\rb$
		\begin{enumerate}
		\item[] \begin{enumerate}
				
			\item [Subcase 2.1] The edges $u_1u_2$ and $u_3u_4$ exist in $G$.
			\item [Subcase 2.2] Just the edge $u_1u_2$ is contained in $G$.
			\item [Subcase 2.3] $U=\set{u_1,u_2,u_3,u_4}$ is a stable set in $G$.
		\end{enumerate}
		\end{enumerate}
\end{enumerate}
\end{enumerate}
Subcase 1.1: Suppose $\abs{C}=c\geq 5$ and no edge of $C$ is contained in $G$.\\
Hence the set $U=\set{u_1,\dots,u_c}$ of the vertices forming $C$ in $G^2$ forms a stable set in $G$. With $u_i$ and $u_{i\pm1\lb \!\!\!\mod c\rb}$ being adjacent in $G^2$ there exists paths of length $2$ between them in $G$ and therefore a set $W=\set{w_1,\dots,w_c}$ of additional vertices exists, realizing those paths and being disjoint from $U$. Clearly the vertices in $W$ have to be disjoint and for all $i\in\set{1,\dots,c}$ $\knb{G}{w_i}\cap U=\set{u_i,u_{i+1\lb \!\!\!\mod c\rb}}$ holds, otherwise $C$ would have a chord.\\
This leads to a cycle $C'$ of length $2c$ in $G$, alternating between vertices from $U$ and $V$. With $c\geq 5$ this cycle must contain a chord, which cannot join two vertices in $U$ nor can it join a vertex in $u\in U$ to some vertex in $W$ which is not already a neighbor of $u$ on $C'$. Hence a chord in $C'$ must be of the form $w_iw_j$. If in addition $j= i\pm 1\lb \!\!\!\mod c\rb$ holds this chord shortens $C'$ just by $1$.

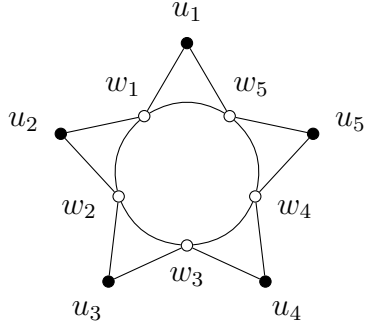
\begin{figure}[H]
	\begin{center}
		\begin{tikzpicture}
		
		\node (center) [inner sep=1.5pt] {};
		
		\node (label) [inner sep=1.5pt,position=135:1.7cm from center] {};
		
		\node (u1) [inner sep=1.5pt,position=90:1.6cm from center,draw,circle,fill] {};
		\node (u2) [inner sep=1.5pt,position=162:1.6cm from center,draw,circle,fill] {};
		\node (u3) [inner sep=1.5pt,position=234:1.6cm from center,draw,circle,fill] {};
		\node (u4) [inner sep=1.5pt,position=306:1.6cm from center,draw,circle,fill] {};
		\node (u5) [inner sep=1.5pt,position=18:1.6cm from center,draw,circle,fill] {};
		
		\node (v1) [inner sep=1.5pt,position=126:0.8cm from center,draw,circle] {};
		\node (v2) [inner sep=1.5pt,position=198:0.8cm from center,draw,circle] {};
		\node (v3) [inner sep=1.5pt,position=270:0.8cm from center,draw,circle] {};
		\node (v4) [inner sep=1.5pt,position=342:0.8cm from center,draw,circle] {};
		\node (v5) [inner sep=1.5pt,position=54:0.8cm from center,draw,circle] {};

		\node (lu1) [position=90:0.07cm from u1] {$u_1$};
		\node (lu2) [position=162:0.07cm from u2] {$u_2$};
		\node (lu3) [position=234:0.07cm from u3] {$u_3$};
		\node (lu4) [position=306:0.07cm from u4] {$u_4$};
		\node (lu5) [position=18:0.07cm from u5] {$u_5$};
		
		\node (lv1) [position=126:0.07cm from v1] {$w_1$};
		\node (lv2) [position=198:0.07cm from v2] {$w_2$};
		\node (lv3) [position=270:0.07cm from v3] {$w_3$};
		\node (lv4) [position=342:0.07cm from v4] {$w_4$};
		\node (lv5) [position=54:0.07cm from v5] {$w_5$};

		\path
		(u1) edge (v1)
		edge (v5)
		(u2) edge (v1)
		edge (v2)
		(u3) edge (v2)
		edge (v3)
		(u4) edge (v3)
		edge (v4)
		(u5) edge (v4)
		edge (v5)
		;
		
		\path[bend right] 
		(v1) edge (v2)
		(v2) edge (v3)
		(v3) edge (v4)
		(v4) edge (v5)
		(v5) edge (v1)
		;
		
		\end{tikzpicture}
		\caption{Shortened cycle in the case $c=5$.}
		\label{fig3.5}
	\end{center}
\end{figure} 
Even if all of those chords exist there is still a cycle of at least length $5$ remaining. Hence an edge of the form $w_iw_j$ with $j\neq i\pm1\lb \!\!\!\mod c\rb$ must exist and with that we get $\induz{G}{\set{w_i,w_j,u_i,u_{i+1\lb\!\!\!\mod\abs{C}\rb}}}\cong K_{1,3}$ which is a contradiction. With that subcase 1.1 is closed.\\

Subcase 1.2:  Suppose $\abs{C}=c\geq 5$ and at least one edge of $C$ is contained in $G$.\\
W.l.o.g. suppose the edge $u_1u_2$ exists in $G$, then by Lemma \autoref{lemma3.7} the edges $u_1u_c$ and $u_2u_3$ cannot exists, and again we need some additional vertices $W=\set{w_1,\dots,w_q}$ which form paths of length $2$ between those vertices in $U$ that are not already adjacent but have to be in $G^2$ in order to form $C$. By Corollary \autoref{cor3.3} $q\geq\aufr{\frac{c}{2}}$ and again a cycle $C'$ with $\abs{C'}\geq c+\aufr{\frac{c}{2}}$ exists in $G$. I.e. $C'$ consists of the vertices in $U$ and $V$ and does not contain three consecutive vertices from $U$.
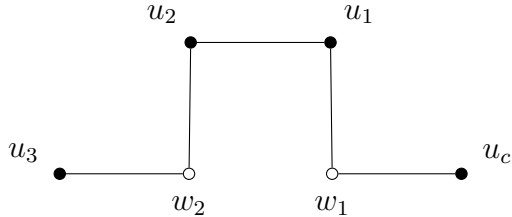
\begin{figure}[H]
	\begin{center}
		\begin{tikzpicture}
		
		\node (center) [inner sep=1.5pt] {};
		\node (anchor1) [inner sep=1.5pt,position=0:0.8cm from center] {};
		\node (anchor2) [inner sep=1.5pt,position=180:0.8cm from center] {};
		
		\node (label) [inner sep=1.5pt,position=135:1.7cm from center] {};
		
		\node (u1) [inner sep=1.5pt,position=90:1.6cm from anchor1,draw,circle,fill] {};
		\node (u2) [inner sep=1.5pt,position=90:1.6cm from anchor2,draw,circle,fill] {};
		\node (u3) [inner sep=1.5pt,position=180:2.5cm from center,draw,circle,fill] {};
		\node (u4) [inner sep=1.5pt,position=0:2.5cm from center,draw,circle,fill] {};
		
		\node (v1) [inner sep=1.5pt,position=0:0.8cm from center,draw,circle] {};
		\node (v2) [inner sep=1.5pt,position=180:0.8cm from center,draw,circle] {};

		\node (lu1) [position=45:0.07cm from u1] {$u_1$};
		\node (lu2) [position=135:0.07cm from u2] {$u_2$};
		\node (lu3) [position=150:0.07cm from u3] {$u_3$};
		\node (lu4) [position=30:0.07cm from u4] {$u_c$};
		
		\node (lv1) [position=270:0.07cm from v1] {$w_1$};
		\node (lv2) [position=270:0.07cm from v2] {$w_2$};

		\path
		(u1) edge (v1)
		(u2) edge (u1)
		edge (v2)
		(u3) edge (v2)
		(u4) edge (v1)
		;

		\end{tikzpicture}
		\caption{Direct surroundings of the edge $u_1u_2$ in $G$.}
		\label{fig3.6}
	\end{center}
\end{figure}
With $c\geq 5$ $C'$ has at least length $8$ and therefore must contain a chord. Again the only chords possible join vertices of $W$ and if we join $w_1$ and $w_2$ to the $w_j$ with $j\geq3$ the obtained cycles have at least length $5$. Hence the edge $w_1w_2$ must exist in $G$ and $\induz{G}{u_1,u_c,w_1,w_2}\cong K_{1,3}$ which again poses a contradiction and closes case 1.\\

Subcase 2.1: Suppose $C=\lb u_1u_2u_3u_4\rb$ and the edges $u_1u_2$ and $u_3u_4$ exist in $G$.\\
Then $G$ must contain two additional vertices $w_1$ and $w_2$ that form a cycle of length $6$ together with the vertices of $C$.
\begin{figure}[H]
	\begin{center}
		\begin{tikzpicture}
		
		\node (center) [inner sep=1.5pt] {};
		
		\node (label) [inner sep=1.5pt,position=135:1.7cm from center] {};
		
		\node (u1) [inner sep=1.5pt,position=60:1.4cm from center,draw,circle,fill] {};
		\node (u2) [inner sep=1.5pt,position=120:1.4cm from center,draw,circle,fill] {};
		\node (u3) [inner sep=1.5pt,position=240:1.4cm from center,draw,circle,fill] {};
		\node (u4) [inner sep=1.5pt,position=300:1.4cm from center,draw,circle,fill] {};
		
		\node (v1) [inner sep=1.5pt,position=180:1.4cm from center,draw,circle] {};
		\node (v2) [inner sep=1.5pt,position=0:1.4cm from center,draw,circle] {};

		\node (lu1) [position=60:0.07cm from u1] {$u_1$};
		\node (lu2) [position=120:0.07cm from u2] {$u_2$};
		\node (lu3) [position=240:0.07cm from u3] {$u_3$};
		\node (lu4) [position=300:0.07cm from u4] {$u_4$};
		
		\node (lv1) [position=180:0.07cm from v1] {$w_1$};
		\node (lv2) [position=0:0.07cm from v2] {$w_2$};

		\path
		(u1) edge (u2)
		edge (v2)
		(u2) edge (v1)
		(u3) edge (v1)
		edge (u4)
		(u4) edge (v2)
		;
		
		\path[dotted]
		(v1) edge (v2)
		;

		\end{tikzpicture}
		\caption{$C'$ with just one possible chord.}
		\label{fig3.7}
	\end{center}
\end{figure}
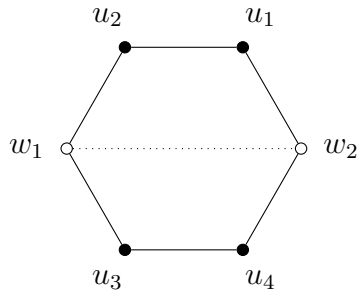
Still the only allowed chords would join vertices of $W$ and with just two of those vertices $w_1w_2$ is the sole possible chord. So again we find a $K_{1,3}$ in $G$ which contradicts our assumption and the subcase 2.1 is settled.\\

Subcase 2.2: Suppose $C=\lb u_1u_2u_3u_4\rb$ and just the edge $u_1u_2$ exists in $G$.\\
Now we get three additional vertices $w_1$, $w_2$ and $w_3$ that form a cycle $C'$ of length $7$ together with $\V{C}$ in $G$.
\begin{figure}[H]
	\begin{center}
		\begin{tikzpicture}
		
		\node (center) [inner sep=1.5pt] {};
		
		\node (label) [inner sep=1.5pt,position=135:1.7cm from center] {};
		
		\node (u1) [inner sep=1.5pt,position=60:1.4cm from center,draw,circle,fill] {};
		\node (u2) [inner sep=1.5pt,position=120:1.4cm from center,draw,circle,fill] {};
		\node (u3) [inner sep=1.5pt,position=220:1.4cm from center,draw,circle,fill] {};
		\node (u4) [inner sep=1.5pt,position=320:1.4cm from center,draw,circle,fill] {};
		
		\node (v1) [inner sep=1.5pt,position=170:1.4cm from center,draw,circle] {};
		\node (v2) [inner sep=1.5pt,position=270:1.4cm from center,draw,circle] {};
		\node (v3) [inner sep=1.5pt,position=10:1.4cm from center,draw,circle] {};

		\node (lu1) [position=60:0.07cm from u1] {$u_1$};
		\node (lu2) [position=120:0.07cm from u2] {$u_2$};
		\node (lu3) [position=220:0.07cm from u3] {$u_3$};
		\node (lu4) [position=320:0.07cm from u4] {$u_4$};
		
		\node (lv1) [position=170:0.07cm from v1] {$w_1$};
		\node (lv2) [position=270:0.07cm from v2] {$w_2$};
		\node (lv3) [position=10:0.07cm from v3] {$w_3$};

		\path
		(u1) edge (u2)
		(u2) edge (v1)
		(v1) edge (u3)
		(u3) edge (v2)
		(v2) edge (u4)
		(u4) edge (v3)
		(v3) edge (u1)
		;
		
		\path[dotted]
		(v1) edge (v3)
		(v1) edge (v2)
		(v2) edge (v3)
		;

		\end{tikzpicture}
		\caption{$C'$ with three possible chords.}
	\end{center}
\end{figure}
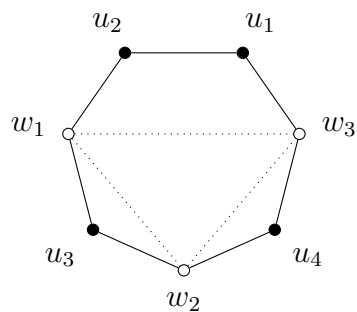
Here now we have three possible chords. But we have already seen in the former subcase that the chord $w_1w_3$ would result in an induced $K_{1,3}$, so $w_1w_3$ is forbidden too. But $w_1w_2$ and $w_2w_3$ each shorten the cycle just by $1$ and thus a cycle of length $5$ still remains in $G$. Hence this case is impossible too and we can close subcase 2.2 as well.\\

Subcase 2.3: Suppose $C=\lb u_1u_2u_3u_4\rb$ and $U=\set{u_1,u_2,u_3,u_4}$ is a stable set in $G$.\\
Now we finally need $4$ additional vertices $W=\set{w_1,w_2,w_3,w_4}$ in order for $U$ to induce a cycle in $G^2$. This results in another cycle $C'$, this time of length $8$ and two chords of $C'$ can be dismissed beforehand because we already saw the resulting in induced $K_{1,3}$'s.
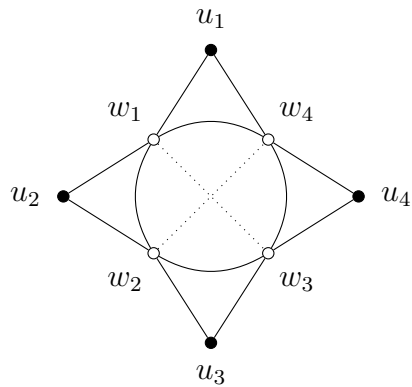
\begin{figure}[H]
	\begin{center}
		\begin{tikzpicture}
		
		\node (center) [inner sep=1.5pt] {};
		
		\node (label) [inner sep=1.5pt,position=135:1.7cm from center] {};
		
		\node (u1) [inner sep=1.5pt,position=90:1.8cm from center,draw,circle,fill] {};
		\node (u2) [inner sep=1.5pt,position=180:1.8cm from center,draw,circle,fill] {};
		\node (u3) [inner sep=1.5pt,position=270:1.8cm from center,draw,circle,fill] {};
		\node (u4) [inner sep=1.5pt,position=0:1.8cm from center,draw,circle,fill] {};
		
		\node (v1) [inner sep=1.5pt,position=135:0.9cm from center,draw,circle] {};
		\node (v2) [inner sep=1.5pt,position=225:0.9cm from center,draw,circle] {};
		\node (v3) [inner sep=1.5pt,position=315:0.9cm from center,draw,circle] {};
		\node (v4) [inner sep=1.5pt,position=45:0.9cm from center,draw,circle] {};

		\node (lu1) [position=90:0.07cm from u1] {$u_1$};
		\node (lu2) [position=180:0.07cm from u2] {$u_2$};
		\node (lu3) [position=270:0.07cm from u3] {$u_3$};
		\node (lu4) [position=0:0.07cm from u4] {$u_4$};
		
		\node (lv1) [position=135:0.07cm from v1] {$w_1$};
		\node (lv2) [position=225:0.07cm from v2] {$w_2$};
		\node (lv3) [position=315:0.07cm from v3] {$w_3$};
		\node (lv4) [position=45:0.07cm from v4] {$w_4$};

		\path
		(u1) edge (v1)
		edge (v4)
		(u2) edge (v1)
		edge (v2)
		(u3) edge (v2)
		edge (v3)
		(u4) edge (v3)
		edge (v4)
		;
		
		\path [bend right]
		(v1) edge (v2)
		(v2) edge (v3)
		(v3) edge (v4)
		(v4) edge (v1)
		;
		
		\path[dotted]
		(v1) edge (v3)
		(v2) edge (v4)
		;

\end{tikzpicture}
\caption{Nonchordal sunflower of size $4$ with two forbidden chords.}
\label{fig3.8}
\end{center}
\end{figure}
This leaves the edges $w_1w_2$, $w_2w_3$, $w_3w_4$ and $w_4w_1$ as the only possible chords in $C'$. Each of those edges shortens the cycle just by $1$ and thus all four of them must exist. The resulting graph, as seen in \autoref{fig3.8}, is a general sunflower of size $4$ which is non-chordal due to the two forbidden chords and cannot be suspended, otherwise $C$ would have a chord. This closes the second case and completes the proof.
\end{proof}

Note that there are just four general sunflowers of size $4$ and three of them are chordal or even a sun. This leaves the general sunflower in \autoref{fig3.4} as the only non-chordal sunflower of this size and we obtain the following corollary.

\begin{corollary}\label{cor3.4}
Let $G$ be a graph. If $G$ does not contain a $K_{1,3}$ nor a $C_n$ with $n\geq5$ and all its non-chordal sunflowers of size $4$ are suspended, then $G^2$ is chordal. 	
\end{corollary}

This still is not a necessary condition for the square of a graph to be chordal. Especially excluding the $K_{1,3}$ seems to be a problem because by Corollary \autoref*{cor3.1} any power of a tree is chordal and obviously $K_{1,3}$ is a tree. A very small one at least.\\
But allowing the $K_{1,3}$ brings a lot of other possible graphs an the table that could produce cycles when squared. Sunflowers of bigger size than $4$ come to mind. But there are even more graphs that are seemingly similar to general sunflowers. Consider the following two graphs as examples.

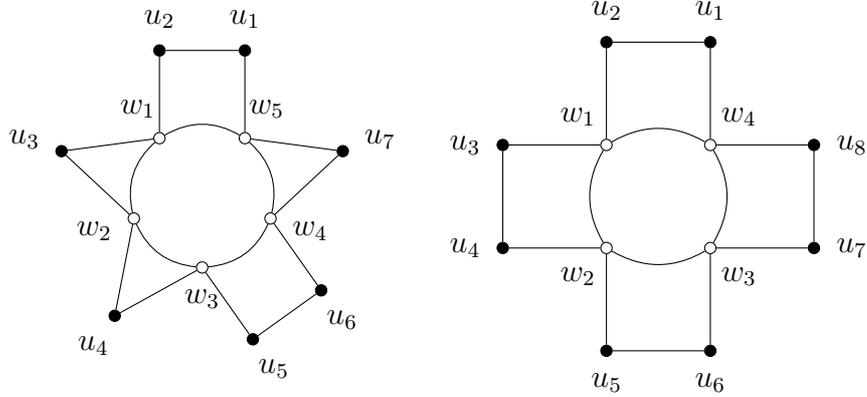
\begin{figure}[H]\label{fig3.11}
	\begin{center}
		\begin{tikzpicture}
		
		\node (center) [inner sep=1.5pt] {};
		
		\node (v1) [inner sep=1.5pt,position=135:0.8cm from center,draw,circle] {};
		\node (v2) [inner sep=1.5pt,position=225:0.8cm from center,draw,circle] {};
		\node (v3) [inner sep=1.5pt,position=315:0.8cm from center,draw,circle] {};
		\node (v4) [inner sep=1.5pt,position=45:0.8cm from center,draw,circle] {};

		\node (u1) [inner sep=1.5pt,position=90:1.2cm from v4,draw,circle,fill] {};
		\node (u2) [inner sep=1.5pt,position=90:1.2cm from v1,draw,circle,fill] {};
		\node (u3) [inner sep=1.5pt,position=180:1.2cm from v1,draw,circle,fill] {};
		\node (u4) [inner sep=1.5pt,position=180:1.2cm from v2,draw,circle,fill] {};
		\node (u5) [inner sep=1.5pt,position=270:1.2cm from v2,draw,circle,fill] {};
		\node (u6) [inner sep=1.5pt,position=270:1.2cm from v3,draw,circle,fill] {};
		\node (u7) [inner sep=1.5pt,position=0:1.2cm from v3,draw,circle,fill] {};
		\node (u8) [inner sep=1.5pt,position=0:1.2cm from v4,draw,circle,fill] {};

		\node (lu1) [position=90:0.07cm from u1] {$u_1$};
		\node (lu2) [position=90:0.07cm from u2] {$u_2$};
		\node (lu3) [position=180:0.07cm from u3] {$u_3$};
		\node (lu4) [position=180:0.07cm from u4] {$u_4$};
		\node (lu5) [position=270:0.07cm from u5] {$u_5$};
		\node (lu6) [position=270:0.07cm from u6] {$u_6$};
		\node (lu7) [position=0:0.07cm from u7] {$u_7$};
		\node (lu8) [position=0:0.07cm from u8] {$u_8$};
		
		\node (lv1) [position=135:0.07cm from v1] {$w_1$};
		\node (lv2) [position=225:0.07cm from v2] {$w_2$};
		\node (lv3) [position=315:0.07cm from v3] {$w_3$};
		\node (lv4) [position=45:0.07cm from v4] {$w_4$};

		\path
		(u1) edge (v4)
		(u2) edge (v1)
		(u3) edge (v1)
		(u4) edge (v2)
		(u5) edge (v2)
		(u6) edge (v3)
		(u7) edge (v3)
		(u8) edge (v4)
		;
		
		\path
		(u1) edge (u2)
		(u3) edge (u4)
		(u5) edge (u6)
		(u7) edge (u8)
		;
		
		\path [bend right]
		(v1) edge (v2)
		(v2) edge (v3)
		(v3) edge (v4)
		(v4) edge (v1)
		;

		\node (center2) [inner sep=1.5pt,left of=center,node distance=6cm] {};
		
		\node (label) [inner sep=1.5pt,position=135:1.7cm from center2] {};

		\node (v1) [inner sep=1.5pt,position=126:0.8cm from center2,draw,circle] {};
		\node (v2) [inner sep=1.5pt,position=198:0.8cm from center2,draw,circle] {};
		\node (v3) [inner sep=1.5pt,position=270:0.8cm from center2,draw,circle] {};
		\node (v4) [inner sep=1.5pt,position=342:0.8cm from center2,draw,circle] {};
		\node (v5) [inner sep=1.5pt,position=54:0.8cm from center2,draw,circle] {};

		\node (u1) [inner sep=1.5pt,position=90:1cm from v5,draw,circle,fill] {};
		\node (u2) [inner sep=1.5pt,position=90:1cm from v1,draw,circle,fill] {};
		\node (u3) [inner sep=1.5pt,position=162:1.8cm from center2,draw,circle,fill] {};
		\node (u4) [inner sep=1.5pt,position=234:1.8cm from center2,draw,circle,fill] {};
		\node (u5) [inner sep=1.5pt,position=305:1cm from v3,draw,circle,fill] {};
		\node (u6) [inner sep=1.5pt,position=305:1cm from v4,draw,circle,fill] {};
		\node (u7) [inner sep=1.5pt,position=18:1.8cm from center2,draw,circle,fill] {};

		\node (lu1) [position=90:0.07cm from u1] {$u_1$};
		\node (lu2) [position=90:0.07cm from u2] {$u_2$};
		\node (lu3) [position=162:0.07cm from u3] {$u_3$};
		\node (lu4) [position=234:0.07cm from u4] {$u_4$};
		\node (lu5) [position=305:0.07cm from u5] {$u_5$};
		\node (lu6) [position=305:0.07cm from u6] {$u_6$};
		\node (lu7) [position=18:0.07cm from u7] {$u_7$};
		
		\node (lv1) [position=126:0.07cm from v1] {$w_1$};
		\node (lv2) [position=198:0.07cm from v2] {$w_2$};
		\node (lv3) [position=270:0.07cm from v3] {$w_3$};
		\node (lv4) [position=342:0.07cm from v4] {$w_4$};
		\node (lv5) [position=54:0.07cm from v5] {$w_5$};

		\path
		(u1) edge (v5)
		(u2) edge (v1)
		(u3) edge (v1)
		edge (v2)
		(u4) edge (v2)
		edge (v3)
		(u5) edge (v3)
		(u6) edge (v4)
		(u7) edge (v4)
		edge (v5)
		;
		
		\path [bend right]
		(v1) edge (v2)
		(v2) edge (v3)
		(v3) edge (v4)
		(v4) edge (v5)
		(v5) edge (v1)
		;
		
		\path
		(u1) edge (u2)
		(u5) edge (u6)
		;

		\end{tikzpicture}
	\end{center}
	\caption{Examples of graphs with non-chordal squares.}
	\label{fig3.9}
\end{figure}

This points to a further generalization of our already general sunflowers. While even simple cycles may produce new but shorter induced cycles when squared it seems to be a very basic property of those graphs that they contain a Hamiltonian cycle consisting only of shortest paths between the vertices of $U$. Combining this observation with the results of Lemma \autoref{lemma3.7} and the corresponding Corollary \autoref{cor3.3} leads to the general concept of flowers.

\begin{definition}[Flower]
A {\em flower} of size $n$ is a graph $F=\lb U\cup W, E\rb$ with $U=\set{u_1,\dots,u_n}$ and $W=\set{w_1,\dots,w_q}$ with $\aufr{\frac{n}{2}}\leq q\leq n$ satisfying the following conditions:
\begin{enumerate}[i)]
	\item There is a cycle $C$ containing all vertices of $W$ in the order $w_1,\dots,w_q$.
	
	\item The set $U=\set{u_1,\dots,u_n}$ is sorted by the appearance order of its elements along $C$ with $u_1w_q, u_2w_1\in E$ and $u_iu_j\notin E$ for $j\neq i\pm1\lb\!\!\!\mod n\rb$.
	
	\item If $w_iw_{i+1}\in\E{C}$, then there is exactly one $u\in U\setminus\V{C}$ with $\fkt{N_F}{u}=\set{w_i,w_{i+1}}$, those vertices $u$ are called {\em pending}.
	
	\item If $w_iw_{i+1}\notin\E{C}$, then there either is one $u\in U\cap\V{C}$ adjacent to $w_i$ and $w_{i+1}$, or there are exactly two vertices $u,t\in u\cap\V{C}$, such that $w_iutw_{i+1}$ is part of $C$.
	
	\item The pending vertices are pairwise nonadjacent and all vertices $u\in U$ that are not pending are contained in $C$.
\end{enumerate}
The family of all flowers of size $n$ is denoted by $\mathfrak{F}_n$.\\
If $F$ is contained in some graph $G$ and there exists an additional vertex $v$ with $vu_i, vu_j\in\E{G}$ and $j\neq i\pm1\lb \!\!\!\mod n\rb$, $F$ is called a {\em withered flower} or just {\em withered}. 	
\end{definition}

Note that edges $w_iw_j$, even $w_iw_{i+1}$ might exist in a given flower $F$, but are not necessarily part of the cycle $C$ of $i)$ in the definition. The flowers in \autoref{fig3.10} are examples of this possibility.

\begin{figure}[H]
	\begin{center}
		\begin{tikzpicture}
		
		\node (anchor1) [] {};
		\node (anchor2) [position=0:5.2cm from anchor1] {};
		\node (anchor3) [position=0:5.2cm from anchor2] {};
		
		\node (u1) [draw,circle,fill,inner sep=1.5pt,position=90:1.6cm from anchor1] {};
		\node (u2) [draw,circle,fill,inner sep=1.5pt,position=180:1.6cm from anchor1] {};
		\node (u3) [draw,circle,fill,inner sep=1.5pt,position=270:1.6cm from anchor1] {};
		\node (u4) [draw,circle,fill,inner sep=1.5pt,position=0:1.6cm from anchor1] {};
		
		\node (w1) [draw,circle,fill,inner sep=1.5pt,position=135:0.8cm from anchor1] {};
		\node (w2) [draw,circle,fill,inner sep=1.5pt,position=225:0.8cm from anchor1] {};
		\node (w3) [draw,circle,fill,inner sep=1.5pt,position=315:0.8cm from anchor1] {};
		\node (w4) [draw,circle,fill,inner sep=1.5pt,position=45:0.8cm from anchor1] {};
		
		\node (lu1) [position=90:0.07 from u1] {$u_1$};
		\node (lu2) [position=180:0.07 from u2] {$u_2$};
		\node (lu3) [position=270:0.07 from u3] {$u_3$};
		\node (lu4) [position=0:0.07 from u4] {$u_4$};
		
		\node (lw1) [position=135:0.07cm from w1] {$w_1$};
		\node (lw2) [position=225:0.07cm from w2] {$w_2$};
		\node (lw3) [position=315:0.07cm from w3] {$w_3$};
		\node (lw4) [position=45:0.07cm from w4] {$w_4$};
		
		\node (n) [position=270:2.8cm from anchor1] {$q=4$};
		
		\path
		(u1) edge (w1)
		edge (w4)
		(u2) edge (w1)
		edge (w2)
		(u3) edge (w2)
		edge (w3)
		(u4) edge (w3)
		edge (w4)
		;
		
		\path [bend right]
		(w1) edge (w2)
		(w2) edge (w3)
		(w3) edge (w4)
		(w4) edge (w1)
		;
		
		\path
		(w2) edge (w4)
		;
		
		
		\node (u1) [draw,circle,fill,inner sep=1.5pt,position=61.42:1cm from anchor2] {};
		\node (u2) [draw,circle,fill,inner sep=1.5pt,position=112.84:1cm from anchor2] {};
		\node (u3) [draw,circle,fill,inner sep=1.5pt,position=215.68:1cm from anchor2] {};
		\node (u4) [draw,circle,fill,inner sep=1.5pt,position=318.52:1cm from anchor2] {};
		
		\node (w1) [draw,circle,fill,inner sep=1.5pt,position=164.26:1cm from anchor2] {};
		\node (w2) [draw,circle,fill,inner sep=1.5pt,position=267.1:1cm from anchor2] {};
		\node (w3) [draw,circle,fill,inner sep=1.5pt,position=10:1cm from anchor2] {};
		
		\node (lu1) [position=61.42:0.07 from u1] {$u_1$};
		\node (lu2) [position=112.84:0.07 from u2] {$u_2$};
		\node (lu3) [position=215.68:0.07 from u3] {$u_3$};
		\node (lu4) [position=318.52:0.07 from u4] {$u_4$};
		
		\node (lw1) [position=164.26:0.07cm from w1] {$w_1$};
		\node (lw2) [position=267.1:0.07cm from w2] {$w_2$};
		\node (lw3) [position=10:0.07cm from w3] {$w_3$};
		
		\node (n) [position=270:2.8cm from anchor2] {$q=3$};
		\path
		(u1) edge (u2)
		(u2) edge (w1)
		(w1) edge (u3)
		(u3) edge (w2)
		(w2) edge (u4)
		(u4) edge (w3)
		(w3) edge (u1)
		;
		
		\path [bend left]
		(w1) edge (w2)
		edge (w3)
		;
		
		
		\node (u1) [draw,circle,fill,inner sep=1.5pt,position=60:1cm from anchor3] {};
		\node (u2) [draw,circle,fill,inner sep=1.5pt,position=120:1cm from anchor3] {};
		\node (u3) [draw,circle,fill,inner sep=1.5pt,position=240:1cm from anchor3] {};
		\node (u4) [draw,circle,fill,inner sep=1.5pt,position=300:1cm from anchor3] {};
		
		\node (w1) [draw,circle,fill,inner sep=1.5pt,position=0:1cm from anchor3] {};
		\node (w2) [draw,circle,fill,inner sep=1.5pt,position=180:1cm from anchor3] {};
		
		\node (lu1) [position=60:0.07 from u1] {$u_1$};
		\node (lu2) [position=120:0.07 from u2] {$u_2$};
		\node (lu3) [position=240:0.07 from u3] {$u_3$};
		\node (lu4) [position=300:0.07 from u4] {$u_4$};
		
		\node (lw1) [position=0:0.07cm from w1] {$w_1$};
		\node (lw2) [position=180:0.07cm from w2] {$w_2$};
		
		\node (n) [position=270:2.8cm from anchor3] {$q=2$};
		
		\path
		(u1) edge (u2)
		(u2) edge (w2)
		(w1) edge (u1)
		(u3) edge (u4)
		(u4) edge (w1)
		(w2) edge (u3)
		;

		\end{tikzpicture}
	\end{center}
	\caption{Some additional examples of flowers of size $4$ with $q=2,3,4$.}
	\label{fig3.10}
\end{figure}
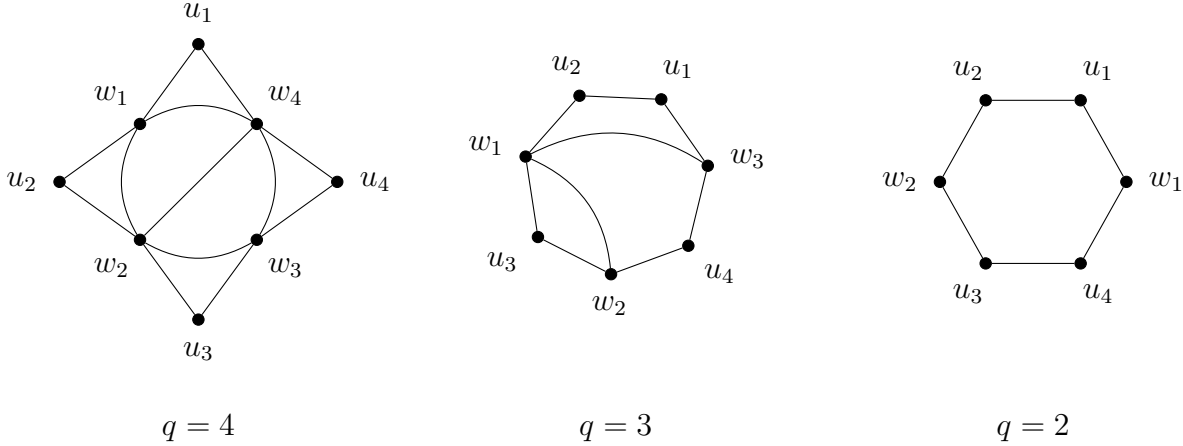

\begin{theorem}\label{thm3.17}
A graph $G$ contains a flower of size $n$, which is not withered, if and only if $G^2$ contains a $C_n$.
\end{theorem}

\begin{proof}
Let $F\in\mathscr{F}_n$ be an unwithered flower of size $n$ in $G$. With $F$ not being withered it follows $\distg{G}{u_i}{u_j}\geq 3$ for all $j\neq i\pm1\lb \!\!\!\mod n\rb$ and $\distg{G}{u_i}{u_j}\leq2$ for $j= i\pm1\lb \!\!\!\mod n\rb$. Hence the $u_i$ form an induced cycle of length $n$ in $G^2$.\\
\\
Now let $C_n$ be an induced cycle in $G^2$. By Corollary \autoref{cor3.3} $G$ can contain at most $\abr{\frac{n}{2}}$ edges of $C_n$. Hence for at least every second edge $u_iu_{i+1\lb \!\!\!\mod n\rb}$ in $C_n$ there must exist some vertex $w_k$ with $u_iw_k, u_{i+1\lb \!\!\!\mod n\rb}\in\E{G}$ in $G$ additionally satisfying $w_ku_j\notin\E{G}$ for all $j\in\set{1,\dots n}\setminus\set{u_i,u_{i+1\lb \!\!\!\mod n\rb}}$. Such vertices $w_i$ exist necessarily for every pair of consecutive vertices of $C_n$ that are not already adjacent in $G$, i.e. only those necessary $w_i$ are taken into account and there are at least $\aufr{\frac{n}{2}}$ of them.\\
The $u_i$ together with the $w_k$ form a circle $C'$ of length $n+\aufr{\frac{n}{2}}\leq n+q\leq 2\,n$ in which no two $w_k$ are adjacent. Now for the pending $u$-vertices we have a look at those pairs $w_i$, $w_{i+1}$ with $w_iw_{i+1}\in\E{G}$ that are adjacent to a common $u$-vertex. For each such pair of $w$-vertices we shorten the cycle $C'$ by removing their common neighbor $u$, which thus becomes a pending vertex, from it and adding the edge $w_iw_{i+1}$. The result is the cycle $C$ of condition $i)$ in the definition. No other edge between two $w$-vertices is part of $C$ and no more than two consecutive $u$-vertices appear on it due to the application of Corollary \autoref{cor3.3}, thus $ii)$, $iii)$ and $iv)$ are satisfied as well.\\
In order to prevent $C_n$ from having a chord every $w_k$ may only be adjacent to those $u$ vertices adjacent to it in $C$ and besides the edges of $C_n$ already present in $G$ no other edges between the vertices in $U$ can exist satisfying condition $v)$ of the definition and therefore a flower of size $n$ that is not withered exists in $G$. If $F$ were withered $C_n$ would contain a chord in $G^2$ and therefore would not have been an induced cycle in the first place.
\end{proof}

The proof of Theorem \autoref{thm3.17} indicates that our basic cycle $C$, which is guaranteed by condition $i)$ of the flower definition can be extended, so that it contains all vertices of the flower $F$. A graph with such a cycle is called Hamiltonian and this property is another well studied aspect of graph theory.

\begin{definition}[Hamiltonian Cycle]
Let $G$ be a graph. A {\em hamiltonian cycle} $C$ is a, not necessarily induced, cycle in $G$ with $\V{C}=\V{G}$. If $G$ possesses at least one hamiltonian cycle, $G$ itself is called {\em hamiltonian}.
\end{definition}

\begin{lemma}\label{lemma3.14}
Let $F\in\mathscr{F}_n$ be a flower of size $n$, then $F$ is hamiltonian and the hamiltonian cycle of $F$ does not contain an edge of the form $w_iw_{i+1}$.	
\end{lemma}

\begin{proof}
In order to find this hamiltonian cycle we simply reverse a construction step in the proof of Theorem \autoref{thm3.17}. Let $C$ be the cycle of condition $i)$ in our flower. Then $C$ already contains all $w$-vertices and all $u$-vertices that are not pending.\\
Now, by definition, any $u$-vertex that is pending is adjacent to two $w$-vertices, say $w_i$ and $w_{i+1}$, such that $w_iw_{i+1}$ is an edge of $C$. In addition there is no edge joining two $w$-vertices contained in $C$ other than those excluding pending $u$-vertices from the cycle. By removing those edges $w_iw_{i+1}$ from $C$ and instead adding the corresponding pending vertex $u$ together with the two edges $uw_i$ and $uw_{i+1}$ to $C$ we obtain a cycle containing all vertices of $F$, hence a hamiltonian cycle, and not containing any edge joining two $w$-vertices.
\end{proof}

Theorem \autoref{thm3.17} gives a first, very easy necessary and sufficient condition for the square of some graph to be chordal. But the mere description given by the definition of flowers does not seem very on the point.

\begin{corollary}\label{cor3.5}
Let $G$ be a graph, then $G^2$ is chordal if and only if all flowers of size $n\geq4$ in $G$ are withered.
\end{corollary} 

While $2$-strong cliques and $2$-strong stable sets are the extension of cliques and stable sets from $G$ to $G^2$ flowers seem to be extensions of cycles in the same manner. In fact one could call a flower of size $n$ in accordance to nomenclature of this thesis a $2$-strong cycle of length $n$ and the term withered translates more or less to a $2$-strong cycle having a $2$-strong chord. In this terms Corollary \autoref{cor3.5} could be rephrased as follows:
\begin{center}
A graph $G$ is called {\em $2$-strong chordal} if every $2$-strong cycle in $G$ of length $\geq 4$ has a $2$-strong chord.
\end{center}
In the following we will see even stronger relations and analogies that will strengthen the idea of such a generalization of chordality.\\
Fulkerson and Gross (see \cite{fulkerson1965incidence}) gave a very nice characterization of chordal graphs in terms of separators that can be easily expanded to our generalization.

\begin{definition}[Separator]
Let $G=\lb V,E\rb$ be a connected graph and $S\subseteq V$. $S$ is called a {\em separator} of $G$ if $G-S$ is not connected.\\
$S$ is called an {\em $a,b$-separator} if $S$ is a separator and $G-S$ contains two distinct components $A$ and $B$ with $a\in A$ and $b\in B$.
\end{definition}

\begin{theorem}[Fulkerson, Gross. 1965 \cite{fulkerson1965incidence}]\label{thm3.18}
A graph $G$ is chordal if and only if every minimal separator of $G$ induces a clique.
\end{theorem}

\begin{definition}[$k$-strong Separator]
Let $G=\lb V,E\rb$ be a connected graph and $S\subseteq V$. $S$ is called a {\em $k$-strong separator} of $G$ if $G-S$ is not connected and for all components $U$, $W$ of $G-S$ with $U\neq W$ every path from $U$ to $W$ in $G$ contains a subpath of length $k-1$ completely contained in $\induz{G}{S}$.\\
$S$ is called a {\em $k$-strong $a,b$-Separator} if $S$ is a $k$-strong separator of $G$ and $G-S$ contains two distinct components $A$ and $B$ with $a\in A$ and $b\in B$.
\end{definition}

\begin{lemma}\label{lemma3.9}
Let $G=\lb V,E\rb$ be a graph, $S\subsetneq V$ and $k\in\N$, $S$ is a minimal $k$-strong separator of $G$ if and only if $S$ is a minimal separator of $G^k$.
\end{lemma}

\begin{proof}
Let $S$ be a minimal $k$-strong separator of $G$. We begin the proof by showing that $S$ separates $G^k$ as well. Let $U$ and $W$ be different components of $G-S$, then every path from $U$ to $W$ contains a path of length $k-1$ that is completely contained in $\induz{G}{S}$, hence for all $u\in U$ and $w\in W$ it holds $\distg{G}{u}{w}\geq k+1$. And so $\distg{G^k}{u}{w}\geq 2$. With all $u,w$-paths using at least $k$ consecutive vertices of $S$ every $u,w$-path in $G^k$ has to contain at least one vertex of $S$, thus $S$ separates $u$ and $w$ in $G^k$. With that not only is $S$ a separator of $G^k$, the components of $G-S$ and $G^k-S$ are the same. Hence if $S$ is a minimal $k$-strong $a,b$-separator in $G$, $S$ is an $a,b$-separator in $G^k$.\\
Now suppose $S$ is no minimal separator of $G^k$, with $S$ being a separator of $G^k$ there exists some $S'\subsetneq S$ which is minimal. Let $U'$ and $W'$ be different components of $G^k-S'$, then every path from $U$ to $W$ in $G^k$ contains at least one vertex of $S'$ and with $u'\in U'$ and $w'\in W'$ being in different components of $G^k-S'$ they cannot be adjacent, thus $\distg{G}{u'}{w'}\geq k+1$. Suppose $S'$ is no $k$-strong separator of $G$, then there exists a pair of vertices $u'$, $w'$ from different components with minimal distance in $G$ such that an $u',w'$-path in $G$ does not contain a subpath of length $k-1$ consisting only of vertices from $S'$. Then there exists some vertex $x\in \V{G-S'-U'}$ with $\distg{G}{u'}{x}\leq k$ and $y\in\V{G-S'-W'}$ with $\distg{G}{w'}{y}\leq k$. If there is some component $K\subseteq\V{G^k-S'}$ containing both $x$ and $y$ there would be a path from $u'$ to $w'$ in $G^k-S'$ which yields a contradiction. So there must exist two different components $X$ and $Y$ of $G^k-S'$ with $x\in X$ and $y\in Y$.\\
Please not that $x$ and $u'$, and $y$ and $w'$ being adjacent in $G^k$ does not necessarily mean that they are contained in the same components of $G-S'$.\\
Now let $P$ be a shortest $u',w'$-path in $G$ contradicting $S'$ being a $k$-strong separator in $G$. As we have seen there must exist at least two components in $G$ containing vertices $x$ and $y$ as described above with $x,y\in\V{P}$. Let $X$ and $Y$ be the two components of $G-S'$ met by $P$ with the smallest distance to each other, so chose a subpath $P'\subseteq P$ connecting two components $X$ and $Y$ such that $P'$ does not contain any vertex of any component of $G-S'$ besides $x$ and $y$. Since $P$ contradicts $S'$ being $k$-strong $\V{P'-x-y}\subseteq S'$ contains at most $k-1$ vertices and therefore $\distg{G}{x}{y}\leq k$. Hence $x$ and $y$ are adjacent in $G^k$ and cannot belong to different components of $G^k-S'$.\\
By iterating the argument over the two subpaths $P_x$ and $P_y$ of $P$ connecting $x$ to $u'$ and $y$ to $w'$ we finally obtain a $u'$-$w'$-path in $G^k$, which is a contradiction to $U$ and $W$ being different components of $G^k-S'$.\\  
Thus our assumption was wrong and $S'\subsetneq S$ is a $k$-strong separator of $G$ contradicting the minimality of $S$ and so $S$ is a minimal separator of $G^k$.\\ 
Now, for the converse direction, let $S$ be a minimal separator of $G^k$. Obviously every separator of $G^k$ separates $G$ too. By repeating the arguments for the minimality of $S$ in $G^k$ from the first part of this proof we again obtain a contradiction to the assumption of $S$ not being a minimal $k$-strong separator in $G$. With that our proof is complete. 
\end{proof}

Combining this lemma with the theorem of Fulkerson and Gross we obtain a another criterion for the chordality of $G^2$, and even for higher powers.

\begin{corollary}\label{cor3.6}
Let $G$ be a graph and $k\in\N$. $G^k$ is chordal if and only if every minimal $k$-strong separator of $G$ is a $k$-strong clique in $G$.
\end{corollary}	

Another characterization was given by Buneman in 1974. This one gives a very nice structure for the cliques of a chordal graph and can easily be extended to higher chordal powers.

\begin{theorem}[Buneman. 1974 \cite{buneman1974characterisation}]\label{thm3.19}
A graph $G$ is chordal if and only if there exists a tree $T=\lb K,L\rb$ where the vertex set $K$ corresponds to the set of all maximal cliques in $G$ and the set $K_v\define\condset{Q\in K}{v\in Q}$ induces a subtree of $T$ for all $v\in\V{G}$.
\end{theorem}

We have already seen that every $k$-strong clique in a graph $G$ corresponds to an ordinary clique in $G^k$. So, if $G^k$ is chordal, there exists such a tree as described in Theorem \ref{thm3.19} whose vertices correspond to the maximal cliques in $G^k$, hence the maximal $k$-strong cliques in $G$. Hence a comparable tree-structure exists for $G$ and its maximal $k$-strong cliques. We obtain the following result.

\begin{lemma}
Let $G$ be a graph and $k\in\N$. $G^k$ is chordal if and only if there exists a tree $T_k=\lb K_k,L_k\rb$ where the vertex set $K_k$ corresponds to the set of all maximal $k$-strong cliques in $G$ and the set $K_v\define\condset{Q\in K_k}{v\in Q}$ induces a subtree of $T_k$ for all $v\in\V{G}$.
\end{lemma}

Now we are able to give a number of equivalent properties that are necessary and sufficient for a graph $G$ to have a chordal square. Just one analogue property of ordinary chordal graphs remains open and raises the question:
\begin{center}
If $G^2$ is chordal, can $G$ be described as an intersection graph of some family $\mathscr{F}$?
\end{center}

\begin{theorem}\label{thm3.20}
Let $G$ be a graph, the following properties are equivalent:
\begin{enumerate}[i)]
\item $G^2$ is chordal.

\item All flowers of size $n\geq 4$ in $G$ are withered.

\item Every minimal $2$-strong separator of $G$ is a $2$-strong clique in $G$.

\item There exists a tree $T_2=\lb K_2,L_2\rb$ where the vertex set $K_2$ corresponds to the set of all maximal $2$-strong cliques of $G$ and the set $K_v\define\condset{Q\in\K_2}{v\in Q}$ induces a subtree of $T_2$ for all $v\in\V{G}$.
\end{enumerate}
\end{theorem}

We will close this subsection by making a first attempt at finding the graphs, whose squares are perfect. According to the Strong Perfect Graph Theorem a graph is perfect if and only if it does not contain a cycle of odd length and no induced subgraph isomorphic to the complement of an odd cycle. Theorem \autoref{thm3.17} gives us a description of all structures in $G$ responsible for $G^2$ having an induced cycle of certain length. With that we obtain to following two corollaries.

\begin{corollary}\label{cor3.8}
A graph $G$ contains a flower of odd size, which is not withered, if and only if $G^2$ contains an odd cycle.
\end{corollary}

\begin{corollary}
Let $G$ be a graph. If $G^2$ is perfect, $G$ does not contain a flower of odd size that is not withered.
\end{corollary}

The question which structural properties of $G$ will result in $\overline{G^2}$ containing an odd cycle is raised.

\chapter{Strong Edge Coloring}


In this chapter we will move from vertex coloring to the coloring of edges. As stated in the title of the chapter we will mainly study the so called strong edge coloring, which is the $2$-strong edge coloring of a graph, or the $2$-strong coloring of its line graph.\\
We will give a brief introduction and state some known results in this field. Our main goal will be the application of results from Chapter 3 on graph powers to the line graph and especially the characterization of graphs whose squared line graphs are chordal. In contrast to general graph powers this can be done in terms of forbidden subgraphs.\\
Again we will make some attempts to expand this theory to find graphs whose squared line graph becomes perfect.
\vspace{-2mm}

\section{Introduction}

First introduced by Fouquet and Jolivet in 1983 (see \cite{fouquet1983strong}) and suggested by Erd\H{o}s and Ne\v{s}et\v{r}il, the strong edge coloring comes from the study of induced matchings in graphs. An induced matching in a graph $G$ is an induced subgraph of $G$ that forms a matching. So it is a set of pairwise disjoint edges of $G$ with no two of those edges being adjacent to the same edge in $G$. Clearly this is a $2$-strong matching. A strong edge coloring of a graph is a coloring of the edges, such that each color class forms an induced matching, which corresponds to a $2$-strong edge coloring.\\
First we will revisit some general and well known bounds in term of a strong edge coloring alongside with the introduction of some further notation which is heavily used in the literature.

\begin{definition}[Pair Degree]\label{def4.1}
Let $G$ be a graph. The {\em pair degree} or {\em edge degree} of an edge $xy=e\in\E{G}$ is given by $\fkt{\operatorname{s}}{e}=\fkt{\deg}{x}+\fkt{\deg}{y}-1$ and it holds $\fkt{\operatorname{s}}{e}=\fkt{\deg_{\lineg{G}}}{e}$. We denote the {\em maximum pair degree} of $G$ with $\fkt{\sigma}{G}=\max_{e\in E\lb G\rb}\fkt{\operatorname{s}}{e}=\fkt{\Delta}{\lineg{G}}+1$.
\end{definition}
\vspace{-2mm}
With that we obtain the following bounds on the $2$-strong chromatic index of a graph $G$.
\begin{enumerate}[i)]
\item $\fkt{\sigma}{G}\leq\kam{2}{G}\leq\stronki{2}{G}$

\item $\frac{\abs{\E{G}}}{\kmat{2}{G}}\leq\stronki{2}{G}$
\end{enumerate}
A very nice upper bound can be obtained by applying Theorem \autoref{thm3.1} to the line graph of $G$.
\begin{theorem}\label{thm4.1}
Let $G$ be a graph, then it holds
\begin{align*}
\stronki{2}{G}\leq \lb\fkt{\sigma}{G}-1\rb^2-2\fkt{\sigma}{G}-1\leq 2\fkt{\Delta}{G}^2-2\fkt{\Delta}{G}+1. 
\end{align*}
\end{theorem}

Now consider the following constructions to motivate an even stronger conjecture by Erd\H{o}s and Ne\v{s}et\v{r}il. For an even number $\Delta$ take a cycle of length $5$ and replace each of its vertices by $\frac{\Delta}{2}$ new ones. For an odd number $\Delta$ take two consecutive vertices of a $C_5$ and replace them by $\frac{\Delta+1}{2}$, $\Delta\geq 3$, new vertices and the remaining three original ones are each replaced by $\frac{\Delta-1}{2}$ new vertices. The graphs obtained by this constructions are $2$-strong anti matchings with maximum degree $\Delta$ and $\frac{5}{4}\Delta^2$ edges in the even and $\frac{5}{4}\Delta^2-\frac{1}{2}\Delta+\frac{1}{4}$ edges in the odd case. The conjecture states that these are general upper bounds on the $2$-strong chromatic index of graphs.

\begin{conjecture}[Erd\H{o}s, Ne\v{s}et\v{r}il. 1985]\label{con4.1} 
Let $G$ be a graph. Then
\begin{align*}
\stronki{2}{G}\leq\stueckfkt{\frac{5}{4}\fkt{\Delta}{G}^2}{\text{if}~\fkt{\Delta}{G}~\text{is even}}{\frac{5}{4}\fkt{\Delta}{G}^2-\frac{1}{2}\fkt{\Delta}{G}+\frac{1}{4}}{\text{otherwise.}}
\end{align*}
\end{conjecture}

Two interesting special cases of the Erd\H{o}s-Ne\v{s}et\v{r}il Conjecture arise.

\begin{conjecture}\label{con4.2}
Let $G$ be a graph. Then
\begin{align*}
\kam{2}{G} \leq\stueckfkt{\frac{5}{4}\fkt{\Delta}{G}^2}{\text{if}~\fkt{\Delta}{G}~\text{is even}}{\frac{5}{4}\fkt{\Delta}{G}^2-\frac{1}{2}\fkt{\Delta}{G}+\frac{1}{4}}{\text{otherwise.}}
\end{align*}
\end{conjecture} 

\begin{conjecture}\label{con4.3}
Let $G$ be a graph. Then
\begin{align*}
\abs{\E{G}} \leq\stueckfkt{\kmat{2}{G}\frac{5}{4}\fkt{\Delta}{G}^2}{\text{if}~\fkt{\Delta}{G}~\text{is even}}{\kmat{2}{G}\lb\frac{5}{4}\fkt{\Delta}{G}^2-\frac{1}{2}\fkt{\Delta}{G}+\frac{1}{4}\rb}{\text{otherwise.}}
\end{align*}
\end{conjecture} 

Conjecture \autoref{con4.3} reduced to graphs that are $2$-strong anti matchings, in other words the case $\kmat{2}{G}=1$, has been asked by Bermond et. al. in 1983 (see \cite{bermond1983surveys}) and was proven by Chung et. al. in 1990 (see \cite{chung1990maximum}). This results in the following theorem.

\begin{theorem}[Chung, Gy{\'a}rf{\'a}s, Tuza and Trotter. 1990 \cite{chung1990maximum}]\label{thm4.2}
Let $G$ be a graph with $\distg{G}{x}{y}\leq 2$ for all $x,y\in\V{G}$. Then
\begin{align*}
\abs{\E{G}} \leq\stueckfkt{\frac{5}{4}\fkt{\Delta}{G}^2}{\text{if}~\fkt{\Delta}{G}~\text{is even}}{\frac{5}{4}\fkt{\Delta}{G}^2-\frac{1}{2}\fkt{\Delta}{G}+\frac{1}{4}}{\text{otherwise.}}
\end{align*}
\end{theorem} 

Note that this is no solution to Conjecture \autoref{con4.2}. There exist graphs with anti matchings that are no induced subgraphs, which are not included in Theorem \autoref{thm4.2}.\\
By restricting ourselves to bipartite graphs, i.e. graphs with chromatic number $2$, some analogue conjectures can be stated and even proved to be correct.

\begin{conjecture}\label{con4.4}
If $G$ is a bipartite graph then $\stronki{2}{G}\leq\fkt{\Delta}{G}^2$.
\end{conjecture}

\begin{theorem}[Faudree, Gy{\'a}rf{\'a}s, Schelp, and Tuza. 1990 \cite{GyarfasTuza1990strongedge}]\label{}
If $G$ is a bipartite graph then $\kam{2}{G}\leq\fkt{\Delta}{G}^2$.
\end{theorem}

\begin{theorem}[Faudree, Gy{\'a}rf{\'a}s, Schelp and Tuza. 1989 \cite{faudree1989induced}]\label{thm4.3}
If $G$ is a bipartite graph then $\E{G}\leq\kmat{2}{G}\fkt{\Delta}{G}^2$.
\end{theorem}

For the general chromatic index the maximum degree poses a natural lower bound. A similar bound can be obtained for the $2$-strong chromatic index in terms of the so called average degree.

\begin{definition}[Average Degree]
Let $G$ be a graph. The {\em average degree} of $G$ is given by $\fkt{\deg}{G}=\frac{1}{\abs{\V{G}}}\sum_{v\in\V{G}}\fkt{\deg}{v}$.
\end{definition}

\begin{theorem}[D{\k{e}}bski, Grytczuk and {\'S}leszy{\'n}ska-Nowak. 2015 \cite{DebskiGrytczuk2015strongedgesparse}]\label{thm4.7}
Let $G$ be a graph then $\stronki{2}{G}\geq 2\fkt{\deg}{G}-1$.
\end{theorem}

Molloy and Reed (see \cite{molloy1997bound}) established another bound for arbitrary graphs with the probabilistic method.

\begin{theorem}[Molloy, Reed. 1997 \cite{molloy1997bound}]\label{thm4.8}
Let $G$ be a graph with $\fkt{\Delta}{G}$ sufficiently large, then $\stronki{2}{G}\leq1.998\fkt{\Delta}{G}^2$. 
\end{theorem}

On special graph classes some stronger bounds for the $2$-strong chromatic index have been obtained or even given in terms of functions in the maximum pair degree or the maximum degree of those graphs.

\begin{lemma}[Lai, Lih, and Tsai. 2012 \cite{lai2012strong}]\label{lemma4.1}
\begin{align*}
\stronki{2}{C_n}=\stueckfktd{3}{\text{if}~n=0\lb\!\!\! \mod 3\rb}{5}{\text{if}~n=5}{4}{otherwise.}
\end{align*}
\end{lemma}

\begin{theorem}[Faudree, Gy{\'a}rf{\'a}s, Schelp and Tuza. 1990 \cite{chung1990maximum}]\label{thm4.4}
If $T$ is a tree then $\stronki{2}{G}=\fkt{\sigma}{T}$.
\end{theorem}

\begin{theorem}[Cameron. 1989 \cite{cameron1989induced}]\label{thm4.5}
If $G$ is chordal then $\stronki{2}{G}=\kam{2}{G}$.
\end{theorem}

\begin{theorem}[Cameron, Sritharan and Tang. 2003 \cite{cameron2003finding}]\label{thm4.6}
If neither $G$ nor $\overline{G}$ contains an induced $C_n$ with $n\geq 5$, then $\stronki{2}{G}=\kam{2}{G}$.
\end{theorem}

\section{Squares of the Line Graph}

With strong edge coloring corresponding to the coloring of the squared line graph of a graph it seems to be a natural approach to study structural properties of the line graph and its square.\\
In this section we will give a brief introduction to the structure of line graphs in general and then give some further investigation in their squares. Our goal is to find analogons to the main results of Chapter 3 to give necessary and sufficient conditions for the chordality of the squared line graph. While for general graph powers it is not possible to give a finite set of forbidden subgraphs to characterize the graphs with chordal squares, we will find that there is such a set of induced subgraphs that is responsible for the non-chordality of the squared line graph.\\
Furthermore another attempt of describing perfect squares, now those of linegraphs, is made.

\subsection{The Line Graph}

We start this section by recalling the definition of a line graph and stating some basic and well known facts. Furthermore we will give a list of common parameters used in the description of graphs that can be translated to the language of line graphs. We have already seen in Chapter 2 that especially concepts like cliques and stable sets can be translated from the line graph into matchings and anti matchings in the original graph. All those translations might seem rather obvious, but it will prove very convenient to have a good overview for we will often switch between a graph and its line graph in the upcoming proofs.

\begin{definition}[Line Graph]
The {\em line graph} $\lineg{G}$ of a graph $G$ is the graph defined on the edge set
$\E{G}$ (each vertex in $\lineg{G}$ represents a unique edge in $G$) such that two vertices in
$\lineg{G}$ are defined to be adjacent if their corresponding edges in $G$ share a common
endpoint. The distance between two edges in G is defined to be the distance between
their corresponding vertices in $\lineg{G}$. For two edges $e, e'\in\E{G}$ we write $\distg{G}{e}{e'}=\distlg{G}{e}{e'}$.
\end{definition}

\begin{remark}\label{rem4.1} Let $G=\lb V,E\rb$ be a graph and $\lineg{G}$ its line graph, then
\begin{enumerate}[i)]

\item $\V{\lineg{G}}=E$,

\item $\fkt{\deg_{\lineg{G}}}{e}=\fkt{\deg}{x}+\fkt{\deg}{y}-2=\fkt{\operatorname{s}}{e}-1$ for $xy=e\in E$,

\item $\fkt{\Delta}{\lineg{G}}=\fkt{\sigma}{G}-1$,

\item $\fkt{\omega}{\lineg{G}}=\fkt{\Delta}{G}$,

\item $\fkt{\alpha}{\lineg{G}}=\fkt{\nu}{G}$,

\item $\distg{\lineg{G}}{e_1}{e_2}\geq 2\Rightarrow e_1\cap e_2=\emptyset$,

\item $\abs{\E{\lineg{G}}}=\frac{1}{2}\sum_{v\in V}\fkt{\deg_G}{v}^2-\abs{E}$.

\end{enumerate}
\end{remark}

\begin{theorem}[see \cite{volkmann2006graphen}]\label{thm4.9}
A connected graph $G$ is isomorphic to its line graph if and only if $G$ is a cycle. 
\end{theorem} 

\begin{theorem}[Whitney. 1992 \cite{whitney1992congruent}]\label{thm4.12}
If $G$ and $G'$ are graphs with $\lineg{G}\cong\lineg{G'}$ then either $G\cong G'$ or $G\cong K_{1,3}$ and $G'\cong K_3$.
\end{theorem}

\begin{theorem}[Krausz. 1943 \cite{krausz1943demonstration}]\label{thm4.10}
A graph $G$ is the line graph of a graph $H$ if and only if there exists a decomposition of $G$ into complete and edge disjoint subgraphs such that each vertex of $G$ is contained in at most two of those subgraphs.
\end{theorem}

\begin{corollary}\label{cor4.1}
If $G$ is a line graph, it does not contain a $K_{1,3}$.
\end{corollary}

\begin{theorem}[Harary. 1969 \cite{harary6graph}]\label{thm4.11}
A graph is the line graph of a tree if and only if it is a block graph in which each separating vertex is contained in exactly two blocks.
\end{theorem}

Note that this theorem gives us the chordality of line graphs of trees. I.e. with $T$ being a tree and thus $\lineg{T}$ being a block graph, $\lineg{T}$ cannot contain a sunflower which is not suspended. Hence $\lineg{T}^k$ is chordal for all $k\in\N$. Later we will see an alternative proof for this fact.\\
We conclude this brief summary of results on line graphs with another well known fact, which will be an important part of further investigations.

\begin{lemma}\label{lemma4.2}
A graph $G$ contains a, not necessary induced, cycle $C$ of length $\abs{C}\geq 4$ if and only if $\lineg{G}$ contains an induced cycle $C_L$ with $\abs{C_L}=\abs{C}$.  
\end{lemma}

\begin{proof}
Let $C_L$ be an induced cycle of length $c=\abs{C_L}\geq 4$ in $\lineg{G}$. Then there exist distinct edges $e_1,\dots,e_c$ in $G$ corresponding to the vertices of $C_L$ in $\lineg{G}$. For all $i\in\set{1,\dots, c}$ the edges $e_i$ and $e_j$, with $j=i\pm1\lb \!\!\!\mod c\rb$, are adjacent, therefore share a common vertex. On the other hand, if $j\neq i\pm1\lb \!\!\!\mod c\rb$, $\distg{\lineg{G}}{e_i}{e_j}\geq 2$ holds, hence $e_i$ and $e_j$ are disjoint. So $e_1,\dots,e_c$ form a cycle of length $c$ in $G$.\\
Now let $C$ be a cycle of length $c=\abs{C}\geq 4$ in $G$. The line graph of $C$ is a cycle $C_L$ itself and its edges correspond to the vertices of $C_L$. Suppose $C_L$ has a chord in $\lineg{G}$, then there exists an edge joining $e_i$ and $e_j$ with $j\neq i\pm1\lb\!\!\!\mod c\rb$ in the line graph, hence $e_i$ and $e_j$ share a common vertex. This contradicts $C$ being a cycle in $G$ and the proof is complete. 
\end{proof} 

\begin{corollary}\label{cor4.2}
The line graph of a graph $G$ is chordal if and only if $\fkt{\operatorname{girth}}{G}\leq 3$.
\end{corollary}

\begin{corollary}\label{cor4.3}
The line graph of a graph $G$ does not contain an induced cycle of length $\geq5$ if and only if $\fkt{\operatorname{girth}}{G}\leq4$.
\end{corollary}

\begin{corollary}\label{cor4.4}
If $G$ is a graph with $\fkt{\operatorname{girth}}{G}\leq3$, then $\lineg{G}^k$ is chordal for all $k\in\N$.
\end{corollary}

\begin{proof}
With $\fkt{\operatorname{girth}}{G}\leq3$ the line graph $\lineg{G}$ does not contain an induced cycle of length $\geq 4$ thus it is chordal, i.e. $\lineg{G}$ cannot contain a non-chordal sunflower of size $\geq4$, and by being a line graph $\lineg{G}$ does not contain a $K_{1,3}$ as well. Hence by Corollary \autoref{cor3.4} $\lineg{G}^2$ is chordal and Duchet's Theorem yields the chordality of $\lineg{G}^k$ for all $k\in\N$. 
\end{proof}

\subsection{Chordality of the Squared Line Graph}

Corollary \autoref{cor4.2} provides the application of Corollary \autoref{cor3.4} to the line graph. However the class of graphs satisfying the conditions of Corollary \autoref{cor4.2} is not very interesting in this context. So the goal of this section will be to broaden the spectrum of graphs with chordal line graph squares.\\
We will start our investigation by looking into structural reasons for the appearance of sunflowers in line graphs.
\begin{definition}[Sunflower Sprout]
A {\em sunflower sprout} of size $n$ is a graph $S=\lb U\cup V,E\rb$, in which the vertices of $V$ induce a cycle of length $n$ and $U=\set{u_1,\dots,u_n}$ is a set of, not necessarily dissimilar, vertices with $V\cap U=\emptyset$ such that the following holds:
\vspace{-1mm}
\begin{enumerate}[i)]

\item $j\neq i\pm1\lb \!\!\!\mod n\rb\Rightarrow u_i\neq u_j$ and
\vspace{-1mm}
\item $v_iu_i\in E$ for all $i\in\set{1,\dots,n}$.
\end{enumerate}
\vspace{-1mm}
The family of all sunflower sprouts of size $n$ is denoted by $\mathcal{S}_n$.\\
The sunflower sprout is called {\em infertile} if there exists the edge $u_iu_j$ or $u_iv_j$ for an $i\in\set{1,\dots,n}$ and $j\neq i\pm1\lb \!\!\!\mod n\rb$.
\end{definition}
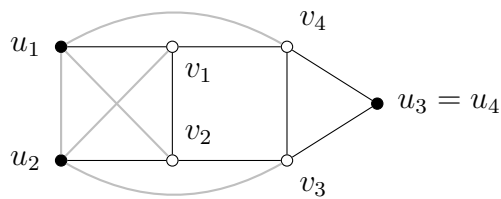
\begin{figure}[H]
\begin{center}
\begin{tikzpicture}

\node (center) [inner sep=1.5pt] {};

\node (label) [inner sep=1.5pt,position=135:1.7cm from center] {};

\node (v1) [inner sep=1.5pt,position=135:0.9cm from center,draw,circle] {};
\node (v2) [inner sep=1.5pt,position=225:0.9cm from center,draw,circle] {};
\node (v3) [inner sep=1.5pt,position=315:0.9cm from center,draw,circle] {};
\node (v4) [inner sep=1.5pt,position=45:0.9cm from center,draw,circle] {};

\node (u1) [inner sep=1.5pt,position=180:1.3cm from v1,draw,circle,fill] {};
\node (u2) [inner sep=1.5pt,position=180:1.3cm from v2,draw,circle,fill] {};
\node (u3) [inner sep=1.5pt,position=0:1.8cm from center,draw,circle,fill] {};

\node (lu1) [position=180:0.04cm from u1] {$u_1$};
\node (lu2) [position=180:0.04cm from u2] {$u_2$};
\node (lu3) [position=0:0.04cm from u3] {$u_3=u_4$};

\node (lv1) [position=315:0.03cm from v1] {$v_1$};
\node (lv2) [position=45:0.03cm from v2] {$v_2$};
\node (lv3) [position=315:0.03cm from v3] {$v_3$};
\node (lv4) [position=45:0.03cm from v4] {$v_4$};

\path
(u1) edge (v1)
(u2) edge (v2)
(u3) edge (v3)
	 edge (v4)
; 

\path 
(v1) edge (v2)
(v2) edge (v3)
(v3) edge (v4)
(v4) edge (v1)
;

\path [color=lightgray,thick]
(u2) edge [bend right] (v3)
	 edge (v1)
(u1) edge (v2)
	 edge [bend left] (v4)
	 edge (u2)
;

\end{tikzpicture}
\vspace{-1mm}
\caption{A sunflower sprout of size $4$ with $u_3=u_4$.}
\label{fig4.1}
\end{center}
\end{figure}
\vspace{-2em}
Note that the sunflower sprout given in \autoref{fig4.1} only represents the basic structure. The four gray edges may exist in any combination, but are not mandatory for the graph to be proper sunflower sprout (i.e. their existence does not make the sprout infertile). Hence there are $16$ different sunflower sprouts in total that have the black edges of the graph above as a basis.
\vspace{-0.5mm}
\newpage
\begin{lemma}\label{lemma4.3}
The line graph $\lineg{G}$ of a graph $G$ contains a not suspended, non-chordal sunflower of size $n$ if and only if $G$ contains a fertile sunflower sprout of size $n$.
\end{lemma}
\vspace{-1.5mm}
\begin{proof}
Suppose $G$ contains a fertile sunflower sprout of size $n$. The line graph of such a sprout contains a non-chordal sunflower of the same size. For this sunflower to be suspended there must exist an edge, corresponding to a vertex in the line graph, adjacent to edges $u_iv_i$ and $u_jv_j$ with $j\neq i\pm1\lb \!\!\!\mod n\rb$. With $\lb v_1,\dots,v_n\rb$ being an induced cycle, the edge $v_iv_j$ cannot exist. All other edges, namely $v_iu_j$, $v_ju_i$ and $u_iu_j$, would result in the sunflower sprout being infertile.\\
\\
Now let $S_n$ be an unsuspended, non-chordal sunflower in $\lineg{G}$. The cycle $C_n$ of this sunflower is induced, i.e. it does not contain a chord, and thus, by Lemma \autoref{lemma4.2}, $G$ contains a cycle $C$ of the same length.\\
For every two adjacent edges on $C$ there must exist an additional edge $e_i$, adjacent to both of them, which corresponds to a $u$-vertex $u_i$ of $S_n$ in the line graph. This edge cannot be adjacent to another such edge $e_j$ with $j\neq i\pm1\lb \!\!\!\mod n\rb$, otherwise $u_i$ and $u_j$ would be adjacent in $\lineg{G}$ and $S_n$ would not be a sunflower.\\
It is left to show that $C$ does not contain a chord. Suppose there is a chord $e$ in $C$, then $e$ is adjacent to two edges $e_i$ and $e_j$, $j\neq i\pm 1\lb \!\!\!\mod n\rb$ and therefore has to correspond to a vertex $w$ adjacent to $u_i$ and $u_j$ in the line graph and the sunflower $S_n$ is suspended, resulting in a contradiction. 
\end{proof}
\vspace{-1mm}
For sunflower sprouts of size $4$ there are very limited variants. Although there are numerous possible combinations of additional edges, displayed in gray in \autoref{fig4.2}, all possible sprouts can be reduced to three different basic structures. As seen in the sunflower sprout of \autoref{fig4.1} two consecutive $u$-vertices can be joined into one, forming a triangle with an edge of the induced circle. A sunflower sprout of size $4$ can contain none, one or two of those triangles, thus limiting $\mathcal{S}_4$ to the following $4113$ graphs. 
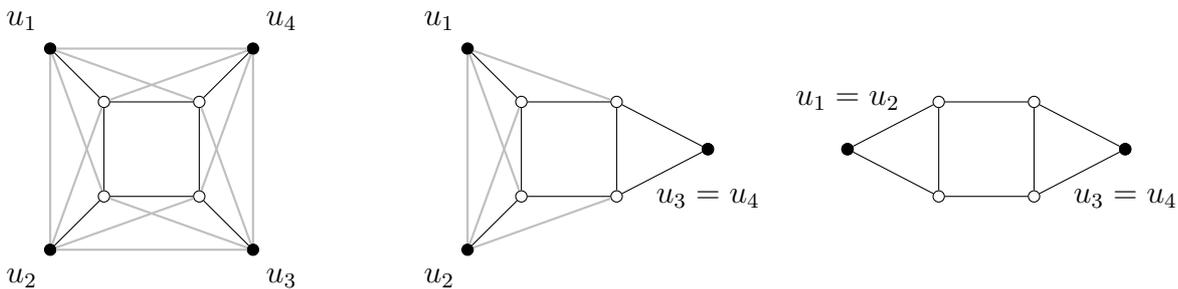
\begin{figure}[H]
\begin{center}
\begin{tikzpicture}

\node (anchor1) [] {};
\node (anchor2) [position=0:5.2cm from anchor1] {};
\node (anchor3) [position=0:5.2cm from anchor2] {};

\node (v1) [inner sep=1.5pt,position=135:0.6cm from anchor1,draw,circle] {};
\node (v2) [inner sep=1.5pt,position=225:0.6cm from anchor1,draw,circle] {};
\node (v3) [inner sep=1.5pt,position=315:0.6cm from anchor1,draw,circle] {};
\node (v4) [inner sep=1.5pt,position=45:0.6cm from anchor1,draw,circle] {};

\path
(v1) edge (v2)
(v2) edge (v3)
(v3) edge (v4)
(v4) edge (v1)
;

\node (u1) [inner sep=1.5pt,position=135:1.6cm from anchor1,draw,circle,fill] {};
\node (u2) [inner sep=1.5pt,position=225:1.6cm from anchor1,draw,circle,fill] {};
\node (u3) [inner sep=1.5pt,position=315:1.6cm from anchor1,draw,circle,fill] {};
\node (u4) [inner sep=1.5pt,position=45:1.6cm from anchor1,draw,circle,fill] {};

\path
(u1) edge (v1)
(u2) edge (v2)
(u3) edge (v3)
(u4) edge (v4)
;

\path[color=lightgray, thick]
(u1) edge (u2)
(u2) edge (u3)
(u3) edge (u4)
(u4) edge (u1)

(u1) edge (v2)
	 edge (v4)
(u2) edge (v3)
	 edge (v1)
(u3) edge (v4)
	 edge (v2)
(u4) edge (v1)
	 edge (v3)
;

\node (lu1) [position=135:0.07 from u1] {$u_1$};
\node (lu2) [position=225:0.07 from u2] {$u_2$};
\node (lu3) [position=315:0.07 from u3] {$u_3$};
\node (lu4) [position=45:0.07 from u4] {$u_4$};


\node (v1) [inner sep=1.5pt,position=135:0.6cm from anchor2,draw,circle] {};
\node (v2) [inner sep=1.5pt,position=225:0.6cm from anchor2,draw,circle] {};
\node (v3) [inner sep=1.5pt,position=315:0.6cm from anchor2,draw,circle] {};
\node (v4) [inner sep=1.5pt,position=45:0.6cm from anchor2,draw,circle] {};

\path
(v1) edge (v2)
(v2) edge (v3)
(v3) edge (v4)
(v4) edge (v1)
;

\node (u1) [inner sep=1.5pt,position=135:1.6cm from anchor2,draw,circle,fill] {};
\node (u2) [inner sep=1.5pt,position=225:1.6cm from anchor2,draw,circle,fill] {};
\node (u3) [inner sep=1.5pt,position=0:1.6cm from anchor2,draw,circle,fill] {};
\node (u4) [inner sep=1.5pt,position=0:1.6cm from anchor2,draw,circle,fill] {};

\path
(u1) edge (v1)
(u2) edge (v2)
(u3) edge (v3)
(u4) edge (v4)
;

\path[color=lightgray, thick]
(u1) edge (u2)

(u1) edge (v2)
	 edge (v4)
(u2) edge (v3)
	 edge (v1)
;

\node (lu1) [position=135:0.07 from u1] {$u_1$};
\node (lu2) [position=225:0.07 from u2] {$u_2$};
\node (lu3) [position=270:0.3 from u3] {$u_3=u_4$};


\node (v1) [inner sep=1.5pt,position=135:0.6cm from anchor3,draw,circle] {};
\node (v2) [inner sep=1.5pt,position=225:0.6cm from anchor3,draw,circle] {};
\node (v3) [inner sep=1.5pt,position=315:0.6cm from anchor3,draw,circle] {};
\node (v4) [inner sep=1.5pt,position=45:0.6cm from anchor3,draw,circle] {};

\path
(v1) edge (v2)
(v2) edge (v3)
(v3) edge (v4)
(v4) edge (v1)
;

\node (u1) [inner sep=1.5pt,position=180:1.6cm from anchor3,draw,circle,fill] {};
\node (u2) [inner sep=1.5pt,position=180:1.6cm from anchor3,draw,circle,fill] {};
\node (u3) [inner sep=1.5pt,position=0:1.6cm from anchor3,draw,circle,fill] {};
\node (u4) [inner sep=1.5pt,position=0:1.6cm from anchor3,draw,circle,fill] {};

\path
(u1) edge (v1)
(u2) edge (v2)
(u3) edge (v3)
(u4) edge (v4)
;

\node (lu1) [position=90:0.3 from u1] {$u_1=u_2$};
\node (lu3) [position=270:0.3 from u3] {$u_3=u_4$};

\end{tikzpicture}
\vspace{-1mm}
\end{center}
\caption{Sunflower sprouts of size $4$.}
\label{fig4.2}
\end{figure}
\vspace{-1mm}
\newpage
An obvious first analogy to Corollary \autoref{cor3.4} is the following, in addition we just have to deal with one particular sunflower sprout, seen in \autoref{fig4.3}.

\begin{figure}[H]
\begin{center}
\begin{tikzpicture}

\node (anchor1) [] {};

\node (v1) [inner sep=1.5pt,position=135:0.6cm from anchor1,draw,circle] {};
\node (v2) [inner sep=1.5pt,position=225:0.6cm from anchor1,draw,circle] {};
\node (v3) [inner sep=1.5pt,position=315:0.6cm from anchor1,draw,circle] {};
\node (v4) [inner sep=1.5pt,position=45:0.6cm from anchor1,draw,circle] {};

\path
(v1) edge (v2)
(v2) edge (v3)
(v3) edge (v4)
(v4) edge (v1)
;

\node (u1) [inner sep=1.5pt,position=135:1.6cm from anchor1,draw,circle,fill] {};
\node (u2) [inner sep=1.5pt,position=225:1.6cm from anchor1,draw,circle,fill] {};
\node (u3) [inner sep=1.5pt,position=315:1.6cm from anchor1,draw,circle,fill] {};
\node (u4) [inner sep=1.5pt,position=45:1.6cm from anchor1,draw,circle,fill] {};

\path
(u1) edge (v1)
(u2) edge (v2)
(u3) edge (v3)
(u4) edge (v4)
;

\node (lu1) [position=135:0.07 from u1] {$u_1$};
\node (lu2) [position=225:0.07 from u2] {$u_2$};
\node (lu3) [position=315:0.07 from u3] {$u_3$};
\node (lu4) [position=45:0.07 from u4] {$u_4$};

\end{tikzpicture}
\end{center}
\caption{The only sunflower sprout of size $4$, $S'_4$, with $\fkt{\operatorname{girth}}{S_4'}\leq 4$.}
\label{fig4.3}
\end{figure}

\begin{corollary}\label{cor4.5}
Let $G$ be a graph with $\fkt{\operatorname{girth}}{G}\leq4$, then $\lineg{G}^2$ is chordal if and only if $G$ does not contain the sunflower sprout $S_4'$ in \autoref{fig4.3}.
\end{corollary}

Out next goal will be a relaxation of the condition on the girth of $G$ in the corollaries above. With this we will be able to give a complete analog to Corollary \autoref{cor3.4} for the squared line graph.

\begin{lemma}\label{lemma4.4}
Let $G$ be a graph that does not contain a $C_n$ with $n\geq 5$. If $\lineg{G}^2$ is not chordal, $G$ contains a fertile sunflower sprout $S'\in\mathcal{S}_4$. 
\end{lemma}

\begin{proof}
With $\lineg{G}^2$ not being chordal there exists an induced cycle $C_L^2=\lb u_1\dots u_c\rb$ with $c=\abs{C_L^2}\geq 4$ in $\lineg{G}^2$.\\
As in the proof of Lemma \autoref{lemma3.8}, this proof is divided into subcases as follows.
\begin{enumerate}
\item[] \begin{enumerate}
	\item [Case 1] The set $\set{u_1,\dots,u_c}$ is stable in $\lineg{G}$.
		\begin{enumerate}
		\item[] \begin{enumerate}
				
			\item [Subcase 1.1] The resulting cycle $C_L$ in $\lineg{G}$ is induced.
			\item [Subcase 1.2] $C_L$ contains a chord.
		\end{enumerate}
		\end{enumerate}

	\item [Case 2] The edge $u_1u_2$ exists in $\lineg{G}$.
		\begin{enumerate}
		\item[] \begin{enumerate}
				
			\item [Subcase 2.1] The resulting cycle $C_L$ in $\lineg{G}$ is induced.
			\item [Subcase 2.2] $C_L$ contains a chord.
		\end{enumerate}
		\end{enumerate}
\end{enumerate}
\end{enumerate}

Subcase 1.1: The set $\set{u_1,\dots,u_c}$ is stable in $\lineg{G}$ and the resulting cycle $C_L$ in $\lineg{G}$ is induced.\\
We get $\distg{\lineg{G}}{u_i}{u_j}=2$ for all $j= i\pm1\lb \!\!\!\mod c\rb$ and thus a set $W=\set{w_1^,\dots,w_c}$ exists, which realizes the paths of length $2$ between the $u_i$. Hence the $u$-vertices and the $w$-vertices together form a cycle $C_L$ of length $2\,c$ in $\lineg{G}$.\\
Suppose $C_L$ is an induced cycle, then, by Lemma \autoref{lemma4.2} $G$ contains a cycle $C_G$ of the same length. With $c\geq 4$ the cycle $C_G$ has at least length $8$ and thus must contain at least two chords. The edges of $C_G$ correspond to the vertices of $C_L$ and therefore alternate between $u$- and $w$-edges.

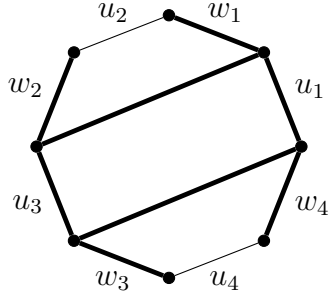
\begin{figure}[H]
	\begin{center}
		\begin{tikzpicture}
		
		\node (center) [inner sep=1.5pt] {};
		
		\node (u1) [draw,circle,fill,inner sep=1.5pt,position=0:1.6cm from center] {};
		\node (u2) [draw,circle,fill,inner sep=1.5pt,position=45:1.6cm from center] {};
		\node (u3) [draw,circle,fill,inner sep=1.5pt,position=90:1.6cm from center] {};
		\node (u4) [draw,circle,fill,inner sep=1.5pt,position=135:1.6cm from center] {};
		\node (u5) [draw,circle,fill,inner sep=1.5pt,position=180:1.6cm from center] {};
		\node (u6) [draw,circle,fill,inner sep=1.5pt,position=225:1.6cm from center] {};
		\node (u7) [draw,circle,fill,inner sep=1.5pt,position=270:1.6cm from center] {};
		\node (u8) [draw,circle,fill,inner sep=1.5pt,position=315:1.6cm from center] {};
		
		\node (l1) [position=22.5:1.57cm from center] {$u_1$};
		\node (l2) [position=67.5:1.57cm from center] {$w_1$};
		\node (l3) [position=112.5:1.57cm from center] {$u_2$};
		\node (l4) [position=157.5:1.57cm from center] {$w_2$};
		\node (l5) [position=202.5:1.57cm from center] {$u_3$};
		\node (l6) [position=247.5:1.57cm from center] {$w_3$};
		\node (l7) [position=292.5:1.57cm from center] {$u_4$};
		\node (l8) [position=337.5:1.57cm from center] {$w_4$};
		
		\path
		(u1) edge [line width=1.8pt] (u2)
		(u2) edge [line width=1.8pt] (u3)
		(u3) edge (u4)
		(u4) edge [line width=1.8pt] (u5)
		(u5) edge [line width=1.8pt] (u6)
		(u6) edge [line width=1.8pt] (u7)
		(u7) edge (u8)
		(u8) edge [line width=1.8pt] (u1)
		;
		
		\path
		(u1) edge [line width=1.8pt] (u6)
		(u2) edge [line width=1.8pt] (u5)
		;
		\end{tikzpicture}
	\end{center}
	\caption{$C_G$ with $c=4$, wo chords and a sunflower sprout.}
	\label{fig4.4}
\end{figure}
If two chords of $C_G$ form a cycle of length $4$ together with two edges of $C_G$, alternating between chords and original edges of the cycle, they form a sunflower sprout of size $4$ as seen in \autoref{fig4.4}. So suppose no two chords form such an alternating $4$-cycle. Thus there exist two chords in $C_G$ sharing a common endpoint which belongs to an edge $u_i$ and one of the other two endpoints belongs to another edge  $u_j$ with $\distg{C_L}{u_i}{u_j}\geq 4$. Hence either one of those chords results in a path of length $2$ from $u_i$ to $u_j$ in $\lineg{G}$ which furthermore results in a chord in $C_L^2$ or every chord skips exactly one $w$-edge, resulting in another induced cycle of length at least $c$. If $c=4$ this cycle is part of a fertile sunflower sprout and for $c\geq 5$ this is a contradiction. Thus Case 1.1 is closed.
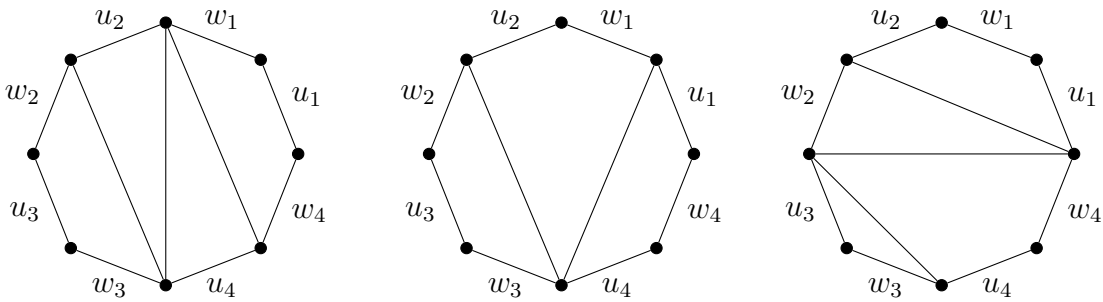
\begin{figure}[H]
\begin{center}
\begin{tikzpicture}
\node (anchor) [] {};
\node (center) [inner sep=1.5pt] {};

\node (u1) [draw,circle,fill,inner sep=1.5pt,position=0:1.6cm from center] {};
\node (u2) [draw,circle,fill,inner sep=1.5pt,position=45:1.6cm from center] {};
\node (u3) [draw,circle,fill,inner sep=1.5pt,position=90:1.6cm from center] {};
\node (u4) [draw,circle,fill,inner sep=1.5pt,position=135:1.6cm from center] {};
\node (u5) [draw,circle,fill,inner sep=1.5pt,position=180:1.6cm from center] {};
\node (u6) [draw,circle,fill,inner sep=1.5pt,position=225:1.6cm from center] {};
\node (u7) [draw,circle,fill,inner sep=1.5pt,position=270:1.6cm from center] {};
\node (u8) [draw,circle,fill,inner sep=1.5pt,position=315:1.6cm from center] {};

\node (l1) [position=22.5:1.57cm from center] {$u_1$};
\node (l2) [position=67.5:1.57cm from center] {$w_1$};
\node (l3) [position=112.5:1.57cm from center] {$u_2$};
\node (l4) [position=157.5:1.57cm from center] {$w_2$};
\node (l5) [position=202.5:1.57cm from center] {$u_3$};
\node (l6) [position=247.5:1.57cm from center] {$w_3$};
\node (l7) [position=292.5:1.57cm from center] {$u_4$};
\node (l8) [position=337.5:1.57cm from center] {$w_4$};

\path
(u1) edge (u2)
(u2) edge (u3)
(u3) edge (u4)
(u4) edge (u5)
(u5) edge (u6)
(u6) edge (u7)
(u7) edge (u8)
(u8) edge (u1)
;

\path
(u8) edge (u3)
(u7) edge (u4)
(u7) edge (u3)
;

\node (center) [inner sep=1.5pt,position=0:5cm from anchor] {};

\node (u1) [draw,circle,fill,inner sep=1.5pt,position=0:1.6cm from center] {};
\node (u2) [draw,circle,fill,inner sep=1.5pt,position=45:1.6cm from center] {};
\node (u3) [draw,circle,fill,inner sep=1.5pt,position=90:1.6cm from center] {};
\node (u4) [draw,circle,fill,inner sep=1.5pt,position=135:1.6cm from center] {};
\node (u5) [draw,circle,fill,inner sep=1.5pt,position=180:1.6cm from center] {};
\node (u6) [draw,circle,fill,inner sep=1.5pt,position=225:1.6cm from center] {};
\node (u7) [draw,circle,fill,inner sep=1.5pt,position=270:1.6cm from center] {};
\node (u8) [draw,circle,fill,inner sep=1.5pt,position=315:1.6cm from center] {};

\node (l1) [position=22.5:1.57cm from center] {$u_1$};
\node (l2) [position=67.5:1.57cm from center] {$w_1$};
\node (l3) [position=112.5:1.57cm from center] {$u_2$};
\node (l4) [position=157.5:1.57cm from center] {$w_2$};
\node (l5) [position=202.5:1.57cm from center] {$u_3$};
\node (l6) [position=247.5:1.57cm from center] {$w_3$};
\node (l7) [position=292.5:1.57cm from center] {$u_4$};
\node (l8) [position=337.5:1.57cm from center] {$w_4$};

\path
(u1) edge (u2)
(u2) edge (u3)
(u3) edge (u4)
(u4) edge (u5)
(u5) edge (u6)
(u6) edge (u7)
(u7) edge (u8)
(u8) edge (u1)
;

\path
(u7) edge (u4)
(u7) edge (u2)
;

\node (center) [inner sep=1.5pt,position=0:10cm from anchor] {};

\node (u1) [draw,circle,fill,inner sep=1.5pt,position=0:1.6cm from center] {};
\node (u2) [draw,circle,fill,inner sep=1.5pt,position=45:1.6cm from center] {};
\node (u3) [draw,circle,fill,inner sep=1.5pt,position=90:1.6cm from center] {};
\node (u4) [draw,circle,fill,inner sep=1.5pt,position=135:1.6cm from center] {};
\node (u5) [draw,circle,fill,inner sep=1.5pt,position=180:1.6cm from center] {};
\node (u6) [draw,circle,fill,inner sep=1.5pt,position=225:1.6cm from center] {};
\node (u7) [draw,circle,fill,inner sep=1.5pt,position=270:1.6cm from center] {};
\node (u8) [draw,circle,fill,inner sep=1.5pt,position=315:1.6cm from center] {};

\node (l1) [position=22.5:1.57cm from center] {$u_1$};
\node (l2) [position=67.5:1.57cm from center] {$w_1$};
\node (l3) [position=112.5:1.57cm from center] {$u_2$};
\node (l4) [position=157.5:1.57cm from center] {$w_2$};
\node (l5) [position=202.5:1.57cm from center] {$u_3$};
\node (l6) [position=247.5:1.57cm from center] {$w_3$};
\node (l7) [position=292.5:1.57cm from center] {$u_4$};
\node (l8) [position=337.5:1.57cm from center] {$w_4$};

\path
(u1) edge (u2)
(u2) edge (u3)
(u3) edge (u4)
(u4) edge (u5)
(u5) edge (u6)
(u6) edge (u7)
(u7) edge (u8)
(u8) edge (u1)
;

\path
(u1) edge (u4)
(u1) edge (u5)
(u5) edge (u7)
;
\end{tikzpicture}
\end{center}
\caption{Three examples for $C_G$ with $c=4$ and chords with a common endpoint.}
\label{fig4.5}
\end{figure}

Subcase 1.2: $C_L$ contains a chord.
So $C_L$ is no induced cycle. A chord in $C_L$ must not be of the form $u_iu_j$ and a chord of the form $u_iw_j$ with $j\neq i\pm1\lb\!\!\!\mod c\rb$ must not exist either for the existence of such a chord would result in a chord in $C_L^2$. Thus the only legal chords are those joining $w$-vertices.\\
Suppose the chord $w_iw_j$ with $i\in\set{1,\dots,c}$ and $j\neq i\pm1\lb \!\!\!\mod c\rb$ exists in $C_L$. Then the corresponding edges in $C_G$ contain a common vertex $v_1$, furthermore there is a common vertex $v_2$ for $w_i$ and $u_i$ and $w_i$ shares another vertex $v_3$ with $u_{i+1\lb \!\!\!\mod c\rb}$. With $C_L$ being a cycle we get $v_1\neq v_2$ and $v_1\neq v_3$. With $U=\set{u_1,\dots,c_u}$ being stable, $u_i$ and $u_{i+1\lb \!\!\!\mod c\rb}$ cannot share a vertex in $G$, hence $w_i$ must contain $3$ vertices, contradicting $G$ being a simple graph. Thus all chords in $C_L$ are of the form $w_iw_{i+1\lb \!\!\!\mod c\rb}$ for some $i\in\set{1,\dots,c}$.\\
Now suppose $C_L$ contains $k\in\set{1,\dots,c-1}$ of all $c$ of those possible chords. Each of those chords shortens the cycle $C_L$ by one, thus resulting in an induced cycle of length $2\,c-k\geq c+1$ in $\lineg{G}$ and with Lemma \autoref{lemma4.2} $G$ contains a cycle $C_G'$ of the same length and with $c+1\geq 5$ $C_G'$ contains a chord.

\begin{figure}[H]
\begin{center}
\begin{tikzpicture}
\node (anchor) [] {};
\node (center) [inner sep=1.5pt] {};

\node (u1) [draw,circle,fill,inner sep=1.5pt,position=90:1.3cm from center] {};
\node (u2) [draw,circle,fill,inner sep=1.5pt,position=162:1.3cm from center] {};
\node (u3) [draw,circle,fill,inner sep=1.5pt,position=234:1.3cm from center] {};
\node (u4) [draw,circle,fill,inner sep=1.5pt,position=306:1.3cm from center] {};
\node (u5) [draw,circle,fill,inner sep=1.5pt,position=16:1.3cm from center] {};

\node (u6) [draw,circle,fill,inner sep=1.5pt,position=90:2.4cm from center] {};
\node (u7) [draw,circle,fill,inner sep=1.5pt,position=16:2.4cm from center] {};
\node (u8) [draw,circle,fill,inner sep=1.5pt,position=162:2.4cm from center] {};

\node (l1) [position=126:1.15cm from center] {$w_1$};
\node (l2) [position=198:1.25cm from center] {$w_1$};
\node (l3) [position=270:1.25cm from center] {$u_2$};
\node (l4) [position=342:1.25cm from center] {$w_2$};
\node (l5) [position=52:1.15cm from center] {$w_3$};
\node (l6) [position=90:0.07cm from u6] {$u_4$};
\node (l7) [position=166:0.1cm from u7] {$u_3$};
\node (l8) [position=12:0.1cm from u8] {$u_1$};

\path
(u1) edge (u2)
(u2) edge (u3)
(u3) edge (u4)
(u4) edge (u5)
(u5) edge (u1)
;

\path
(u1) edge (u6)
(u5) edge (u7)
(u2) edge (u8)
;

\path[dotted]
(u2) edge (u4)
(u3) edge (u5)
;

\end{tikzpicture}
\end{center}
\caption{The sole possibility for $\abs{C_G'}=5$ with the two legal chords.}
\label{fig4.6}
\end{figure}
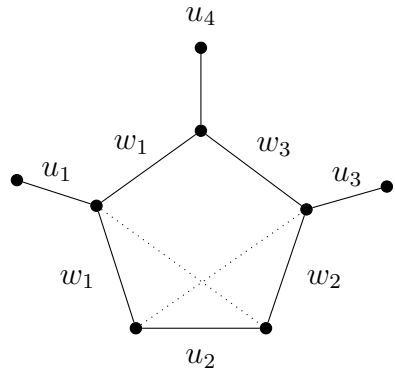 
If $\abs{C_G'}=5$ holds, each legal chord generates a fertile sunflower sprout in $G$ and all possibilities to make such a sunflower sprout infertile are illegal chords.\\
Hence $\abs{C_G'}\geq 6$ must hold and still the only  legal chords are those edges $e$ joining vertices of $C_G'$, such that $e$ is not adjacent to $u_i$ and $u_j$ with $j\neq i\pm1\lb \!\!\!\mod c\rb$.
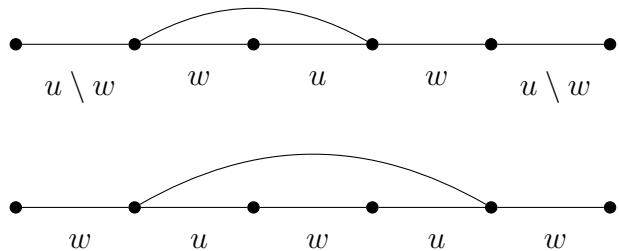
\begin{figure}[H]
\begin{center}
\begin{tikzpicture}
\node (v1) [draw,circle,fill,inner sep=1.5pt] {};
\node (v2) [draw,circle,fill,inner sep=1.5pt,position=0:1.4cm from v1] {};
\node (v3) [draw,circle,fill,inner sep=1.5pt,position=0:1.4cm from v2] {};
\node (v4) [draw,circle,fill,inner sep=1.5pt,position=0:1.4cm from v3] {};
\node (v5) [draw,circle,fill,inner sep=1.5pt,position=0:1.4cm from v4] {};
\node (v6) [draw,circle,fill,inner sep=1.5pt,position=0:1.4cm from v5] {};

\node (e1) [position=0:0.62cm from v1] {};
\node (e2) [position=0:0.62cm from v2] {};
\node (e3) [position=0:0.62cm from v3] {};
\node (e4) [position=0:0.62cm from v4] {};
\node (e5) [position=0:0.62cm from v5] {};

\node (el1) [position=270:0.07 from e1] {$u\setminus w$};
\node (el2) [position=270:0.07 from e2] {$w$};
\node (el3) [position=270:0.07 from e3] {$u$};
\node (el4) [position=270:0.07 from e4] {$w$};
\node (el5) [position=270:0.07 from e5] {$u\setminus w$};

\path
(v1) edge (v2)
(v2) edge (v3)
(v3) edge (v4)
(v4) edge (v5)
(v5) edge (v6)
;

\path
(v2) edge [bend left] (v4)
;

\node (bv1) [draw,circle,fill,inner sep=1.5pt,position=270:2cm from v1] {};
\node (bv2) [draw,circle,fill,inner sep=1.5pt,position=0:1.4cm from bv1] {};
\node (bv3) [draw,circle,fill,inner sep=1.5pt,position=0:1.4cm from bv2] {};
\node (bv4) [draw,circle,fill,inner sep=1.5pt,position=0:1.4cm from bv3] {};
\node (bv5) [draw,circle,fill,inner sep=1.5pt,position=0:1.4cm from bv4] {};
\node (bv6) [draw,circle,fill,inner sep=1.5pt,position=0:1.4cm from bv5] {};

\node (be1) [position=0:0.62cm from bv1] {};
\node (be2) [position=0:0.62cm from bv2] {};
\node (be3) [position=0:0.62cm from bv3] {};
\node (be4) [position=0:0.62cm from bv4] {};
\node (be5) [position=0:0.62cm from bv5] {};

\node (bel1) [position=270:0.07 from be1] {$w$};
\node (bel2) [position=270:0.07 from be2] {$u$};
\node (bel3) [position=270:0.07 from be3] {$w$};
\node (bel4) [position=270:0.07 from be4] {$u$};
\node (bel5) [position=270:0.07 from be5] {$w$};

\path
(bv1) edge (bv2)
(bv2) edge (bv3)
(bv3) edge (bv4)
(bv4) edge (bv5)
(bv5) edge (bv6)
;

\path
(bv2) edge [bend left] (bv5)
;

\end{tikzpicture}
\end{center}
\caption{The two types of legal chords in $C_G'$.}
\label{fig4.7}
\end{figure}

\autoref{fig4.7} displays the two types of legal chords. It follows that such a chord, shortening the cycle by $1$ or $2$, always skips exactly one $w$-edge and at least one $u$-edge. Each induced cycle consisting of edges of the cycle $C_G'$ and its chords use $u$-edges, $w$-edges and chords of the two types. For $c=4$ we already closed the case $k=3$, so the cases $k=2$ and $k=1$ remain. If $k=2$ and there exists a chord of the $u$-$w$-$u$ type this results in a fertile sunflower sprout of size $4$, so suppose only chords of the $u$-$w$ type exist. If just one of those chords exists an induced cycle of length $5$ remains in $G$, so both chords must be there and the result, again, is a fertile sunflower sprout of size $4$. The case $k=1$ can be handled analogue.\\
Hence $c\geq 5$. Now every induced cycle either contains a $w$-edge or a chord that skips these edges, either way there are $c$ such edges and so there is an induced cycle of length $\geq c\geq5$ that cannot contain a chord. This contradicts our assumption and closes the first case.

Subcase 2.1: The edge $u_1u_2$ exists in $\lineg{G}$ and the resulting cycle $C_L$ in $\lineg{G}$ is induced.\\
By Lemma \autoref{lemma3.7} the edges $u_2u_3$ and $u_1u_c$ cannot exist. Again we gain vertices $w_1,\dots,w_q$ realizing paths of length $2$ between $u$-vertices that are not adjacent. Corollary \autoref{cor3.3} yields $q\geq \aufr{\frac{c}{2}}$ and we obtain a cycle $C_L$ in $\lineg{G}$ with $\abs{C_L}\geq c+\aufr{\frac{c}{2}}$.

\begin{figure}[H]
\begin{center}
\begin{tikzpicture}

\node (center) [inner sep=1.5pt] {};
\node (anchor1) [inner sep=1.5pt,position=0:0.8cm from center] {};
\node (anchor2) [inner sep=1.5pt,position=180:0.8cm from center] {};

\node (label) [inner sep=1.5pt,position=135:1.7cm from center] {};

\node (u1) [inner sep=1.5pt,position=90:1.6cm from anchor1,draw,circle,fill] {};
\node (u2) [inner sep=1.5pt,position=90:1.6cm from anchor2,draw,circle,fill] {};
\node (u3) [inner sep=1.5pt,position=180:2.5cm from center,draw,circle,fill] {};
\node (u4) [inner sep=1.5pt,position=0:2.5cm from center,draw,circle,fill] {};

\node (v1) [inner sep=1.5pt,position=0:0.8cm from center,draw,circle] {};
\node (v2) [inner sep=1.5pt,position=180:0.8cm from center,draw,circle] {};

\node (lu1) [position=45:0.07cm from u1] {$u_1$};
\node (lu2) [position=135:0.07cm from u2] {$u_2$};
\node (lu3) [position=150:0.07cm from u3] {$u_3$};
\node (lu4) [position=30:0.07cm from u4] {$u_c$};

\node (lv1) [position=270:0.07cm from v1] {$w_1$};
\node (lv2) [position=270:0.07cm from v2] {$w_2$};

\path
(u1) edge (v1)
(u2) edge (u1)
	 edge (v2)
(u3) edge (v2)
(u4) edge (v1)
;

\end{tikzpicture}
\caption{The edge $u_1u_2$ in $\lineg{G}$.}
\label{fig4.8}
\end{center}
\end{figure}
By Assumption $C_L$ is an induced cycle in $\lineg{G}$, so Lemma \autoref{lemma4.2} yields the existence of a corresponding cycle $C_G$ in $G$ of the same length whose edges correspond to the vertices of $C_L$. With $c\geq 4$ we have $\abs{C_L}=\abs{C_G}\geq 6$ and $C_G$ must contain a chord.\\
There are no two consecutive $w$-edges in $C_G$ (i.e. no two $w$-vertices are adjacent in $C_L$), so for every chord each endpoint is adjacent to a $u$-vertex. Thus for every chord there exists a pair $i,j\in\set{1,\dots,c}$ with $j\neq i$ and $\distg{C_L}{u_i}{u_j}\geq 2$ such that both $u_i$ and $u_j$ are adjacent to the chord in $G$. Hence $\distg{\lineg{G}}{u_i}{u_j}=2$ and thus the legal chords are those satisfying $j=i\pm1\lb \!\!\!\mod c\rb$ for all those pairs, otherwise this would result in a chord in $C_L^2$.\\
Let $c=4$ and $q=2$, then all chords in $C_G$ dividing the cycle into two cycles of length $4$ are illegal, thus $C_G$ has to contain at least $2$ chords. \autoref{fig4.9} displays all illegal and the remaining two legal chords in such a $C_G$, those two edges form an alternating and induced $C_4$ and so we obtain a fertile sunflower sprout of size $4$ in $G$.

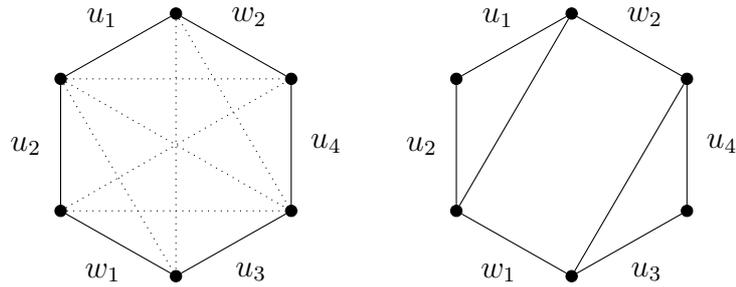
\begin{figure}[H]
\begin{center}
\begin{tikzpicture}
\node (anchor) [] {};
\node (center) [inner sep=1.5pt] {};

\node (u1) [draw,circle,fill,inner sep=1.5pt,position=90:1.6cm from center] {};
\node (u2) [draw,circle,fill,inner sep=1.5pt,position=150:1.6cm from center] {};
\node (u3) [draw,circle,fill,inner sep=1.5pt,position=210:1.6cm from center] {};
\node (u4) [draw,circle,fill,inner sep=1.5pt,position=270:1.6cm from center] {};
\node (u5) [draw,circle,fill,inner sep=1.5pt,position=330:1.6cm from center] {};
\node (u6) [draw,circle,fill,inner sep=1.5pt,position=30:1.6cm from center] {};

\node (l1) [position=120:1.57cm from center] {$u_1$};
\node (l2) [position=180.5:1.57cm from center] {$u_2$};
\node (l3) [position=240.5:1.57cm from center] {$w_1$};
\node (l4) [position=300.5:1.57cm from center] {$u_3$};
\node (l5) [position=0.5:1.57cm from center] {$u_4$};
\node (l6) [position=60.5:1.57cm from center] {$w_2$};

\path
(u1) edge (u2)
(u2) edge (u3)
(u3) edge (u4)
(u4) edge (u5)
(u5) edge (u6)
(u6) edge (u1)
;

\path[dotted]
(u1) edge (u4)
	 edge (u5)
(u2) edge (u4)
	 edge (u5)
	 edge (u6)
(u3) edge (u6)
	 edge (u5)
;

\node (center) [inner sep=1.5pt,position=0:5cm from anchor] {};

\node (u1) [draw,circle,fill,inner sep=1.5pt,position=90:1.6cm from center] {};
\node (u2) [draw,circle,fill,inner sep=1.5pt,position=150:1.6cm from center] {};
\node (u3) [draw,circle,fill,inner sep=1.5pt,position=210:1.6cm from center] {};
\node (u4) [draw,circle,fill,inner sep=1.5pt,position=270:1.6cm from center] {};
\node (u5) [draw,circle,fill,inner sep=1.5pt,position=330:1.6cm from center] {};
\node (u6) [draw,circle,fill,inner sep=1.5pt,position=30:1.6cm from center] {};

\node (l1) [position=120:1.57cm from center] {$u_1$};
\node (l2) [position=180.5:1.57cm from center] {$u_2$};
\node (l3) [position=240.5:1.57cm from center] {$w_1$};
\node (l4) [position=300.5:1.57cm from center] {$u_3$};
\node (l5) [position=0.5:1.57cm from center] {$u_4$};
\node (l6) [position=60.5:1.57cm from center] {$w_2$};

\path
(u1) edge (u2)
(u2) edge (u3)
(u3) edge (u4)
(u4) edge (u5)
(u5) edge (u6)
(u6) edge (u1)
;

\path
(u1) edge (u3)
(u4) edge (u6)
;

\end{tikzpicture}
\end{center}
\caption{The illegal and legal chords in $C_G$ with $\abs{C_G}=6$ and $q=2$.}
\label{fig4.9}
\end{figure}
\vspace{-3mm}
So now suppose either $c\geq 5$ or $q\geq 3$, hence $\abs{C_L}=\abs{C_G}\geq 7$. The conditions on the legality of possible chords obtained above still hold, so there are the two possible types of chords, displayed in \autoref{fig4.7} and such a chord shortens a cycle by $1$ or $2$. In Addition the following chord, shortening a cycle by $1$ may exist.
\begin{figure}[H]
	\begin{center}
		\begin{tikzpicture}
		\node (v1) [draw,circle,fill,inner sep=1.5pt] {};
		\node (v2) [draw,circle,fill,inner sep=1.5pt,position=0:1.4cm from v1] {};
		\node (v3) [draw,circle,fill,inner sep=1.5pt,position=0:1.4cm from v2] {};
		\node (v4) [draw,circle,fill,inner sep=1.5pt,position=0:1.4cm from v3] {};
		\node (v5) [draw,circle,fill,inner sep=1.5pt,position=0:1.4cm from v4] {};
		\node (v6) [draw,circle,fill,inner sep=1.5pt,position=0:1.4cm from v5] {};
		
		\node (e1) [position=0:0.62cm from v1] {};
		\node (e2) [position=0:0.62cm from v2] {};
		\node (e3) [position=0:0.62cm from v3] {};
		\node (e4) [position=0:0.62cm from v4] {};
		\node (e5) [position=0:0.62cm from v5] {};
		
		\node (el1) [position=270:0.07 from e1] {$w$};
		\node (el2) [position=270:0.07 from e2] {$u$};
		\node (el3) [position=270:0.07 from e3] {$u$};
		\node (el4) [position=270:0.07 from e4] {$w$};
		\node (el5) [position=270:0.07 from e5] {$u\setminus w$};
		
		\path
		(v1) edge (v2)
		(v2) edge (v3)
		(v3) edge (v4)
		(v4) edge (v5)
		(v5) edge (v6)
		;
		
		\path
		(v2) edge [bend left] (v4)
		;

		\end{tikzpicture}
	\end{center}
	\caption{The third type of legal chords in $C_G$.}
	\label{fig4.10}
\end{figure}
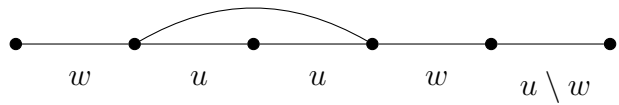
\vspace{-3mm}
As we did before, we start with the case $c=4$ and $q\geq 3$, the case $q=4$ reassembles Case 1.2, so suppose $q=4$. \autoref{fig4.11} shows the resulting cycle $C_G$ and the sole possibility of a chord of the type $u$-$w$-$u$ which results in an induced cycle of length $5$ with no legal chords left to shorten it further.\\
Hence only chords of the $u$-$w$ type are legal and there can exist at most two of them, so at least one $w$-edge and the two consecutive $u$-edges cannot be skipped and again we obtain an induced cycle of length at least $5$. Thus the case $q=3$ cannot occur and we get $c\geq5$.\\
As observed in Case 1.1 no two chords of $C_G$ may form an alternating cycle of length $4$ together with two edges of $C_G$, otherwise a fertile sunflower sprout would be the result. Again each chord, producing a shorter cycle, skips exactly one $w$-edge that was part of the cycle before. Pairs of consecutive $u$-edges may be skipped too and there are exactly $c-q$ such pairs and thus a cycle of length at least $q+\lb c-q\rb= c\geq 5$ exists. A contradiction that closes Case 2.1.
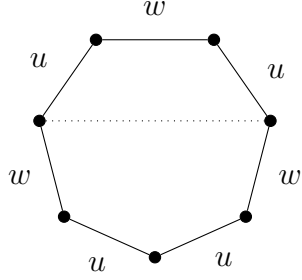
\begin{figure}[H]
\begin{center}
\begin{tikzpicture}

\node (center) [inner sep=1.5pt] {};

\node (label) [inner sep=1.5pt,position=135:1.7cm from center] {};

\node (u1) [inner sep=1.5pt,position=60:1.4cm from center,draw,circle,fill] {};
\node (u2) [inner sep=1.5pt,position=120:1.4cm from center,draw,circle,fill] {};
\node (u3) [inner sep=1.5pt,position=220:1.4cm from center,draw,circle,fill] {};
\node (u4) [inner sep=1.5pt,position=320:1.4cm from center,draw,circle,fill] {};

\node (v1) [inner sep=1.5pt,position=170:1.4cm from center,draw,circle,fill] {};
\node (v2) [inner sep=1.5pt,position=270:1.4cm from center,draw,circle,fill] {};
\node (v3) [inner sep=1.5pt,position=10:1.4cm from center,draw,circle,fill] {};

\node (lu1) [position=90:1.47cm from center] {$w$};
\node (lu2) [position=145:1.47cm from center] {$u$};
\node (lu3) [position=245:1.47cm from center] {$u$};
\node (lu4) [position=345:1.47cm from center] {$w$};

\node (lv1) [position=195:1.47cm from center] {$w$};
\node (lv2) [position=300:1.47cm from center] {$u$};
\node (lv3) [position=30:1.47cm from center] {$u$};

\path
(u1) edge (u2)
(u2) edge (v1)
(v1) edge (u3)
(u3) edge (v2)
(v2) edge (u4)
(u4) edge (v3)
(v3) edge (u1)
;

\path[dotted]
(v1) edge (v3)
;

\end{tikzpicture}
\caption{$C_G$ with $c=4$, $q=3$ and the sole legal $u$-$w$-$u$-type chord.}
\label{fig4.11}
\end{center}
\end{figure}
\vspace{-3mm}

Subcase 2.2: $C_L$ contains a chord.\\
As seen in Subcase 1.2 just chords of the form $w_iw_j$ are allowed in $C_L$, otherwise $C_L^2$ would contain a chord. If $j\neq i\pm 1\lb \!\!\!\mod q\rb$, the two direct neighbors of $w_i$ on $C_L$, $u_i$ and $u_{i+1\lb \!\!\!\mod c\rb}$ are neither adjacent, nor is one of them adjacent to $w_j$, which would be an illegal chord, and so $\set{u_i,u_{i+1\lb \!\!\!\mod c\rb},w_i,w_j}$ induces a $K_{1,3}$ in $\lineg{G}$. Thus reducing the legal chords to edges of the form $w_iw_j$ with $j=i+1\lb \!\!\!\mod q\rb$. The last possible restriction on the legality of chords is an edge $w_iw_j$ with $j=i+1\lb \!\!\!\mod q\rb$ and $\distg{C_L}{w_i}{w_j}=3$, which results in another induced $K_{1,3}$ as seen in \autoref{fig4.12}. Hence a chord in $C_L$ shortens the cycle by exactly $1$ and pose a shortcut skipping exactly one $u$-vertex. 

\begin{figure}[H]
\begin{center}
\begin{tikzpicture}

\node (center) [inner sep=1.5pt] {};
\node (anchor1) [inner sep=1.5pt,position=0:0.8cm from center] {};
\node (anchor2) [inner sep=1.5pt,position=180:0.8cm from center] {};

\node (label) [inner sep=1.5pt,position=135:1.7cm from center] {};

\node (u1) [inner sep=1.5pt,position=90:1.6cm from anchor1,draw,circle,fill] {};
\node (u2) [inner sep=1.5pt,position=90:1.6cm from anchor2,draw,circle,fill] {};
\node (u3) [inner sep=1.5pt,position=180:2.5cm from center,draw,circle,fill] {};
\node (u4) [inner sep=1.5pt,position=0:2.5cm from center,draw,circle,fill] {};

\node (v1) [inner sep=1.5pt,position=0:0.8cm from center,draw,circle] {};
\node (v2) [inner sep=1.5pt,position=180:0.8cm from center,draw,circle] {};

\node (lu1) [position=45:0.07cm from u1] {$u_{k+3\lb\!\!\!\mod c\rb}$};
\node (lu2) [position=135:0.07cm from u2] {$u_{k+2\lb\!\!\!\mod c\rb}$};
\node (lu3) [position=150:0.07cm from u3] {$u_{k+1\lb\!\!\!\mod c\rb}$};
\node (lu4) [position=30:0.07cm from u4] {$u_{u_k}$};

\node (lv1) [position=270:0.07cm from v1] {$w_i$};
\node (lv2) [position=270:0.07cm from v2] {$w_{i+1\lb\!\!\!\mod q\rb}$};

\path
(u1) edge [line width=1.8pt] (v1)
(u2) edge (u1)
	 edge (v2)
(u3) edge (v2)
(u4) edge [line width=1.8pt] (v1)
(v1) edge [line width=1.8pt] (v2)
;

\end{tikzpicture}
\caption{The chord $w_iw_{i+1\lb \!\!\!\mod q\rb}$ and the induced $K_{1,3}$.}
\label{fig4.12}
\end{center}
\end{figure}
\vspace{-3mm}
Thus every chord in $C_L$ joins two $w$-vertices, both adjacent to a common $u$-vertex. Furthermore every pair of adjacent $u$-vertices cannot be skipped by a chord and contribute $3$ edges to every cycle consisting purely of edges of $C_L$ and its chord. In total we get at least $q$ edges, either joining two adjacent $u$-vertices, representing a chord or being one of the two edges that could be skipped by a legal chord. In addition there are $c-q$ pairs of adjacent $u$-vertices and each of them adds another two edges to the length of an induced cycle, even the shortest possible. So any induced cycle $C_L'$ obtained from $C_L$ by shortening with legal chords has a length of at least $\abs{C_L'}\geq q+2\,\lb c-q\rb=2\,c-q\geq c+1$.\\
Thus, by Lemma \autoref{lemma4.2}, a cycle $C_G'$ with $\abs{C_G'}=\abs{C_L'}\geq c+1\geq 5$ exists in $G$ and must contain a chord. The case $c=4$ is similar to Subcase 2.1, so suppose $c\geq5$. Still all legal chords of $C_G'$ are of the form displayed in \autoref{fig4.7} or  in \autoref{fig4.10}. Thus a chord in $C_G'$ is possible if and only if $C_L'$ does contain a consecutive vertex-triple of the form $w$-$u$-$w$, or a tuple of two consecutive $u$-edges. As seen before for every $w$-edges and for every consecutive $u$-tuple there may exist at most one legal chord shortening a cycle, skipping either the tuple or the $w$-edge and thus an induced cycle of at least length $c\geq5$ remains. This closes Case 2 and completes the proof.  
\end{proof}

\begin{lemma}\label{lemma4.5}
Let $G$ be a graph without induced cycles of length $\geq 5$. If $G$ does not contain a fertile sunflower sprout of size $4$, $\lineg{G}^2$ is chordal.
\end{lemma}

By applying Theorem \autoref{thm3.20} to the line graph it is easy to see that Lemma \autoref{lemma4.5} gives a complete description of the subclass of graph without induced cycles of length $\geq 5$ with a chordal line graph square. This leads to the following theorem.

\begin{theorem}\label{thm4.13}
Let $G$ be a graph without induced cycles of length $\geq 5$, then $\lineg{G}^2$ is chordal if and only if $G$ does not contain a fertile sunflower sprout of size $4$.
\end{theorem}

\begin{proof}
If $G$ does not contain a fertile sunflower sprout of size $4$ Lemma \autoref{lemma4.5} yields the chordality of $\lineg{G}^2$.\\
Now suppose $\lineg{G}^2$ is chordal and $G$ contains a fertile sunflower sprout of size $4$. By Lemma \autoref{lemma4.3} $\lineg{G}$ contains an unsuspended, non-chordal sunflower of size $4$, which is a flower of size $4$ that is not withered and therefore a contradiction to Theorem \autoref{thm3.20}.
\end{proof}

While there seems to be no inherent difference between induced cycles and cycles with chords in $G$ for the existence of an induced cycle in the line graph, the above result indicates a stronger relation for the squared line graph. A basic observation on the translation of a chord in $G$ to the line graph gives a nice clue:\\
Let $C$ be a cycle with some chord $e$ and let $a$, $b$, $c$ and $d$ be the edges of $C$ adjacent to $e$. We get $\distg{\lineg{G}}{a}{b}=1$ and $\distg{\lineg{G}}{c}{d}=1$, furthermore suppose $\distg{\lineg{G}}{x}{y}\geq 3$ for $x\in\set{a,b}$ and $y\in\set{c,d}$. \autoref{fig4.13} illustrates this translation.
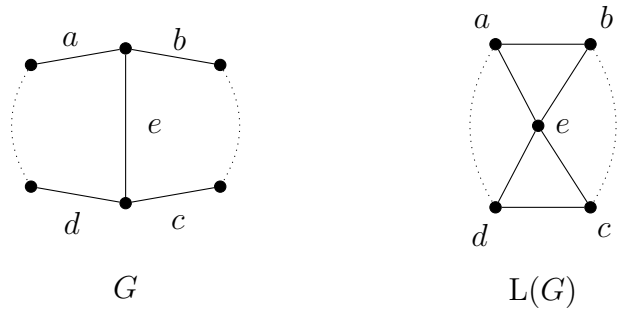
\begin{figure}[H]
	\begin{center}
		\begin{tikzpicture}
		
		\node (anchor1) [] {};
		\node (anchor2) [position=0:5.2cm from anchor1] {};
		
		\node (u1) [draw,circle,fill,inner sep=1.5pt,position=90:0.8cm from anchor1] {};
		\node (u2) [draw,circle,fill,inner sep=1.5pt,position=350:1.1cm from u1] {};
		\node (u3) [draw,circle,fill,inner sep=1.5pt,position=190:1.1cm from u1] {};
		\node (u4) [draw,circle,fill,inner sep=1.5pt,position=270:0.8cm from anchor1] {};
		\node (u5) [draw,circle,fill,inner sep=1.5pt,position=10:1.1cm from u4] {};
		\node (u6) [draw,circle,fill,inner sep=1.5pt,position=170:1.1cm from u4] {};
		
		\path
		(u1) edge (u2)
			 edge (u3)
		(u4) edge (u5)
			 edge (u6)
		;
		
		\path
		(u1) edge (u4)
		;
		
		\node (n) [position=270:1.7cm from anchor1] {$G$};
		
		\node (e) [position=0:0cm from anchor1] {$e$};
		
		\node (a) [position=170:0.4cm from u1] {$a$};
		\node (b) [position=10:0.4cm from u1] {$b$};
		\node (c) [position=340:0.4cm from u4] {$c$};
		\node (d) [position=200:0.4cm from u4] {$d$};
		
		\path
		(u2) edge [bend left,dotted] (u5)
		(u3) edge [bend right, dotted] (u6)
		;
		
		
		\node (u1) [draw,circle,fill,inner sep=1.5pt,position=0:5.2cm from anchor1] {};
		\node (e) [position=0:0cm from u1] {$e$};
		\node (u2) [draw,circle,fill,inner sep=1.5pt,position=60:1cm from anchor2] {};
		\node (b) [position=60:0cm from u2] {$b$};
		\node (u3) [draw,circle,fill,inner sep=1.5pt,position=120:1cm from anchor2] {};
		\node (a) [position=120:0cm from u3] {$a$};
		\node (u4) [draw,circle,fill,inner sep=1.5pt,position=240:1cm from anchor2] {};
		\node (d) [position=240:0cm from u4] {$d$};
		\node (u5) [draw,circle,fill,inner sep=1.5pt,position=300:1cm from anchor2] {};
		\node (c) [position=300:0cm from u5] {$c$};
				
		\path
		(u1) edge (u2)
			 edge (u3)
			 edge (u4)
			 edge (u5)
		(u2) edge (u3)
		(u4) edge (u5)
		;
		
		\node (n) [position=270:1.7cm from anchor2] {$\lineg{G}$};
		
		\path
		(u2) edge [bend left,dotted] (u5)
		(u3) edge [bend right, dotted] (u4)
		;
		
		\end{tikzpicture}
	\end{center}
	\caption{Translation of a chord in $G$ to $\lineg{G}$.}
	\label{fig4.13}
\end{figure}
\vspace{-3mm}
So a chord in a cycle in $G$ poses a path of length $2$ between at least two nonadjacent vertices of the resulting cycle in $\lineg{G}$. This results in a chord in $\lineg{G}^2$.\\
We will prove this observation with the methods used in the proof of Lemma \autoref{lemma4.5}. 

\begin{lemma}\label{lemma4.6}
If $G$ does not contain an induced cycle of length $l\geq f\geq 4$, then $\lineg{G}^2$ does not contain an induced cycle of length $l\geq f$ as well.
\end{lemma}

\begin{proof}
Let a graph $G$ without cycles of length $l'\geq f\geq 4$ be given. Suppose $\lineg{G}^2$ contains an induced cycle $C_{L^2}$ of length $l\geq f$ with vertex set $u_1,\dots,u_l$. Theorem \autoref{thm3.17} yields the existence of a flower $F$ of size $l$ in $\lineg{G}$ with $u=\set{u_1,\dots,u_l}$ and the additional set $W=\set{w_1,\dots,w_q}$.\\
This flower yields a cycle $C_L$ of length $l+q$ on the vertices $U\cup W$, by Lemma \autoref{lemma3.14} a flower is hamiltonian. As seen before, chords in $C_L$ may only join vertices of $W$. Since $\lineg{G}$ is a line graph, it must not contain an induced $K_{1,3}$ and thus a legal chord must join two consecutive $w$-vertices with exactly one $u$-vertex in between. Again such a chord shortens a cycle by at most $1$ and there are at most $q$ chords of this type possible. Hence $\lineg{G}$ contains an induced cycle $C_L'$ of length $\tilde{l}\geq l$ and so, by Lemma \autoref{lemma4.2} $G$ contains a cycle $C_G$ of the same length consisting of $u$- and $w$-edges. We denote its vertex set by $v_1,\dots,v_{\tilde{l}}$.\\
Please note, that in $C_G$ at most two $u$-edges are consecutive and every set of consecutive $u$-edges (including the sets of size one) is lead by a $w$-edge if we go through the cycle in the order given by the definition of flowers (i.e. the ordering of the $u$-vertices corresponding to the $u$-edges in $C_G$). If there are $k$ such "leading" $w$-edges, the length of $C_G$ is exactly $\tilde{l}=l+k$.\\
By assumption $C_G$ contains chords, since $G$ has no induced cycles of length $l\geq f$. But most types of chords would correspond to a chord in $C_{L^2}$, respectively to the withering of the flower $F$. The forbidden chords can be classified in three different types as follows:
\begin{enumerate}[i)]
\item chords of the form $v_iv_j$ with $\abs{i-j}\lb \!\!\!\mod \tilde{l}\rb\geq 4$,

\item chords of the form $v_iv_j$ with $\abs{i-j}\lb \!\!\!\mod \tilde{l}\rb=3$ if they close a cycle of length $4$ either containing two $w$-edges or two consecutive $u$-edges, and

\item chords with one endpoint being adjacent to two consecutive $u$-edges.
\end{enumerate}
Hence the possible chords are of the form 
\begin{enumerate}[i)]

\item $v_{i+1}v_{i+3}$, if $v_{i+1}$ and $v_{i+3}$ are either contained in an alternating $u$-$w$ path or the path $v_{i+1}v_{i+2}v_{i+3}$ consists only of $u$-edges, and,

\item $v_{i+1}v_{i+4}$ but only if they close a $4$-cycle containing the sequence $uwu$. 
\end{enumerate}
\begin{figure}[H]
\begin{center}
\begin{tikzpicture}
\node (v1) [draw,circle,fill,inner sep=1.5pt] {};
\node (v2) [draw,circle,fill,inner sep=1.5pt,position=0:1.4cm from v1] {};
\node (v3) [draw,circle,fill,inner sep=1.5pt,position=0:1.4cm from v2] {};
\node (v4) [draw,circle,fill,inner sep=1.5pt,position=0:1.4cm from v3] {};
\node (v5) [draw,circle,fill,inner sep=1.5pt,position=0:1.4cm from v4] {};
\node (v6) [draw,circle,fill,inner sep=1.5pt,position=0:1.4cm from v5] {};

\node (e1) [position=0:0.62cm from v1] {};
\node (e2) [position=0:0.62cm from v2] {};
\node (e3) [position=0:0.62cm from v3] {};
\node (e4) [position=0:0.62cm from v4] {};
\node (e5) [position=0:0.62cm from v5] {};

\node (el1) [position=270:0.07 from e1] {$u\setminus w$};
\node (el2) [position=270:0.07 from e2] {$w$};
\node (el3) [position=270:0.07 from e3] {$u$};
\node (el4) [position=270:0.07 from e4] {$w$};
\node (el5) [position=270:0.07 from e5] {$u\setminus w$};

\path
(v1) edge [{Latex[length=1.5mm,width=2.5mm]}-] (v2)
(v2) edge [{Latex[length=1.5mm,width=2.5mm]}-] (v3)
(v3) edge [{Latex[length=1.5mm,width=2.5mm]}-] (v4)
(v5) edge [{Latex[length=1.5mm,width=2.5mm]}-] (v6)
;

\draw[snake=coil,segment aspect=0,thick,-{Latex[length=1.5mm,width=2.5mm]}] (v5) -- (v4);

\path
(v2) edge [bend left] (v4)
;

\node (bv1) [draw,circle,fill,inner sep=1.5pt,position=270:2cm from v1] {};
\node (bv2) [draw,circle,fill,inner sep=1.5pt,position=0:1.4cm from bv1] {};
\node (bv3) [draw,circle,fill,inner sep=1.5pt,position=0:1.4cm from bv2] {};
\node (bv4) [draw,circle,fill,inner sep=1.5pt,position=0:1.4cm from bv3] {};
\node (bv5) [draw,circle,fill,inner sep=1.5pt,position=0:1.4cm from bv4] {};
\node (bv6) [draw,circle,fill,inner sep=1.5pt,position=0:1.4cm from bv5] {};

\node (be1) [position=0:0.62cm from bv1] {};
\node (be2) [position=0:0.62cm from bv2] {};
\node (be3) [position=0:0.62cm from bv3] {};
\node (be4) [position=0:0.62cm from bv4] {};
\node (be5) [position=0:0.62cm from bv5] {};

\node (bel1) [position=270:0.07 from be1] {$w$};
\node (bel2) [position=270:0.07 from be2] {$u$};
\node (bel3) [position=270:0.07 from be3] {$u$};
\node (bel4) [position=270:0.07 from be4] {$w$};
\node (bel5) [position=270:0.07 from be5] {$u\setminus w$};

\path
(bv1) edge [{Latex[length=1.5mm,width=2.5mm]}-] (bv2)
(bv2) edge [{Latex[length=1.5mm,width=2.5mm]}-] (bv3)
(bv3) edge [{Latex[length=1.5mm,width=2.5mm]}-] (bv4)
(bv5) edge [{Latex[length=1.5mm,width=2.5mm]}-] (bv6)
;

\draw[snake=coil,segment aspect=0,thick,-{Latex[length=1.5mm,width=2.5mm]}] (bv5) -- (bv4);

\path
(bv2) edge [bend left] (bv4)
;

\node (cv1) [draw,circle,fill,inner sep=1.5pt,position=270:2cm from bv1] {};
\node (cv2) [draw,circle,fill,inner sep=1.5pt,position=0:1.4cm from cv1] {};
\node (cv3) [draw,circle,fill,inner sep=1.5pt,position=0:1.4cm from cv2] {};
\node (cv4) [draw,circle,fill,inner sep=1.5pt,position=0:1.4cm from cv3] {};
\node (cv5) [draw,circle,fill,inner sep=1.5pt,position=0:1.4cm from cv4] {};
\node (cv6) [draw,circle,fill,inner sep=1.5pt,position=0:1.4cm from cv5] {};

\node (ce1) [position=0:0.62cm from cv1] {};
\node (ce2) [position=0:0.62cm from cv2] {};
\node (ce3) [position=0:0.62cm from cv3] {};
\node (ce4) [position=0:0.62cm from cv4] {};
\node (ce5) [position=0:0.62cm from cv5] {};

\node (cel1) [position=270:0.07 from ce1] {$w$};
\node (cel2) [position=270:0.07 from ce2] {$u$};
\node (cel3) [position=270:0.07 from ce3] {$w$};
\node (cel4) [position=270:0.07 from ce4] {$u$};
\node (cel5) [position=270:0.07 from ce5] {$w$};

\path
(cv1) edge [{Latex[length=1.5mm,width=2.5mm]}-] (cv2)
(cv2) edge [{Latex[length=1.5mm,width=2.5mm]}-] (cv3)
(cv4) edge [{Latex[length=1.5mm,width=2.5mm]}-] (cv5)
;

\draw[snake=coil,segment aspect=0,thick,-{Latex[length=1.5mm,width=2.5mm]}] (cv4) -- (cv3);
\draw[snake=coil,segment aspect=0,thick,-{Latex[length=1.5mm,width=2.5mm]}] (cv6) -- (cv5);

\path
(cv2) edge [bend left] (cv5)
;

\end{tikzpicture}
\end{center}
\caption{The two types of legal chords in $C_G$.}
\label{fig4.14}
\end{figure}
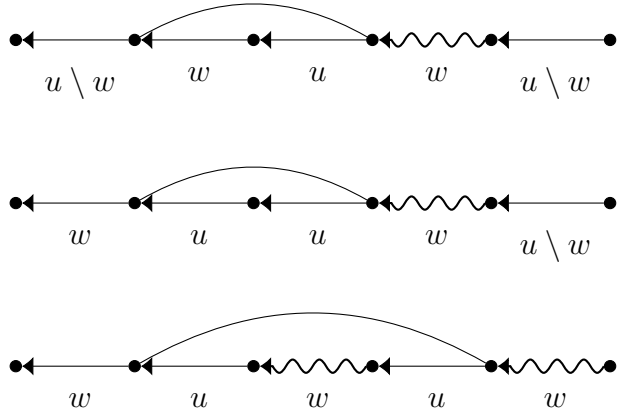
\vspace{-3mm}
\autoref{fig4.14} illustrates those two types of possible chords. Apparently, every chord of type i) reduces the length of $C_G$ by one and is associated with exactly one "leading" $w$-edge. A chord of type ii) reduces the length of $C_G$ by two, but here are two "leading" $w$-edges (zig-zag edges) involved. With this $C_G$ cannot be shortened by more than $k$, since this is the number of "leading" $w$-edges and thus an induced cycle of at least length $l$ remains in $G$, a contradiction.
\end{proof}

Directly we obtain the following theorem as a corollary. This was first proven by Kathie Cameron in 1989 (see \cite{cameron1989induced}) in her investigation of $2$-strong matchings.

\begin{theorem}[Cameron. 1989 \cite{cameron1989induced}]\label{thm4.14}
Let $G$ be chordal, then $\lineg{G}^2$ is chordal.
\end{theorem}

In a similar fashion the general property of $G$ being an intersection graph is inherited by the squared line graph $\lineg{G}^2$.

\begin{theorem}[Cameron. 2004 \cite{cameron2004induced}]\label{thm5.17}
Let $G$ be the intersection graph of some family $\mathscr{F}$, then $\lineg{G}^2$ is the intersection graph of the family
\begin{align*}
\mathscr{F}'\define\condset{u\cup v}{u,v\in\mathscr{F},u\neq v~\text{and}~u\cap v\neq\emptyset}.
\end{align*} 
\end{theorem}

The proof of Lemma \autoref{lemma4.6} gives us another important clue on the structures in a graph $G$ resulting in an induced cycle in $\lineg{G}^2$. To be more precise, $\lineg{G}^2$ contains an induced cycle of length $\geq 4$ if and only if $\lineg{G}$ contains a flower of the same size which is not withered. As a result in order to control cycles in the squared line graph we must find the graphs whose line graphs become flowers, as we did before with the sunflower sprouts. With that we have to generalize the concept of sunflower sprouts.

\begin{definition}[Sprout]
A {\em sprout} of size $n$ is a graph $S=\lb V,U\cup W\cup E\rb$ with $\abs{U}=n$ and $\abs{W}=q$, $U$, $W$ and $E$ pairwise having an empty intersection and $\aufr{\frac{n}{2}}\leq q\leq n$ satisfying the following conditions:
\begin{enumerate}[i)]

\item There is a cycle $C$ with $\E{C}\supseteq W$ containing the edges of $W$ in the order $w_1,\dots,w_q$.

\item The set $U=\set{u_1,\dots,u_n}$ is sorted by the appearance order of its elements along $C$ with $u_1\cap w_q\neq\emptyset$ and $u_2\cap w_1\neq\emptyset$, in addition $u_i\cap u_j=\emptyset$ for $j\neq i\pm 1\lb \!\!\!\mod n\rb$.

\item If $w_i\cap w_{i+1}\neq\emptyset$, then there is exactly one $u\in U$ with $w_i\cap w_{i+1}\cap u\neq\emptyset$, those edges are called {\em pending}.

\item If $w_i\cap w_{i+1}=\emptyset$, then there either is one $u\in U$ connecting $w_i$ and $w_{i+1}$ in $C$, or there are exactly two edges $t,u\in U$, such that the path $w_ituw_{i+1}$ is part of $C$.

\item The pending $u$-edges are pairwise nonadjacent and all $u$-edges that are not pending are edges of $C$.
\end{enumerate}
The family of all sprouts of size $n$ is denoted by $\mathfrak{S}_n$.\\
If a sprout $S$ contains an edge $e\in E$ connecting two non-consecutive $u$-edges, we say $S$ is {\em infertile}, otherwise $S$ is called {\em fertile}.
\end{definition}

\begin{theorem}\label{thm4.16}
Let $G$ be a graph, then $\lineg{G}^2$ contains an induced cycle of length $n$ if and only of $G$ contains a fertile sprout of size $n$.
\end{theorem}

\begin{proof}
In the following, we show that the existence of a fertile sprout is equivalent to the existence of a non-withered flower of the same size in $\lineg{G}$. Then the assertion follows directly from Theorem \autoref{thm3.17}.

Let $S$ be a fertile sprout with edge classes $U$, $W$ and $E$ in $G$, furthermore let $C_G$ be the cycle in $S$ containing all $w$-edges. This cycle may contain chords $w\in E$, which again are of the three types seen in \autoref{fig4.14}. Lemma \autoref{lemma4.2} yields the existence of an induced cycle $C_L'$ in $\lineg{G}$ whose vertices correspond to the edges of $C_G$, which are completely contained in $U\cup W$.\\
Any edge $u\in U$ which is not in $C_G$ is pending and therefore forms a star together with its two adjacent $w$-edges. Hence in $\lineg{G}$ we obtain a triangle. We choose the cycle $C_L'$ as the cycle required in $i)$ of the flower definition and with any pending edge ending up as a $u$-vertex in the line graph which is adjacent to exactly two $w$-vertices that are adjacent themselves conditions $ii)$, $iii)$ and $v)$ are satisfied as well. At last condition $iv)$ of the sprout definition guarantees that a $u$-vertex in $C_L'$ either is adjacent to no other $u$-vertex, or to exactly one. This corresponds to condition $iv)$ of the flower definition which therefore is also satisfied.  So we obtain a flower $F$ with size $\abs{U}$ in $\lineg{G}$. Any vertex in $\lineg{G}$ responsible for $F$ being withered would correspond to an edge $e\in E$, rendering the sprout infertile.\\
\\
So now suppose there is a non-withered flower $F$ in $\lineg{G}$. We define $C_L'$ to be the induced cycle in $F$ containing as much $w$-vertices as possible. If not all $w$-vertices lie on $C_L'$, each cycle (e.g the cycle $C$ containing all $u$- and $w$-vertices except the pending ones) covering all $w$-vertices contains a chord. Such a chord either joins two consecutive $u$-vertices that both are not adjacent to another $u$-vertex or contradicts either the definition of a flower or $\lineg{G}$ being a line graph by generating an induced $K_{1,3}$.\\
Now for the $u$-$u$ chord. If there is another chord joining the skipped $w$-vertex to another, consecutive $w$-vertex, the skipped one can be included in an induced cycle containing all $w$-vertices, if there is no such chord we are in the case displayed in \autoref{fig4.15}.

\begin{figure}[H]
	\begin{center}
		\begin{tikzpicture}
		\node (v1) [draw,circle,fill,inner sep=1.5pt] {};
		\node (v2) [draw,circle,fill,inner sep=1.5pt,position=25:1.2cm from v1] {};
		\node (v3) [draw,circle,fill,inner sep=1.5pt,position=25:1.2cm from v2] {};
		\node (v4) [draw,circle,fill,inner sep=1.5pt,position=335:1.2cm from v3] {};
		\node (v5) [draw,circle,fill,inner sep=1.5pt,position=335:1.2cm from v4] {};
		
		\node (l1) [position=90:0.07 from v1] {$w$};
		\node (l2) [position=90:0.07 from v2] {$u$};
		\node (l3) [position=90:0.07 from v3] {$w'$};
		\node (l4) [position=90:0.07 from v4] {$u$};
		\node (l5) [position=90:0.07 from v5] {$w$};

		\path
		(v1) edge (v2)
		(v2) edge (v3)
		(v3) edge (v4)
		(v4) edge (v5)
		;
		
		\path 
		(v2) edge (v4)
		;
		
		\end{tikzpicture}
	\end{center}
	\caption{The sole legal chord excluding a $w$-vertex from $C_L'$.}
	\label{fig4.15}
\end{figure}
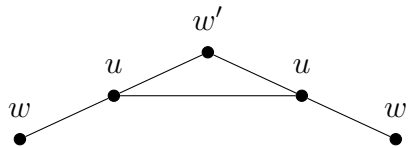
\vspace{-3mm}
In this case the excluded $w$-vertex $w'$ is not necessary and $F'=\lb U\cup W\setminus\set{w'}, E'\rb$ is a flower of the same size. Hence all $w'$-vertices lie on $C_L'$ and Lemma \autoref{lemma4.2} yields the existence of a cycle $C_G$ in $G$ completely consisting of all edges corresponding to the $w$-vertices in $\lineg{G}$ and some $u$-vertices. Hence condition $i)$ of the sprout definition is satisfied.\\
The remaining $u$-vertices cannot be put into the cycle and thus they correspond to the pending edges. Hence we obtain a sprout in $G$ which is fertile since any violation of the other conditions would result in $F$ not being a flower. 
\end{proof}
\begin{corollary}\label{cor4.6}
A graph $G$ without induced cycles of length $k\geq c$ does not contain a fertile sprout of size $k\geq c$.
\end{corollary}

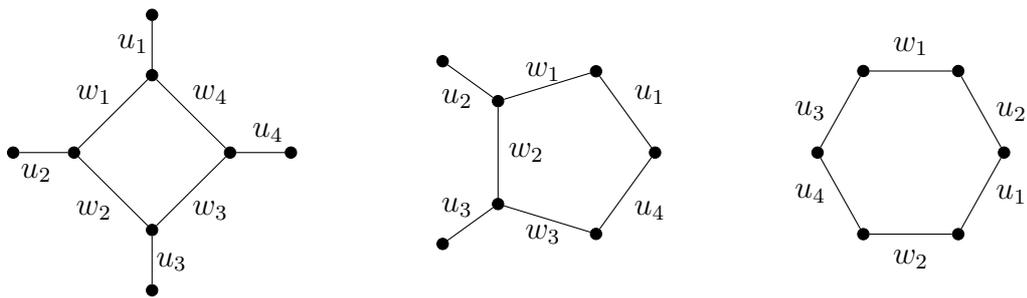
\begin{figure}[!h]
\begin{center}
\begin{tikzpicture}

\node (anchor1) [] {};
\node (anchor2) [position=0:5.2cm from anchor1] {};
\node (anchor3) [position=0:4.2cm from anchor2] {};

\node (u1) [draw,circle,fill,inner sep=1.5pt,position=90:1.6cm from anchor1] {};
\node (u2) [draw,circle,fill,inner sep=1.5pt,position=180:1.6cm from anchor1] {};
\node (u3) [draw,circle,fill,inner sep=1.5pt,position=270:1.6cm from anchor1] {};
\node (u4) [draw,circle,fill,inner sep=1.5pt,position=0:1.6cm from anchor1] {};

\node (w1) [draw,circle,fill,inner sep=1.5pt,position=90:0.8cm from anchor1] {};
\node (w2) [draw,circle,fill,inner sep=1.5pt,position=180:0.8cm from anchor1] {};
\node (w3) [draw,circle,fill,inner sep=1.5pt,position=270:0.8cm from anchor1] {};
\node (w4) [draw,circle,fill,inner sep=1.5pt,position=0:0.8cm from anchor1] {};

\node (lu1) [position=100:1.05 from anchor1] {$u_1$};
\node (lu2) [position=190:1.05 from anchor1] {$u_2$};
\node (lu3) [position=280:1.05 from anchor1] {$u_3$};
\node (lu4) [position=10:1.05 from anchor1] {$u_4$};

\node (lw1) [position=135:0.5cm from anchor1] {$w_1$};
\node (lw2) [position=225:0.5cm from anchor1] {$w_2$};
\node (lw3) [position=315:0.5cm from anchor1] {$w_3$};
\node (lw4) [position=45:0.5cm from anchor1] {$w_4$};

\path
(u1) edge (w1)
(u2) edge (w2)
(u3) edge (w3)
(u4) edge (w4)
;

\path
(w1) edge (w2)
(w2) edge (w3)
(w3) edge (w4)
(w4) edge (w1)
;



\node (v1) [draw,circle,fill,inner sep=1.5pt,position=0:0.9cm from anchor2] {};
\node (v2) [draw,circle,fill,inner sep=1.5pt,position=72:0.9cm from anchor2] {};
\node (v3) [draw,circle,fill,inner sep=1.5pt,position=144:0.9cm from anchor2] {};
\node (v4) [draw,circle,fill,inner sep=1.5pt,position=216:0.9cm from anchor2] {};
\node (v5) [draw,circle,fill,inner sep=1.5pt,position=288:0.9cm from anchor2] {};

\path
(v1) edge (v2)
(v2) edge (v3)
(v3) edge (v4)
(v4) edge (v5)
(v5) edge (v1)
;

\node (cl1) [position=36:0.7cm from anchor2] {$u_1$};
\node (cl2) [position=108:0.7cm from anchor2] {$w_1$};
\node (cl3) [position=180:0.05cm from anchor2] {$w_2$};
\node (cl4) [position=252:0.7cm from anchor2] {$w_3$};
\node (cl5) [position=324:0.7cm from anchor2] {$u_4$};


\node (a1) [draw,circle,fill,inner sep=1.5pt,position=144:1.8cm from anchor2] {};
\node (a2) [draw,circle,fill,inner sep=1.5pt,position=216:1.8cm from anchor2] {};

\path
(v3) edge (a1)
(v4) edge (a2)
;

\node (al1) [position=154:1.1cm from anchor2] {$u_2$};
\node (al2) [position=206:1.1cm from anchor2] {$u_3$};


\node (u1) [draw,circle,fill,inner sep=1.5pt,position=60:1cm from anchor3] {};
\node (u2) [draw,circle,fill,inner sep=1.5pt,position=120:1cm from anchor3] {};
\node (u3) [draw,circle,fill,inner sep=1.5pt,position=240:1cm from anchor3] {};
\node (u4) [draw,circle,fill,inner sep=1.5pt,position=300:1cm from anchor3] {};

\node (w1) [draw,circle,fill,inner sep=1.5pt,position=0:1cm from anchor3] {};
\node (w2) [draw,circle,fill,inner sep=1.5pt,position=180:1cm from anchor3] {};

\node (lu1) [position=337.5:0.9cm from anchor3] {$u_1$};
\node (lu2) [position=22.5:0.9cm from anchor3] {$u_2$};
\node (lu3) [position=157.5:0.9cm from anchor3] {$u_3$};
\node (lu4) [position=202.5:0.9cm from anchor3] {$u_4$};

\node (lw1) [position=90:1cm from anchor3] {$w_1$};
\node (lw2) [position=270:1cm from anchor3] {$w_2$};

\path
(u1) edge (u2)
(u2) edge (w2)
(w1) edge (u1)
(u3) edge (u4)
(u4) edge (w1)
(w2) edge (u3)
;

\end{tikzpicture}
\end{center}
\caption{Examples of sprouts of size $4$.}
\label{fig.16}
\end{figure}
\vspace{-3mm}
We will now have a deeper look into the relation between induced cycles of a certain length and the existence of fertile sprouts of a certain size in the same graph.
\begin{lemma}\label{lemma4.7}
For each $n\geq4$ the only fertile sprouts $S$ of size at least $n$ with a longest induced cycle $C$ of length $n$, $C$ does not contain any $u$-edge and $S\in\mathcal{S}_n$.
\end{lemma}
\vspace{-1.5em}
\begin{proof}
Let $S=\lb V,U\cup W\cup E\rb$ be a fertile sprout of size $n\geq4$ with a longest induced cycle $C$ of length $n$. Now let $C$ contain $i$ $u$-edges with $1\leq i\leq \abr{\frac{n}{2}}$. So $C$ contains at most $n-i\leq n-1$ $w$-edges, hence there are at most $n-i-1$ vertices on $C$ adjacent to exactly two $w$-edges and thus there are at most $n-i-1$ pending edges. With $i$ $u$-edges already contained in $C$ this makes a total of at most $n-i-1+i=n-1$ $u$-edges in $S$, a contradiction to the size of $S$ and thus $C$ does not contain a single $u$-edge.\\
Now, since no $u$-edges are contained in the longest induced cycle of $S$, all $u$-edges must either be pending, or, for each pair of adjacent $u$-edges, there is an edge $e\in E$ which is part of $C$, obviously with $S$ being a sprout we get, for the vertex set $U'=\lb\bigcup_{u\in U}u\rb\setminus \V{C}$ $j\neq i\pm1\lb \!\!\!\mod n\rb\Rightarrow u'_i\neq u'_j$ with $u_i',u_j'\in U'$. In addition, since all $u$-edges are adjacent to $C$ we get $v_iu_i'\in U$, with $\set{v_i}=u_i\setminus\set{u_i'}$. Hence $S$ is a fertile sunflower sprout of size $n$. 
\end{proof}
\vspace{-2mm}
This brings up an important relation between sunflower sprouts and sprouts. While sunflower sprouts are defined via their vertex sets, sprouts are defined via their edges and some edges necessary for a graph to be a sunflower sprout are neither $u$- nor $w$-edges.
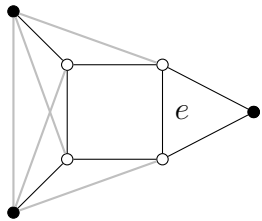
\begin{figure}[H]
\begin{center}
\begin{tikzpicture}

\node (anchor1) [] {};
\node (anchor2) [position=0:5.2cm from anchor1] {};
\node (anchor3) [position=0:5.2cm from anchor2] {};


\node (v1) [inner sep=1.5pt,position=135:0.6cm from anchor2,draw,circle] {};
\node (v2) [inner sep=1.5pt,position=225:0.6cm from anchor2,draw,circle] {};
\node (v3) [inner sep=1.5pt,position=315:0.6cm from anchor2,draw,circle] {};
\node (v4) [inner sep=1.5pt,position=45:0.6cm from anchor2,draw,circle] {};

\path
(v1) edge (v2)
(v2) edge (v3)
(v3) edge (v4)
(v4) edge (v1)
;

\node (u1) [inner sep=1.5pt,position=135:1.6cm from anchor2,draw,circle,fill] {};
\node (u2) [inner sep=1.5pt,position=225:1.6cm from anchor2,draw,circle,fill] {};
\node (u3) [inner sep=1.5pt,position=0:1.6cm from anchor2,draw,circle,fill] {};
\node (u4) [inner sep=1.5pt,position=0:1.6cm from anchor2,draw,circle,fill] {};

\path
(u1) edge (v1)
(u2) edge (v2)
(u3) edge (v3)
(u4) edge (v4)
;

\path[color=lightgray, thick]
(u1) edge (u2)

(u1) edge (v2)
	 edge (v4)
(u2) edge (v3)
	 edge (v1)
;

\node (e) [position=0:5mm from anchor2] {$e$};


\end{tikzpicture}
\end{center}
\caption{The sunflower sprouts of size $4$ and type II.}
\label{fig4.17}
\end{figure}
\vspace{-1em}
While \autoref{fig4.17} shows all possible fertile sunflower sprouts of size $4$ with one identical pair of $u$-vertices, in the sense of sunflower sprouts (note that the gray edges are optional and may exist in any possible combination), the following \autoref{fig4.18} shows all possible fertile sprouts of size $4$ with one pair of adjacent $u$-edges. 
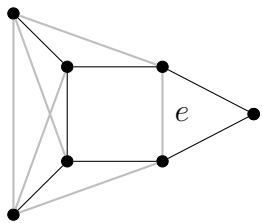
\begin{figure}[H]
\begin{center}
\begin{tikzpicture}

\node (anchor1) [] {};
\node (anchor2) [position=0:5.2cm from anchor1] {};
\node (anchor3) [position=0:5.2cm from anchor2] {};


\node (v1) [inner sep=1.5pt,position=135:0.6cm from anchor2,draw,circle,fill] {};
\node (v2) [inner sep=1.5pt,position=225:0.6cm from anchor2,draw,circle,fill] {};
\node (v3) [inner sep=1.5pt,position=315:0.6cm from anchor2,draw,circle,fill] {};
\node (v4) [inner sep=1.5pt,position=45:0.6cm from anchor2,draw,circle,fill] {};

\path
(v1) edge (v2)
(v2) edge (v3)
(v4) edge (v1)
;

\node (u1) [inner sep=1.5pt,position=135:1.6cm from anchor2,draw,circle,fill] {};
\node (u2) [inner sep=1.5pt,position=225:1.6cm from anchor2,draw,circle,fill] {};
\node (u3) [inner sep=1.5pt,position=0:1.6cm from anchor2,draw,circle,fill] {};
\node (u4) [inner sep=1.5pt,position=0:1.6cm from anchor2,draw,circle,fill] {};

\path
(u1) edge (v1)
(u2) edge (v2)
(u3) edge (v3)
(u4) edge (v4)
;

\path[color=lightgray, thick]
(u1) edge (u2)

(u1) edge (v2)
	 edge (v4)
(u2) edge (v3)
	 edge (v1)
(v3) edge (v4)
;

\node (e) [position=0:5mm from anchor2] {$e$};


\end{tikzpicture}
\end{center}
\caption{The sprouts of size $4$ with one pair of adjacent $u$-edges.}
\label{fig4.18}
\end{figure}
\vspace{-2em}
\newpage
Here now the edge $e$, necessary for the graph to be a sunflower sprout, becomes one of the optional edges and thus making longer induced cycles possible.\\
While the fertile sprouts with the smallest induced cycles are sunflower sprouts, there is another extreme, where a sprout reassembles a single induced cycle, without any additional edges.

\begin{lemma}\label{lemma4.8}
For all cycles $C_n$, $n\geq6$, there is an $n_C\geq 4$ such that $C_n\in\mathfrak{S}_{n_C}$.
\end{lemma}

\begin{proof}
Let $n=k+l$ with $k\geq 4$ and $\aufr{\frac{k}{2}}\leq l\leq k$. We chose $n_C=k$, now there are exactly $n_C-l$ pairs of adjacent $u$-edges on $C$, all other $l$ edges are chosen as $w$-edges and no two of those $w$-edges are adjacent, thus $C_n$ is a fertile sprout of size $k=n_C$.\\
Such a decomposition of $n$ exists for all $n\geq 6$ and can be obtained by choosing $k\leq n\leq 2\,k$ with $l=n-k\geq\aufr{\frac{k}{2}}$.
\end{proof}

This yields another important clue on the class of graphs with a chordal line graph square.

\begin{corollary}
If a graph $G$ contains an induced cycle of length at least $6$, $\lineg{G}^2$ is not chordal.
\end{corollary}

Induced cycles of length $4$ can only result in fertile sunflower sprouts of size $4$ as seen in Lemma \autoref{lemma4.7}. An induced cycle of length at least $6$ results in the squared line graph not to be chordal. So what about induced cycles of length $5$?\\
Again by Lemma \autoref{lemma4.7} one possible fertile sprout containing such an induced cycle is a fertile sunflower sprout of size $5$ and flowers of greater sizes are excluded. So just sprouts of size $4$ that are no sunflower sprouts need to be taken into account.\\
For this we will investigate two lemmas in order to describe those fertile sprouts containing an induced cycle of length $5$.

\begin{lemma}\label{lemma4.9}
Let $S=\lb V,U\cup W\cup E\rb$ be a fertile sprout of size $4$ with a longest induced cycle $C$ of length $5$ and $\E{C}\cap E=\emptyset$, then $S$ has either two or three pending edges. In addition those edges are incident with consecutive vertices of $C$.
\end{lemma}

\begin{proof}
With $S$ being a fertile sprout of size $4$ the number $p$ of pending edges is at least $0$ and at most $4$. So we will begin by discussing the cases $p\in\set{0,1,4}$.\\
For $p=0$ all four $u$-edges must be contained in a cycle, containing at least $2$ $w$-edges, which leads to a cycle of length at least $6$, so all possible induced cycles of length $5$ must contain a chord and therefore $\E{C}\cap E\neq\emptyset$. So now suppose $p=1$, now there are at least $3$ $w$-edges together with $3$ $u$-edges again forming a cycle of at least length $6$ and $C$ must consist of at least one chord of this larger cycle. The only fertile sprout of size $4$ with $p=4$ is, according to Lemma \autoref{lemma4.7}, a sunflower sprout and therefore must not contain an induced cycle of length $5$.\\
If there are three pending edges, those must be consecutive, as the sprout is of size $4$. So suppose there are two pending edges and those are not adjacent to a common $w$-edge on $C$, then, for each pending edge, there are exactly two $w$-edges. So we have $4$ $w$-edges in total and $2$-additional $u$-edges, all six contained in $C$ and thus this case cannot appear.
\end{proof}

\begin{figure}[!h]
\begin{center}
\begin{tikzpicture}
\node (anchor1) [] {};
\node (i) [position=270:2.25cm from anchor1] {I};

\node (anchor4) [position=0:5.2cm from anchor1] {};
\node (iv) [position=270:2.25cm from anchor4] {II};



\node (v1) [draw,circle,fill,inner sep=1.5pt,position=0:0.9cm from anchor1] {};
\node (v2) [draw,circle,fill,inner sep=1.5pt,position=72:0.9cm from anchor1] {};
\node (v3) [draw,circle,fill,inner sep=1.5pt,position=144:0.9cm from anchor1] {};
\node (v4) [draw,circle,fill,inner sep=1.5pt,position=216:0.9cm from anchor1] {};
\node (v5) [draw,circle,fill,inner sep=1.5pt,position=288:0.9cm from anchor1] {};

\path
(v1) edge (v2)
(v2) edge (v3)
(v3) edge (v4)
(v4) edge (v5)
(v5) edge (v1)
;

\node (cl1) [position=36:0.7cm from anchor1] {$u_1$};
\node (cl2) [position=108:2mm from anchor1] {$w_1$};
\node (cl3) [position=180:0.05cm from anchor1] {$w_2$};
\node (cl4) [position=252:2mm from anchor1] {$w_3$};
\node (cl5) [position=324:0.7cm from anchor1] {$u_4$};


\node (a1) [draw,circle,fill,inner sep=1.5pt,position=144:1.8cm from anchor1] {};
\node (a2) [draw,circle,fill,inner sep=1.5pt,position=216:1.8cm from anchor1] {};

\path
(v3) edge (a1)
(v4) edge (a2)
;

\node (al1) [position=144:2cm from anchor1] {$u_2$};
\node (al2) [position=216:2cm from anchor1] {$u_3$};


\path[color=lightgray,thick]
(a1) edge [bend left] (v2)
	 edge (v4)
(a2) edge (v3)
	 edge [bend right] (v5)
(a1) edge [bend right] (a2)
;



\node (v1) [draw,circle,fill,inner sep=1.5pt,position=0:0.9cm from anchor4] {};
\node (v2) [draw,circle,fill,inner sep=1.5pt,position=72:0.9cm from anchor4] {};
\node (v3) [draw,circle,fill,inner sep=1.5pt,position=144:0.9cm from anchor4] {};
\node (v4) [draw,circle,fill,inner sep=1.5pt,position=216:0.9cm from anchor4] {};
\node (v5) [draw,circle,fill,inner sep=1.5pt,position=288:0.9cm from anchor4] {};

\path
(v1) edge (v2)
(v2) edge (v3)
(v3) edge (v4)
(v4) edge (v5)
(v5) edge (v1)
;

\node (cl1) [position=36:0.2mm from anchor4] {$w_1$};
\node (cl2) [position=108:0.7cm from anchor4] {$w_2$};
\node (cl3) [position=180:0.7cm from anchor4] {$u_3$};
\node (cl4) [position=252:0.7cm from anchor4] {$w_4$};
\node (cl5) [position=324:0.2mm from anchor4] {$w_5$};


\node (a1) [draw,circle,fill,inner sep=1.5pt,position=0:1.8cm from anchor4] {};
\node (a2) [draw,circle,fill,inner sep=1.5pt,position=72:1.8cm from anchor4] {};
\node (a3) [draw,circle,fill,inner sep=1.5pt,position=288:1.8cm from anchor4] {};

\path
(v1) edge (a1)
(v2) edge (a2)
(v5) edge (a3)
;

\node (al1) [position=0:2cm from anchor4] {$u_1$};
\node (al2) [position=72:2cm from anchor4] {$u_2$};
\node (al2) [position=288:2cm from anchor4] {$u_4$};


\path[color=lightgray,thick]
(a1) edge (v2)
	 edge (v5)
	 edge [bend right] (a2)
	 edge [bend left] (a3)
(a2) edge (v1)
	 edge [bend right] (v3)
(a3) edge (v1)
	 edge [bend left] (v4)
;

\end{tikzpicture}
\end{center}
\caption{The two types of fertile sprouts satisfying the conditions of Lemma \autoref{lemma4.9}.}
\label{fig4.19}
\end{figure}
\vspace{-5mm}

It is rather easy to see that all possible chords in the cycles of these sprouts would result in an induced cycle of length $4$ and with that these sprouts would contain a fertile sunflower sprout of size $4$. Because of this additional chords do not need to be taken into account in this case.\\
Following this idea to find already forbidden sprouts in other sprouts we continue our investigation by increasing the length of the induced cycle in $S$.

\begin{lemma}\label{lemma4.10}
Let $S=\lb V,U\cup W\cup E\rb$ be a fertile sprout of size $4$ with a longest induced cycle $C$ of length $5$ and $\E{C}\cap E\neq\emptyset$, then $S$ has at most $2$ pending edges and contains a cycle $C'$ of length at least $6$ with $\E{C'}\cap E=\emptyset$. 
\end{lemma}

\begin{proof}
Since $S$ is a sprout of size four either three pending edges, or a cycle of length $5$ just consisting of $u$- and $w$-edges would result in a fertile sprout of type II as seen in \autoref{fig4.19}.
\end{proof}

There are three possible lengths of cycles $C$ with $\E{C}\cap E=\emptyset$ and $\abs{C}\geq 6$ in sprouts of size $4$. It holds $6\leq \abs{C}\leq 8$ for such a cycle and since the number of possible chords is limited, the number of different subtypes is limited too. In the following we will give an overview on the possible graphs in form of figures and then show that it is possible to drastically reduce the number of forbidden subgraphs.\\
Type III has a basic $u$-$w$ cycle of length $6$ and, in order to obtain an induced cycle of length $5$, contains at least one chord. Now such a sprout may contain up to two pending edges. We start without any pending edges and again the gray edges may exist as additional edges contained in $E$ in any possible combination. As there are some edges, that are not necessary for the graphs to be a fertile sprout, but must exist in order to create the induced cycle of length $5$, those edges are colored black, but have a dotted line. 
 
\begin{lemma}\label{lemma4.11}
Let $S=\lb V,U\cup W\cup E\rb$ be a fertile sprout of size $4$ with a longest induced cycle $C$ of length $5$, $\E{C}\cap E\neq\emptyset$ and a cycle $C'$ with $\E{C'}\cap E=\emptyset$ as well as $\abs{C'}=6$, then $S$ contains a fertile sprout $S'$ of type I (see \autoref{fig4.19}) or III$_a$ (see \autoref{fig4.20}). 
\end{lemma}

\begin{proof}
Let $C'$ be the cycle in $S$ consisting purely of $u$- and $w$-edges, in the following we will call $C'$ the {\em base cycle}, and $C$ be the induced cycle of length $5$. As \autoref{fig4.20} suggests we will divide the proof in three cases and, while the first case obviously holds since the sole graph without any pending edges fulfilling all conditions is, in fact, the III$_a$. The cases 2 and 3 will be subdivided into several subcases as follows.
\begin{enumerate}
	\item[] \begin{enumerate}
		\item [Case 1] The fertile sprout $S$ has no pending edges.
		
		\item [Case 2] The fertile sprout $S$ has exactly one pending edge.
		\begin{enumerate}
			\item[] \begin{enumerate}
				
				\item [Subcase 2.1] The required chord $c$ skips two adjacent $u$-edges.
				\item [Subcase 2.2] The chord $c$ skips one $u$- and one $w$-edge but is not adjacent to the pending edge $u'$.
				\item [Subcase 2.3] The chord $c$ skips one $u$- and one $w$-edge and is adjacent to the pending edge $u'$.
			\end{enumerate}
		\end{enumerate}
		
		\item [Case 3] The fertile sprout $S$ has exactly two pending edges.
		\begin{enumerate}
			\item[] \begin{enumerate}
				
				\item [Subcase 3.1] The two pending edges are not adjacent to a common $w$-edge.
				\item [Subcase 3.2] The two pending edges are adjacent to a common $w$-edge and the required chord $c$ in $C$ is adjacent to one of the pending edges.
				\item [Subcase 3.3] The two pending edges are adjacent to a common $w$-edge and the required chord $c$ in $C$ is not adjacent to any of the pending edges.
			\end{enumerate}
		\end{enumerate}
	\end{enumerate}
\end{enumerate}

Case 1: The fertile sprout $S$ has no pending edges.\\
Without any pending edges, in order to realize a fertile sprout of size $4$ with a base cycle of length $6$ there will be exactly two $w$-edges and two pairs of adjacent $u$-edges. Still all possible chords are of the three types seen in \autoref{fig4.14}, so here the necessary chord $c$ which shortens $C'$ by one and produces $C$ skips two adjacent $u$-edges. Since no chord may join two vertices of $C$ the resulting graph is exactly the III$_a$.
\begin{figure}[!h]
	\begin{center}
		\begin{tikzpicture}
		\node (anchor1) [inner sep=1.5pt] {};
		\node (label1) [position=270:1.5cm from anchor1] {III$_a$};
		
		
		\node (u1) [draw,circle,fill,inner sep=1.5pt,position=90:1.1cm from anchor1] {};
		\node (u2) [draw,circle,fill,inner sep=1.5pt,position=150:1.1cm from anchor1] {};
		\node (u3) [draw,circle,fill,inner sep=1.5pt,position=210:1.1cm from anchor1] {};
		\node (u4) [draw,circle,fill,inner sep=1.5pt,position=270:1.1cm from anchor1] {};
		\node (u5) [draw,circle,fill,inner sep=1.5pt,position=330:1.1cm from anchor1] {};
		\node (u6) [draw,circle,fill,inner sep=1.5pt,position=30:1.1cm from anchor1] {};
		
		\node (l1) [position=120:1cm from anchor1] {$u$};
		\node (l2) [position=180.5:1cm from anchor1] {$w$};
		\node (l3) [position=240.5:1cm from anchor1] {$u$};
		\node (l4) [position=300.5:1cm from anchor1] {$u$};
		\node (l5) [position=0.5:1cm from anchor1] {$w$};
		\node (l6) [position=60.5:1cm from anchor1] {$u$};
		
		\path
		(u1) edge (u2)
		(u2) edge (u3)
		(u3) edge (u4)
		(u4) edge (u5)
		(u5) edge (u6)
		(u6) edge (u1)
		;
		
		\path[dotted,thick]
		(u2) edge (u6)
		;
		
		\end{tikzpicture}
	\caption{The sprout III$_a$.}
	\label{fig4.20}		
	\end{center}
\end{figure}
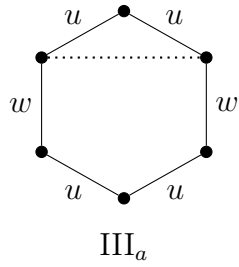
\vspace{-1mm}
Case 2: The fertile sprout $S$ has exactly one pending edge.\\
Here we will handle all three subcases simultaneously. One pending edge requires two adjacent $w$-edges, which leaves $4$ other edges on the base cycle and still $3$ $u$-edges are required, thus we obtain one pair of adjacent $u$-edges and a single one. The chord $c$ producing the shortened cycle $C$ either skips these two $u$-edges or a $u$-$w$-pair. In Subcase 2.1 we have a problem similar to Case 1, there is no legal edge joining the common vertex of the two adjacent $u$-edges with another vertex on the cycle and thus the vertices of $C'$ induce a III$_a$.\\
In the Subcases 2.2 and 2.3 now one additional chord in $C'$ is possible, in Subcase 2.2 this chord is exactly the required one of Subcase 2.3 and vice versa, so both cases are similar. If this additional chord does not exist, the vertices of $C'$ induce the III$_a$, so now suppose both chords exist. Let $C''$ be the cycle obtained by shortening $C'$ via the chord adjacent to the pending edge. So there exists exactly one $u$-edge $u''$ on $C'$, which is not part of $C''$, one of its endpoints in part of $C''$ and the other one is incident with the skipped $w$-edge and the other $w$-edge, $u''$ is adjacent to. By taking $C''$ as a new base cycle and the pending edge together with $u''$ as two new pending edges, we obtain a graph of type I. Hence Case 2 is closed.
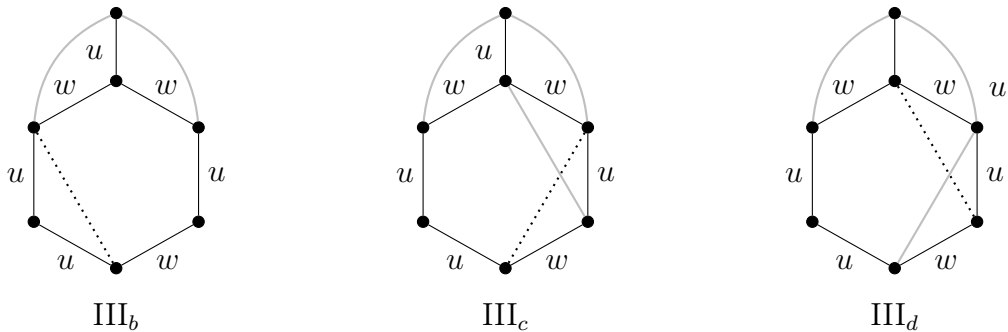
\begin{figure}[!h]
	\begin{center}
		\begin{tikzpicture}
		\node (anchor1) [inner sep=1.5pt] {};
		\node (label1) [position=270:1.5cm from anchor1] {III$_b$};
		\node (anchor2) [inner sep=1.5pt,position=0:5cm from anchor1] {};
		\node (label1) [position=270:1.5cm from anchor2] {III$_c$};
		\node (anchor3) [inner sep=1.5pt,position=0:5cm from anchor2] {};
		\node (label1) [position=270:1.5cm from anchor3] {III$_d$};
		
		
		\node (u1) [draw,circle,fill,inner sep=1.5pt,position=90:1.1cm from anchor1] {};
		\node (u2) [draw,circle,fill,inner sep=1.5pt,position=150:1.1cm from anchor1] {};
		\node (u3) [draw,circle,fill,inner sep=1.5pt,position=210:1.1cm from anchor1] {};
		\node (u4) [draw,circle,fill,inner sep=1.5pt,position=270:1.1cm from anchor1] {};
		\node (u5) [draw,circle,fill,inner sep=1.5pt,position=330:1.1cm from anchor1] {};
		\node (u6) [draw,circle,fill,inner sep=1.5pt,position=30:1.1cm from anchor1] {};
		
		\node (v1) [draw,circle,fill,inner sep=1.5pt,position=90:2cm from anchor1] {};
		
		\node (l1) [position=120:1cm from anchor1] {$w$};
		\node (l2) [position=180.5:1cm from anchor1] {$u$};
		\node (l3) [position=240.5:1cm from anchor1] {$u$};
		\node (l4) [position=300.5:1cm from anchor1] {$w$};
		\node (l5) [position=0.5:1cm from anchor1] {$u$};
		\node (l6) [position=60.5:1cm from anchor1] {$w$};
		\node (l8) [position=100:1.35cm from anchor1] {$u$};

		\path
		(u1) edge (u2)
		(u2) edge (u3)
		(u3) edge (u4)
		(u4) edge (u5)
		(u5) edge (u6)
		(u6) edge (u1)
		(v1) edge (u1)
		;
		
		\path[dotted,thick]
		(u2) edge (u4)
		;
		
		\path[color=lightgray,thick]
		(v1) edge [bend right] (u2)
		edge [bend left] (u6)
		;
		
		
		\node (u1) [draw,circle,fill,inner sep=1.5pt,position=90:1.1cm from anchor2] {};
		\node (u2) [draw,circle,fill,inner sep=1.5pt,position=150:1.1cm from anchor2] {};
		\node (u3) [draw,circle,fill,inner sep=1.5pt,position=210:1.1cm from anchor2] {};
		\node (u4) [draw,circle,fill,inner sep=1.5pt,position=270:1.1cm from anchor2] {};
		\node (u5) [draw,circle,fill,inner sep=1.5pt,position=330:1.1cm from anchor2] {};
		\node (u6) [draw,circle,fill,inner sep=1.5pt,position=30:1.1cm from anchor2] {};
		
		\node (v1) [draw,circle,fill,inner sep=1.5pt,position=90:2cm from anchor2] {};
		
		\node (l1) [position=120:1cm from anchor2] {$w$};
		\node (l2) [position=180.5:1cm from anchor2] {$u$};
		\node (l3) [position=240.5:1cm from anchor2] {$u$};
		\node (l4) [position=300.5:1cm from anchor2] {$w$};
		\node (l5) [position=0.5:1cm from anchor2] {$u$};
		\node (l6) [position=60.5:1cm from anchor2] {$w$};
		\node (l8) [position=100:1.35cm from anchor2] {$u$};

		\path
		(u1) edge (u2)
		(u2) edge (u3)
		(u3) edge (u4)
		(u4) edge (u5)
		(u5) edge (u6)
		(u6) edge (u1)
		(v1) edge (u1)
		;
		
		\path[dotted,thick]
		(u6) edge (u4)
		;
		
		\path[color=lightgray,thick]
		(v1) edge [bend right] (u2)
		edge [bend left] (u6)
		(u1) edge (u5)
		;
		

		\node (u1) [draw,circle,fill,inner sep=1.5pt,position=90:1.1cm from anchor3] {};
		\node (u2) [draw,circle,fill,inner sep=1.5pt,position=150:1.1cm from anchor3] {};
		\node (u3) [draw,circle,fill,inner sep=1.5pt,position=210:1.1cm from anchor3] {};
		\node (u4) [draw,circle,fill,inner sep=1.5pt,position=270:1.1cm from anchor3] {};
		\node (u5) [draw,circle,fill,inner sep=1.5pt,position=330:1.1cm from anchor3] {};
		\node (u6) [draw,circle,fill,inner sep=1.5pt,position=30:1.1cm from anchor3] {};
		
		\node (v1) [draw,circle,fill,inner sep=1.5pt,position=90:2cm from anchor3] {};
		
		\node (l1) [position=120:1cm from anchor3] {$w$};
		\node (l2) [position=180.5:1cm from anchor3] {$u$};
		\node (l3) [position=240.5:1cm from anchor3] {$u$};
		\node (l4) [position=300.5:1cm from anchor3] {$w$};
		\node (l5) [position=0.5:1cm from anchor3] {$u$};
		\node (l6) [position=60.5:1cm from anchor3] {$w$};
		\node (l7) [position=40:1.35cm from anchor3] {$u$};

		\path
		(u1) edge (u2)
		(u2) edge (u3)
		(u3) edge (u4)
		(u4) edge (u5)
		(u5) edge (u6)
		(u6) edge (u1)
		(v1) edge (u1)
		;

		\path[dotted,thick]
		(u1) edge (u5)
		;
		
		\path[color=lightgray,thick]
		(v1) edge [bend right] (u2)
		edge [bend left] (u6)
		(u4) edge (u6)
		;

		\end{tikzpicture}
	\caption{The sprouts of Case 2.}
	\label{fig4.21}
	\end{center}
\end{figure}
\vspace{-1mm}
~\\
Subcase 3.1: The fertile sprout $S$ has exactly two pending edges and the two pending edges are not adjacent to a common $w$-edge.\\
Now $C'$ consists of two pairs of adjacent $w$-edges and two $u$-edges separating them. The chord $c$ producing $C$ has to be adjacent to one of the pending edges and is of the $u$-$w$-type. The sole possible additional edge is the one joining the skipped vertex on $C'$ with the common endpoint of $C$ with the other pending edge, a mirror of $c$. If this edge does not exist, the vertices of $C'$ induce a III$_a$, otherwise take $C$ as a new base cycle and the pending edge adjacent to $c$ together with the skipped $u$-edge as a new pair of pending edges. The induced subgraph of $S$ is a graph of type I and thus Subcase 3.1 is closed.
\begin{figure}[!h]
	\begin{center}
		\begin{tikzpicture}
		
		\node (anchor3) [inner sep=1.5pt] {};
		\node (label1) [position=270:1.5cm from anchor3] {III$_e$};


		\node (u1) [draw,circle,fill,inner sep=1.5pt,position=90:1.1cm from anchor3] {};
		\node (u2) [draw,circle,fill,inner sep=1.5pt,position=150:1.1cm from anchor3] {};
		\node (u3) [draw,circle,fill,inner sep=1.5pt,position=210:1.1cm from anchor3] {};
		\node (u4) [draw,circle,fill,inner sep=1.5pt,position=270:1.1cm from anchor3] {};
		\node (u5) [draw,circle,fill,inner sep=1.5pt,position=330:1.1cm from anchor3] {};
		\node (u6) [draw,circle,fill,inner sep=1.5pt,position=30:1.1cm from anchor3] {};

		\node (v1) [draw,circle,fill,inner sep=1.5pt,position=150:2cm from anchor3] {};
		\node (v2) [draw,circle,fill,inner sep=1.5pt,position=330:2cm from anchor3] {};

		\node (l1) [position=120:1cm from anchor3] {$w$};
		\node (l2) [position=180.5:1cm from anchor3] {$w$};
		\node (l3) [position=240.5:1cm from anchor3] {$u$};
		\node (l4) [position=300.5:1cm from anchor3] {$w$};
		\node (l5) [position=0.5:1cm from anchor3] {$w$};
		\node (l6) [position=60.5:1cm from anchor3] {$u$};
		\node (l7) [position=160:1.35cm from anchor3] {$u$};
		\node (l8) [position=340:1.35cm from anchor3] {$u$};

		\path
		(u1) edge (u2)
		(u2) edge (u3)
		(u3) edge (u4)
		(u4) edge (u5)
		(u5) edge (u6)
		(u6) edge (u1)
		(v1) edge (u2)
		(v2) edge (u5)
		;

		\path[dotted,thick]
		(u2) edge (u6)
		;

		\path[color=lightgray,thick]
		(u1) edge (u5)
		(v1) edge [bend left] (u1)
			 edge [bend right] (u3)
		(v2) edge [bend left] (u4)
			 edge [bend right] (u6)
		;

		\end{tikzpicture}
		\caption{The sprouts of Subcase 3.1.}
		\label{fig4.22}		
	\end{center}
\end{figure}
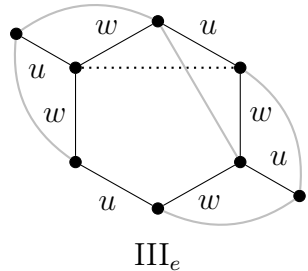
\vspace{-1mm}
Subcase 3.2: The two pending edges are adjacent to a common $w$-edge and the required chord $c$ in $C'$ is adjacent to one of the pending edges.\\
With two pending edges adjacent to a common $w$-edge $C'$ must consist of two nonadjacent $u$-edges and $4$ $w$-edges, so $c$ is a chord of the $u$-$w$-type in $C'$. The vertex $v$ of $C'$ not contained in $C$ must be the endpoint of all other chords in $C'$ and thus there is exactly one chord of the $u$-$w$-type and one of the $uwu$-type possible. If none of the two additional edges exist we again obtain a III$_a$. If any of the chords exist we take $C$ as a new base-cycle together with the two pending edges and obtain a graph of type I.
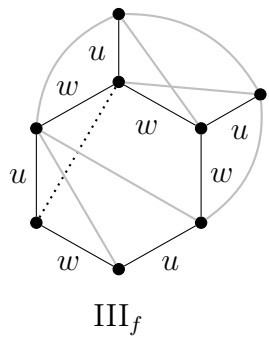
\begin{figure}[H]
	\begin{center}
		\begin{tikzpicture}
		\node (anchor4) [inner sep=1.5pt] {};
		\node (label1) [position=270:1.5cm from anchor4] {III$_f$};
			
		
		\node (u1) [draw,circle,fill,inner sep=1.5pt,position=90:1.1cm from anchor4] {};
		\node (u2) [draw,circle,fill,inner sep=1.5pt,position=150:1.1cm from anchor4] {};
		\node (u3) [draw,circle,fill,inner sep=1.5pt,position=210:1.1cm from anchor4] {};
		\node (u4) [draw,circle,fill,inner sep=1.5pt,position=270:1.1cm from anchor4] {};
		\node (u5) [draw,circle,fill,inner sep=1.5pt,position=330:1.1cm from anchor4] {};
		\node (u6) [draw,circle,fill,inner sep=1.5pt,position=30:1.1cm from anchor4] {};
		
		\node (v2) [draw,circle,fill,inner sep=1.5pt,position=30:2cm from anchor4] {};
		\node (v1) [draw,circle,fill,inner sep=1.5pt,position=90:2cm from anchor4] {};
		
		\node (l1) [position=120:1cm from anchor4] {$w$};
		\node (l2) [position=180.5:1cm from anchor4] {$u$};
		\node (l3) [position=240.5:1cm from anchor4] {$w$};
		\node (l4) [position=300.5:1cm from anchor4] {$u$};
		\node (l5) [position=0.5:1cm from anchor4] {$w$};
		\node (l6) [position=60.5:4mm from anchor4] {$w$};
		\node (l7) [position=20:1.35cm from anchor4] {$u$};
		\node (l8) [position=100:1.35cm from anchor4] {$u$};

		\path
		(u1) edge (u2)
		(u2) edge (u3)
		(u3) edge (u4)
		(u4) edge (u5)
		(u5) edge (u6)
		(u6) edge (u1)
		(v2) edge (u6)
		(v1) edge (u1)
		;
		
		\path[dotted,thick]
		(u1) edge (u3)
		;
		
		\path[color=lightgray,thick]
		(u2) edge (u4)
		edge (u5)
		(v1) edge [bend right] (u2)
		edge (u6)
		(v2) edge (u1)
		edge [bend left] (u5)
		(v1) edge [bend left] (v2)
		;
		\end{tikzpicture}
		\caption{The sprouts of Subcase 3.2.}
		\label{fig4.23}		
	\end{center}
\end{figure}
\vspace{-1mm}
Subcase 3.3: The two pending edges are adjacent to a common $w$-edge and the required chord $c$ in $C'$ is not adjacent to any of the pending edges.\\
This is very similar to Subcase 3.2, the difference lies in the possible additional edges, since $c$ is not adjacent to one of the pending edges there is no chord of the $uwu$-type possible, but now we have two $u$-$w$-type edges that may exist. Without any of those edges we again obtain a III$_a$ as seen before and if at least one of the edges exists $C$ together with the two pending edges induces a graph of type I. This closes Case 3 and completes the proof.
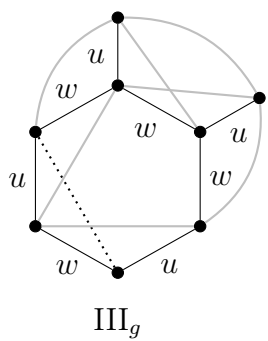
\begin{figure}[!h]
	\begin{center}
		\begin{tikzpicture}
		\node (anchor5) [inner sep=1.5pt] {};
		\node (label1) [position=270:1.5cm from anchor5] {III$_g$};
		
		
		\node (u1) [draw,circle,fill,inner sep=1.5pt,position=90:1.1cm from anchor5] {};
		\node (u2) [draw,circle,fill,inner sep=1.5pt,position=150:1.1cm from anchor5] {};
		\node (u3) [draw,circle,fill,inner sep=1.5pt,position=210:1.1cm from anchor5] {};
		\node (u4) [draw,circle,fill,inner sep=1.5pt,position=270:1.1cm from anchor5] {};
		\node (u5) [draw,circle,fill,inner sep=1.5pt,position=330:1.1cm from anchor5] {};
		\node (u6) [draw,circle,fill,inner sep=1.5pt,position=30:1.1cm from anchor5] {};
		
		\node (v2) [draw,circle,fill,inner sep=1.5pt,position=30:2cm from anchor5] {};
		\node (v1) [draw,circle,fill,inner sep=1.5pt,position=90:2cm from anchor5] {};
		
		\node (l1) [position=120:1cm from anchor5] {$w$};
		\node (l2) [position=180.5:1cm from anchor5] {$u$};
		\node (l3) [position=240.5:1cm from anchor5] {$w$};
		\node (l4) [position=300.5:1cm from anchor5] {$u$};
		\node (l5) [position=0.5:1cm from anchor5] {$w$};
		\node (l6) [position=60.5:4mm from anchor5] {$w$};
		\node (l7) [position=20:1.35cm from anchor5] {$u$};
		\node (l8) [position=100:1.35cm from anchor5] {$u$};

		\path
		(u1) edge (u2)
		(u2) edge (u3)
		(u3) edge (u4)
		(u4) edge (u5)
		(u5) edge (u6)
		(u6) edge (u1)
		(v2) edge (u6)
		(v1) edge (u1)
		;
		
		\path[dotted,thick]
		(u2) edge (u4)
		;
		
		\path[color=lightgray,thick]
		(u3) edge (u1)
		edge (u5)
		(v1) edge [bend right] (u2)
		edge (u6)
		(v2) edge (u1)
		edge [bend left] (u5)
		(v1) edge [bend left] (v2)
		;
		
		\end{tikzpicture}
		\caption{The sprouts of Subcase 3.3.}
		\label{fig4.24}		
	\end{center}
\end{figure}
\vspace{-1cm}
\end{proof}

Now we increase the length of the base cycle by $1$ and investigate the structure of fertile sprouts of size $4$ with a base cycle of length $7$ and a longest induced cycle of length $5$. With this now two pending edges are no longer possible, so all graphs of type IV have either none or one pending edge.\\
Again we can find some of the simpler and smaller sprouts in these graphs, though this time just the sprouts of type I and III$_a$ are not enough. Since in some cases induced cycles of length $4$ may appear, it seems we need to take the sunflower sprouts of size $4$ into account as well.\\
As before the dotted edges in the following figures are necessary chords in order to shorten the cycle and gray ones are additional edges that may exist in any possible combination.

\begin{lemma}\label{lemma4.12}
	Let $S=\lb V,U\cup W\cup E\rb$ be a fertile sprout of size $4$ with a longest induced cycle $C$ of length $5$, $\E{C}\cap E\neq\emptyset$ and a cycle $C'$ with $\E{C'}\cap E=\emptyset$ as well as $\abs{C'}=7$, then $S$ contains a fertile sprout $S'$ of type I (see \autoref{fig4.19}) or III$_a$ (see \autoref{fig4.20}) or a fertile sunflower sprout of size $4$ (see \autoref{fig4.2}). 
\end{lemma}

\begin{proof}
Let $C'$ be the cycle in $S$ consisting purely of $u$- and $w$-edges, the base cycle, and $C$ be the induced cycle of length $5$. The proof is divided into two cases since there either is no pending edge, or one. These two cases will be divided further into subcases since there are again several possibilities for $C'$ to be shortened in order to obtain a $C$.
\begin{enumerate}
	\item[] \begin{enumerate}
		\item [Case 1] The fertile sprout $S$ has no pending edge.
		\begin{enumerate}
			\item[] \begin{enumerate}
				
				\item [Subcase 1.1] There is just one chord shortening $C'$.
				\item [Subcase 1.2] There are two nonadjacent chords and one of them is of the $u$-$u$-type.
				\item [Subcase 1.3] There are two nonadjacent $u$-$w$-type chords.
				\item [Subcase 1.4] There are two adjacent chords and one of them is of the $u$-$u$-type.
				\item[Subcase 1.5] There are two adjacent $u$-$w$-type chords.
			\end{enumerate}
		\end{enumerate}
		
		\item [Case 2] The fertile sprout $S$ has exactly one pending edge $u'$.
		\begin{enumerate}
			\item[] \begin{enumerate}
				
				\item [Subcase 2.1] Both chords are adjacent to the pending edge $u'$.
				\item [Subcase 2.2] Just one of the two necessary chords is adjacent to $u'$ and the chords are adjacent.
				\item [Subcase 2.3] Just one of the two necessary chords is adjacent to $u'$ and the chords are not adjacent.
				\item [Subcase 2.4] No necessary chord is adjacent to the pending edge, but both chords are adjacent to each other.
				\item [Subcase 2.5] No necessary chord is adjacent to the pending edge nor are those chords adjacent.
				\item [Subcase 2.6] There is just one chord shortening $C'$.
			\end{enumerate}
		\end{enumerate}
	\end{enumerate}
\end{enumerate}
\vspace{-2mm}
Case 1: The fertile sprout $S$ has no pending edge.\\
Since $C'$ has length $7$ and no pending edge, it must consist of $4$ $u$- and $3$ $w$-edges, hence there is exactly one pair of adjacent $u$-edges.\\
Subcase 1.1: There is just one chord shortening $C'$.\\
With $C'$ being a cycle of length $7$ the chord $c$ producing $C$ must be of the $uwu$-type and since there is a pair of adjacent $u$-edges on $C'$, there is just one possible position for this chord. The edge $c$ skips exactly one $w$-edge $w'$ and all other possible chords in $C'$ must contain one endpoint of $w'$. Since no more $uwu$-edges are possible, all such chords shorten $C'$ by $1$, hence if we take $C$ as a new base cycle and the two $u$-edges adjacent to $w'$ as pending edges, we always obtain a fertile sprout of type I.
\begin{figure}[H]
	\begin{center}
		\begin{tikzpicture}
		\node (anchor1) [] {};
		\node (label1) [position=270:1.4cm from anchor1] {IV$_a$};
		
		
		
		\node (v1) [draw,circle,fill,inner sep=1.5pt,position=90:1cm from anchor1] {};
		\node (v2) [draw,circle,fill,inner sep=1.5pt,position=141.42:1cm from anchor1] {};
		\node (v3) [draw,circle,fill,inner sep=1.5pt,position=192.84:1cm from anchor1] {};
		\node (v4) [draw,circle,fill,inner sep=1.5pt,position=244.26:1cm from anchor1] {};
		\node (v5) [draw,circle,fill,inner sep=1.5pt,position=295.68:1cm from anchor1] {};
		\node (v6) [draw,circle,fill,inner sep=1.5pt,position=347.1:1cm from anchor1] {};
		\node (v7) [draw,circle,fill,inner sep=1.5pt,position=38.52:1cm from anchor1] {};
		
		\path
		(v1) edge (v2)
		(v2) edge (v3)
		(v3) edge (v4)
		(v4) edge (v5)
		(v5) edge (v6)
		(v6) edge (v7)
		(v7) edge (v1)
		;
		
		\node (labelv1) [position=115.5:9mm from anchor1] {$u$};
		\node (labelv2) [position=166.9:9mm from anchor1] {$w$};
		\node (labelv3) [position=218.3:9mm from anchor1] {$u$};
		\node (labelv4) [position=269.7:9mm from anchor1] {$w$};
		\node (labelv5) [position=321.2:9mm from anchor1] {$u$};
		\node (labelv6) [position=12.6:9mm from anchor1] {$w$};
		\node (labelv7) [position=64:9mm from anchor1] {$u$};
		

		\path[dotted,thick]
		(v3) edge (v6)
		;
		
		\path[color=lightgray,thick]
		(v4) edge (v2)
		edge (v6)
		(v5) edge (v3)
		edge (v7)
		;
		\end{tikzpicture}
	\end{center}
	\caption{The sprouts of Subcase 1.1.}
	\label{fig4.25}
\end{figure}
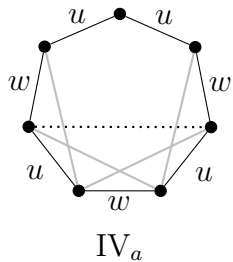
\vspace{-1mm}
Subcase 1.2: There are two nonadjacent chords and one of them is of the $u$-$u$-type, and Subcase 1.4: There are two adjacent chords and one of them is of the $u$-$u$-type.\\
Since the common endpoint $v$ of the two adjacent $u$-edges cannot be the endpoint of any legal chord in $C'$, the vertices of the induced cycle $C$ together with $v$ induce a III$_a$.
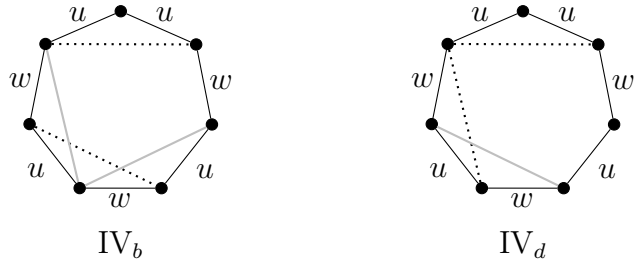
\begin{figure}[!h]
	\begin{center}
		\begin{tikzpicture}
		
		\node (anchor2) [] {};
		\node (label2) [position=270:1.4cm from anchor2] {IV$_b$};
		
		\node (anchor1) [position=0:5cm from anchor2] {};
		\node (label1) [position=270:1.4cm from anchor1] {IV$_d$};

		
		
		\node (v1) [draw,circle,fill,inner sep=1.5pt,position=90:1cm from anchor2] {};
		\node (v2) [draw,circle,fill,inner sep=1.5pt,position=141.42:1cm from anchor2] {};
		\node (v3) [draw,circle,fill,inner sep=1.5pt,position=192.84:1cm from anchor2] {};
		\node (v4) [draw,circle,fill,inner sep=1.5pt,position=244.26:1cm from anchor2] {};
		\node (v5) [draw,circle,fill,inner sep=1.5pt,position=295.68:1cm from anchor2] {};
		\node (v6) [draw,circle,fill,inner sep=1.5pt,position=347.1:1cm from anchor2] {};
		\node (v7) [draw,circle,fill,inner sep=1.5pt,position=38.52:1cm from anchor2] {};
		
		\path
		(v1) edge (v2)
		(v2) edge (v3)
		(v3) edge (v4)
		(v4) edge (v5)
		(v5) edge (v6)
		(v6) edge (v7)
		(v7) edge (v1)
		;
		
		\node (labelv1) [position=115.5:9mm from anchor2] {$u$};
		\node (labelv2) [position=166.9:9mm from anchor2] {$w$};
		\node (labelv3) [position=218.3:9mm from anchor2] {$u$};
		\node (labelv4) [position=269.7:9mm from anchor2] {$w$};
		\node (labelv5) [position=321.2:9mm from anchor2] {$u$};
		\node (labelv6) [position=12.6:9mm from anchor2] {$w$};
		\node (labelv7) [position=64:9mm from anchor2] {$u$};
		

		\path[dotted,thick]
		(v2) edge (v7)
		(v3) edge (v5)
		;
		
		\path[color=lightgray,thick]
		(v4) edge (v2)
		edge (v6)
		;

		
		
		\node (v1) [draw,circle,fill,inner sep=1.5pt,position=90:1cm from anchor1] {};
		\node (v2) [draw,circle,fill,inner sep=1.5pt,position=141.42:1cm from anchor1] {};
		\node (v3) [draw,circle,fill,inner sep=1.5pt,position=192.84:1cm from anchor1] {};
		\node (v4) [draw,circle,fill,inner sep=1.5pt,position=244.26:1cm from anchor1] {};
		\node (v5) [draw,circle,fill,inner sep=1.5pt,position=295.68:1cm from anchor1] {};
		\node (v6) [draw,circle,fill,inner sep=1.5pt,position=347.1:1cm from anchor1] {};
		\node (v7) [draw,circle,fill,inner sep=1.5pt,position=38.52:1cm from anchor1] {};
		
		\path
		(v1) edge (v2)
		(v2) edge (v3)
		(v3) edge (v4)
		(v4) edge (v5)
		(v5) edge (v6)
		(v6) edge (v7)
		(v7) edge (v1)
		;
		
		\node (labelv1) [position=115.5:9mm from anchor1] {$u$};
		\node (labelv2) [position=166.9:9mm from anchor1] {$w$};
		\node (labelv3) [position=218.3:9mm from anchor1] {$u$};
		\node (labelv4) [position=269.7:9mm from anchor1] {$w$};
		\node (labelv5) [position=321.2:9mm from anchor1] {$u$};
		\node (labelv6) [position=12.6:9mm from anchor1] {$w$};
		\node (labelv7) [position=64:9mm from anchor1] {$u$};
		

		\path[dotted,thick]
		(v2) edge (v7)
		edge (v4)
		;
		
		\path[color=lightgray,thick]
		(v3) edge (v5)
		;

		\end{tikzpicture}
	\end{center}
	\caption{The sprouts of Subcase 1.2 and Subcase 1.4.}
	\label{fig4.26}
\end{figure}
\vspace{-1mm}
~\\
Subcase 1.3: There are two nonadjacent chords $u$-$w$-type chords.\\
Now exactly the $w$-edge $w'$ that was skipped in Subcase 1.1 is the only $w$-edge contained in the shortened cycle $C$. All optional chords of $C'$ in this case join endpoints of the two $u$-edges adjacent to $w'$. Therefore, by choosing $C$ as a new base cycle and those two $u$-edges as pending, we again obtain a fertile sprout of type I.
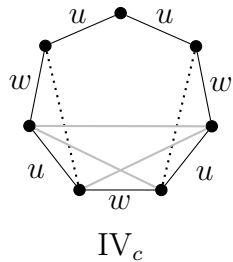
\begin{figure}[!h]
	\begin{center}
		\begin{tikzpicture}
		\node (anchor3) [] {};
		\node (label1) [position=270:1.4cm from anchor3] {IV$_c$};
		
		
		
		\node (v1) [draw,circle,fill,inner sep=1.5pt,position=90:1cm from anchor3] {};
		\node (v2) [draw,circle,fill,inner sep=1.5pt,position=141.42:1cm from anchor3] {};
		\node (v3) [draw,circle,fill,inner sep=1.5pt,position=192.84:1cm from anchor3] {};
		\node (v4) [draw,circle,fill,inner sep=1.5pt,position=244.26:1cm from anchor3] {};
		\node (v5) [draw,circle,fill,inner sep=1.5pt,position=295.68:1cm from anchor3] {};
		\node (v6) [draw,circle,fill,inner sep=1.5pt,position=347.1:1cm from anchor3] {};
		\node (v7) [draw,circle,fill,inner sep=1.5pt,position=38.52:1cm from anchor3] {};
		
		\path
		(v1) edge (v2)
		(v2) edge (v3)
		(v3) edge (v4)
		(v4) edge (v5)
		(v5) edge (v6)
		(v6) edge (v7)
		(v7) edge (v1)
		;
		
		\node (labelv1) [position=115.5:9mm from anchor3] {$u$};
		\node (labelv2) [position=166.9:9mm from anchor3] {$w$};
		\node (labelv3) [position=218.3:9mm from anchor3] {$u$};
		\node (labelv4) [position=269.7:9mm from anchor3] {$w$};
		\node (labelv5) [position=321.2:9mm from anchor3] {$u$};
		\node (labelv6) [position=12.6:9mm from anchor3] {$w$};
		\node (labelv7) [position=64:9mm from anchor3] {$u$};
		

		\path[dotted,thick]
		(v2) edge (v4)
		(v5) edge (v7)
		;
		
		\path[color=lightgray,thick]
		(v3) edge (v5)
		edge (v6)
		(v6) edge (v4)
		;
				
		\end{tikzpicture}
	\end{center}
	\caption{The sprouts of Subcase 1.3.}
	\label{fig4.27}
\end{figure}
\vspace{-1mm}
~\\
~\\
Subcase 1.5: There are two adjacent $u$-$w$-type chords.\\
Now two consecutive $u$-edges, both adjacent to a common $w$-edge, are skipped. There is one optional $uwu$-edge, but by taking $C$ as a new base cycle and the two skipped $u$-edges as pending, this $uwu$-edge joins a neighbor of the endpoint of one of those pending edges on $C$ with the other endpoint of this pending edge. Thus the resulting graph is a fertile sprout of type I and Case 1 is closed.
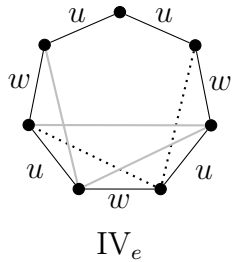
\begin{figure}[!h]
	\begin{center}
		\begin{tikzpicture}
		\node (anchor2) [] {};
		\node (label1) [position=270:1.4cm from anchor2] {IV$_e$};
		
		
		
		\node (v1) [draw,circle,fill,inner sep=1.5pt,position=90:1cm from anchor2] {};
		\node (v2) [draw,circle,fill,inner sep=1.5pt,position=141.42:1cm from anchor2] {};
		\node (v3) [draw,circle,fill,inner sep=1.5pt,position=192.84:1cm from anchor2] {};
		\node (v4) [draw,circle,fill,inner sep=1.5pt,position=244.26:1cm from anchor2] {};
		\node (v5) [draw,circle,fill,inner sep=1.5pt,position=295.68:1cm from anchor2] {};
		\node (v6) [draw,circle,fill,inner sep=1.5pt,position=347.1:1cm from anchor2] {};
		\node (v7) [draw,circle,fill,inner sep=1.5pt,position=38.52:1cm from anchor2] {};
		
		\path
		(v1) edge (v2)
		(v2) edge (v3)
		(v3) edge (v4)
		(v4) edge (v5)
		(v5) edge (v6)
		(v6) edge (v7)
		(v7) edge (v1)
		;
		
		\node (labelv1) [position=115.5:9mm from anchor2] {$u$};
		\node (labelv2) [position=166.9:9mm from anchor2] {$w$};
		\node (labelv3) [position=218.3:9mm from anchor2] {$u$};
		\node (labelv4) [position=269.7:9mm from anchor2] {$w$};
		\node (labelv5) [position=321.2:9mm from anchor2] {$u$};
		\node (labelv6) [position=12.6:9mm from anchor2] {$w$};
		\node (labelv7) [position=64:9mm from anchor2] {$u$};
		

		\path[dotted,thick]
		(v5) edge (v3)
		edge (v7)
		;
		
		\path[color=lightgray,thick]
		(v6) edge (v3)
		edge (v4)
		(v2) edge (v4)
		;

		\end{tikzpicture}
	\end{center}
	\caption{The sprouts of Subcase 1.5.}
	\label{fig4.28}
\end{figure}
\vspace{-1mm}
Case 2: The fertile sprout $S$ has exactly one pending edge $u'$.\\
Now with one pending edge and a base cycle of length $7$, there must be $4$ $w$-edges and thus no two $u$-edges can be adjacent.\\
~\\
Subcase 2.1: Both chords are adjacent to the pending edge $u'$.\\
All chords in $C'$ are either of the $u$-$w$-type or of the $uwu$-type. If two edges, both adjacent to the pending edge, produce the induced cycle $C$ of length $5$, those must be $u$-$w$-type chords. Two vertices, $v_1$ and $v_2$, of $C'$ are not contained in $C$, both of them are endpoints of a $w$-edge adjacent to the pending edge $u'$. Now all optional chords in $C'$ either have $V_1$ or $v_2$ as an endpoint. There is one other $u$-edge, the only one contained in $C$, and all possible chords join one of its endpoints with a $v_i$. For each endpoint there is one $u$-$w$-type chord and one $uwu$-type chord possible.\\
Now, if no optional edge exists, we can choose $C$ as a base cycle and add one of the two skipped vertices to obtain a III$_a$. And, even if some additional edges exist, as long as one of the skipped vertices is not incident with any chord it can be chosen to obtain a III$_a$.\\
So suppose both $u$-$w$-type edges exist and none of the $uwu$-chords, then chose the cycle given by the two skipped $w$-edges, the $u$-edge on $C$ and those two optional edges as a base cycle and add $u'$ together with one of the skipped $u$-edges as pending, thus a type I sprout is obtained.\\
So let at least one $uwu$-type chord exist, say with $v_1$ as an endpoint, and another chord with the endpoint $v_2$. Now an induced cycle of length $4$ exists and does not contain any $u$-edges, but every vertex of this cycle is adjacent to exactly one $u$-edge. Hence the induced cycle of length $4$ together with all four $u$-edges form a fertile sunflower sprout.
\begin{figure}[!h]
	\begin{center}
		\begin{tikzpicture}
		\node (anchor1) [] {};
		\node (label1) [position=270:1.4cm from anchor1] {IV$_f$};
		
		
		
		\node (v1) [draw,circle,fill,inner sep=1.5pt,position=90:1cm from anchor1] {};
		\node (v2) [draw,circle,fill,inner sep=1.5pt,position=141.42:1cm from anchor1] {};
		\node (v3) [draw,circle,fill,inner sep=1.5pt,position=192.84:1cm from anchor1] {};
		\node (v4) [draw,circle,fill,inner sep=1.5pt,position=244.26:1cm from anchor1] {};
		\node (v5) [draw,circle,fill,inner sep=1.5pt,position=295.68:1cm from anchor1] {};
		\node (v6) [draw,circle,fill,inner sep=1.5pt,position=347.1:1cm from anchor1] {};
		\node (v7) [draw,circle,fill,inner sep=1.5pt,position=38.52:1cm from anchor1] {};
		
		\path
		(v1) edge (v2)
		(v2) edge (v3)
		(v3) edge (v4)
		(v4) edge (v5)
		(v5) edge (v6)
		(v6) edge (v7)
		(v7) edge (v1)
		;
		
		\node (labelv1) [position=115.5:9mm from anchor1] {$w$};
		\node (labelv2) [position=166.9:9mm from anchor1] {$u$};
		\node (labelv3) [position=218.3:9mm from anchor1] {$w$};
		\node (labelv4) [position=269.7:9mm from anchor1] {$u$};
		\node (labelv5) [position=321.2:9mm from anchor1] {$w$};
		\node (labelv6) [position=12.6:9mm from anchor1] {$u$};
		\node (labelv7) [position=64:9mm from anchor1] {$w$};
		
		
		\node (a) [draw,circle,fill,inner sep=1.5pt,position=90:1.8cm from anchor1] {};
		\path
		(v1) edge (a)
		;
		\node (alabel) [position=90:2cm from anchor1] {$u$};
		
		\path[dotted,thick]
		(v1) edge (v3)
		edge (v6)
		;
		
		\path[color=lightgray,thick]
		(v4) edge (v2)
		edge (v7)
		(v5) edge (v2)
		edge (v7)
		;
		
		\path[color=lightgray,thick]
		(a) edge [bend right] (v2)
		edge [bend left] (v7)
		;
		\end{tikzpicture}
	\end{center}
	\caption{The sprouts of Subcase 2.1.}
	\label{fig4.29}
\end{figure}
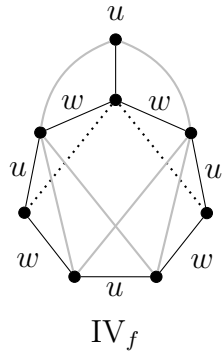
\vspace{-1mm}
~\\
Subcase 2.2: Just one of the two necessary chords is adjacent to $u'$ and the chords are adjacent.\\
Both chords $c$ and $c'$ share a common endpoint $v$ which is incident to a $u$- and a $w$-edge. The two neighbors $v_u$ and $v_w$ of $v$ are possible endpoints for all $uwu$-type chords that may exist. If there is no such edge and $v_uv_w$ does not exist, the endpoint of the $u$-edge incident with $v$, $v_u$, together with the vertices of $C$ induces a III$_a$. If $v_uv_w$ exists but still no $uwu$-type chord, take $C$, $u'$ and $vv_u$ to obtain a type I sprout.\\
Now suppose the $uwu$-type chord with endpoint $v_u$ exists, then again $vv_u$ can be taken as a second pending edge together with $u'$, but this time the base cycle is given by the induced cycle of length $5$ produced by the $uwu$-chord and $C'$. If $v_u$ is not incident to such a chord, but the other $uwu$-type chord exists there is an induced cycle of length $4$ given by this chord, one of the necessary chords $c$ and $c'$ and two $w$-edges of $C$. Each of the four vertices of this cycle is incident with exactly one $u$-edge and thus a sunflower sprout is obtained.
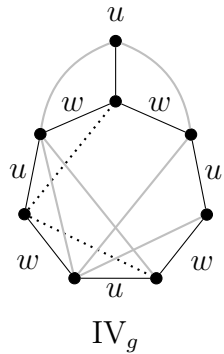
\begin{figure}[!h]
	\begin{center}
		\begin{tikzpicture}
		\node (anchor2) [] {};
		\node (label1) [position=270:1.4cm from anchor2] {IV$_g$};
		
		
		
		\node (v1) [draw,circle,fill,inner sep=1.5pt,position=90:1cm from anchor2] {};
		\node (v2) [draw,circle,fill,inner sep=1.5pt,position=141.42:1cm from anchor2] {};
		\node (v3) [draw,circle,fill,inner sep=1.5pt,position=192.84:1cm from anchor2] {};
		\node (v4) [draw,circle,fill,inner sep=1.5pt,position=244.26:1cm from anchor2] {};
		\node (v5) [draw,circle,fill,inner sep=1.5pt,position=295.68:1cm from anchor2] {};
		\node (v6) [draw,circle,fill,inner sep=1.5pt,position=347.1:1cm from anchor2] {};
		\node (v7) [draw,circle,fill,inner sep=1.5pt,position=38.52:1cm from anchor2] {};
		
		\path
		(v1) edge (v2)
		(v2) edge (v3)
		(v3) edge (v4)
		(v4) edge (v5)
		(v5) edge (v6)
		(v6) edge (v7)
		(v7) edge (v1)
		;
		
		\node (labelv1) [position=115.5:9mm from anchor2] {$w$};
		\node (labelv2) [position=166.9:9mm from anchor2] {$u$};
		\node (labelv3) [position=218.3:9mm from anchor2] {$w$};
		\node (labelv4) [position=269.7:9mm from anchor2] {$u$};
		\node (labelv5) [position=321.2:9mm from anchor2] {$w$};
		\node (labelv6) [position=12.6:9mm from anchor2] {$u$};
		\node (labelv7) [position=64:9mm from anchor2] {$w$};
		
		
		\node (a) [draw,circle,fill,inner sep=1.5pt,position=90:1.8cm from anchor2] {};
		\path
		(v1) edge (a)
		;
		\node (alabel) [position=90:2cm from anchor2] {$u$};
		
		\path[dotted,thick]
		(v3) edge (v1)
		edge (v5)
		;
		
		\path[color=lightgray,thick]
		(v2) edge (v5)
		(v4) edge (v2)
		edge (v6)
		edge (v7)
		;
		
		\path[color=lightgray,thick]
		(a) edge [bend right] (v2)
		edge [bend left] (v7)
		;

		\end{tikzpicture}
	\end{center}
	\caption{The sprouts of Subcase 2.2.}
	\label{fig4.30}
\end{figure}
\vspace{-1mm} 
~\\
Subcase 2.3: Just one of the two necessary chords is adjacent to $u'$ and the chords are not adjacent.\\
In this case all possible endpoint for $uwu$-type chords are part of $C$ and therefore no such chord can exist. Thus in all cases $C$ is a base cycle together with $u'$ and the $u$-edge which is not contained in $C$, but is adjacent to a $w$-edge adjacent to $u'$ as pending edges form a type I sprout.
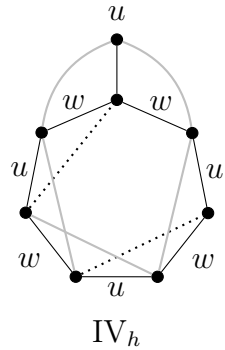
\begin{figure}[!h]
	\begin{center}
		\begin{tikzpicture}
		\node (anchor3) [] {};
		\node (label1) [position=270:1.4cm from anchor3] {IV$_h$};
		
		
		
		\node (v1) [draw,circle,fill,inner sep=1.5pt,position=90:1cm from anchor3] {};
		\node (v2) [draw,circle,fill,inner sep=1.5pt,position=141.42:1cm from anchor3] {};
		\node (v3) [draw,circle,fill,inner sep=1.5pt,position=192.84:1cm from anchor3] {};
		\node (v4) [draw,circle,fill,inner sep=1.5pt,position=244.26:1cm from anchor3] {};
		\node (v5) [draw,circle,fill,inner sep=1.5pt,position=295.68:1cm from anchor3] {};
		\node (v6) [draw,circle,fill,inner sep=1.5pt,position=347.1:1cm from anchor3] {};
		\node (v7) [draw,circle,fill,inner sep=1.5pt,position=38.52:1cm from anchor3] {};
		
		\path
		(v1) edge (v2)
		(v2) edge (v3)
		(v3) edge (v4)
		(v4) edge (v5)
		(v5) edge (v6)
		(v6) edge (v7)
		(v7) edge (v1)
		;
		
		\node (labelv1) [position=115.5:9mm from anchor3] {$w$};
		\node (labelv2) [position=166.9:9mm from anchor3] {$u$};
		\node (labelv3) [position=218.3:9mm from anchor3] {$w$};
		\node (labelv4) [position=269.7:9mm from anchor3] {$u$};
		\node (labelv5) [position=321.2:9mm from anchor3] {$w$};
		\node (labelv6) [position=12.6:9mm from anchor3] {$u$};
		\node (labelv7) [position=64:9mm from anchor3] {$w$};
		
		
		\node (a) [draw,circle,fill,inner sep=1.5pt,position=90:1.8cm from anchor3] {};
		\path
		(v1) edge (a)
		;
		\node (alabel) [position=90:2cm from anchor3] {$u$};
		
		\path[dotted,thick]
		(v3) edge (v1)
		(v4) edge (v6)
		;
		
		\path[color=lightgray,thick]
		(v3) edge (v5)
		(v4) edge (v2)
		(v5) edge (v7)
		;
		
		\path[color=lightgray,thick]
		(a) edge [bend right] (v2)
		edge [bend left] (v7)
		;
		\end{tikzpicture}
	\end{center}
	\caption{The sprouts of Subcase 2.3.}
	\label{fig4.31}
\end{figure}
\vspace{-1mm} 
~\\
~\\
Subcase 2.4: No necessary chord is adjacent to the pending edge, but both chords are adjacent to each other.\\
Now one vertex $v$ of $C'$ exists, that is not contained in $C$ but may be the endpoint of a $uwu$-type chord. By taking $C$ as a new base cycle and the skipped $u$-edge that does not contain $v$ as a second pending edge together with $u'$ a type I sprout is obtained.
\begin{figure}[!h]
	\begin{center}
		\begin{tikzpicture}
		\node (anchor1) [] {};
		\node (label1) [position=270:1.4cm from anchor1] {IV$_i$};
		
		
		
		\node (v1) [draw,circle,fill,inner sep=1.5pt,position=90:1cm from anchor1] {};
		\node (v2) [draw,circle,fill,inner sep=1.5pt,position=141.42:1cm from anchor1] {};
		\node (v3) [draw,circle,fill,inner sep=1.5pt,position=192.84:1cm from anchor1] {};
		\node (v4) [draw,circle,fill,inner sep=1.5pt,position=244.26:1cm from anchor1] {};
		\node (v5) [draw,circle,fill,inner sep=1.5pt,position=295.68:1cm from anchor1] {};
		\node (v6) [draw,circle,fill,inner sep=1.5pt,position=347.1:1cm from anchor1] {};
		\node (v7) [draw,circle,fill,inner sep=1.5pt,position=38.52:1cm from anchor1] {};
		
		\path
		(v1) edge (v2)
		(v2) edge (v3)
		(v3) edge (v4)
		(v4) edge (v5)
		(v5) edge (v6)
		(v6) edge (v7)
		(v7) edge (v1)
		;
		
		\node (labelv1) [position=115.5:9mm from anchor1] {$w$};
		\node (labelv2) [position=166.9:9mm from anchor1] {$u$};
		\node (labelv3) [position=218.3:9mm from anchor1] {$w$};
		\node (labelv4) [position=269.7:9mm from anchor1] {$u$};
		\node (labelv5) [position=321.2:9mm from anchor1] {$w$};
		\node (labelv6) [position=12.6:9mm from anchor1] {$u$};
		\node (labelv7) [position=64:9mm from anchor1] {$w$};
		
		
		\node (a) [draw,circle,fill,inner sep=1.5pt,position=90:1.8cm from anchor1] {};
		\path
		(v1) edge (a)
		;
		\node (alabel) [position=90:2cm from anchor1] {$u$};
		
		\path[dotted,thick]
		(v4) edge (v2)
		edge (v6)
		;
		
		\path[color=lightgray,thick]
		(v3) edge (v1)
		edge (v5)
		(v5) edge (v7)
		edge (v2)
		;
		
		\path[color=lightgray,thick]
		(a) edge [bend right] (v2)
		edge [bend left] (v7)
		;

		\end{tikzpicture}
	\end{center}
	\caption{The sprouts of Subcase 2.4.}
	\label{fig4.32}
\end{figure}
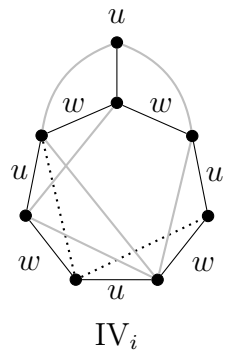
\vspace{-1mm}
~\\
Subcase 2.5: No necessary chord is adjacent to the pending edge nor are those chords adjacent.\\
Like in Subcase 2.3 all vertices that are possible endpoints for $uwu$-type chords are contained in $C$. But as a difference some of the optional $u$-$w$-type chords might represent illegal  chords in a type I sprout. Let $v$ be one of the skipped vertices on $C'$. The $u$-edge containing $v$ can be chosen as a second pending edge to the new base cycle $C$ if and only if the optional edge joining $v$ to the endpoint of the other necessary chords. So suppose both of those edges exist. Now consider the induced cycle $C''$ of length $5$ given by the two adjacent $w$-edges of $u'$, one $u$-edge, one of the two optional edges and the adjacent necessary edge. Take $u'$ and the $u$-edge adjacent to the necessary chord on $C''$ as pending edges and again a type I sprout is obtained.
\begin{figure}[!h]
	\begin{center}
		\begin{tikzpicture}
		\node (anchor3) [] {};
		\node (label1) [position=270:1.4cm from anchor3] {IV$_j$};
		
		
		
		\node (v1) [draw,circle,fill,inner sep=1.5pt,position=90:1cm from anchor3] {};
		\node (v2) [draw,circle,fill,inner sep=1.5pt,position=141.42:1cm from anchor3] {};
		\node (v3) [draw,circle,fill,inner sep=1.5pt,position=192.84:1cm from anchor3] {};
		\node (v4) [draw,circle,fill,inner sep=1.5pt,position=244.26:1cm from anchor3] {};
		\node (v5) [draw,circle,fill,inner sep=1.5pt,position=295.68:1cm from anchor3] {};
		\node (v6) [draw,circle,fill,inner sep=1.5pt,position=347.1:1cm from anchor3] {};
		\node (v7) [draw,circle,fill,inner sep=1.5pt,position=38.52:1cm from anchor3] {};
		
		\path
		(v1) edge (v2)
		(v2) edge (v3)
		(v3) edge (v4)
		(v4) edge (v5)
		(v5) edge (v6)
		(v6) edge (v7)
		(v7) edge (v1)
		;
		
		\node (labelv1) [position=115.5:9mm from anchor3] {$w$};
		\node (labelv2) [position=166.9:9mm from anchor3] {$u$};
		\node (labelv3) [position=218.3:9mm from anchor3] {$w$};
		\node (labelv4) [position=269.7:9mm from anchor3] {$u$};
		\node (labelv5) [position=321.2:9mm from anchor3] {$w$};
		\node (labelv6) [position=12.6:9mm from anchor3] {$u$};
		\node (labelv7) [position=64:9mm from anchor3] {$w$};
		
		
		\node (a) [draw,circle,fill,inner sep=1.5pt,position=90:1.8cm from anchor3] {};
		\path
		(v1) edge (a)
		;
		\node (alabel) [position=90:2cm from anchor3] {$u$};
		
		\path[dotted,thick]
		(v2) edge (v4)
		(v5) edge (v7)
		;
		
		\path[color=lightgray,thick]
		(v1) edge (v3)
		(v4) edge (v6)
		(v6) edge (v1)
		(v3) edge (v5)
		;
		
		\path[color=lightgray,thick]
		(a) edge [bend right] (v2)
		edge [bend left] (v7)
		;
		\end{tikzpicture}
	\end{center}
	\caption{The sprouts of Subcase 2.5.}
	\label{fig4.33}
\end{figure}
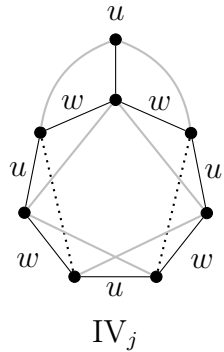
\vspace{-1mm}
\newpage
Subcase 2.6: There is just one chord shortening $C'$.\\
A chord shortening a cycle of length $7$ to a cycle of length $5$ in a sprout must be of the $uwu$-type. As seen before there are just two such chords possible and one endpoint of such a chord is the endpoint $v$ of one of the $w$-edges adjacent to $u'$. Take $C$ as a new base cycle and two pending edges, $u'$ and the $u$-edge incident with $v$, and again a sprout of type I is obtained. This closes Case 2 and completes the proof.
\begin{figure}[!h]
	\begin{center}
		\begin{tikzpicture}
		\node (anchor4) [] {};
		\node (label4) [position=270:1.4cm from anchor4] {IV$_k$};

		
		
		\node (v1) [draw,circle,fill,inner sep=1.5pt,position=90:1cm from anchor4] {};
		\node (v2) [draw,circle,fill,inner sep=1.5pt,position=141.42:1cm from anchor4] {};
		\node (v3) [draw,circle,fill,inner sep=1.5pt,position=192.84:1cm from anchor4] {};
		\node (v4) [draw,circle,fill,inner sep=1.5pt,position=244.26:1cm from anchor4] {};
		\node (v5) [draw,circle,fill,inner sep=1.5pt,position=295.68:1cm from anchor4] {};
		\node (v6) [draw,circle,fill,inner sep=1.5pt,position=347.1:1cm from anchor4] {};
		\node (v7) [draw,circle,fill,inner sep=1.5pt,position=38.52:1cm from anchor4] {};
		
		\path
		(v1) edge (v2)
		(v2) edge (v3)
		(v3) edge (v4)
		(v4) edge (v5)
		(v5) edge (v6)
		(v6) edge (v7)
		(v7) edge (v1)
		;
		
		\node (labelv1) [position=115.5:9mm from anchor4] {$w$};
		\node (labelv2) [position=166.9:9mm from anchor4] {$u$};
		\node (labelv3) [position=218.3:9mm from anchor4] {$w$};
		\node (labelv4) [position=269.7:9mm from anchor4] {$u$};
		\node (labelv5) [position=321.2:9mm from anchor4] {$w$};
		\node (labelv6) [position=12.6:9mm from anchor4] {$u$};
		\node (labelv7) [position=64:9mm from anchor4] {$w$};
		
		
		\node (a) [draw,circle,fill,inner sep=1.5pt,position=90:1.8cm from anchor4] {};
		\path
		(v1) edge (a)
		;
		\node (alabel) [position=90:2cm from anchor4] {$u$};
		
		\path[dotted,thick]
		(v2) edge (v5)
		;
		
		\path[color=lightgray,thick]
		(v1) edge (v3)
		(v4) edge (v6)
		(v4) edge (v7)
		(v3) edge (v5)
		;
		
		\path[color=lightgray,thick]
		(a) edge [bend right] (v2)
		edge [bend left] (v7)
		;
		\end{tikzpicture}
	\end{center}
	\caption{The sprouts of Subcase 2.6.}
	\label{fig4.34}
\end{figure}
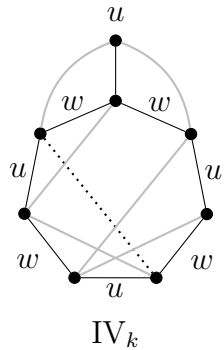
\vspace{-1mm}    
\end{proof}

So at last there is just one case left to be investigated: the base cycles of length $8$. As we did before, we simply enumerate all possibilities that can occur and show how they can be reduced to a subgraph which was previously investigated.

\begin{lemma}\label{lemma4.13}
	Let $S=\lb V,U\cup W\cup E\rb$ be a fertile sprout of size $4$ with a longest induced cycle $C$ of length $5$, $\E{C}\cap E\neq\emptyset$ and a cycle $C'$ with $\E{C'}\cap E=\emptyset$ as well as $\abs{C'}=8$, then $S$ contains a fertile sprout $S'$ of type I (see \autoref{fig4.19}) or III$_a$ (see \autoref{fig4.20}) or a fertile sunflower sprout of size $4$ (see \autoref{fig4.2}). 
\end{lemma}

\begin{proof}
With $C'$ being of length $8$, there cannot be any pending edges, nor can adjacent pairs of $u$-edges exist, therefore the edges of $C'$ alternate between $u$- and $w$-edges. Such a cycle may contain $u$-$w$-type chords and $uwu$-type chords, so this time we will divide the proof in two cases, one in which $C$ does not contain any $uwu$-type edge and one in which it does.\\
If $C$ contains no such edge, it must contain exactly three $u$-$w$-type chords, otherwise there is exactly one $uwu$-type and one $u$-$w$-type chord.
\begin{enumerate}
	\item[] \begin{enumerate}
		\item [Case 1] The cycle $C$ does not contain a $uwu$-type chord of $C'$.
		\begin{enumerate}
			\item[] \begin{enumerate}
				
				\item [Subcase 1.1] The chords $c_1$, $c_2$ and $c_3$ form a path of length 2.
				\item [Subcase 1.2] The chord $c_1$ is neither adjacent to $c_2$, nor to $c_3$.
			\end{enumerate}
		\end{enumerate}
		
		\item [Case 2] The cycle $C$ contains a $uwu$-type chord of $C'$.
		\begin{enumerate}
			\item[] \begin{enumerate}
				
				\item [Subcase 2.1] The $uwu$-type chord $c$ is adjacent to the $u$-$w$-type chord $c'$.
				\item [Subcase 2.2] The $uwu$-type chord $c$ and the $u$-$w$-type chord $c'$ are not adjacent.
			\end{enumerate}
		\end{enumerate}
	\end{enumerate}
\end{enumerate}
\vspace{-2mm}
Case 1: The cycle $C$ does not contain a $uwu$-type chord of $C'$.\\
As said before $C$ must consist of exactly three $u$-$w$-type chords and two edges of the base cycle $C$, those three necessary chords will be called $c_1$, $c_2$ and $c_3$.\\
Subcase 1.1: The chords $c_1$. $c_2$ and $c_3$ form a path of length 2.\\
With this $C$ consists of the three necessary chords and an $u$-edge together with one of its adjacent $w$-edges. This $w$-edge is adjacent to another $u$-edge, which is not contained in $C$, this will be the first pending edge. In addition this $u$-edge is, together with another $u$-edge, adjacent to  $c_3$. This second $u$-edge, which also is not part of $C$, will be chosen as the second pending edge and forms a type I sprout together with $C$ as a base cycle. 
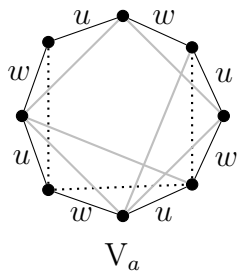
\begin{figure}[!h]
	\begin{center}
		\begin{tikzpicture}
		\node (anchor) [] {};
		\node (label) [position=270:1.4cm from anchor] {V$_a$};


		
		\node (v1) [draw,circle,fill,inner sep=1.5pt,position=90:1.1cm from anchor] {};
		\node (v2) [draw,circle,fill,inner sep=1.5pt,position=135:1.1cm from anchor] {};
		\node (v3) [draw,circle,fill,inner sep=1.5pt,position=180:1.1cm from anchor] {};
		\node (v4) [draw,circle,fill,inner sep=1.5pt,position=225:1.1cm from anchor] {};
		\node (v5) [draw,circle,fill,inner sep=1.5pt,position=270:1.1cm from anchor] {};
		\node (v6) [draw,circle,fill,inner sep=1.5pt,position=315.1:1cm from anchor] {};
		\node (v7) [draw,circle,fill,inner sep=1.5pt,position=0:1.1cm from anchor] {};
		\node (v8) [draw,circle,fill,inner sep=1.5pt,position=45:1cm from anchor] {};
		
		\path
		(v1) edge (v2)
		(v2) edge (v3)
		(v3) edge (v4)
		(v4) edge (v5)
		(v5) edge (v6)
		(v6) edge (v7)
		(v7) edge (v8)
		(v8) edge (v1)
		;
		
		\node (labelv1) [position=112.5:10mm from anchor] {$u$};
		\node (labelv2) [position=157.5:10mm from anchor] {$w$};
		\node (labelv3) [position=202.5:10mm from anchor] {$u$};
		\node (labelv4) [position=247.5:10mm from anchor] {$w$};
		\node (labelv5) [position=292.5:10mm from anchor] {$u$};
		\node (labelv6) [position=337.2:10mm from anchor] {$w$};
		\node (labelv7) [position=22.5:10mm from anchor] {$u$};
		\node (labelv8) [position=67.5:10mm from anchor] {$w$};
		

		\path[dotted,thick]
		(v2) edge (v4)
		(v4) edge (v6)
		(v6) edge (v8)
		;
		
		\path[color=lightgray,thick]
		(v3) edge (v1)
			 edge (v5)
			 edge (v6)
		(v5) edge (v7)
			 edge (v8)
		(v7) edge (v1)
		;

		\end{tikzpicture}
	\end{center}
	\caption{The sprouts of Subcase 1.1.}
	\label{fig4.35}
\end{figure}
\vspace{-1mm}
Subcase 1.2: The chord $c_1$ is neither to $c_2$, nor to $c_3$ adjacent.\\
Again $C$ consists of one $u$- and one $w$-edge, $u'$ and $w'$, but this time both are adjacent to the sole chord $c_1$. Let $v$ be the vertex on $C'$ skipped by $c_1$, then $v$ may be adjacent to all four endpoints of $u'$ and $w'$, two of those edges are part of $C'$ and two are optional ones. If neither of the two edges exist, $v$ together with $C$ forms a III$_a$. So suppose just one of the two exists. This gives us another induced cycle of length $5$, given by the pair of adjacent necessary chords $c_2$ and $c_3$, the optional chord, which is either adjacent to $c_2$ or $c_3$ and a pair of adjacent original edges of $C'$. The resulting graph resembles Subcase 1.1 and so a type I sprout exists.\\
So now consider the case where both of these chords exist. The two chord together with $c_2$ and $c_3$ form an induced cycle of length $4$ which does not contain any $u$-edges, but each of the four vertices is adjacent to exactly one $u$-edge. This gives us a fertile sunflower sprout.
\begin{figure}[!h]
	\begin{center}
		\begin{tikzpicture}
		\node (anchor) [] {};
		\node (label) [position=270:1.4cm from anchor] {V$_a$};

		
		
		\node (v1) [draw,circle,fill,inner sep=1.5pt,position=90:1.1cm from anchor] {};
		\node (v2) [draw,circle,fill,inner sep=1.5pt,position=135:1.1cm from anchor] {};
		\node (v3) [draw,circle,fill,inner sep=1.5pt,position=180:1.1cm from anchor] {};
		\node (v4) [draw,circle,fill,inner sep=1.5pt,position=225:1.1cm from anchor] {};
		\node (v5) [draw,circle,fill,inner sep=1.5pt,position=270:1.1cm from anchor] {};
		\node (v6) [draw,circle,fill,inner sep=1.5pt,position=315.1:1cm from anchor] {};
		\node (v7) [draw,circle,fill,inner sep=1.5pt,position=0:1.1cm from anchor] {};
		\node (v8) [draw,circle,fill,inner sep=1.5pt,position=45:1cm from anchor] {};
		
		\path
		(v1) edge (v2)
		(v2) edge (v3)
		(v3) edge (v4)
		(v4) edge (v5)
		(v5) edge (v6)
		(v6) edge (v7)
		(v7) edge (v8)
		(v8) edge (v1)
		;
		
		\node (labelv1) [position=112.5:10mm from anchor] {$u$};
		\node (labelv2) [position=157.5:10mm from anchor] {$w$};
		\node (labelv3) [position=202.5:10mm from anchor] {$u$};
		\node (labelv4) [position=247.5:10mm from anchor] {$w$};
		\node (labelv5) [position=292.5:10mm from anchor] {$u$};
		\node (labelv6) [position=337.2:10mm from anchor] {$w$};
		\node (labelv7) [position=22.5:10mm from anchor] {$u$};
		\node (labelv8) [position=67.5:10mm from anchor] {$w$};
		

		\path[dotted,thick]
		(v2) edge (v4)
		(v5) edge (v7)
		(v7) edge (v1)
		;
		
		\path[color=lightgray,thick]
		(v3) edge (v1)
		edge (v5)
		edge (v6)
		(v6) edge (v4)
		edge (v8)
		(v5) edge (v8)
		;

		\end{tikzpicture}
	\end{center}
	\caption{The sprouts of Subcase 1.2.}
	\label{fig4.36}
\end{figure}
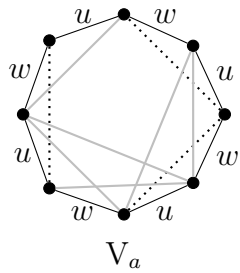
\vspace{-1mm}
~\\
Case 2: The cycle $C$ does contain a $uwu$-type chord of $C'$.\\
Now $C$ consists of three edges of the base cycle $C'$ and just of two chords. Let $c$ be the $uwu$-type chord and $c'$ the $u$-$w$-type one.\\
~\\
Subcase 2.1: The $uwu$-type chord $c$ is adjacent to the $u$-$w$-type chord $c'$.\\
There are two $u$-edges adjacent to $c'$ and all optional edges they can be adjacent to either join two endpoints of those $u$-edges, or join one of its endpoints to a neighbor of the other endpoint of the $u$-edge. So $C$ together with the $u$-edges adjacent to $c'$ forms a type I sprout.
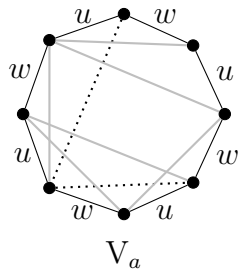
\begin{figure}[!h]
	\begin{center}
		\begin{tikzpicture}
		\node (anchor) [] {};
		\node (label) [position=270:1.4cm from anchor] {V$_a$};

		
		
		\node (v1) [draw,circle,fill,inner sep=1.5pt,position=90:1.1cm from anchor] {};
		\node (v2) [draw,circle,fill,inner sep=1.5pt,position=135:1.1cm from anchor] {};
		\node (v3) [draw,circle,fill,inner sep=1.5pt,position=180:1.1cm from anchor] {};
		\node (v4) [draw,circle,fill,inner sep=1.5pt,position=225:1.1cm from anchor] {};
		\node (v5) [draw,circle,fill,inner sep=1.5pt,position=270:1.1cm from anchor] {};
		\node (v6) [draw,circle,fill,inner sep=1.5pt,position=315.1:1cm from anchor] {};
		\node (v7) [draw,circle,fill,inner sep=1.5pt,position=0:1.1cm from anchor] {};
		\node (v8) [draw,circle,fill,inner sep=1.5pt,position=45:1cm from anchor] {};
		
		\path
		(v1) edge (v2)
		(v2) edge (v3)
		(v3) edge (v4)
		(v4) edge (v5)
		(v5) edge (v6)
		(v6) edge (v7)
		(v7) edge (v8)
		(v8) edge (v1)
		;
		
		\node (labelv1) [position=112.5:10mm from anchor] {$u$};
		\node (labelv2) [position=157.5:10mm from anchor] {$w$};
		\node (labelv3) [position=202.5:10mm from anchor] {$u$};
		\node (labelv4) [position=247.5:10mm from anchor] {$w$};
		\node (labelv5) [position=292.5:10mm from anchor] {$u$};
		\node (labelv6) [position=337.2:10mm from anchor] {$w$};
		\node (labelv7) [position=22.5:10mm from anchor] {$u$};
		\node (labelv8) [position=67.5:10mm from anchor] {$w$};
		

		\path[dotted,thick]
		(v1) edge (v4)
		(v4) edge (v6)
		;
		
		\path[color=lightgray,thick]
		(v2) edge (v4)
		edge (v7)
		edge (v8)
		(v3) edge (v5)
		edge (v6)
		(v5) edge (v7)
		;

		\end{tikzpicture}
	\end{center}
	\caption{The sprouts of Subcase 2.1.}
	\label{fig4.37}
\end{figure}
\vspace{-1mm}
\newpage
Subcase 2.2: The $uwu$-type chord $c$ and the $u$-$w$-type chord $c'$ are not adjacent.\\
This subcase is a little more complicated than the previous ones, so we will need some more work before going into more subcases. Now, with $C$ consisting of a $uwu$-type and a $u$-$w$-type chord, it contains exactly one $u$-edge and two $w$-edges, furthermore with $c$ and $c'$ not being adjacent the three original edges do not form a path. So there is one pair of adjacent original edges, a $u$- and a $w$-one, let $v_{uw}$ be their common vertex, and there is one single $w$-edge adjacent to both $c$, with the endpoint $v$ and $c'$, with the endpoint $v'$. In addition both $c$ and $c'$ skip exactly one $w$-edge, let $x_1$ and $x_2$ be the endpoints of this edge skipped by $c$. Each of these vertices may be joined to exactly one endpoint of the other $w$-edge, skipped by $c'$, let $y_1$ be the one which can be joined to $x_1$ and $y_2$ the other one.\\
Now suppose the edge $y_2v_{uw}$ does not exist, then $C$  forms a type I sprout together with the two $u$-edges incident with $v$ and $v'$.\\
Therefore let $y_2v_{uw}$ exist in the following.\\
In addition suppose that $x_1y_1$ exists, but $x_2y_2$ does not. If $v'x_2$ does not exist too we obtain an induced cycle of length $5$ consisting of three consecutive edges of $C'$, and a $uwu$-type chord as well as a $u$-$w$-type chord and both chords are adjacent. With this we resemble the graph in Subcase 2.1 and are done. So let $v'x_2$ exist. Now $v'x_2$, $c'$ and $x_1y_1$ together with one $w$-edge form an induced cycle of length $4$ and by choosing the $u$-edges as pending we obtain a fertile sunflower sprout.\\
So suppose it is the other way around, $x_2y_2$ exists, but $x_1y_1$ does not. If furthermore the edge $v_{uw}x_1$ does not exist we again have another induced cycle of length $5$ resembling Subcase 2.1. So let $v_{uw}x_1$ exist, then another induced cycle of length $4$ without any $u$-edges is obtained, resulting in the existence of a fertile sunflower sprout.\\
At last let $x_1y_1$ together with $x_2y_2$ exist. Those two chords together with the two $w$-edges adjacent to both chords form another induced cycle of length $4$, another fertile sunflower sprout is obtained and the case is closed.
\begin{figure}[!h]
	\begin{center}
		\begin{tikzpicture}
		\node (anchor) [] {};
		\node (label) [position=270:1.4cm from anchor] {V$_a$};

		
		
		\node (v1) [draw,circle,fill,inner sep=1.5pt,position=90:1.1cm from anchor] {};
		\node (v2) [draw,circle,fill,inner sep=1.5pt,position=135:1.1cm from anchor] {};
		\node (v3) [draw,circle,fill,inner sep=1.5pt,position=180:1.1cm from anchor] {};
		\node (v4) [draw,circle,fill,inner sep=1.5pt,position=225:1.1cm from anchor] {};
		\node (v5) [draw,circle,fill,inner sep=1.5pt,position=270:1.1cm from anchor] {};
		\node (v6) [draw,circle,fill,inner sep=1.5pt,position=315.1:1cm from anchor] {};
		\node (v7) [draw,circle,fill,inner sep=1.5pt,position=0:1.1cm from anchor] {};
		\node (v8) [draw,circle,fill,inner sep=1.5pt,position=45:1cm from anchor] {};
		
		\path
		(v1) edge (v2)
		(v2) edge (v3)
		(v3) edge (v4)
		(v4) edge (v5)
		(v5) edge (v6)
		(v6) edge (v7)
		(v7) edge (v8)
		(v8) edge (v1)
		;
		
		\node (labelv1) [position=112.5:10mm from anchor] {$u$};
		\node (labelv2) [position=157.5:10mm from anchor] {$w$};
		\node (labelv3) [position=202.5:10mm from anchor] {$u$};
		\node (labelv4) [position=247.5:10mm from anchor] {$w$};
		\node (labelv5) [position=292.5:10mm from anchor] {$u$};
		\node (labelv6) [position=337.2:10mm from anchor] {$w$};
		\node (labelv7) [position=22.5:10mm from anchor] {$u$};
		\node (labelv8) [position=67.5:10mm from anchor] {$w$};
		

		\path[dotted,thick]
		(v1) edge (v4)
		(v5) edge (v7)
		;
		
		\path[color=lightgray,thick]
		(v2) edge (v4)
		edge (v7)
		edge (v8)
		(v3) edge (v5)
		edge (v6)
		(v4) edge (v6)
		(v6) edge (v8)
		;

		\end{tikzpicture}
	\end{center}
	\caption{The sprouts of Subcase 2.2.}
	\label{fig4.38}
\end{figure}
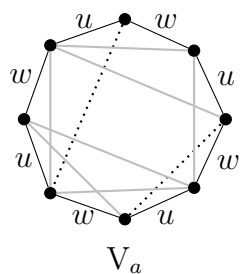
\vspace{-1em}    
\end{proof}
It is very easy to see that a fertile sunflower of size $5$ always contains a fertile sprout of type I as seen in \autoref{fig4.19}. So in summary all graphs that are responsible for the existence of an induced cycle in the squared line graph of a graph without induced cycles of length at least $6$ can be reduced to the $33$ graphs in \autoref{fig4.39}, namely a sprout of type I, or the III$_a$.
\begin{figure}[!h]
	\begin{center}
		\begin{tikzpicture}
		\node (anchor1) [] {};
		\node (i) [position=270:2.25cm from anchor1] {I};
		
		\node (anchor4) [position=0:5.2cm from anchor1] {};
		\node (iv) [position=270:2.25cm from anchor4] {III$_a$};
		
		
		
		\node (v1) [draw,circle,fill,inner sep=1.5pt,position=0:0.9cm from anchor1] {};
		\node (v2) [draw,circle,fill,inner sep=1.5pt,position=72:0.9cm from anchor1] {};
		\node (v3) [draw,circle,fill,inner sep=1.5pt,position=144:0.9cm from anchor1] {};
		\node (v4) [draw,circle,fill,inner sep=1.5pt,position=216:0.9cm from anchor1] {};
		\node (v5) [draw,circle,fill,inner sep=1.5pt,position=288:0.9cm from anchor1] {};
		
		\path
		(v1) edge (v2)
		(v2) edge (v3)
		(v3) edge (v4)
		(v4) edge (v5)
		(v5) edge (v1)
		;
		
		\node (cl1) [position=36:0.7cm from anchor1] {$u_1$};
		\node (cl2) [position=108:2mm from anchor1] {$w_1$};
		\node (cl3) [position=180:0.05cm from anchor1] {$w_2$};
		\node (cl4) [position=252:2mm from anchor1] {$w_3$};
		\node (cl5) [position=324:0.7cm from anchor1] {$u_4$};
		
		
		\node (a1) [draw,circle,fill,inner sep=1.5pt,position=144:1.8cm from anchor1] {};
		\node (a2) [draw,circle,fill,inner sep=1.5pt,position=216:1.8cm from anchor1] {};
		
		\path
		(v3) edge (a1)
		(v4) edge (a2)
		;
		
		\node (al1) [position=144:2cm from anchor1] {$u_2$};
		\node (al2) [position=216:2cm from anchor1] {$u_3$};
		
		
		\path[color=lightgray,thick]
		(a1) edge [bend left] (v2)
		edge (v4)
		(a2) edge (v3)
		edge [bend right] (v5)
		(a1) edge [bend right] (a2)
		;

		
		\node (u1) [draw,circle,fill,inner sep=1.5pt,position=90:1.1cm from anchor4] {};
		\node (u2) [draw,circle,fill,inner sep=1.5pt,position=150:1.1cm from anchor4] {};
		\node (u3) [draw,circle,fill,inner sep=1.5pt,position=210:1.1cm from anchor4] {};
		\node (u4) [draw,circle,fill,inner sep=1.5pt,position=270:1.1cm from anchor4] {};
		\node (u5) [draw,circle,fill,inner sep=1.5pt,position=330:1.1cm from anchor4] {};
		\node (u6) [draw,circle,fill,inner sep=1.5pt,position=30:1.1cm from anchor4] {};
		
		\node (l1) [position=120:1cm from anchor4] {$u$};
		\node (l2) [position=180.5:1cm from anchor4] {$w$};
		\node (l3) [position=240.5:1cm from anchor4] {$u$};
		\node (l4) [position=300.5:1cm from anchor4] {$u$};
		\node (l5) [position=0.5:1cm from anchor4] {$w$};
		\node (l6) [position=60.5:1cm from anchor4] {$u$};
		
		\path
		(u1) edge (u2)
		(u2) edge (u3)
		(u3) edge (u4)
		(u4) edge (u5)
		(u5) edge (u6)
		(u6) edge (u1)
		;
		
		\path[dotted,thick]
		(u2) edge (u6)
		;

		\end{tikzpicture}
	\end{center}
	\caption{The two types of forbidden sprouts.}
	\label{fig4.39}
\end{figure}
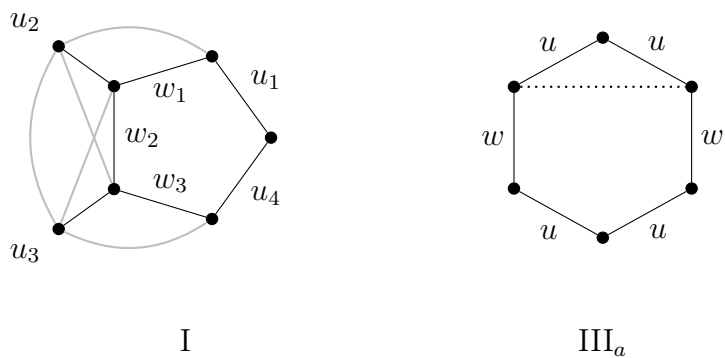
\vspace{-1mm}
~\\
Combining all these results we finally obtain a characterization of those graphs with a chordal line graph square in terms of forbidden subgraphs.
\begin{theorem}\label{thm4.15}
Let $G$ be a graph. Then $\lineg{G}^2$ is chordal if and only if $G$ does not contain a $C_n$ with $n\geq6$, a fertile sunflower sprout of size $4$, III$_a$ or an induced subgraph of type I (see \autoref{fig4.39}). 	
\end{theorem}

\subsection{A Look Into Perfection}

The Strong Perfect Graph Theorem (\autoref{thm2.3}) states that a graph is perfect if and only of it neither contains an induced cycle of odd length at least $5$, a hole, nor the induced complement of such a cycle, an anti hole.\\
Hence to find the graphs whose squared line graph are perfect we need to find those subgraphs responsible for induced cycles and induced complements of cycles of odd length. Theorem \autoref{thm4.16} gives us a description of those graphs responsible for the existence of holes: the fertile sprouts of odd size.\\
As we did in Lemma \autoref{lemma4.8}, we will start by giving an upper bound on the length of an induced cycle that may be contained in a graph with a perfect line graph square. A simple lower bound is given by Theorem \autoref{thm4.15} since chordal graphs are perfect.
\begin{lemma}\label{lemma4.14}
For all $n\geq 4$ the cycle $C_j$ with $n+\aufr{\frac{n}{2}}\leq j\leq 2\,n$ is a fertile sprout.	
\end{lemma}
\begin{proof}
By Lemma \autoref{lemma3.7} there are at most $\abr{\frac{n}{2}}$ pairs of adjacent $u$-edges in a fertile sprout of size $n$.\\
We start with $j=n+\aufr{\frac{n}{2}}$. If there are exactly $\abr{\frac{n}{2}}$ pairs of adjacent $u$-edges, then the remaining $\aufr{\frac{n}{2}}-\abr{\frac{n}{2}}\in\set{0,1}$ $u$-edges cannot be adjacent to any other $u$-edge. Hence we need exactly $\aufr{\frac{n}{2}}$ $w$-edges to complete the sprout. By alternating between $u$-edge pairs, $w$-edges and up to one single $u$-edge we obtain a cycle of length $n+\aufr{\frac{n}{2}}$. This cycle may contain chords, but all those additional edges are optional, so the $C_j$ is a fertile sprout of size $n$.\\
Each of those $\abr{\frac{n}{2}}$ pairs of $u$-edges may be split by an additional $w$-edge, hence with $k\leq\abr{\frac{n}{2}}$ such splits we can produce a cycle of length $n+\aufr{\frac{n}{2}}+k\leq2\,n$ which again is a fertile sprout of size $n$.
\end{proof}
It is easy to check that $n+\aufr{\frac{n}{2}}\leq  2\,\lb n-2\rb+1$ holds for $n\geq 7$ and for $n=5$ the cycles of Lemma \autoref{lemma4.14} are of length $j\in\set{8,9,10}$. With $7+\aufr{\frac{7}{2}}=11$ we obtain the following corollary.
\begin{corollary}\label{cor4.7}
A graph $G$ with an induced cycle of length $l\geq 8$ contains a fertile sprout of odd size.	
\end{corollary}
While at this point we are not able to give a complete description of the structures producing anti holes in the squared line graph, some basic observations can be made to reduce the length of allowed induced cycles even further.
\begin{lemma}\label{lemma4.15}
If a graph $G$ contains a $C_7$, $\lineg{G}^2$ contains an anti hole of the same size.	
\end{lemma}
\begin{proof}
With $G$ containing a $C_7$ Lemma \autoref{lemma4.2} yields the existence of an induced cycle $C$ of the same length in $\lineg{G}$. For each pair of nonadjacent vertices $ u,w\in\E{\lineg{G}}$ of $C$, which does not have a common neighbor on $C$, it holds $\distg{\lineg{G}}{u}{w}\geq 3$, since a path of length $2$ between two such vertices would correspond to a chord in the $C_7$, which does not exist.
\begin{figure}[H]
	\begin{center}
		\begin{tikzpicture}
		
		\node (anchor2) [] {};
		\node (label2) [position=270:1.4cm from anchor2] {$\lineg{G}$};
		
		\node (anchor1) [position=0:5cm from anchor2] {};
		\node (label1) [position=270:1.4cm from anchor1] {$\lineg{G}^2$};

		
		
		\node (v1) [draw,circle,fill,inner sep=1.5pt,position=90:1cm from anchor2] {};
		\node (v2) [draw,circle,fill,inner sep=1.5pt,position=141.42:1cm from anchor2] {};
		\node (v3) [draw,circle,fill,inner sep=1.5pt,position=192.84:1cm from anchor2] {};
		\node (v4) [draw,circle,fill,inner sep=1.5pt,position=244.26:1cm from anchor2] {};
		\node (v5) [draw,circle,fill,inner sep=1.5pt,position=295.68:1cm from anchor2] {};
		\node (v6) [draw,circle,fill,inner sep=1.5pt,position=347.1:1cm from anchor2] {};
		\node (v7) [draw,circle,fill,inner sep=1.5pt,position=38.52:1cm from anchor2] {};
		
		\path
		(v1) edge (v3)
		(v3) edge (v5)
		(v5) edge (v7)
		(v7) edge (v2)
		(v2) edge (v4)
		(v4) edge (v6)
		(v6) edge (v1)
		;
		
		\node (labelv1) [position=90:11mm from anchor2] {$1$};
		\node (labelv2) [position=141.42:11mm from anchor2] {$5$};
		\node (labelv3) [position=192.84:11mm from anchor2] {$2$};
		\node (labelv4) [position=244.26:11mm from anchor2] {$6$};
		\node (labelv5) [position=295.68:11mm from anchor2] {$3$};
		\node (labelv6) [position=347.1:11mm from anchor2] {$7$};
		\node (labelv7) [position=38.52:11mm from anchor2] {$4$};
		

		
		
		\node (v1) [draw,circle,fill,inner sep=1.5pt,position=90:1cm from anchor1] {};
		\node (v2) [draw,circle,fill,inner sep=1.5pt,position=141.42:1cm from anchor1] {};
		\node (v3) [draw,circle,fill,inner sep=1.5pt,position=192.84:1cm from anchor1] {};
		\node (v4) [draw,circle,fill,inner sep=1.5pt,position=244.26:1cm from anchor1] {};
		\node (v5) [draw,circle,fill,inner sep=1.5pt,position=295.68:1cm from anchor1] {};
		\node (v6) [draw,circle,fill,inner sep=1.5pt,position=347.1:1cm from anchor1] {};
		\node (v7) [draw,circle,fill,inner sep=1.5pt,position=38.52:1cm from anchor1] {};
		
		\path[color=gray]
		(v1) edge (v3)
		(v3) edge (v5)
		(v5) edge (v7)
		(v7) edge (v2)
		(v2) edge (v4)
		(v4) edge (v6)
		(v6) edge (v1)
		;
		
		\node (labelv1) [position=90:11mm from anchor1] {$1$};
		\node (labelv2) [position=141.42:11mm from anchor1] {$5$};
		\node (labelv3) [position=192.84:11mm from anchor1] {$2$};
		\node (labelv4) [position=244.26:11mm from anchor1] {$6$};
		\node (labelv5) [position=295.68:11mm from anchor1] {$3$};
		\node (labelv6) [position=347.1:11mm from anchor1] {$7$};
		\node (labelv7) [position=38.52:11mm from anchor1] {$4$};
	

		\path
		(v1) edge (v4)
		(v4) edge (v7)
		(v7) edge (v3)
		(v3) edge (v6)
		(v6) edge (v2)
		(v2) edge (v5)
		(v5) edge (v1)
		;
		\end{tikzpicture}
	\end{center}
	\vspace{-4mm}
	\caption{The cycle $C$ in the line graph and the squared line graph.}
	\label{fig4.40}
\end{figure}
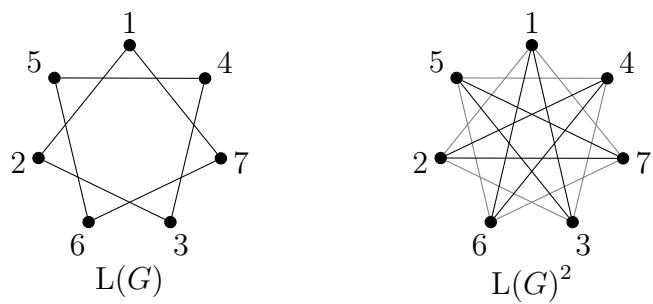
So in $\lineg{G}^2$ each vertex of $C$ is adjacent to exactly $4$ other vertices of $C$ which leaves exactly two vertices of $C$ to whom it is not adjacent. Hence $C^2$ is an anti hole.	
\end{proof}
\begin{lemma}\label{lemma4.16}
Let $G$ be a graph, if $\lineg{G}^2$ is perfect, $G$ does not contain a $C_n$ with $n\geq 7$.
\end{lemma}
Corollary \autoref{cor4.6} states that a graph without induced cycles of length $c\geq 7$ does not contain a fertile sprout of size $k\neq c$, hence such a graph's squared line graph may just contain holes of size $5$. In order to prevent such a hole we need to forbid all fertile sprouts of size $5$ in $G$. Again we can reduce the number of forbidden graphs by having a deeper look into the structure of such sprouts.\\
Since we can exclude the existence of induced cycles of length $c\geq 7$, we just have to take the sunflower sprouts of size $5$ and flowers of size $5$ and a longest induced cycle, or base cycle, of length $6$ into account. \autoref{fig4.41} shows the three possible types of size $5$ sunflower sprouts. As usual optional edges are depicted in gray and may exist in any combination.
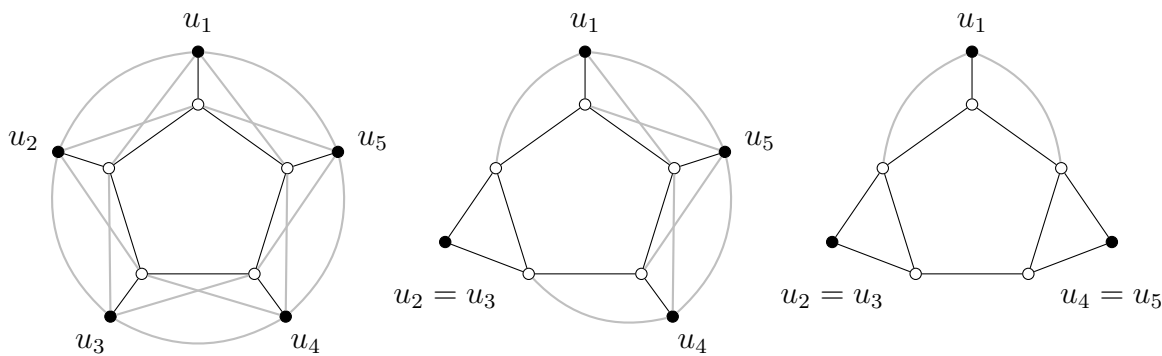
\begin{figure}[!h]
\begin{center}
\begin{tikzpicture}
\node (anchor1) [] {};

\node (anchor2) [position=0:48mm from anchor1] {};

\node (anchor3) [position=0:48mm from anchor2] {};


\node (v1) [position=90:10mm from anchor1,draw,circle,inner sep=1.5pt] {};
\node (v2) [position=162:10mm from anchor1,draw,circle,inner sep=1.5pt] {};
\node (v3) [position=234:10mm from anchor1,draw,circle,inner sep=1.5pt] {};
\node (v4) [position=306:10mm from anchor1,draw,circle,inner sep=1.5pt] {};
\node (v5) [position=18:10mm from anchor1,draw,circle,inner sep=1.5pt] {};

\path
(v1) edge (v2)
(v2) edge (v3)
(v3) edge (v4)
(v4) edge (v5)
(v5) edge (v1)
;


\node (u1) [position=90:17mm from anchor1,draw,circle,inner sep=1.5pt,fill] {};
\node (u2) [position=162:17mm from anchor1,draw,circle,inner sep=1.5pt,fill] {};
\node (u3) [position=234:17mm from anchor1,draw,circle,inner sep=1.5pt,fill] {};
\node (u4) [position=306:17mm from anchor1,draw,circle,inner sep=1.5pt,fill] {};
\node (u5) [position=18:17mm from anchor1,draw,circle,inner sep=1.5pt,fill] {};

\node (ulabel1) [position=90:19mm from anchor1] {$u_1$};
\node (ulabel2) [position=162:19mm from anchor1] {$u_2$};
\node (ulabel3) [position=234:19mm from anchor1] {$u_3$};
\node (ulabel4) [position=306:19mm from anchor1] {$u_4$};
\node (ulabel5) [position=18:19mm from anchor1] {$u_5$};

\path
(v1) edge (u1)
(v2) edge (u2)
(v3) edge (u3)
(v4) edge (u4)
(v5) edge (u5)
;

\path[color=lightgray,thick]
(u1) edge [bend right] (u2)
	 edge (v2)
	 edge (v5)
(u2) edge [bend right] (u3)
	 edge (v3)
	 edge (v1)
(u3) edge [bend right] (u4)
	 edge (v4)
	 edge (v2)
(u4) edge [bend right] (u5)
	 edge (v5)
	 edge (v3)
(u5) edge [bend right] (u1)
	 edge (v1)
	 edge (v4)
;


\node (v1) [position=90:10mm from anchor2,draw,circle,inner sep=1.5pt] {};
\node (v2) [position=162:10mm from anchor2,draw,circle,inner sep=1.5pt] {};
\node (v3) [position=234:10mm from anchor2,draw,circle,inner sep=1.5pt] {};
\node (v4) [position=306:10mm from anchor2,draw,circle,inner sep=1.5pt] {};
\node (v5) [position=18:10mm from anchor2,draw,circle,inner sep=1.5pt] {};

\path
(v1) edge (v2)
(v2) edge (v3)
(v3) edge (v4)
(v4) edge (v5)
(v5) edge (v1)
;


\node (u1) [position=90:17mm from anchor2,draw,circle,inner sep=1.5pt,fill] {};
\node (u2) [position=198:17mm from anchor2,draw,circle,inner sep=1.5pt,fill] {};
\node (u4) [position=306:17mm from anchor2,draw,circle,inner sep=1.5pt,fill] {};
\node (u5) [position=18:17mm from anchor2,draw,circle,inner sep=1.5pt,fill] {};

\node (ulabel1) [position=90:19mm from anchor2] {$u_1$};
\node (ulabel2) [position=270:4mm from u2] {$u_2=u_3$};
\node (ulabel4) [position=306:19mm from anchor2] {$u_4$};
\node (ulabel5) [position=18:19mm from anchor2] {$u_5$};

\path
(v1) edge (u1)
(v2) edge (u2)
(v3) edge (u2)
(v4) edge (u4)
(v5) edge (u5)
;

\path[color=lightgray,thick]
(u1) edge [bend right] (v2)
	 edge (v5)
(u4) edge [bend right] (u5)
	 edge (v5)
	 edge [bend left] (v3)
(u5) edge [bend right] (u1)
	 edge (v1)
	 edge (v4)
;


\node (v1) [position=90:10mm from anchor3,draw,circle,inner sep=1.5pt] {};
\node (v2) [position=162:10mm from anchor3,draw,circle,inner sep=1.5pt] {};
\node (v3) [position=234:10mm from anchor3,draw,circle,inner sep=1.5pt] {};
\node (v4) [position=306:10mm from anchor3,draw,circle,inner sep=1.5pt] {};
\node (v5) [position=18:10mm from anchor3,draw,circle,inner sep=1.5pt] {};

\path
(v1) edge (v2)
(v2) edge (v3)
(v3) edge (v4)
(v4) edge (v5)
(v5) edge (v1)
;


\node (u1) [position=90:17mm from anchor3,draw,circle,inner sep=1.5pt,fill] {};
\node (u2) [position=198:17mm from anchor3,draw,circle,inner sep=1.5pt,fill] {};
\node (u4) [position=342:17mm from anchor3,draw,circle,inner sep=1.5pt,fill] {};

\node (ulabel1) [position=90:19mm from anchor3] {$u_1$};
\node (ulabel2) [position=270:4mm from u2] {$u_2=u_3$};
\node (ulabel4) [position=270:4mm from u4] {$u_4=u_5$};

\path
(v1) edge (u1)
(v2) edge (u2)
(v3) edge (u2)
(v4) edge (u4)
(v5) edge (u4)
;

\path[color=lightgray,thick]
(u1) edge [bend right] (v2)
	 edge [bend left] (v5)
;

\end{tikzpicture}
\end{center}
\vspace{-4mm}
\caption{The sunflower sprouts of size $5$.}
\label{fig4.41}
\end{figure}
\vspace{-1mm}

We will proceed by discussing the case in which the longest induced cycle $C$ of length $6$ is the base cycle, which therefore is of length $6$ too.

\begin{lemma}\label{lemma4.17}
Let $S=\lb V,U\cup W\cup E\rb$ be a fertile sprout of size $5$ with a longest induced cycle $C$ of length $6$ and $\E{C}\cap E=\emptyset$, then $S$ has either two or three pending edges. In addition those edges are incident with consecutive vertices of $C$.	
\end{lemma}

\begin{proof}
With $S$ being a fertile sprout of size $5$ the number of pending edges is at least $0$ and at most $5$. If $S$ had $5$ pending edges, it would have exactly $5$ $w$-edges which would form the largest cycle not containing any edge of $E$, so the number of pending edges in $S$ is at most four.\\
Suppose $S$ has no pending edge, then $C$ consists of $5$ $u$-edges and by Lemma \autoref{lemma3.7} there are at least $\aufr{\frac{5}{2}}=3$ $w$-edges, which must be contained in $C$ as well since $C$ is the base cycle of $S$. This is not possible with $\abs{C}=6$.\\
With one pending edge there remain $4$ $u$-edges together with at least $3$ $w$-edges on $C$, which still is not possible, so suppose there are two pending edges. Now there are at least three $w$-edges necessary for the pending edges and since three $u$-edges may not form a path on $C$ we need at least one additional $w$-edge, thus $C$ must now consist of $3$ $u$-edges and at least $4$ $w$-edges, again exceeding the length of $C$.\\
If there are four pending edges and they are not adjacent to consecutive vertices of $C$, the base cycle must contain at least $6$ $w$-edges and an additional $u$-edge and if there are just three pending edges not being adjacent to consecutive vertices of $C$, $5$ $w$-edges and $2$ additional $u$-edges are required. 	
\end{proof}

This leaves us with just two possible types of fertile sprouts with a base cycle of length $6$ as pictured in the following figure. 

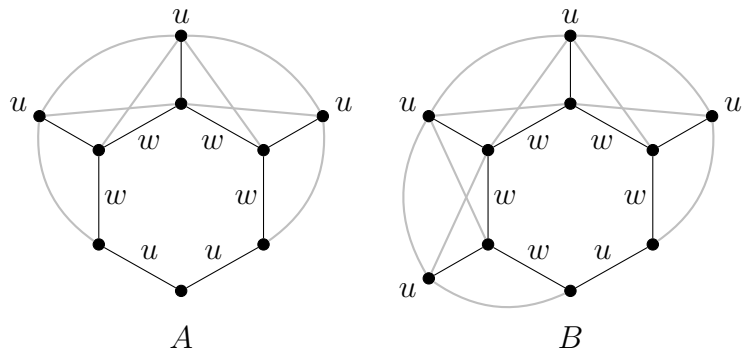
\begin{figure}[!h]
	\begin{center}
		\begin{tikzpicture}
		\node (anchor2) [inner sep=1.5pt] {};
		\node (label1) [position=270:1.5cm from anchor2] {$A$};		
		\node (anchor1) [inner sep=1.5pt,position=0:5cm from anchor2] {};
		\node (label1) [position=270:1.5cm from anchor1] {$B$};
		
		
		\node (u1) [draw,circle,fill,inner sep=1.5pt,position=90:1.1cm from anchor1] {};
		\node (u2) [draw,circle,fill,inner sep=1.5pt,position=150:1.1cm from anchor1] {};
		\node (u3) [draw,circle,fill,inner sep=1.5pt,position=210:1.1cm from anchor1] {};
		\node (u4) [draw,circle,fill,inner sep=1.5pt,position=270:1.1cm from anchor1] {};
		\node (u5) [draw,circle,fill,inner sep=1.5pt,position=330:1.1cm from anchor1] {};
		\node (u6) [draw,circle,fill,inner sep=1.5pt,position=30:1.1cm from anchor1] {};
		
		\node (v1) [draw,circle,fill,inner sep=1.5pt,position=90:2cm from anchor1] {};
		\node (v2) [draw,circle,fill,inner sep=1.5pt,position=150:2cm from anchor1] {};
		\node (v3) [draw,circle,fill,inner sep=1.5pt,position=210:2cm from anchor1] {};
		\node (v4) [draw,circle,fill,inner sep=1.5pt,position=30:2cm from anchor1] {};

		\node (l1) [position=120:5mm from anchor1] {$w$};
		\node (l2) [position=180.5:5mm from anchor1] {$w$};
		\node (l3) [position=240.5:5mm from anchor1] {$w$};
		\node (l4) [position=300.5:5mm from anchor1] {$u$};
		\node (l5) [position=0.5:5mm from anchor1] {$w$};
		\node (l6) [position=60.5:5mm from anchor1] {$w$};
		\node (l8) [position=90:2.1cm from anchor1] {$u$};
		\node (l9) [position=150:2.1cm from anchor1] {$u$};
		\node (l10) [position=210:2.1cm from anchor1] {$u$};
		\node (l11) [position=30:2.1cm from anchor1] {$u$};

		\path
		(u1) edge (u2)
		(u2) edge (u3)
		(u3) edge (u4)
		(u4) edge (u5)
		(u5) edge (u6)
		(u6) edge (u1)
		(v1) edge (u1)
		(v2) edge (u2)
		(v3) edge (u3)
		(v4) edge (u6)
		;

		\path[color=lightgray,thick]
		(v1) edge [bend right] (v2)
			 edge (u2)
			 edge (u6)
		(v2) edge [bend right] (v3)
			 edge (u3)
			 edge (u1)
		(v3) edge [bend right] (u4)
			 edge (u2)
		(v4) edge [bend right] (v1)
			 edge [bend left] (u5)
			 edge (u1)
		;
		
		
		\node (u1) [draw,circle,fill,inner sep=1.5pt,position=90:1.1cm from anchor2] {};
		\node (u2) [draw,circle,fill,inner sep=1.5pt,position=150:1.1cm from anchor2] {};
		\node (u3) [draw,circle,fill,inner sep=1.5pt,position=210:1.1cm from anchor2] {};
		\node (u4) [draw,circle,fill,inner sep=1.5pt,position=270:1.1cm from anchor2] {};
		\node (u5) [draw,circle,fill,inner sep=1.5pt,position=330:1.1cm from anchor2] {};
		\node (u6) [draw,circle,fill,inner sep=1.5pt,position=30:1.1cm from anchor2] {};
		
		\node (v1) [draw,circle,fill,inner sep=1.5pt,position=90:2cm from anchor2] {};
		\node (v2) [draw,circle,fill,inner sep=1.5pt,position=150:2cm from anchor2] {};
		\node (v4) [draw,circle,fill,inner sep=1.5pt,position=30:2cm from anchor2] {};

		\node (l1) [position=120:5mm from anchor2] {$w$};
		\node (l2) [position=180.5:5mm from anchor2] {$w$};
		\node (l3) [position=240.5:5mm from anchor2] {$u$};
		\node (l4) [position=300.5:5mm from anchor2] {$u$};
		\node (l5) [position=0.5:5mm from anchor2] {$w$};
		\node (l6) [position=60.5:5mm from anchor2] {$w$};
		\node (l8) [position=90:2.1cm from anchor2] {$u$};
		\node (l9) [position=150:2.1cm from anchor2] {$u$};
		\node (l11) [position=30:2.1cm from anchor2] {$u$};

		\path
		(u1) edge (u2)
		(u2) edge (u3)
		(u3) edge (u4)
		(u4) edge (u5)
		(u5) edge (u6)
		(u6) edge (u1)
		(v1) edge (u1)
		(v2) edge (u2)
		(v4) edge (u6)
		;

		\path[color=lightgray,thick]
		(v1) edge [bend right] (v2)
		edge (u2)
		edge (u6)
		(v2) edge [bend right] (u3)
		edge (u1)
		(v4) edge [bend right] (v1)
		edge [bend left] (u5)
		edge (u1)
		;

		\end{tikzpicture}
		\caption{The sprouts of size $5$ with an induced base cycle of length $5$.}
		\label{fig4.42}
	\end{center}
\end{figure}
\vspace{-1mm}

To this point it seems to be very similar to the reduction of forbidden subgraphs we did in order to give a better characterization of the graphs with a chordal line graph square. But in the chordal case, in addition to the type I sprouts which correspond to the type $A$ sprouts here, we got the III$_a$. Now with a smallest induced cycle that is a fertile sprout of size $5$ being of length $8$ there is not sprout possible which has an induced cycle of length $6$ and has just one chord, forming a triangle with two skipped edges of the base cycle.\\
This leads to the following lemma. The according proof is extremely similar to the proofs of the Lemmas \autoref{lemma4.11}, \autoref{lemma4.12} and \autoref{lemma4.13} and since this is the case we leave it to the reader.

\begin{lemma}\label{lemma4.18}
Let $S=\lb V,U\cup W\cup E\rb$ be a fertile sprout of size $5$ with a longest induced cycle $C$ of length $6$, $\E{C}\cap E\neq\emptyset$ and a cycle $C'$ with $\E{C'}\cap E=\emptyset$ with $\abs{C'}\in\set{7,8,9,10,11,12,13,14}$, then $S$ contains a fertile sprout $S'$ of type A (see \autoref{fig4.42}) or a fertile sunflower sprout of size $5$ (see \autoref{fig4.41}).	
\end{lemma}

This leads to our final result in this chapter, giving a first impression of how a graph with a perfect line graph square might look like.

\begin{theorem}\label{thm5.16}
Let $G$ be a graph and $\lineg{G}^2$ be perfect, then $G$ does not contain a $C_n$ with $n\geq7$, a fertile sunflower sprout of size $5$ or a fertile sprout of type $A$ (see \autoref{fig4.42}).	
\end{theorem}
\vspace{-1em}
Unlike induced cycles, that can be described very well in terms of fertile sprouts and flowers, anti holes seem to have be much more complicated in their structure. As it was done with induced cycles, the first step in finding these graphs would be to find structures that contain anti holes when squared and then trying to translate these structures to the language of line graphs. A simple example of such a structure is given in the following figure.

\begin{figure}[H]
	\begin{center}
		\begin{tikzpicture}
		\node (center) [] {};
		
		\node (u1) [draw,circle,fill,inner sep=1.5pt,position=90:1.2cm from center] {};
		\node (u2) [draw,circle,fill,inner sep=1.5pt,position=162:1.2cm from center] {};
		\node (u3) [draw,circle,fill,inner sep=1.5pt,position=234:1.2cm from center] {};
		\node (u4) [draw,circle,fill,inner sep=1.5pt,position=306:1.2cm from center] {};
		\node (u5) [draw,circle,fill,inner sep=1.5pt,position=16:1.2cm from center] {};
		
		\node (v1) [draw,circle,inner sep=1.5pt,position=54:6mm from center] {};
		\node (v2) [draw,circle,inner sep=1.5pt,position=124:6mm from center] {};
		\node (v3) [draw,circle,inner sep=1.5pt,position=198:6mm from center] {};
		\node (v4) [draw,circle,inner sep=1.5pt,position=270:6mm from center] {};
		\node (v5) [draw,circle,inner sep=1.5pt,position=340:6mm from center] {};
		
		\path [bend left]
		(u1) edge (v1)
		(u2) edge (v2)
		(u3) edge (v3)
		(u4) edge (v4)
		(u5) edge (v5)
		;
		
		\path [bend left]
		(v1) edge (u3)
		(v2) edge (u4)
		(v3) edge (u5)
		(v4) edge (u1)
		(v5) edge (u2)
		;

		\end{tikzpicture}
	\end{center}
	\caption{A graph whose square contains an anti hole of size $5$.}
	\label{fig4.43}
\end{figure}
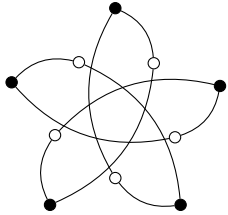 
\vspace{-1.2em}
For anti holes of greater sizes those graphs seem to rapidly increase in complexity. The next figure shows a graph that contains an anti hole of size $7$ which was constructed in the same way as the graph in \autoref{fig4.43} was.\\
For the translation into a line graph such a graph would not be possible, since this construction contains a $K_{1,3}$. In order to prevent such additional subgraphs at least more edges would be necessary, increasing the complexity of such graphs even more. This raises another question: What benefit does such knowledge about the square of a graph hold?
\vspace{-3mm}
\begin{figure}[!h]
	\begin{center}
		\begin{tikzpicture}
		\node (center) [] {};
		
		\node (u1) [draw,circle,fill,inner sep=1.5pt,position=90:22mm from center] {};
		\node (u2) [draw,circle,fill,inner sep=1.5pt,position=141.428:22mm from center] {};
		\node (u3) [draw,circle,fill,inner sep=1.5pt,position=192.856:22mm from center] {};
		\node (u4) [draw,circle,fill,inner sep=1.5pt,position=244.284:22mm from center] {};
		\node (u5) [draw,circle,fill,inner sep=1.5pt,position=295.712:22mm from center] {};
		\node (u6) [draw,circle,fill,inner sep=1.5pt,position=347.14:22mm from center] {};
		\node (u7) [draw,circle,fill,inner sep=1.5pt,position=38.568:22mm from center] {};
				
		\node (v12) [draw,circle,inner sep=1.5pt,position=102.85:26mm from center] {};
		\node (v21) [draw,circle,inner sep=1.5pt,position=128.55:26mm from center] {};
		\node (v22) [draw,circle,inner sep=1.5pt,position=154.35:26mm from center] {};
		\node (v31) [draw,circle,inner sep=1.5pt,position=179.95:26mm from center] {};
		\node (v32) [draw,circle,inner sep=1.5pt,position=205.65:26mm from center] {};
		\node (v41) [draw,circle,inner sep=1.5pt,position=231.35:26mm from center] {};
		\node (v42) [draw,circle,inner sep=1.5pt,position=257.05:26mm from center] {};
		\node (v51) [draw,circle,inner sep=1.5pt,position=282.75:26mm from center] {};
		\node (v52) [draw,circle,inner sep=1.5pt,position=308.45:26mm from center] {};
		\node (v61) [draw,circle,inner sep=1.5pt,position=334.15:26mm from center] {};
		\node (v62) [draw,circle,inner sep=1.5pt,position=359.85:26mm from center] {};
		\node (v71) [draw,circle,inner sep=1.5pt,position=25.55:26mm from center] {};
		\node (v72) [draw,circle,inner sep=1.5pt,position=51.25:26mm from center] {};
		\node (v11) [draw,circle,inner sep=1.5pt,position=76.95:26mm from center] {};
						
		\path 
		(u1) edge (v11)
			 edge (v12)
		(u2) edge (v21)
			 edge (v22)
		(u3) edge (v31)
			 edge (v32)
		(u4) edge (v41)
			 edge (v42)
		(u5) edge (v51)
			 edge (v52)
		(u6) edge (v61)
			 edge (v62)
		(u7) edge (v71)
			 edge (v72)
		;
		
		\path 
		(v12) edge (u3)
		(v22) edge (u4)
		(v32) edge (u5)
		(v42) edge (u6)
		(v52) edge (u7)
		(v62) edge (u1)
		(v72) edge (u2)
		;

		\path 
		(v11) edge (u5)
		(v21) edge (u6)
		(v31) edge (u7)
		(v41) edge (u1)
		(v51) edge (u2)
		(v61) edge (u3)
		(v71) edge (u4)
		;
				
		\end{tikzpicture}
	\end{center}
	\caption{A graph whose square contains an anti hole of size $7$.}
	\label{fig4.44}
\end{figure}
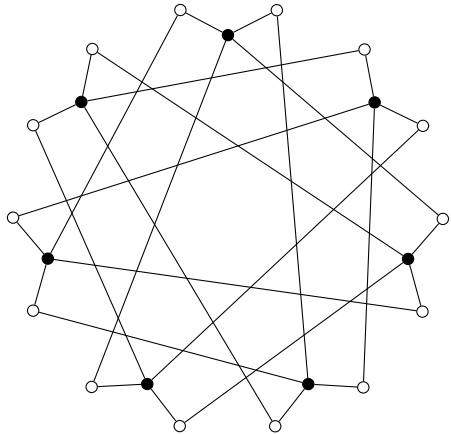 

\chapter{Computing Strong Colorings}


This chapter will focus on the application of the theory of chordal graph powers even beyond coloring problems. We will start by giving a short introduction to complexity theory and give some additional coloring problems that share similarities with strong colorings, in addition we will introduce some other well known combinatorial problems that are in general hard to solve.\\
Then we take it one step further and introduce parameterized complexity alongside two very important parameters, no only in algorithmic chromatic graph theory: Treewidth and Vertex Cover. Some results on the parameterized complexity of various coloring problems are presented together with some algorithmic approaches.\\
Since the main aspect of this thesis is the investigation of chordal graph powers we will give a short introduction in the algorithmic aspects of chordality, which will later be extended to chordal powers of graphs.\\
At last we will combine the concepts introduced in the first section of this chapter with our knowledge on chordal powers of graphs and introduce the power of chordality as a new and seemingly powerful parameter which allows us to easily compute upper bounds on both the treewidth and the vertex cover number of graphs. We will give polynomial algorithms for computing the power of chordality and some related parameters, which then will be used for parameterized algorithms for hard problems.

\section{Parameterized Complexity}
What is it that makes a problem {\em hard}? Why is a problem harder than another? These are the questions of the classical complexity theory.\\
This section will give a quick introduction to this broad field including an overview on those problems that are considered hard in terms of this theory and of greater relevance for the study of strong colorings.\\
We then will take one step further and introduce one possible way to approach hard problems: Parameterized Complexity.  In addition we will give an introduction to treewidth and vertex cover, two graph parameters that have proven to be extremely useful when it comes to parameterized algorithms.\\
With a few exceptions this first section will forgo the concept of chordality since this is the main topic of the second section of this chapter.\\
Since this first section is an introduction we will rely on the book "Optimization Theory" by Hubertus Jongen, Klaus Meer and Eberhard Triesch (see \cite{jongen2004optimization}) and the book on exact algorithms for hard graph problems by Frank Gurski, Irene Rothe, Jörg Rothe and Egon Wanke (see \cite{gurski2010exakte}). In addition we will use the survey "Parameterized Complexity" which was organized by Downey, Fellows, Niedermeier and Rossmanith (see \cite{downey2012parameterized}) and not to forget the work "On the Computational Complexity of Algorithms" by Hartmanis and Stearns (see \cite{hartmanis1965computational}) from the year 1965, which first gave important formal definitions concerning the running time and computational complexity of algorithms and was based on the celebrated work of Alan Turing (see \cite{turing1936computable}).

\subsection{Introduction to Complexity Theory}

An important measure of the quality of an algorithm is the so called running time it takes to compute the desired output. This time is measured in the number of elementary computational steps an algorithm takes, elementary steps are for example allocations, arithmetic operations and comparisons.\\
Besides the running time of an algorithm usually the space plays an important rule in complexity theory. However we restrict ourselves to running time consideration only.\\
We begin by giving a, somewhat informal, definition of algorithms and then proceed to the running time, or time complexity of an algorithm. Usually this is done by considering {\em Turing Machines}, but since this is just a short introduction, it will do without.

\begin{definition}[Decision Problem]
A {\em decision problem} $\Pi$ is, given an instance of the instance set $\mathcal{I}$, the question if the set of solutions is empty or not.
\end{definition}

\begin{definition}[Algorithm]
A ({\em deterministic}) {\em algorithm} is a finitely describable, step wise procedure	to solve a problem. The steps must be unambiguously described and executable. An algorithm {\em terminates} if an output is computed after a finite number of steps. 
\end{definition}

\begin{definition}[Time Complexity]
Let $\mathcal{I}_n$ be the set of all inputs of size $n\in\N$ and $\fkt{A_T}{I}$ the number of elementary steps of algorithm $A$ for an input $I\in \mathcal{I}_n$. The {\em worst case time complexity}	of algorithm $A$ for an input of size $n$ is $\fkt{T_A^{WC}}{n}=\sup\condset{\fkt{A_T}{I}}{I\in \mathcal{I}_n}$.\\
If $\fkt{T_A^{WC}}{n}\in\Ord{\fkt{f}{n}}$ for some function $f$, we say that $A$ has a {\em running time} of $\Ord{\fkt{f}{n}}$ and $f$ is a {\em complexity function}.
\end{definition}

This leads to a natural classification of algorithms in terms of their running time.

\begin{definition}[Deterministic Time Classes]\label{def5.1}
For some (computable) complexity function $t$ the class $\DTIME{t}$ of the in time $t$ deterministically solvable problems given through
\begin{align*}
\DTIME{t}=\condset{L}{\text{there is an algorithm}~A~\text{for}~L~\text{with}~\fkt{T_A^{WC}}{n}\leq\fkt{t}{n}~\text{for all}~n\in\N}.
\end{align*}  	
\end{definition}

With this we can give a formal definition of the class $\mathcal{P}$ of all problems that have a polynomial running time. Those problems are said to be solvable efficiently. As it is possible to give a class of polynomial solvable problems, it is possible to give a larger class of problems solvable in exponential time $\mathcal{EXP}$.

\begin{definition}[$\mathcal{P}$]
\begin{align*}
\mathcal{P}=\bigcup_{k\in\N_0}\DTIME{n^k}
\end{align*}	
\end{definition}

\begin{definition}[$\mathcal{EXP}$]
	\begin{align*}
	\mathcal{EXP}=\bigcup_{k\in\N_0}\DTIME{2^{n^k}}
	\end{align*}	
\end{definition}

\begin{theorem}[Hierarchy Theorem, see \cite{gurski2010exakte}]\label{thm5.1}
Let $t_1$ and $t_2$ be (computable) complexity functions with $t_2\notin\Ord{t_1}$ and $\fkt{t_1}{n}\geq n$ for all $n\in N$, then
\begin{align*}
\DTIME{t_2\fkt{\log}{t_2}}\subsetneq\DTIME{t_1}.
\end{align*}	
\end{theorem}

For graphs it is of high interest whether a certain parameter can be computed in polynomial time or not. We will now give a short overview of those problems which so far have appeared in this thesis and are known to be in $\mathcal{P}$.
\begin{itemize}
	\item \textsc{Maximum Degree}: Given a graph $G$ and a positive integer $k\in\N$, check 
	whether $\fkt{\Delta}{G}\leq k$.
	
	\item \textsc{Distance}: Given a graph $G$, a distance function $c\colon 
	\V{G}\rightarrow\R_{\geq 0}$ and two distinct vertices $v,w\in\V{G}$ together with a fixed 
	number $r\in\R_{\geq 0}$, check whether $\distg{G}{v}{w}\leq r$.
	
	\item \textsc{Shortest Path}: Given a graph $G$, a distance function $c\colon 
	\V{G}\rightarrow\R_{\geq 0}$ and two distinct vertices $v,w\in\V{G}$ together with a 
	$v$-$w$-path $P$ in $G$, check whether $P$ is a shortest $u$-$w$-path. 
	
	\item \textsc{Matching} and \textsc{Stable Set on Line Graphs}: Given a graph $G$ and a 
	positive integer $k\in\N$, check whether $G$ contains a matching of size $k$ (resp. check 
	whether $\lineg{G}$ contains a stable set of size $k$).
	
	\item \textsc{2-Coloring}: Given a graph $G$, check whether $\X{G}\leq 2$.
\end{itemize}

As we did with deterministic time classes it is also possible to define non-deterministic classes. Non-determinism is a theoretical tool which destroys the typical features of what we called an algorithm so far.\\
Suppose we have a guess about a correct solution to a given problem, we can at least {\em verify} in polynomial time whether this guess is in fact a solution. Of course, this only gives a satisfying answer to the initial problem if we were lucky and guessed correctly.\\
It is this guessing procedure which builds the core of non-deterministic algorithms. Now, at some point the former unambiguously described steps may differ. At any time in which a deterministic algorithms demands a certain step, a non-deterministic one may chose between different steps and the outcome of this decision is not necessary the same each time the non-deterministic algorithm reaches such a point.\\
The concept of running time remains the same, with the exception that a basic operation now too may be a non-deterministic decision.

\begin{definition}[Non-Deterministic Time Classes]\label{def5.2}
	For some (computable) complexity function $t$ the class $\NTIME{t}$ of the in time $t$ non-deterministically solvable problems given through
	\begin{align*}
	\NTIME{t}=\condset{L}{\text{there is a n.d.\ algorithm}~A~\text{for}~L~\text{with}~\fkt{T_A^{WC}}{n}\leq\fkt{t}{n}~\text{for all}~n\in\N}.
	\end{align*}  	
\end{definition}

\begin{definition}[$\mathcal{NP}$]
	\begin{align*}
	\mathcal{NP}=\bigcup_{k\in\N_0}\NTIME{n^k}
	\end{align*}	
\end{definition}

\begin{lemma}[see \cite{jongen2004optimization}]\label{lemma5.1}
The class $\mathcal{P}$ is a subclass of $\mathcal{NP}$. Any problem in $\mathcal{NP}$ is decidable in {\em exponential time}, that is by an algorithm with running time in $\Ord{2^{c\,p\lb n\rb}}$ for a fixed constant $c$ and a fixed polynomial $p$.	
\end{lemma}

The crucial question in relation to the two polynomial time-type classes $\mathcal{P}$ and $\mathcal{NP}$ is whether $\mathcal{P}\subsetneq\mathcal{NP}$ or $\mathcal{P}=\mathcal{NP}$. While the answer to this question remains open to date, there are even problems for which it is not known whether they are members of $\mathcal{NP}$ or not.

There is a rather vast class of problems that are known to be contained in $\mathcal{NP}$, but no polynomial time algorithms are known. Those problems all share certain characteristics which makes it possible to sort of translate one problem into another. 

\begin{definition}[Polynomial time reducibility]
Let $\Pi_1$ and $\Pi_2$ be two decision problems with instance sets $\mathcal{I}_1$ and $\mathcal{I}_2$. $\Pi_1$ is {\em polynomial time reducable} to $\Pi_2$ if there exists a function $f\colon\mathcal{I}_1\rightarrow\mathcal{I}_2$ such that
\begin{enumerate}[i)]
	
	\item $f$ is computable in polynomial time and
	
	\item $\fkt{f}{I}$ is a solution (or {\em yes-instance}) of $\Pi_2$ if and only if $I$ is a yes-instance of $\Pi_1$.

\end{enumerate}
\end{definition}

\begin{lemma}[see \cite{jongen2004optimization}]
If $\Pi_1$ is polynomial time reducable to $\Pi_2$ and $\Pi_2\in\mathcal{NP}$, then $\Pi_1\in\mathcal{NP}$. The same holds for $\mathcal{P}$ instead of $\mathcal{NP}$.	
\end{lemma}

\begin{definition}[$\mathcal{NP}$-completeness]
A decision problem $\Pi$ belonging to $\mathcal{NP}$ is called {\em $\npcomp$} if, for other decision problems $\Pi'\in\mathcal{NP}$, $\Pi'$ is polynomial time reducable to $\Pi$.\\
If $\Pi$ is not known to be a member of $\mathcal{NP}$ but still satisfies the second property above it is called {\em $\nphard$}. 
\end{definition}

\begin{theorem}[see \cite{jongen2004optimization}]\label{thm5.2}
	Let $\Pi_1\in\mathcal{NP}$ and $\Pi_2$ be $\npcomp$. If $\Pi_2$ is polynomial time reducable to $\Pi_1$, then $\Pi_1$ is $\npcomp$.
\end{theorem}

Theorem \autoref{thm5.2} gives an important tool to decide whether a given problem is $\npcomp$. In the following we will introduce a selected list of problems that are known to be $\npcomp$ and that have occurred in this thesis. For those problems we rely on the work of Garey and Johnson (see \cite{garey1979computers}).

\begin{itemize}
	
	\item \textsc{Clique}: Given a graph $G$ and a positive integer $k\in\N$, check whether 
	$\fkt{\omega}{G}\leq k$.
	
	\item \textsc{Clique Cover}: Given a graph $G$ and a positive integer $k\in\N$, check whether 
	$\Xq{G}\leq k$.
	
	\item \textsc{Coloring}: Given a graph $G$ and a positive integer $k\in\N$, check whether 
	$\X{G}\leq k$.
	
	\item \textsc{$3$-Coloring}: Given a graph $G$, check whether $\X{G}\leq 3$.

	\item \textsc{Edge Coloring}: Given a graph $G$ and a positive integer $k\in \N$, check whether 
	$\fkt{\chi'}{G}\leq k$.
	
	\item \textsc{Stable Set}: Given a graph $G$ and a positive integer $k\in\N$, check whether 
	$\fkt{\alpha}{G}\leq k$.
	
\end{itemize} 

In terms of strong colorings there are some additional problems that have to be mentioned. Those problems pose, as seen before, the generalization in terms of a distance condition of other problem seen above. The interesting fact is that even some problems known to be in $\mathcal{P}$ become $\npcomp$ if generalized to a distance of $2$.

\begin{itemize}
	
	\item \textsc{$2$-Strong Anti-Matching} (Mahdian. 2002 
	\cite{Mahdian2002complexitystrongedge}): Given a graph $G$ and a positive integer $k\in\N$, 
	check whether $\kam{2}{G}\leq k$.
	
	\item \textsc{$2$-Strong Coloring} (Lloyd, Ramanathan. 1992 \cite{lloyd1992complexity}): Given 
	a graph $G$ and a positive integer $k\in\N$, check whether $\stronk{2}{G}\leq k$.
	
	\item \textsc{$2$-Strong Edge Coloring} (Mahdian. 2002 
	\cite{Mahdian2002complexitystrongedge}): Given a graph $G$ and a positive integer $k\in\N$, 
	check whether $\stronki{2}{G}\leq k$.
	
	\item \textsc{$2$-Strong Matching} (Ko, Shepherd. 1994 \cite{ko1994adding}): Given a graph 
	$G$ and a positive integer $k\in\N$, check whether $\kmat{2}{G}\leq k$.
	
\end{itemize}

This might come somewhat as a surprise since both \textsc{Matching} and \textsc{$1$-Strong Anti 
Matching}, which is equivalent to \textsc{Maximum Degree}, are in $\mathcal{P}$. Later we will 
see even more exciting results regarding the complexity of strong colorings, but first we will 
dedicate ourselves to an important follow up question: If we do not know whether it is possible to 
solve an $\npcomp$ problem efficiently, how can we approach such problems with algorithms and 
avoid exponential running times?\\
There are several possible approaches for this, for most of those problems it is possible just to solve them to a certain degree and give an approximate solution, which we will also do later in this chapter, another approach could be to accept errors in other aspects and design so called random algorithms, that will give a solution with a certain probability. A third possibility is the so called parameterized complexity.

In the classical analysis of complexity for algorithms we investigate the running time. In the parameterized complexity theory, parameters are considered separately from the size of the input, the analysis becomes somewhat multidimensional. For this we will now extend the term decision problem to a parameterized variant.

\begin{definition}[Parameterized Problem and Parameter]
A {\em parameterized problem} is a pair $\lb \Pi, \kappa\rb$, where $\Pi$ is a decision problem with instance set $\mathcal{I}$ and
\begin{align*}
\kappa\colon\mathcal{I}\rightarrow\N,
\end{align*}
the so called {\em parameter}, is a polynomial time computable function.	
\end{definition}

In order to distinguish a non-parameterized decision problem from its parameterized variants we will add the prefix $p$ to the problems name. In many cases the parameter will be more than just a given number and may even depend on the given instance, in such cases we will add additional prefixes the the original problem name.

\begin{example}
If a problem is just parameterized by a certain number, it usually fixes the desired graph parameter to a certain value as in the following problem.
\begin{center}
	\begin{tabular}{ll}
		\toprule
		\multicolumn{2}{c}{$p$-\textsc{Stable Set}}\\
		\midrule
		Input: & A graph $G$ and a number $k\in\N$.\\
		Parameter: & $k$.\\
		Question: & Does $\fkt{\alpha}{G}\leq k$ hold?\\
		\bottomrule
	\end{tabular}
\end{center}
If we want to control the problem even more, other graph parameters may be added as long as they satisfy	the condition of being polynomial time computable.
\begin{center}
	\begin{tabular}{ll}
		\toprule
		\multicolumn{2}{c}{$p$-$deg$-\textsc{Stable Set}}\\
		\midrule
		Input: & A graph $G$ and a number $k\in\N$.\\
		Parameter: & $k+\fkt{\Delta}{G}$.\\
		Question: & Does $\fkt{\alpha}{G}\leq k$ hold?\\
		\bottomrule
	\end{tabular}
\end{center}
\end{example}

It seems that $\nphard$ problems always have an inevitable exponential part in the running time of their algorithms. If it is possible to extract this part and trap it inside a function purely dependent on the parameter, it is, in a certain sense, possible to mitigate it and to obtain a running time that is polynomial in the input size while excluding the parts dependent on the parameter. Problems for whom this is possible are called {\em fixed parameter tractable} and are collected in the complexity class $\FPT$.

\begin{definition}[$\FPT$-Algorithm, Fixed Parameter Tractable, $\FPT$]
Let $\lb \Pi,\kappa\rb$ be a parameterized problem, $\mathcal{I}$ the instance set of $\Pi$ and $\kappa\colon\mathcal{I}\rightarrow\N$ a parameter.
\begin{enumerate}[i)]
	\item An algorithm is called {\em $\FPT$-algorithm} with parameter $\kappa$ if there is a computable function $f\colon\N\rightarrow\N$ and a polynomial $p$ such that for all instances $I\in\mathcal{I}$, the running time of $A$ with input $I$ of size $\abs{I}$ is bounded by
	\begin{align*}
	\fkt{f}{\fkt{\kappa}{I}}\cdot\fkt{p}{\abs{I}}.
	\end{align*}
	
	\item A parameterized problem $\lb \Pi,\kappa\rb$ is called {\em fixed parameter tractable} if there is a $\FPT$-algorithm with parameter $\kappa$	that solves the decision problem $\Pi$.
	
	\item $\FPT$ is the set of all parameterized problems that are fixed parameter tractable.
\end{enumerate}	
\end{definition}

\begin{lemma}[see \cite{gurski2010exakte}]\label{lemma5.2}
Let $\lb \Pi,\kappa\rb$ be a parameterized problem. The following statements are equivalent:
\begin{enumerate}[i)]
	\item $\lb \Pi,\kappa\rb\in\FPT$
	
	\item There is a computable function $f\colon\N\rightarrow\N$, a polynomial $p$ and an algorithm that solves $\Pi$ for an arbitrary instance $I$ in time 
	\begin{align*}
	\fkt{f}{\fkt{\kappa}{I}}+\fkt{p}{\abs{I}}.
	\end{align*}
	
	\item There are computable functions $f,g\colon\N\rightarrow\N$, a polynomial $p$ and an algorithm that solves $\Pi$ for an arbitrary instance in time
	\begin{align*}
	\fkt{g}{\fkt{\kappa}{I}}+\fkt{f}{\fkt{\kappa}{I}}\cdot\fkt{p}{\abs{I}+\fkt{\kappa}{I}}.
	\end{align*}
\end{enumerate}	
\end{lemma}

Obviously parameterization can be very powerful, but still there are some problems that cannot be parameterized if $\mathcal{P}\neq\mathcal{NP}$ holds, which is a fair assumption most researchers can agree on. There is another interpretation of the working of parameters which poses a useful tool to see whether a parameterization is possible or not.

\begin{definition}[$k$th Slice]
Let $\Pi, \kappa$ be a parameterized problem and $k\in\N$. The {$k$th Slice} of $\lb \Pi,\kappa\rb$ is the, not parameterized or classical, decision problem $\Pi$ with the reduced instance set $\mathcal{I}_k=\condset{I\in\mathcal{I}}{\fkt{\kappa}{I}=k}$.
\end{definition}   

\begin{theorem}[see \cite{gurski2010exakte}]\label{thm5.3}
Let $\lb \Pi,\kappa\rb$ be a parameterized problem and $k\in\N$. If $\lb \Pi,\kappa\rb\in\FPT$, then the $k$th slice of $\lb \Pi,\kappa\rb$ is solvable in polynomial time.	
\end{theorem}

\begin{proof}
With $\lb \Pi,\kappa\rb$ there exists a $\FPT$-algorithm $A$ solving $\Pi$. Hence there is a computable function $f\colon\N\rightarrow\N$ and a polynomial $p$, such that the running time of $A$ with input $I$ is bounded by $\fkt{f}{\fkt{\kappa}{I}}\cdot\abs{I}^{\mathcal{O}\lb 1\rb}$. With $\fkt{\kappa}{I}=k$ for all $I\in\mathcal{I}_k$, which is the instance set of the $k$th slice, $\fkt{f}{\fkt{\kappa}{I}}=\fkt{f}{k}$ degenerates into a constant. Therefore $\fkt{f}{k}\cdot\abs{I}^{\mathcal{O}\lb 1\rb}\in\Ord{\abs{I}^{\mathcal{O}\lb 1\rb}}$ and thus the running time of $A$ is polynomial in the input size.	
\end{proof}

With this we will have a first look on coloring problems. The parameterized coloring problem is defined as follows.

\begin{center}
	\begin{tabular}{ll}
		\toprule
		\multicolumn{2}{c}{$p$-\textsc{Coloring}}\\
		\midrule
		Input: & A graph $G$ and a number $k\in\N$.\\
		Parameter: & $k$.\\
		Question: & Does $\X{G}\leq k$ hold?\\
		\bottomrule
	\end{tabular}
\end{center}

\begin{corollary}[see \cite{gurski2010exakte}]\label{cor5.1}
If $\mathcal{P}\neq\mathcal{NP}$, then $p\text{-\textsc{Coloring}}\notin\FPT$.	
\end{corollary}

\begin{proof}
Suppose $p$-\textsc{Coloring} is in $\FPT$, then by Theorem \autoref{thm5.3} the third slice of 
$p$-\textsc{Coloring} is solvable in polynomial time. But as stated above the decision problem 
$3$-\textsc{Coloring} is $\npcomp$ and Theorem \autoref{thm5.2} implies 
$\mathcal{P}=\mathcal{NP}$, a contradiction.
\end{proof}

Sadly this tool is not enough to handle all problems that are not, or better probably not, contained 
in $\FPT$. The above mentioned problem $p$-\textsc{Stable Set} is such a problem, but all of it's 
slices are in $\mathcal{P}$. For the dual problem $p$-\textsc{Clique} the same holds.

\begin{center}
	\begin{tabular}{ll}
		\toprule
		\multicolumn{2}{c}{$p$-\textsc{Clique}}\\
		\midrule
		Input: & A graph $G$ and a number $k\in\N$.\\
		Parameter: & $k$.\\
		Question: & Does $\fkt{\omega}{G}\leq k$ hold?\\
		\bottomrule
	\end{tabular}
\end{center}

By taking a more complex parameter this can be solved.

\begin{theorem}[see \cite{gurski2010exakte}]\label{thm5.4}
$p$-$deg$-\textsc{Stable Set} $\in\FPT$.	
\end{theorem}

Sometimes the parameter cannot be this strictly separated from the polynomial in the input size. This leads to a (probably) larger superclass of $\FPT$.

\begin{definition}[$\XP$-Algorithm, $\XP$]
\begin{enumerate}[i)]~\\
\vspace{-6mm}

\item Let $\lb \Pi,\kappa\rb$ be a parameterized problem and $\mathcal{I}$ the instance set of $\Pi$. An algorithm $A$ is called {\em $\XP$-algorithm} with the parameterization $\kappa\colon\mathcal{I}\rightarrow\N$, if there are computable functions $f,g\colon\N\rightarrow\N$ such that for each instance $I\in\mathcal{I}$ the running time of $A$ with input $I$ is bounded by
\begin{align*}
\fkt{f}{\fkt{\kappa}{I}}\cdot\abs{I}^{g\lb \kappa\lb I\rb\rb}.
\end{align*} 

\item The set of all parameterized problems solvable by a $\XP$-algorithm is collected in the complexity class $\XP$.
\end{enumerate}
\end{definition}

Obviously we obtain $\FPT\subseteq\XP$. There are several complexity classes in between $\FPT$ and $\XP$ that pose a better distinction between problems in $\FPT$ and the seemingly harder problems in $\XP$, the so called $\induz{W}{q}$ classes. For this we need some additional tools. The first is a parameterized version of the reduction which is used to show the $\mathcal{NP}$ completeness of problems.

\begin{definition}[Parameterized Reduction]
A parameterized problem $\lb \Pi_1,\kappa_1\rb$ with instance set $\mathcal{I}_1$ is {\em parameterized reducable} to another parameterized problem $\lb \Pi_2,\kappa_2\rb$ with instance set $\mathcal{I}$, if there is a function $r\colon\mathcal{I}_1\rightarrow\mathcal{I}_2$ satisfying the following conditions:
\begin{enumerate}[i)]

\item $I$ is a {\em yes}-instance of $\Pi_1$ if and only if $\fkt{r}{I}$ is a {\em yes}-instance of $\i_2$ for all $I\in\mathcal{I}_1$,

\item $\fkt{r}{I}$ is for a computable function $f$ and a polynomial $p$ computable in time
\begin{align*}
\fkt{f}{\fkt{\kappa_1}{I}}\cdot\fkt{p}{\abs{I}},
\end{align*}

\item and it holds $\fkt{\kappa_2}{\fkt{r}{I}}\leq\fkt{g}{\fkt{\kappa_1}{I}}$ for all $I\in\mathcal{I}_1$ and a computable function $g$.

\end{enumerate}
\end{definition}

In order to define a $W$-class we need one basic problem which other problems can be reduced to. In the classical complexity these basic problems are so called satisfiability problems on boolean formulas. For a class $\induz{W}{q}$ such a problem was first given by Downey and Fellows (see \cite{downey1995fixed}) in 1995.

\begin{definition}[Boolean Circuit]
A {\em boolean circuit} is a directed, cycle free graph $D=\lb V,A\rb$, whose vertices are labeled 
with a constant value (\textbf{true} or \textbf{false}), one of the input variables $x_1,\dots,x_n$, or 
a boolean operation $\wedge$, $\vee$ or $\neg$.  vertices are called $p$-{\em gates}, where $p$ 
is their label. Each $x_i$-gate is called an {\em input gate}. Input, \textbf{true}- and 
\textbf{false}-gate 
have no, $\neg$-gates exactly one and $\wedge$- and $\vee$-gates each exactly two incoming 
arcs. Each gate may have an arbitrary number of outgoing arcs. Gates without outgoing arcs are 
called {\em output gates}. The set of input gates is denoted by $\operatorname{In}$ and the set of 
output gates by $\operatorname{Out}$.\\
Each instantiation $T$ of the variables $x_1,\dots,x_n$ gives a logical value 
$\fkt{f}{g}\in\set{\text{\textbf{true, false}}}$ for each gate $g$ which is defined as follows:
\begin{enumerate}[i)]

\item $\fkt{f}{g}$ is \textbf{true}, \textbf{false} or $\fkt{T}{x_i}$, if $g$ is a \textbf{true}-gate, 
\textbf{false}-gate or a $x_i$-gate.

\item If $g$ is a $\neg$ gate and $g'$ is the predecessor of $g$, then $\fkt{f}{g}=$\textbf{true} if 
$\fkt{f}{g'}$=\textbf{false} and $\fkt{f}{g}=$\textbf{false} if $\fkt{f}{g'}$=\textbf{true}.

\item Analogue to the $\neg$-gates, both the $\wedge$- and the $\vee$ gates imitate their corresponding boolean functions and $f$ evaluates the logical values of their two predecessors.

\end{enumerate}
\end{definition}

We then expand the concept of boolean circuits by three additional parameters. Those also pose a slight generalization and lead to a new parameterized decision problem.

\begin{definition}[Large Gate, Weft, Depth]
~\\
\vspace{-6mm}
\begin{enumerate}[i)]

\item A $\wedge$- or a a $\vee$-gate with an arbitrary number of incoming arcs is called a {\em large gate}.

\item The {\em weft} of a boolean circuit is the maximum number of large gates on a direct path from an input to an output gate.

\item The {\em depth} of a boolean circuit is the length of a longest directed path from an input to an output gate.

\end{enumerate}
\end{definition}

\begin{center}
	\begin{tabular}{lp{8cm}}
		\toprule
		\multicolumn{2}{c}{$p$-\textsc{Weighted Sat $\lb t,d\rb$}}\\
		\midrule
		Input: & A boolean circuit $C$ with weft $t$ and depth $d$ along with a weighting function $w\colon \operatorname{Out}\rightarrow\N$ and a number $k\in\N$.\\
		Parameter: & $k$.\\
		Question: & Is there a satisfying instantiation of $C$ with weight $k$?\\
		\bottomrule
	\end{tabular}
\end{center}

\begin{definition}[$W$-Hierarchy]
The {\em $W$-hierarchy} consists of the complexity classes $\induz{W}{t}$, $t\geq 1$. A 
parameterized problem $\lb \Pi,\kappa\rb$ belongs to the class $\induz{W}{t}$, if it can be 
parameterized reduced to $p$-\textsc{Weighted Sat $\lb t,d\rb$} for $d\in\N$.
\end{definition}

We obtain the following chain of inclusions for the parameterized complexity classes.
\begin{align*}
\FPT\subseteq\induz{W}{1}\subseteq\induz{W}{2}\subseteq\dots\subseteq\XP
\end{align*}

\begin{definition}[$W$-Hard, $W$-Complete]
~\\\vspace{-6mm}
\begin{enumerate}

\item A parameterized problem $\lb \Pi,\kappa\rb$ is {\em $\induz{W}{t}$-hard} if all problems in $\induz{W}{t}$ can be parameterized reduced to $\lb \Pi,\kappa\rb$.

\item A parameterized problem $\lb \Pi,\kappa\rb$ is {\em $\induz{W}{t}$-complete}, if it is $\induz{W}{t}$-hard and contained in $\induz{W}{t}$.

\end{enumerate}
\end{definition}

\subsection{Treewidth and Vertex Cover}

We will now begin by investigating two structural graph parameters that can be used for the parameterization of $\npcomp$ problems. The first and arguably the most popular structural parameter for graph related problems is the so called treewidth which was invented by Robertson and Seymour in 1986 (see \cite{robertson1986graph}) as part of their work on graph minors.\\
This concept is strongly related to chordal graphs in general, which will be very useful in the next section. For now let us start with a question that is somewhat similar to a question we asked in Chapter 3: How many additional edges does it take for a graph to become chordal?

\begin{definition}[Triangulation]
A {\em triangulation} of a graph $G=\lb V,E\rb$ is a graph $H=\lb V,F\rb$ that contains $G$ as a, usually not induced, subgraph and is chordal.\\
$H$ is a {\em minimal triangulation} of $G$ if there does not exist a triangulation $H'$ of $G$ with $H'$ being a proper subgraph of $H$.\\
$H$ is called a {\em minimum triangulation} of $G$ if for all other triangulations $H'=\lb V,F'\rb$ of $G$ $\abs{F\setminus E}\leq\abs{F'\setminus E}$ holds. The number of edges added to $G$ by a minimum triangulation is called the {\em fill-in number} of $G$ and denoted by $\fkt{\operatorname{F}}{G}$.\\
The decision problem \textsc{Minimum Fill-In} is the question whether for a given graph $G$ and a 
positive integer $k\in\N$ it holds $\fkt{\operatorname{F}}{G}\leq k$. 
\end{definition}

\begin{theorem}[Yannakakis. 1981 \cite{yannakakis1981computing}]\label{thm5.5}
The \textsc{Minimum Fill-In} problem is $\mathcal{NP}$-complete.
\end{theorem}

\begin{definition}[Tree Decomposition, Treewidth]
A {\em tree decomposition} of a graph $G=\lb V,E\rb$ is a pair $\lb \condset{X_i}{i\in I},T=\lb I,F\rb\rb$ with $\condset{X_i}{i\in I}$	a collection of subsets of $V$, called {\em bags}, and $T=\lb I,F\rb$ a tree, such that the following conditions are satisfied.
\begin{enumerate}[i)]
	\item For all $v\in V$, there exists an $i\in I$ with $v\in X_i$.
	
	\item For all $vw\in E$, there exists an $i\in I$ with $v,w\in X_i$.
	
	\item For all $v\in V$, the set $I_v=\condset{i\in I}{v\in X_i}$ forms a connected subgraph of $T$.
	
\end{enumerate}
The {\em width} of tree decomposition $\lb \condset{X_i}{i\in I},T=\lb I,F\rb\rb$ equals $\max_{i\in I}\abs{X_i}-1$, The {\em treewidth} of a graph $G$, $\tw{G}$, is the minimum width of a tree decomposition of $G$.\\
The decision problem \textsc{Treewidth} is the question whether for a given graph $G$ and a 
positive integer $k\in\N$ it holds $\tw{G}\leq k$.
\end{definition}

\begin{theorem}[Arnborg, Corneil, Proskurowski. 1987 \cite{arnborg1987complexity}]
The \textsc{Treewidth} problem is $\mathcal{NP}$-complete.
\end{theorem}

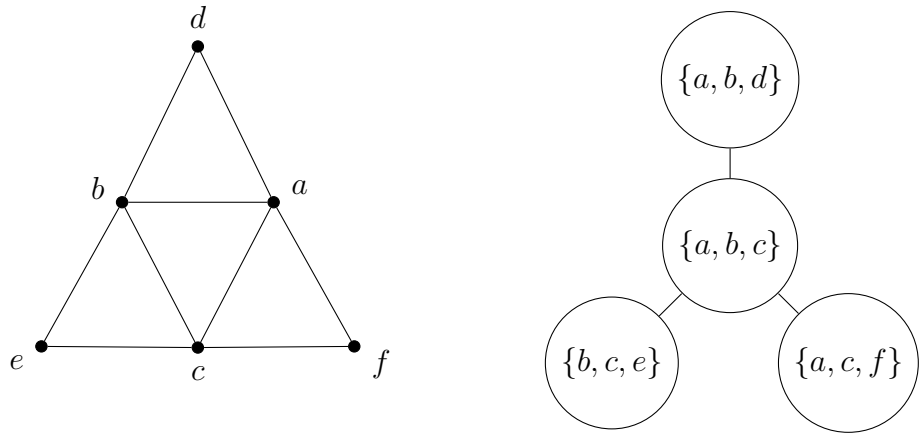
\begin{figure}
\begin{center}
	\begin{tikzpicture}
	
	\node (acenter) [inner sep=1.5pt] {};
	
	\node (a1) [inner sep=1.5pt,position=30:1cm from acenter,draw,circle,fill] {};
	\node (a2) [inner sep=1.5pt,position=150:1cm from acenter,draw,circle,fill] {};
	\node (a3) [inner sep=1.5pt,position=270:1.21cm from acenter,draw,circle,fill] {};
	\node (a4) [inner sep=1.5pt,position=90:2.5cm from acenter,draw,circle,fill] {};
	\node (a5) [inner sep=1.5pt,position=213:2.3cm from acenter,draw,circle,fill] {};
	\node (a6) [inner sep=1.5pt,position=327:2.3cm from acenter,draw,circle,fill] {};
	
	\node (l1) [position=30:0.02cm from a1] {$a$};
	\node (l2) [position=150:0.02cm from a2] {$b$};
	\node (l3) [position=270:0.02cm from a3] {$c$};
	\node (l4) [position=90:0.02cm from a4] {$d$};
	\node (l5) [position=213:0.02cm from a5] {$e$};
	\node (l6) [position=327:0.02cm from a6] {$f$};
	
	\path
	(a1) edge (a2)
	(a2) edge (a3)
	(a3) edge (a1)
	(a4) edge (a1)
	edge (a2)
	(a5) edge (a2)
	edge (a3)
	(a6) edge (a1)
	edge (a3);
	
	\node(center) [draw,circle,right of=acenter,node distance=7cm] {$\set{a,b,c}$};
	\node(above) [draw,circle, above of = center,node distance=2.2cm] {$\set{a,b,d}$};
	\node(belowl) [draw,circle,below left of=center,node distance=2.2cm] {$\set{b,c,e}$};
	\node(belowr) [draw,circle,below right of=center,node distance=2.2cm] {$\set{a,c,f}$};

	\path
	(center) edge (above)
	edge (belowl)
	edge (belowr);
	
	\end{tikzpicture}
\end{center}
\caption{A graph and a tree decomposition of width $2$.}
\label{fig5.1}
\end{figure}

The connection between a triangulation and the tree decomposition, and therefore the relation between the tree decomposition and chordal graphs manifests in the following result by Hans Bodlaender.

\begin{theorem}[Bodlaender. 1998 \cite{bodlaender1998partial}]\label{thm5.6}
Let $G$ be a graph, $k\in\N$ and $H$ a minimum triangulation of $G$, then $\tw{G}\leq k$ if and only if $\fkt{\omega}{H}\leq k+1$.	
\end{theorem}

\begin{corollary}\label{cor5.2}
Let $G$ be a chordal graph, then $\tw{G}\leq\fkt{\omega}{G}-1$.	
\end{corollary}

\begin{corollary}\label{cor5.3}
Let $G$ be a graph in which all flowers of size $n\geq 4$ are withered, then $\tw{G}\leq\kcl{2}{G}-1$.	
\end{corollary}

\begin{lemma}[Bodlaender, Möhring. 1993 \cite{bodlaender1993pathwidth}]\label{lemma5.3}
Let $G$ be a graph and $\lb \mathscr{X}, T\rb$ with $T=\lb V_T,E_T\rb$ a tree decomposition of $G$. Then for every clique $C$ of $G$ there is a vertex $u\in V_T$ with $C\subseteq X_u$.
\end{lemma}

\begin{proof}
We will proof the assertion by induction over the number $n$ of vertices in $C$.\\
For $n=1$ this is trivial, since for every vertex there must exist a bag containing it by definition of a tree decomposition.\\
Now let $n>1$ and $x\in C$. By assumption there is a bag $X_u$ containing $C\setminus\set{x}$ as a subset. We are looking for another vertex of $T$ corresponding to a bag containing all vertices of $c$. Let $T'=\induz{T}{U}$ with $U=\condset{u'\in V_T}{x\in X_{u'}}$.\\
If $u\in U$ holds, then both $C\setminus\set{x}$ and $\set{x}$ are subsets of $X_u$ and we are done.\\
So suppose $u\notin U$. Let $v\in U$ be the vertex with smallest distance to $u$ in $T$. We will show that $C\subseteq X_v$ holds. By definition we have $x\in X_v$. Now each path in $T$ with a starting point in $U$ and endpoint $u$ must contain $v$ and since $x,y\in C$ and $C$ is a clique, $G$ contains the edge $xy$, therefore a vertex $v'\in U$ must exist with $x,y\in X_{v'}$. With $y\in C\setminus\set{x}\subseteq X_u$ and the $v'$-$u$ path in $T$ containing $v$, $X_u$ must contain $y$ by definition of the tree decomposition.
\end{proof}

Another possible interpretation of the \textsc{Minimum Fill-In} problem and the minimum 
triangulation of graphs is the concept of partial $k$-trees.

\begin{definition}[$k$-Tree]
A graph $G$ is called a {\em $k$-tree} if it can be constructed as follows.
\begin{enumerate}[i)]

\item For $k\in\N$ the complete graph $K_k$ is a $k$-tree.

\item If $G=\lb V,E\rb$ a $k$-tree, $v\notin V$ a new vertex and $G'=\lb V', E'\rb$ a complete graph on $k$ vertices of $G$, then $\lb V\cup\set{v}, E\cup\condset{vv'}{v'\in V'}\rb$ is a $k$-tree. 

\end{enumerate}
A graph $G=\lb V,E\rb$ is called a {\em partial $k$-tree} if there is a $k$-tree $G'=\lb V,E'\rb$ with $E\subseteq E'$.
\end{definition}

\begin{theorem}[Schäffer. 1989 \cite{schaffer1989optimal}]\label{thm5.7}
A graph $G$ is a partial $k$-tree if and only if $\tw{G}\leq k$.
\end{theorem}

An important result on the $\npcomp$ problem \textsc{Treewidth} was achieved in 1993 by 
Bodlaender and one of the reasons for the tree decomposition to be such a powerful parameter for 
many graph problems. \textsc{Treewidth} can be parameterized with itself.

\begin{center}
	\begin{tabular}{ll}
		\toprule
		\multicolumn{2}{c}{$p$-$tw$-\textsc{Treewidth}}\\
		\midrule
		Input: & A graph $G$ and a number $k\in\N$.\\
		Parameter: & $\tw{G}$.\\
		Question: & Does $\tw{G}\leq k$ hold?\\
		\bottomrule
	\end{tabular}
\end{center}

\begin{theorem}[Bodlaender. 1993 \cite{bodlaender1993linear}]\label{thm5.9}
The problem $p$-$tw$-\textsc{Treewidth} is in $\FPT$.
\end{theorem}

The power of \textsc{Treewidth} as a parameter becomes clear with the following result by Zhou, 
Kanari and Nishizeki who constructed a parameterized algorithm for $k$-strong colorings on partial 
$k$-trees, i.e. graphs with bounded treewidth.

\begin{center}
	\begin{tabular}{lp{8cm}}
		\toprule
		\multicolumn{2}{c}{$p$-$tw$-\textsc{$k$-Strong Coloring}}\\
		\midrule
		Input: & A graph $G$, a number $k\in\N$ and a number $c\in\N$, as well as a distance function $\omega\colon\E{G}\rightarrow \N$.\\
		Parameter: & $\tw{G}+k$.\\
		Question: & Does $\stronk{k}{G}\leq c$ hold?\\
		\bottomrule
	\end{tabular}
\end{center}

\begin{lemma}[Xiao, Kanari, Nishizeki. 2000 \cite{xiao2000generalized}]\label{lemma5.4}
For any partial $t$-tree $G=\lb V,E\rb$ with a distance function $\omega\colon E\rightarrow\N$ we can construct in polynomial time a $t$-tree $G'=\lb V,E'\rb$ with a distance function $\omega'\colon E'\rightarrow\N$ such that
\begin{enumerate}[i)]

\item $E\subseteq E'$,

\item $\fkt{\omega'}{e}=\distg{G}{v}{u}=\min\condset{\fkt{\omega}{P}}{P~\text{is a}~v\text{-}u~\text{path in}~G}$ for any edge $e=vu\in E'$ and

\item the partial $t$-tree $G$ has a $k$-strong coloring with $c$ colors if and only if the $t$-tree $G'$ has a $k$-strong coloring with $c$-colors. 

\end{enumerate}
\end{lemma}

\begin{proof}
By Theorem \autoref{thm5.9} and the so called Bodlaender's Algorithm it implies both a tree decomposition of width $t$ and a $t$-tree $G'=\lb V,E'\rb$ with $E\subseteq E'$ can be computed in polynomial time. Define $\fkt{\omega'}{e}=\distg{G}{v}{u}=\min\condset{\fkt{\omega}{P}}{P~\text{is a}~v-u~\text{path in}~G}$  for any edge $e=vu\in E'$. Then clearly $\distg{G}{v}{u}=\distg{G'}{v}{u}$ for all vertices $v,u\in V$. Therefore $G$ has a $k$-strong coloring in $c$ colors if and only if $G'$ has a $k$-strong coloring in $c$ colors.
\end{proof}

\begin{theorem}[Xiao, Kanari, Nishizeki.2000 \cite{xiao2000generalized}]\label{thm5.8}
Let $t,k\in\N$ be bounded integers, let $G=\lb V,E\rb$ be a $k$-tree on $n$ vertices given by its tree decomposition and let $c\in\N$ be another integer.\\
Then it can be determined in time $\Ord{n\lb c+1\rb^{2^{2\lb t+1\rb\lb k+2\rb+1}}+n^3}$ whether $G$ has a $k$-strong coloring in $c$ colors. Furthermore, if such a coloring exists, it can be found in the same time. 
\end{theorem}

The algorithm, presented in the form of several lemmas (see \cite{xiao2000generalized}),  uses dynamic programming and generates a table of size $\Ord{n\lb c+1\rb^{2^{2\lb t+1\rb\lb k+2\rb+1}}+n^3}$ in order to find a $k$-strong coloring in $c$ colors for the given $t$-tree $G$. This $t$-tree was constructed by applying the Bodlaender Algorithm to the original graph with bounded treewidth, which we already mentioned before.\\
Please note that the $k$-\textsc{Strong Coloring} problem here is an even more general version of 
distance coloring than the one we originally defined. By choosing $\fkt{\omega}{e}=1$ for all edges 
$e$ of the partial $t$-tree we obtain the original problem. In this sense finding a $t$-tree 
containing the partial $t$-tree as a subgraph could be seen as a more general - and with this a 
harder - version of finding a chordal power of a graph. We will revisit this idea later.

\begin{corollary}\label{cor5.4}
The problem $p$-$tw$-\textsc{$k$-Strong Coloring} is in $\XP$.
\end{corollary}

In fact, if $\mathcal{P}\neq\mathcal{NP}$, even for the case $k=2$, which has been studied in the 
previous chapters, $p$-$tw$-\textsc{$2$-Strong Coloring} is not contained in $\FPT$. As usual 
some additional tools are required the the proof. Especially an additional $\induz{W}{1}$-hard 
coloring problem is needed.

\begin{definition}[Equitable Coloring]
For a graph $G=\lb V,E\rb$ a proper vertex coloring $F$ in $c$ colors is called {\em equitable} if for all color classes $F_i, F_j$, with $i\neq j\in\set{1,\dots,c}$, $\abs{\abs{F_i}-\abs{F_j}}\leq 1$ holds. 
\end{definition}

\begin{center}
	\begin{tabular}{lp{8cm}}
		\toprule
		\multicolumn{2}{c}{$p$-$tw$-$c$-\textsc{Equitable Coloring}}\\
		\midrule
		Input: & A graph $G$ and a number $c\in\N$.\\
		Parameter: & $\tw{G}+c$.\\
		Question: & Does an equitable coloring of $G$ in $c$ colors exist?\\
		\bottomrule
	\end{tabular}
\end{center}

\begin{theorem}[Fellows, Fomin, Lokshtanov, Rosamond, Saurabh, Szeider, Thomassen. 2011 \cite{fellows2011complexity}]\label{thm5.10}
The problem $p$-$tw$-$c$-\textsc{Equitable Coloring} is $\induz{W}{1}$-complete.
\end{theorem}

In addition we have to construct chains of graphs with a certain structure in order to find a graph $H$ that has an equitable coloring in a certain number of colors if and only if some arbitrary graph $G$, on which $H$ was constructed, has a $2$-strong coloring.

\begin{definition}[Chains of $\fkt{F}{l,c}$]
For positive integers $l,c\in\N$ with $l\leq c-2$ the auxiliary graph $F=\fkt{F}{l,c}$ is constructed as follows:
\begin{enumerate}[i)]

\item $\V{F}=\set{a_1,\dots,a_l,b_1,\dots,b_l,c_1,\dots,c_{c-l-2},f,g}$,

\item $fg\in\E{F}$,

\item $fa_i\in\E{F}$ as well as $gb_i\in\E{F}$ with $i\in\set{1,\dots,l}$,

\item $fc_i,gc_i\in\E{F}$ for all $i\in\set{1,\dots,c-l-2}$ and no other edges are present.
\end{enumerate}
A {\em chain} of $\fkt{F}{l,c}$ is formed by merging every $b_i$ with the corresponding $a_i$ of the consequent copy of $\fkt{F}{l,c}$.
\end{definition}

By induction over the length of a chain $\fkt{F}{l,c}$ the following is obtained.

\begin{lemma}[Fiala, Golovach, Kratochv{\'\i}l. 2011 \cite{fiala2011parameterized}]\label{lemma5.5}
For any $2$-strong coloring of a chain of graphs $\fkt{F}{l,c}$ on $c$ colors, the colors used on $a_1,\dots,a_l$ are distinct. These colors are identical with the set of colors used on the $b_1,\dots,b_l$ of the last graph $\fkt{F}{l,c}$ in the chain. Also any coloring of $a_1,\dots,a_l$ with colors from the set $\set{1,\dots,c}$ can be extended to a $2$-strong coloring in $c$ colors of the whole chain. 
\end{lemma}

\begin{definition}[The Auxiliary Graph $H$]
Let $G=\lb V,E\rb$ be a graph on $n$ vertices and $m$ edges and let $r\in\N$ be given with $r$ divides $n$. Then choose $l=\frac{n}{r}$ and $c=n+m+1$.\\
Let $\lb \mathscr{X},T\rb$ be a tree decomposition of $G$ with maximum degree $3$ and $x_1$ corresponding to $X_1$, with $\abs{X_1}=1$, being a leaf. In addition for any two adjacent vertices $x_i$ and $x_j$ of $T$ let $\abs{X_i\setminus X_j}+\abs{X_j\setminus X_i}\leq 1$.\\
Choose a walk $P=1\dots s$ in $T$ that visits every vertex in $T$ at least once and at most three times and let $e_1,\dots,e_m$ be the edges of $G$ sorted by their appearance in the bags $X_1,\dots,X_s$ along $P$.\\
Now we take $r$ disjoint chains of graphs $\fkt{F}{l,c}$ of length $m$. The vertices of these chains are denoted as follows:
\begin{enumerate}[i)]

\item The set $\set{a_1,\dots a_l}$ of the first $F$ in the $i$st chain is denoted by $A_i$.

\item The set $\set{b_1,\dots,b_l}$ of the $j$st copy of $F$ in the $i$st chain is denoted by $B_{i,j}$ and analogue the symbols $C_{i,j}$, $f_{i,j}$ and $g_{i,j}$.

\end{enumerate}
We proceed by adding a copy of $\fkt{F}{n,c}$ and rename its vertices $a_1,\dots,a_n$ to $v_1,\dots,v_n$. The vertices $b_1,\dots,b_n$ are consecutively merged with the vertices in $A_1\cup\dots\cup A_r$. The remaining vertices are $f_0,g_0$ and the vertices in $C_0$.\\
For each edge $e_j=u_pu_q$ of $G$ we add another vertex $w_j$ which then is joined to the vertices $v_p$ and $v_q$ and to $l-1$ vertices of each set $B_{i,j}$ with $i\in\set{1,\dots,r}$, hence $w_j$ is of degree $2+\lb l-1\rb\,r$. The obtained graph $H=\fkt{H}{G}$ is our auxiliary graph.
\end{definition}

\begin{figure}[H]
\begin{center}
	\begin{tikzpicture}
	
	\node (g0) [draw, circle, fill, inner sep=1.5pt] {};
	\node (lg0) [position=270:2mm from g0] {$g_0$};
	
	\node (C0anchor) [position=180:3.5mm from g0] {};
	\node (c01) [draw, circle, fill, inner sep=1.5pt,position=90:5mm from C0anchor] {};
	\node (c02) [draw, circle, fill, inner sep=1.5pt,position=90:0.5mm from C0anchor] {};
	\node (c03) [draw, circle, fill, inner sep=1.5pt,position=270:0.5mm from C0anchor] {};
	\node (c04) [draw, circle, fill, inner sep=1.5pt,position=270:5mm from C0anchor] {};
	\node (lC0) [position=90:1mm from c01] {$C_0$};
	
	\path
	(g0) edge (c01)
		 edge (c02)
		 edge (c03)
		 edge (c04)
	;
	
	\node (f0) [draw,circle,fill,inner sep=1.5pt, position=180:10mm from g0] {};
	\node (lf0) [position=90:2mm from f0] {$f_0$};
	
	\path
	(f0) edge (c01)
		 edge (c02)
		 edge (c03)
		 edge (c04)
	;
	
	\node (Vanchor) [position=180:8mm from f0,inner sep=0pt] {};
	\node (v1) [draw, circle, fill, inner sep=1.5pt,position=90:1.2mm from Vanchor] {};
	\node (v2) [draw, circle, fill, inner sep=1.5pt,position=90:4.5mm from Vanchor] {};
	\node (v3) [draw, circle, fill, inner sep=1.5pt,position=270:1.2mm from Vanchor] {};
	\node (v4) [draw, circle, fill, inner sep=1.5pt,position=270:4.5mm from Vanchor] {};
	\node (v5) [draw, circle, fill, inner sep=1.5pt,position=270:8mm from Vanchor] {};
	\node (v6) [draw, circle, fill, inner sep=1.5pt,position=270:11.5mm from Vanchor] {};
	\node (lv2) [position=180:1mm from v2] {$v_1$};	
	\node (lv3) [position=180:1mm from v3] {$v_p$};	
	\node (lv4) [position=180:1mm from v4] {$v_q$};	
	\node (lv6) [position=180:1mm from v6] {$v_n$};
	
	\path
	(f0) edge (v1)
		 edge (v2)
		 edge (v3)
		 edge (v4)
		 edge (v5)
		 edge (v6)
	;
	
	\node (oanchor1) [draw,circle,fill,inner sep=1.5pt,position=50:10mm from g0] {};	
	\node (oanchor2) [draw,circle,fill,inner sep=1.5pt,position=0:3mm from oanchor1] {};	
	\node (oanchor3) [draw,circle,fill,inner sep=1.5pt,position=0:7mm from oanchor2] {};	
	\node (oanchor4) [draw,circle,fill,inner sep=1.5pt,position=0:3mm from oanchor3] {};	
	\node (oanchor5) [draw,circle,fill,inner sep=1.5pt,position=0:3mm from oanchor4] {};	
	\node (oanchor6) [draw,circle,fill,inner sep=1.5pt,position=0:7mm from oanchor5] {};	
	\node (oanchor7) [draw,circle,fill,inner sep=1.5pt,position=0:3mm from oanchor6] {};	
	\node (oanchor8) [draw,circle,fill,inner sep=1.5pt,position=0:3mm from oanchor7] {};	
	\node (oanchor9) [draw,circle,fill,inner sep=1.5pt,position=0:7mm from oanchor8] {};	
	\node (oanchor10) [draw,circle,fill,inner sep=1.5pt,position=0:3mm from oanchor9] {};	
	\node (oanchor11) [draw,circle,fill,inner sep=1.5pt,position=0:3mm from oanchor10] {};	
	\node (oanchor12) [draw,circle,fill,inner sep=1.5pt,position=0:7mm from oanchor11] {};	
	\node (oanchor13) [draw,circle,fill,inner sep=1.5pt,position=0:3mm from oanchor12] {};	
							
	\path
	(g0) edge (oanchor1)
	(oanchor1) edge (oanchor2)
	(oanchor2) edge (oanchor3)
	(oanchor3) edge (oanchor4)
	(oanchor4) edge (oanchor5)
	(oanchor5) edge (oanchor6)
	(oanchor6) edge (oanchor7)
	(oanchor7) edge (oanchor8)
	(oanchor8) edge (oanchor9)
	(oanchor9) edge (oanchor10)
	(oanchor10) edge (oanchor11)
	(oanchor11) edge (oanchor12)
	(oanchor12) edge (oanchor13)
	;
	
	\node (o11) [draw,circle,fill,inner sep=1.5pt,position=90:2mm from oanchor1] {};
	\node (o12) [draw,circle,fill,inner sep=1.5pt,position=270:2mm from oanchor1] {};
		
	\path
	(g0) edge (o11)
		 edge (o12)
	(oanchor2) edge (o11)
			   edge (o12)
	;
	
	\node (o41) [draw,circle,fill,inner sep=1.5pt,position=90:2mm from oanchor4] {};
	\node (o42) [draw,circle,fill,inner sep=1.5pt,position=270:2mm from oanchor4] {};
		
	\path
	(oanchor3) edge (o41)
		 	   edge (o42)
	(oanchor5) edge (o41)
			   edge (o42)
	;
	
	\node (o71) [draw,circle,fill,inner sep=1.5pt,position=90:2mm from oanchor7] {};
	\node (o72) [draw,circle,fill,inner sep=1.5pt,position=270:2mm from oanchor7] {};
		
	\path
	(oanchor6) edge (o71)
		 	   edge (o72)
	(oanchor8) edge (o71)
			   edge (o72)
	;	
	
	\node (o101) [draw,circle,fill,inner sep=1.5pt,position=90:2mm from oanchor10] {};
	\node (o102) [draw,circle,fill,inner sep=1.5pt,position=270:2mm from oanchor10] {};
		
	\path
	(oanchor9) edge (o101)
		 	   edge (o102)
	(oanchor11) edge (o101)
			    edge (o102)
	;

	\node (o131) [draw,circle,fill,inner sep=1.5pt,position=90:2mm from oanchor13] {};
	\node (o132) [draw,circle,fill,inner sep=1.5pt,position=270:2mm from oanchor13] {};
		
	\path
	(oanchor12) edge (o131)
		 	   edge (o132)
	;		
	
	\node (omid10) [position=0:2.25mm from oanchor2] {};
	\node (omid11) [draw,circle,fill,inner sep=1.5pt,position=90:0.8mm from omid10] {}; 
	\node (omid12) [draw,circle,fill,inner sep=1.5pt,position=90:2mm from omid11] {}; 
	\node (omid13) [draw,circle,fill,inner sep=1.5pt,position=270:0.8mm from omid10] {}; 
	\node (omid14) [draw,circle,fill,inner sep=1.5pt,position=270:2mm from omid13] {};
	
	\path
	(oanchor2) edge (omid11)
			   edge (omid12)
			   edge (omid13)
			   edge (omid14)
	(oanchor3) edge (omid11)
			   edge (omid12)
			   edge (omid13)
			   edge (omid14)
	; 
	
	\node (omid20) [position=0:2.25mm from oanchor5] {};
	\node (omid21) [draw,circle,fill,inner sep=1.5pt,position=90:0.8mm from omid20] {}; 
	\node (omid22) [draw,circle,fill,inner sep=1.5pt,position=90:2mm from omid21] {}; 
	\node (omid23) [draw,circle,fill,inner sep=1.5pt,position=270:0.8mm from omid20] {}; 
	\node (omid24) [draw,circle,fill,inner sep=1.5pt,position=270:2mm from omid23] {};
	
	\path
	(oanchor5) edge (omid21)
			   edge (omid22)
			   edge (omid23)
			   edge (omid24)
	(oanchor6) edge (omid21)
			   edge (omid22)
			   edge (omid23)
			   edge (omid24)
	;	
	
	\node (omid30) [position=0:2.25mm from oanchor8] {};
	\node (omid31) [draw,circle,fill,inner sep=1.5pt,position=90:0.8mm from omid30] {}; 
	\node (omid32) [draw,circle,fill,inner sep=1.5pt,position=90:2mm from omid31] {}; 
	\node (omid33) [draw,circle,fill,inner sep=1.5pt,position=270:0.8mm from omid30] {}; 
	\node (omid34) [draw,circle,fill,inner sep=1.5pt,position=270:2mm from omid33] {};
	
	\path
	(oanchor8) edge (omid31)
			   edge (omid32)
			   edge (omid33)
			   edge (omid34)
	(oanchor9) edge (omid31)
			   edge (omid32)
			   edge (omid33)
			   edge (omid34)
	;	
	
	\node (omid40) [position=0:2.25mm from oanchor11] {};
	\node (omid41) [draw,circle,fill,inner sep=1.5pt,position=90:0.8mm from omid40] {}; 
	\node (omid42) [draw,circle,fill,inner sep=1.5pt,position=90:2mm from omid41] {}; 
	\node (omid43) [draw,circle,fill,inner sep=1.5pt,position=270:0.8mm from omid40] {}; 
	\node (omid44) [draw,circle,fill,inner sep=1.5pt,position=270:2mm from omid43] {};
	
	\path
	(oanchor11) edge (omid41)
			    edge (omid42)
			    edge (omid43)
			    edge (omid44)
	(oanchor12) edge (omid41)
			    edge (omid42)
			    edge (omid43)
			    edge (omid44)
	;

	\node (uanchor1) [draw,circle,fill,inner sep=1.5pt,position=310:10mm from g0] {};	
	\node (uanchor2) [draw,circle,fill,inner sep=1.5pt,position=0:3mm from uanchor1] {};	
	\node (uanchor3) [draw,circle,fill,inner sep=1.5pt,position=0:7mm from uanchor2] {};	
	\node (uanchor4) [draw,circle,fill,inner sep=1.5pt,position=0:3mm from uanchor3] {};	
	\node (uanchor5) [draw,circle,fill,inner sep=1.5pt,position=0:3mm from uanchor4] {};	
	\node (uanchor6) [draw,circle,fill,inner sep=1.5pt,position=0:7mm from uanchor5] {};	
	\node (uanchor7) [draw,circle,fill,inner sep=1.5pt,position=0:3mm from uanchor6] {};	
	\node (uanchor8) [draw,circle,fill,inner sep=1.5pt,position=0:3mm from uanchor7] {};	
	\node (uanchor9) [draw,circle,fill,inner sep=1.5pt,position=0:7mm from uanchor8] {};	
	\node (uanchor10) [draw,circle,fill,inner sep=1.5pt,position=0:3mm from uanchor9] {};	
	\node (uanchor11) [draw,circle,fill,inner sep=1.5pt,position=0:3mm from uanchor10] {};	
	\node (uanchor12) [draw,circle,fill,inner sep=1.5pt,position=0:7mm from uanchor11] {};	
	\node (uanchor13) [draw,circle,fill,inner sep=1.5pt,position=0:3mm from uanchor12] {};	
							
	\path
	(g0) edge (uanchor1)
	(uanchor1) edge (uanchor2)
	(uanchor2) edge (uanchor3)
	(uanchor3) edge (uanchor4)
	(uanchor4) edge (uanchor5)
	(uanchor5) edge (uanchor6)
	(uanchor6) edge (uanchor7)
	(uanchor7) edge (uanchor8)
	(uanchor8) edge (uanchor9)
	(uanchor9) edge (uanchor10)
	(uanchor10) edge (uanchor11)
	(uanchor11) edge (uanchor12)
	(uanchor12) edge (uanchor13)
	;
	
	\node (u11) [draw,circle,fill,inner sep=1.5pt,position=90:2mm from uanchor1] {};
	\node (u12) [draw,circle,fill,inner sep=1.5pt,position=270:2mm from uanchor1] {};
		
	\path
	(g0) edge (u11)
		 edge (u12)
	(uanchor2) edge (u11)
			   edge (u12)
	;
	
	\node (u41) [draw,circle,fill,inner sep=1.5pt,position=90:2mm from uanchor4] {};
	\node (u42) [draw,circle,fill,inner sep=1.5pt,position=270:2mm from uanchor4] {};
		
	\path
	(uanchor3) edge (u41)
		 	   edge (u42)
	(uanchor5) edge (u41)
			   edge (u42)
	;
	
	\node (u71) [draw,circle,fill,inner sep=1.5pt,position=90:2mm from uanchor7] {};
	\node (u72) [draw,circle,fill,inner sep=1.5pt,position=270:2mm from uanchor7] {};
		
	\path
	(uanchor6) edge (u71)
		 	   edge (u72)
	(uanchor8) edge (u71)
			   edge (u72)
	;	
	
	\node (u101) [draw,circle,fill,inner sep=1.5pt,position=90:2mm from uanchor10] {};
	\node (u102) [draw,circle,fill,inner sep=1.5pt,position=270:2mm from uanchor10] {};
		
	\path
	(uanchor9) edge (u101)
		 	   edge (u102)
	(uanchor11) edge (u101)
			    edge (u102)
	;

	\node (u131) [draw,circle,fill,inner sep=1.5pt,position=90:2mm from uanchor13] {};
	\node (u132) [draw,circle,fill,inner sep=1.5pt,position=270:2mm from uanchor13] {};
		
	\path
	(uanchor12) edge (u131)
		 	   edge (u132)
	;		
	
	\node (umid10) [position=0:2.25mm from uanchor2] {};
	\node (umid11) [draw,circle,fill,inner sep=1.5pt,position=90:0.8mm from umid10] {}; 
	\node (umid12) [draw,circle,fill,inner sep=1.5pt,position=90:2mm from umid11] {}; 
	\node (umid13) [draw,circle,fill,inner sep=1.5pt,position=270:0.8mm from umid10] {}; 
	\node (umid14) [draw,circle,fill,inner sep=1.5pt,position=270:2mm from umid13] {};
	
	\path
	(uanchor2) edge (umid11)
			   edge (umid12)
			   edge (umid13)
			   edge (umid14)
	(uanchor3) edge (umid11)
			   edge (umid12)
			   edge (umid13)
			   edge (umid14)
	; 
	
	\node (umid20) [position=0:2.25mm from uanchor5] {};
	\node (umid21) [draw,circle,fill,inner sep=1.5pt,position=90:0.8mm from umid20] {}; 
	\node (umid22) [draw,circle,fill,inner sep=1.5pt,position=90:2mm from umid21] {}; 
	\node (umid23) [draw,circle,fill,inner sep=1.5pt,position=270:0.8mm from umid20] {}; 
	\node (umid24) [draw,circle,fill,inner sep=1.5pt,position=270:2mm from umid23] {};
	
	\path
	(uanchor5) edge (umid21)
			   edge (umid22)
			   edge (umid23)
			   edge (umid24)
	(uanchor6) edge (umid21)
			   edge (umid22)
			   edge (umid23)
			   edge (umid24)
	;	
	
	\node (umid30) [position=0:2.25mm from uanchor8] {};
	\node (umid31) [draw,circle,fill,inner sep=1.5pt,position=90:0.8mm from umid30] {}; 
	\node (umid32) [draw,circle,fill,inner sep=1.5pt,position=90:2mm from umid31] {}; 
	\node (umid33) [draw,circle,fill,inner sep=1.5pt,position=270:0.8mm from umid30] {}; 
	\node (umid34) [draw,circle,fill,inner sep=1.5pt,position=270:2mm from umid33] {};
	
	\path
	(uanchor8) edge (umid31)
			   edge (umid32)
			   edge (umid33)
			   edge (umid34)
	(uanchor9) edge (umid31)
			   edge (umid32)
			   edge (umid33)
			   edge (umid34)
	;	
	
	\node (umid40) [position=0:2.25mm from uanchor11] {};
	\node (umid41) [draw,circle,fill,inner sep=1.5pt,position=90:0.8mm from umid40] {}; 
	\node (umid42) [draw,circle,fill,inner sep=1.5pt,position=90:2mm from umid41] {}; 
	\node (umid43) [draw,circle,fill,inner sep=1.5pt,position=270:0.8mm from umid40] {}; 
	\node (umid44) [draw,circle,fill,inner sep=1.5pt,position=270:2mm from umid43] {};
	
	\path
	(uanchor11) edge (umid41)
			    edge (umid42)
			    edge (umid43)
			    edge (umid44)
	(uanchor12) edge (umid41)
			    edge (umid42)
			    edge (umid43)
			    edge (umid44)
	;
	
	\node (w) [draw,circle,fill,inner sep=1.5pt,position=282:17mm from uanchor7] {};
	\node (lw) [position=270:2mm from w] {$w_j$};	
	
	\path
	(w) edge [bend left] (v3)
		edge [bend left] (v4)
		edge [bend right] (oanchor7)
		edge [bend left] (o72)
		edge [bend right] (uanchor7)
		edge (u72)
	;
	
	\node (lB1j) [position=90:1mm from o71] {$B_{1,j}$};
	\node (lBrj) [position=235:3mm from u72] {$B_{r,j}$};	
	
	\end{tikzpicture}
\end{center}
\caption{Example of the auxiliary graph $H$.}
\label{fig5.2}
\end{figure}
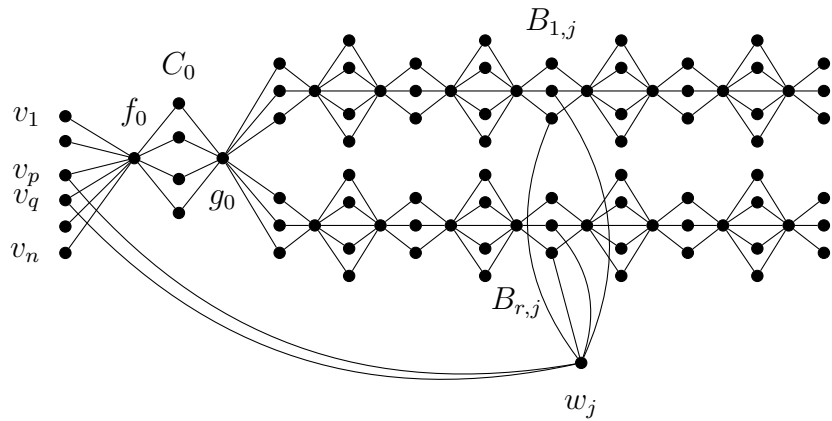

\begin{lemma}[Fiala, Golovach, Kratochv{\'\i}l. 2011 \cite{fiala2011parameterized}]\label{lemma5.6}
Let $G$ be a graph, $r\in\N$ and $H=\fkt{H}{G}$ a corresponding auxiliary graph, then $G$ has an equitable coloring in $r$ colors if and only if $H$ has a $2$-strong coloring in $c=n+m+1$ colors.
\end{lemma}

\begin{proof}
Suppose $G$ has an equitable coloring in $r$ colors. For each color $i$ used on $G$ we will introduce the corresponding colors $\alpha_{i,1},\dots,\alpha_{i,l}$ in $H$. We further use $m+2$ additional colors $\beta_1,\dots,\beta_m$ and $\gamma_1,\gamma_2$. The $2$-strong coloring in $c$ colors of $H$ is obtained as follows:\\
\begin{enumerate}[i)]

\item The vertices of the set $\condset{v_p}{u_p~\text{has color}~i}$ are colored with different $\alpha_{i,1},\dots,\alpha_{i,l}$.

\item For $i\in\set{1,\dots, r}$, vertices of $A_i$ are colored with $\alpha_{i,1},\dots,\alpha_{i,l}$.

\item Each vertex $w_i$ is colored with $\beta_j$.

\item If $w_j$ is adjacent to $v_p$ where $v_p$ has been colored by some $\alpha_{i,y}$, we use the same color $\alpha_{i,y}$ on the only non-neighbor of $w_j$ in $B_{i,j}$.

\item The colors of all other vertices are obtained by Lemma \autoref{lemma5.5}, in particular:
\begin{itemize}
\item colors of the set $B_{i,j}$ are completed arbitrarily to $\alpha_{i,1},\dots,\alpha_{i,l}$,

\item the vertices of $C_{i,j}$ are colored by $\alpha_{k,y}$ with $k\neq i$ and by $\beta_x$, $x=1,\dots,m$, and

\item all vertices $f_{i,j}$ together with $f_0$ are colored with $\gamma_1$ and analogue the vertices $g_{i,j}$ along with $g_0$ are colored $\gamma_2$. 
\end{itemize} 
\end{enumerate}
If a vertex $w_j$ is adjacent to vertices $v_p$ and $v_q$, it means that the corresponding vertices $u_p$ and $u_q$ in $G$ belong to a common edge $e_j$ and therefore are adjacent and the colors $\alpha_{i,j}$ of $v_p$ and $v_q$ differ in the subscript $i$. Hence it suffices that $w_j$ has one non-neighbor in each set $B_{i,j}$ for the coloring to be well defined. With this it becomes easy to check that our $2$-strong coloring of $H$ is proper.

Now assume $H$ to have a $2$-strong coloring in $c$. As we have seen in Lemma \autoref{lemma5.5} the vertices of the sets $A_i$ must be colored with different colors and the same colors are used for the vertices $v_1,\dots,v_n$. The colors used on the vertices of $A_i$ still are denoted by $\alpha_{i,1},\dots,\alpha_{i,l}$.\\
Now we construct an equitable coloring of $G$ in $r$ colors as follows by coloring the vertex $u_p\in\V{G}$ with color $i$, if the corresponding $v_p$ in $H$ is colored by some $\alpha_{i,y}$.\\
With each $\alpha_{i,j}$ appearing exactly once for every pair $i\in\set{1,\dots,r}$ and $j\in\set{1,\dots,l}$, each color class of $G$, given by the previously constructed coloring, consists of exactly $l$ vertices. Assume this coloring is not proper, thus some edge $e_j=u_pu_q\in\E{G}$ exists with $u_p$ and $u_q$ both being colored with the color $i$. Hence the vertices $v_p$ and $v_q$ are colored with some colors $\alpha_{i,x}$ and $\alpha_{i,y}$. By Lemma \autoref{lemma5.5} the vertices of $B_{i,j}$ are colored with $\alpha_{i,1},\dots,\alpha_{i,l}$ and by the definition of $H$, the vertex $w_j$ has $j-1$ neighbors in $B_{i,j}$. Hence at least one of them has the same color as $v_p$ or $v_q$ and the $2$-strong coloring of $H$ was not proper, a contradiction.    
\end{proof}

\begin{lemma}[Fiala, Golovach, Kratochv{\'\i}l. 2011 \cite{fiala2011parameterized}]\label{lemma5.7}
Let $G$ be a graph, the treewidth of the auxiliary graph $\fkt{H}{G}$ is at most $\lb 2\,r+2\rb\tw{G}+3\,r+2$.
\end{lemma}

\begin{proof}
Let $\lb \condset{X_i}{i\in\V{T}},T\rb$ be the tree decomposition of $G$ of width $t$ used to construct $H=fkt{H}{G}$ and $P=1\dots s$ the walk in $T$. We construct a tree decomposition of $H$ from the tree decomposition of $G$.\\
For the same tree $T$, we introduce sets $Y_i$ by first putting $v_j$ to $Y_i$ if and only if $u_j\in X_i$ for $j=1,\dots,n$, and then adding $f_0$ and $g_0$ to each $Y_i$. Then we alter the tree by adding some leaves and obtain the decomposition of $H$ by adding further vertices to the bags.\\
The changes are performed inductively by following the walk $P$. Since $\abs{X_1}=1$, this bag contains no edge of $G$. Let $Y_1\define Y_1\cup\set{f_{1,1},\dots,f_{r,1}}$. For each vertex $z\in C_0\cup\bigcup^r_{i=1}A_i$, a new leaf vertex $d_z$ adjacent to $1$ is added in $T$. We define the corresponding bag as $Y_{d_z}\define\induz{N_H}{z}$ and obtain $\abs{Y_{d_z}}=3$.\\
For the induction we will need some addition variable $h$, initialized by $h\define 1$.\\
Now suppose we have already made modifications of the tree decomposition for a subwalk $1\dots i_{j-1}$ of $P$. Since the edges of $G$ are sorted in the order in which they occur in the bags $X_1,\dots,X_s$, there are some $x$ and $y$ such that $E_j=\set{e_x,\dots e_y}$ is the set of edges that first occur in the bag $X_j$.\\
Now, if $E_j=\emptyset$ we set $Y_{i_j}\define Y_{i_j}\cup\set{f_{1,h},\dots,f_{r,h}}$, if $h\leq m$, otherwise this set is assumed to be empty, and consider the next node of the walk. If $E_j\neq\emptyset$, then we set 
\begin{align*}
Y_{i_j}\define Y_{i_j}\cup\set{w_x,\dots,w_y}\cup\bigcup^y_{p=x}\lb \set{f_{1,p},\dots,f_{r,p}}\cup\set{g_{1,p},\dots,g_{r,p}}\rb\cup\set{f_{1,y+1},\dots,f_{r,y+1}},
\end{align*}
where the last set is empty for $y=m$. Since $\abs{X_{i_j}\setminus X_{i_{j-1}}}\leq1$, $E_j$ contains at most $t$ edges and we added at most $t\,\lb 2\,r+1\rb+r$ vertices. For each vertex $z\in\cup^r_{q=1}\cup^y_{p=x}\lb B_{q,p}\cup C_{q,p}\rb$, a leaf vertex $d_z$ is added to the tree and joined to $i_j$, and gets associated with the bag $Y_{d_z}\define\induz{N_H}{z}$ with $\abs{Y_{d_z}}\leq4$. Finally we set $h\define y+1$.\\
Now all there is left to do is to check that we constructed a valid tree decomposition of $H$. The walk $P$ visits any vertex at most three times. When it visits a vertex $i$ the first time, we add at most $t\cdot\lb 2\,r+1\rb+r$ vertices to the initial set $Y_i$, and when the walk goes to $i$ the second or third time then at most $r$ vertices are added. Therefore each by $Y_i$ contains at most $t+1+2+t\cdot\lb 2\,r+1\rb+r+r+r$ vertices, thus $\tw{H}\leq \lb 2\,r+2\rb\cdot t+3\,r+2$.
\end{proof}

Now, by combining Theorem \autoref{thm5.10} with the lemmas \autoref{lemma5.6} and \autoref{lemma5.7} we obtain the following theorem.

\begin{theorem}[Fiala, Golovach, Kratochv{\'\i}l. 2011 \cite{fiala2011parameterized}]\label{thm5.11}
The problem $p$-$tw$-$k$-\textsc{Strong Coloring} is $\induz{W}{1}$-hard.		
\end{theorem}

It seems that despite treewidth being known as a very powerful parameter for graph problems, a tree decomposition of bounded width alone is not enough to find a strong coloring in polynomial time. So we will close this section by introducing a parameter which poses as an upper bound on the treewidth by inducing a natural tree decomposition with some additional properties the might come in handy whilst searching for a proper strong coloring. 

\begin{definition}[Vertex Cover]
Let $G$ be a graph and $W\subseteq\V{G}$. The set $W$ is called a {\em vertex cover} of $G$ if every edge of $G$ has at least one endpoint in $W$. The {\em vertex cover number} of a graph $G$ is the minimum size of a vertex cover	and is denoted by $\vc{G}$.
\end{definition}

An obvious relation between a vertex cover and stable sets can be stated as well. A set $I\subset\V{G}$ of vertices of a graph $G$ is a maximum stable set if and only if $\V{G}\setminus I$ is a minimum vertex cover. For any vertex cover $W\subseteq\V{G}$ the remaining vertices $\V{G}\setminus W$ always are a stable set. 

\begin{theorem}[Karp. 1972 \cite{karp1972reducibility}]\label{thm5.12}
The minimum vertex cover problem is $\mathcal{NP}$-complete. 	
\end{theorem} 

\begin{center}
	\begin{tabular}{lp{9cm}}
		\toprule
		\multicolumn{2}{c}{$p$-$\theta$-$2$-\textsc{Strong Coloring}}\\
		\midrule
		Input: & A graph $G$ and a number $c\in\N$.\\
		Parameter: & $\vc{G}+c$.\\
		Question: & Does a $2$-strong coloring of $G$ in $c$ colors exist?\\
		\bottomrule
	\end{tabular}
\end{center}

\begin{theorem}[Fiala, Golovach, Kratochv{\'\i}l. 2011 \cite{fiala2011parameterized}]\label{thm5.14}
The $p$-$\theta$-$2$-\textsc{Strong Coloring} problem is in $\FPT$.	
\end{theorem}
\vspace{-4mm}

\section{Chordality}

What is responsible for the ordinary coloring problem, along with many other problems, to be easily solvable on chordal graphs?\\
In this section we finally revisit the rich structure that is a chordal graph and give a broad overview on the algorithmic use of this structure and its limits. Especially treewidth becomes a very nice and easy to handle parameter due to the strong connection between treewidth, $k$-trees and chordality of graphs. This fact will provide some additional structural properties which we can use to obtain upper bounds as well as actual parameterized algorithms for hard to compute parameters on non chordal graphs with chordal powers. For this reason we will introduce the power of chordality as the smallest integer for which the corresponding power of a given graph becomes chordal.\\
We start this section by revisiting several characterizations of chordal graphs that have already been introduced in Chapter 3 and will form the base of this section.
\begin{theorem}[see \cite{fulkerson1965incidence} and \cite{buneman1974characterisation}]\label{thm5.15}
Let $G$ be a graph, the following properties are equivalent:
\begin{enumerate}[i)]

\item $G$ is chordal.

\item Every minimal separator of $G$ induces a clique.

\item $G$ is the intersection graph of a set of subtrees $\set{T_1,\dots,T_k}$ of a tree $T$.

\item There exists a tree $T=\lb K,L\rb$ where the vertex set $K$ corresponds to the set of all maximal cliques in $G$ and the set $K_v\define\condset{Q\in K}{v\in Q}$ induces a subtree of $T$ for all $v\in\V{G}$.

\end{enumerate}
\end{theorem}  
\vspace{-5mm}
\subsection{An Algorithmic Approach on Chordal Graphs}

Now there is one essential alternative characterization of chordal graphs which yields their enormous algorithmic approachability. For this we need to find vertices with a certain property called simpliciality which was introduced by 
Dirac in 1961 (see \cite{dirac1961rigid}).
\begin{definition}[Simplicial Vertex]
Let $G$ be a graph and $v\in\V{G}$ a vertex, $v$ is called {\em simplicial} if its neighborhood $\nb{v}$ induces a complete graph.
\end{definition}
Simplicial vertices can be seen as a generalized version of the leafs of a tree and indeed such vertices can be found in chordal graphs.
\begin{theorem}[Dirac, Lekkeikerker and Boland. 1961,1962 \cite{dirac1961rigid} and \cite{lekkeikerker1962representation}]\label{thm5.16}
Every chordal graph $G$ contains a simplicial vertex. If $G$ is not complete, it contains two nonadjacent simplicial vertices. 
\end{theorem}
\vspace{-2mm}
Since every induced subgraph of a chordal graph $G$ is chordal again, it is easy to see that for any simplicial vertex $v\V{G}$ the graph $G-v$ again is chordal and therefore contains at least one simplicial vertex. With this we are able to deconstruct the whole graph $G$ by consecutive removals of simplicial vertices in the remaining graphs. The removal order of these vertices gives a so called perfect elimination scheme of the chordal graph $G$.
\begin{definition}[Perfect Elimination Scheme (PES)]
Let $G=\lb V,E\rb$ be a graph and $\sigma=\left[ v_1,\dots,v_n\right]$ an ordering of the vertices of $G$.\\
The ordering $\sigma$ is called a {\em perfect elimination scheme} (PES) if the vertex $v_i$ is simplicial in the graph $\induz{G}{\set{v_i,\dots,v_n}}$ for all $i=1,\dots,n$.\\
So the set $X_i\define\condset{v_j\in\nb{v_i}}{j>i}$ induces a clique in $G$.
\end{definition}
\vspace{-2mm}
\begin{theorem}[Fulkerson, Gross. 1965 \cite{fulkerson1965incidence}]\label{thm5.17}
A graph $G$ is chordal if and only if it holds a perfect elimination scheme. Such a PES may start with an arbitrary simplicial vertex.
\end{theorem}
\vspace{-2mm}
Now consider a graph $G$ with some chordal power $G^k$. According to Theorem \autoref{thm5.17} $G^k$ holds a PES $\sigma$, which obviously is an ordering of the vertices of $G$ as well. But as the vertex $v_i$ is simplicial in the graph $\induz{\lb G^k\rb}{\set{v_i,\dots,v_n}}$ and thus $X_i\define\condset{v_j\in\knb{k}{v_i}}{j>i}$ induces a clique in $G^k$. But $X_i$ is not necessarily equal to $\fkt{N_{\lb\induz{G}{\set{v_i,\dots,v_n}}\rb^k}}{v_i}$ as the following example shows.

\begin{example}
Consider the cycle on five vertices $C_5$ and its square, which is chordal. By Theorem \autoref{thm5.17} $C_5^2$ holds a PES $\sigma$, which actually is an arbitrary ordering of the vertices its vertices, since $C_5^2$ is complete.
\begin{figure}[H]
\begin{center}
\begin{tikzpicture}
\node (anchor1) [] {};
\node (i) [position=270:2.25cm from anchor1] {$C_5$};

\node (anchor4) [position=0:5.2cm from anchor1] {};
\node (iv) [position=270:2.25cm from anchor4] {$C_5^2$};



\node (v1) [draw,circle,fill,inner sep=1.5pt,position=0:0.9cm from anchor1] {};
\node (v2) [draw,circle,fill,inner sep=1.5pt,position=72:0.9cm from anchor1] {};
\node (v3) [draw,circle,fill,inner sep=1.5pt,position=144:0.9cm from anchor1] {};
\node (v4) [draw,circle,fill,inner sep=1.5pt,position=216:0.9cm from anchor1] {};
\node (v5) [draw,circle,fill,inner sep=1.5pt,position=288:0.9cm from anchor1] {};

\path
(v1) edge (v2)
(v2) edge (v3)
(v3) edge (v4)
(v4) edge (v5)
(v5) edge (v1)
;

\node (cl1) [position=0:11mm from anchor1] {$v_1$};
\node (cl2) [position=72:11mm from anchor1] {$v_2$};
\node (cl3) [position=144:11mm from anchor1] {$v_3$};
\node (cl4) [position=216:11mm from anchor1] {$v_4$};
\node (cl5) [position=288:11mm from anchor1] {$v_5$};



\node (v1) [draw,circle,fill,inner sep=1.5pt,position=0:0.9cm from anchor4] {};
\node (v2) [draw,circle,fill,inner sep=1.5pt,position=72:0.9cm from anchor4] {};
\node (v3) [draw,circle,fill,inner sep=1.5pt,position=144:0.9cm from anchor4] {};
\node (v4) [draw,circle,fill,inner sep=1.5pt,position=216:0.9cm from anchor4] {};
\node (v5) [draw,circle,fill,inner sep=1.5pt,position=288:0.9cm from anchor4] {};

\path
(v1) edge (v2)
(v2) edge (v3)
(v3) edge (v4)
(v4) edge (v5)
(v5) edge (v1)
;

\path
(v1) edge (v3)
(v2) edge (v4)
(v3) edge (v5)
(v4) edge (v1)
(v5) edge (v2)
;

\node (cl1) [position=0:11mm from anchor4] {$v_1$};
\node (cl2) [position=72:11mm from anchor4] {$v_2$};
\node (cl3) [position=144:11mm from anchor4] {$v_3$};
\node (cl4) [position=216:11mm from anchor4] {$v_4$};
\node (cl5) [position=288:11mm from anchor4] {$v_5$};

\end{tikzpicture}
\end{center}
\caption{The $C_5$ and its chordal square with a simplicial ordering $\sigma=\left[ v_1,v_2,v_3,v_4,v_5\right]$.}
\label{fig5.3}
\end{figure}
\vspace{-8mm}

Now in the graph induced by the vertices of $C_5$ excluding $v_1$, the $2$-strong neighborhood of $v_2$ differs from its neighborhood in the same induced subgraph in $C_5^2$. This is because $v_1$ was responsible for the only path of length two between $v_2$ and $v_5$, with $v_1$ gone this path does not exist anymore and the square of the remaining graph changes.

\begin{figure}[H]
\begin{center}
\begin{tikzpicture}
\node (anchor1) [] {};
\node (i) [position=270:2.25cm from anchor1] {$\induz{C_5^2}{\set{v_2,\dots,v_5}}$};

\node (anchor4) [position=0:5.2cm from anchor1] {};
\node (iv) [position=270:2.25cm from anchor4] {$\lb \induz{C_5}{\set{v_1,\dots,v_5}}\rb^2$};



\node (v2) [draw,circle,fill,inner sep=1.5pt,position=72:0.9cm from anchor1] {};
\node (v3) [draw,circle,fill,inner sep=1.5pt,position=144:0.9cm from anchor1] {};
\node (v4) [draw,circle,fill,inner sep=1.5pt,position=216:0.9cm from anchor1] {};
\node (v5) [draw,circle,fill,inner sep=1.5pt,position=288:0.9cm from anchor1] {};

\path
(v2) edge (v3)
(v3) edge (v4)
(v4) edge (v5)
;

\node (cl2) [position=72:11mm from anchor1] {$v_2$};
\node (cl3) [position=144:11mm from anchor1] {$v_3$};
\node (cl4) [position=216:11mm from anchor1] {$v_4$};
\node (cl5) [position=288:11mm from anchor1] {$v_5$};

\path
(v2) edge (v4)
(v3) edge (v5)
(v5) edge (v2)
;



\node (v2) [draw,circle,fill,inner sep=1.5pt,position=72:0.9cm from anchor4] {};
\node (v3) [draw,circle,fill,inner sep=1.5pt,position=144:0.9cm from anchor4] {};
\node (v4) [draw,circle,fill,inner sep=1.5pt,position=216:0.9cm from anchor4] {};
\node (v5) [draw,circle,fill,inner sep=1.5pt,position=288:0.9cm from anchor4] {};

\path
(v2) edge (v3)
(v3) edge (v4)
(v4) edge (v5)
;

\path
(v2) edge (v4)
(v3) edge (v5)
;

\node (cl2) [position=72:11mm from anchor4] {$v_2$};
\node (cl3) [position=144:11mm from anchor4] {$v_3$};
\node (cl4) [position=216:11mm from anchor4] {$v_4$};
\node (cl5) [position=288:11mm from anchor4] {$v_5$};

\end{tikzpicture}
\end{center}
\caption{The graphs $\induz{C_5^2}{\set{v_2,\dots,v_5}}$ and $\lb \induz{C_5}{\set{v_1,\dots,v_5}}\rb^2$.}.
\label{fig5.4}
\end{figure}
\vspace{-8mm}
\end{example}

This problem makes it somewhat difficult to define a generalized perfect elimination scheme without considering induced subgraphs of the chordal power of a graph. How critical this problem is will become clear in the next example which illustrates even the chordality itself does not have to be inherited by its induced subgraphs, which is due to the problem that sunflowers themselves do not give a characterization of graphs with chordal squares in terms of forbidden induced subgraphs.

\begin{example}
Now consider the withered chordal sunflower $G$ of size $4$. For a chordal sunflower of size $4$ the existence of just one additional vertex connecting two of the $u$-vertices suffices for the square to be chordal. But since in $G$ is chordal, the set $\set{v_4,\dots,v_9}$, which is the neighborhood of $v_1$ in $G^2$, induces a clique and therefore the additional vertex $v_1$ is correctly picked as a simplicial vertex in $G^2$. With that $G-v_1$ is a nonwithered sunflower and $\lb G-v_1\rb^2$ is not chordal anymore.
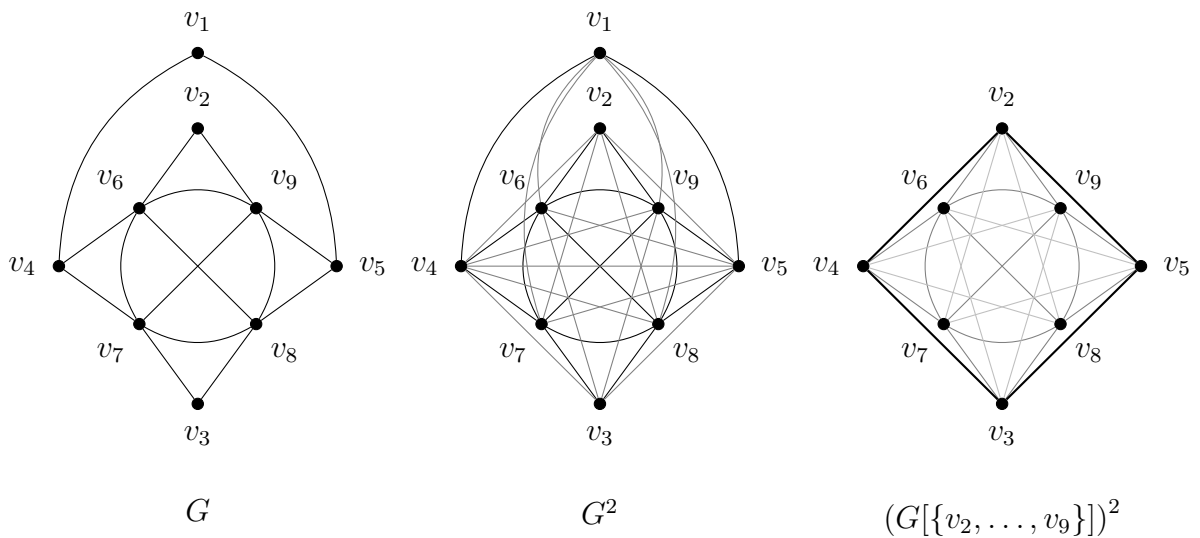
\begin{figure}[H]
	\begin{center}
		\begin{tikzpicture}
		
		\node (anchor1) [] {};
		\node (anchor2) [position=0:50mm from anchor1] {};
		\node (anchor3) [position=0:50mm from anchor2] {};
		
		\node (u1) [draw,circle,fill,inner sep=1.5pt,position=90:1.6cm from anchor1] {};
		\node (u2) [draw,circle,fill,inner sep=1.5pt,position=180:1.6cm from anchor1] {};
		\node (u3) [draw,circle,fill,inner sep=1.5pt,position=270:1.6cm from anchor1] {};
		\node (u4) [draw,circle,fill,inner sep=1.5pt,position=0:1.6cm from anchor1] {};
		
		\node (w1) [draw,circle,fill,inner sep=1.5pt,position=135:0.8cm from anchor1] {};
		\node (w2) [draw,circle,fill,inner sep=1.5pt,position=225:0.8cm from anchor1] {};
		\node (w3) [draw,circle,fill,inner sep=1.5pt,position=315:0.8cm from anchor1] {};
		\node (w4) [draw,circle,fill,inner sep=1.5pt,position=45:0.8cm from anchor1] {};
		
		\node (v) [draw,circle,fill,inner sep=1.5pt,position=90:2.6cm from anchor1] {};
		
		\node (lu1) [position=90:0.07 from u1] {$v_2$};
		\node (lu2) [position=180:0.07 from u2] {$v_4$};
		\node (lu3) [position=270:0.07 from u3] {$v_3$};
		\node (lu4) [position=0:0.07 from u4] {$v_5$};
		
		\node (lw1) [position=135:0.07cm from w1] {$v_6$};
		\node (lw2) [position=225:0.07cm from w2] {$v_7$};
		\node (lw3) [position=315:0.07cm from w3] {$v_8$};
		\node (lw4) [position=45:0.07cm from w4] {$v_9$};

		\node (lv) [position=90:0.07cm from v] {$v_1$};
		
		\node (n) [position=270:2.8cm from anchor1] {$G$};
		
		\path
		(u1) edge (w1)
		edge (w4)
		(u2) edge (w1)
		edge (w2)
		(u3) edge (w2)
		edge (w3)
		(u4) edge (w3)
		edge (w4)
		;
		
		\path [bend right]
		(w1) edge (w2)
		(w2) edge (w3)
		(w3) edge (w4)
		(w4) edge (w1)
		;
		
		\path
		(w2) edge (w4)
		(w1) edge (w3)
		;
		
		\path
		(v) edge [bend right] (u2)
			edge [bend left] (u4)
		;
		
		
		\node (u1) [draw,circle,fill,inner sep=1.5pt,position=90:1.6cm from anchor2] {};
		\node (u2) [draw,circle,fill,inner sep=1.5pt,position=180:1.6cm from anchor2] {};
		\node (u3) [draw,circle,fill,inner sep=1.5pt,position=270:1.6cm from anchor2] {};
		\node (u4) [draw,circle,fill,inner sep=1.5pt,position=0:1.6cm from anchor2] {};
		
		\node (w1) [draw,circle,fill,inner sep=1.5pt,position=135:0.8cm from anchor2] {};
		\node (w2) [draw,circle,fill,inner sep=1.5pt,position=225:0.8cm from anchor2] {};
		\node (w3) [draw,circle,fill,inner sep=1.5pt,position=315:0.8cm from anchor2] {};
		\node (w4) [draw,circle,fill,inner sep=1.5pt,position=45:0.8cm from anchor2] {};
		
		\node (v) [draw,circle,fill,inner sep=1.5pt,position=90:2.6cm from anchor2] {};
		
		\node (lu1) [position=90:0.07 from u1] {$v_2$};
		\node (lu2) [position=180:0.07 from u2] {$v_4$};
		\node (lu3) [position=270:0.07 from u3] {$v_3$};
		\node (lu4) [position=0:0.07 from u4] {$v_5$};
		
		\node (lw1) [position=135:0.07cm from w1] {$v_6$};
		\node (lw2) [position=225:0.07cm from w2] {$v_7$};
		\node (lw3) [position=315:0.07cm from w3] {$v_8$};
		\node (lw4) [position=45:0.07cm from w4] {$v_9$};
		
		\node (lv) [position=90:0.07cm from v] {$v_1$};
		
		\node (n) [position=270:2.8cm from anchor2] {$G^2$};
		
		\path
		(u1) edge (w1)
		edge (w4)
		(u2) edge (w1)
		edge (w2)
		(u3) edge (w2)
		edge (w3)
		(u4) edge (w3)
		edge (w4)
		;
		
		\path [bend right]
		(w1) edge (w2)
		(w2) edge (w3)
		(w3) edge (w4)
		(w4) edge (w1)
		;
		
		\path
		(w2) edge (w4)
		(w1) edge (w3)
		;
		
		\path
		(v) edge [bend right] (u2)
		edge [bend left] (u4)
		;
		
		\path[color=gray]
		(v) edge [bend right] (w1)
			edge [bend right] (w2)
			edge [bend left] (w3)
			edge [bend left] (w4)
		;
		
		\path[color=gray]
		(u1) edge (u2)
			 edge (u4)
			 edge (w2)
			 edge (w3)
		;
		
		\path[color=gray]
		(u2) edge (u4)
			 edge (w3)
			 edge (w4)
		;
		
		\path[color=gray]
		(u3) edge (u2)
			 edge (u4)
			 edge (w1)
			 edge (w4)
		;
		
		\path[color=gray]
		(u4) edge (w1)
			 edge (w2)
		;

		
		\node (u1) [draw,circle,fill,inner sep=1.5pt,position=90:1.6cm from anchor3] {};
		\node (u2) [draw,circle,fill,inner sep=1.5pt,position=180:1.6cm from anchor3] {};
		\node (u3) [draw,circle,fill,inner sep=1.5pt,position=270:1.6cm from anchor3] {};
		\node (u4) [draw,circle,fill,inner sep=1.5pt,position=0:1.6cm from anchor3] {};
		
		\node (w1) [draw,circle,fill,inner sep=1.5pt,position=135:0.8cm from anchor3] {};
		\node (w2) [draw,circle,fill,inner sep=1.5pt,position=225:0.8cm from anchor3] {};
		\node (w3) [draw,circle,fill,inner sep=1.5pt,position=315:0.8cm from anchor3] {};
		\node (w4) [draw,circle,fill,inner sep=1.5pt,position=45:0.8cm from anchor3] {};
		
		\node (lu1) [position=90:0.07 from u1] {$v_2$};
		\node (lu2) [position=180:0.07 from u2] {$v_4$};
		\node (lu3) [position=270:0.07 from u3] {$v_3$};
		\node (lu4) [position=0:0.07 from u4] {$v_5$};
		
		\node (lw1) [position=135:0.07cm from w1] {$v_6$};
		\node (lw2) [position=225:0.07cm from w2] {$v_7$};
		\node (lw3) [position=315:0.07cm from w3] {$v_8$};
		\node (lw4) [position=45:0.07cm from w4] {$v_9$};
		
		\node (n) [position=270:2.8cm from anchor3] {$\lb\induz{G}{\set{v_2,\dots,v_9}}\rb^2$};
		
		\path[color=gray]
		(u1) edge (w1)
		edge (w4)
		(u2) edge (w1)
		edge (w2)
		(u3) edge (w2)
		edge (w3)
		(u4) edge (w3)
		edge (w4)
		;
		
		\path [bend right,color=gray]
		(w1) edge (w2)
		(w2) edge (w3)
		(w3) edge (w4)
		(w4) edge (w1)
		;
		
		\path[color=gray]
		(w2) edge (w4)
		(w1) edge (w3)
		;
		
		\path[color=lightgray]
		(u1) edge[thick,color=black] (u2)
		 edge[thick,color=black] (u4)
		edge (w2)
		edge (w3)
		;
		
		\path[color=lightgray]
		(u2) edge (w3)
		edge (w4)
		;
		
		\path[color=lightgray]
		(u3) edge[thick,color=black] (u2)
		edge[thick,color=black] (u4)
		edge (w1)
		edge (w4)
		;
		
		\path[color=lightgray]
		(u4) edge (w1)
		edge (w2)
		;
		\end{tikzpicture}
	\end{center}
	\caption{A graph $G$ with a PES of its chordal square and an induced subgraph square.}
	\label{fig5.5}
\end{figure}
\end{example}

This suggests a differentiation between two different types of PES for chordal powers of a graph. Those for which the $k$th power of the induced subgraph of the original graph is chordal in each step and those, the weaker ones, for which this is not the case.\\
In this thesis we only suggest the existence of those two different qualities of chordal powers of a graph and will not go any further than stating a definition. For us the weaker version will do.
\begin{definition}[$k$-strong Perfect Elimination Scheme ($k$-PES) (weak version)]
	Let $G=\lb V,E\rb$ be a graph, $k\in\N$ and $\sigma=\left[ v_1,\dots,v_n\right]$ an ordering of the vertices of $G$.\\
	The ordering $\sigma$ is called a {\em $k$-strong perfect elimination scheme} ($k$-PES) if the vertex $v_i$ is simplicial in the graph $\induz{G^k}{\set{v_i,\dots,v_n}}$ for all $i=1,\dots,n$.\\
	So the set $X_i^k\define\condset{v_j\in\knb{k}{v_i}}{j>i}$ induces a clique in $G^k$.
\end{definition}
\begin{definition}[True $k$-strong Perfect Elimination Scheme (True $k$-PES) (strong version)]
	Let $G=\lb V,E\rb$ be a graph, $k\in\N$ and $\sigma=\left[ v_1,\dots,v_n\right]$ an ordering of the vertices of $G$.\\
	The ordering $\sigma$ is called a {\em true $k$-strong perfect elimination scheme} (True $k$-PES) if the vertex $v_i$ is simplicial in the graph $\lb\induz{G}{\set{v_i,\dots,v_n}}\rb^k$ for all $i=1,\dots,n$.\\
	So the set $X_i^k\define\condset{v_j\in\knb{k}{v_i}}{j>i}$ induces a clique in $G^k$.
\end{definition}
The difference in these two definitions is very subtle, but it truly is a difference in the structure of those graphs and might also pose a difference in the algorithmic application of such a $k$-strong PES.

Now we have the existence of a PES as another characterization for chordal graphs. So how to find such a PES and how to verify it? This will be the the next step, leading to the actual computation of graph powers and $k$-strong PES, which is the first step to the application of this concept on the coloring problem.\\
Rose, Tarjan and Lueker showed in 1976 (see \cite{rose1976algorithmic}) that a PES can be found efficiently using lexicographic breadth-first search.

\begin{algorithm}[H]
\caption{Lexicographic Breadth-First Search (\textbf{LexBFS})}\label{alg5.1}
\begin{algorithmic}[1]
\Require Graph $G=\lb V,E\rb$
\Ensure Perfect Elimination Scheme $\sigma$ if $G$ is chordal
\State Assign label $\emptyset$ to each vertex in $V$
\For{$i=n$ \textbf{to} $1$ \textbf{step} $-1$}
\State Pick an unnumbered vertex $v$ with largest label
\State $\fkt{\sigma}{i}=v$
\For{each unnumbered vertex $w\in\nb{v}$}
\State Add $i$ to $\fkt{label}{w}$
\EndFor
\EndFor
\State\Return $\sigma$
\end{algorithmic}
\end{algorithm}
\begin{theorem}[Rose, Tarjan, Lueker. 1976 \cite{rose1976algorithmic}]\label{thm5.18}
The algorithm \textbf{LexBFS} returns a PES $\sigma$ if and only if the input graph $G$ is chordal. 
\textbf{LexBFS} has a running time of $\Ord{\abs{V}+\abs{E}}$.
\end{theorem}
Eight years later Tarjan and Yannakakis proposed a second algorithm with equal running time for the computation of a PES (see \cite{tarjan1984simple}). The so called maximum cardinality search algorithm lets go of sets as labels and now just counts the amount of already numbered vertices in the neighborhood of all vertices, thus reducing the required space of the algorithm.  
\begin{algorithm}[H]
\caption{Maximum Cardinality Search (\textbf{MCS})}\label{alg5.2}
\begin{algorithmic}[1]
\Require Graph $G=\lb V,E\rb$
\Ensure Perfect Elimination Scheme $\sigma$ if $G$ is chordal
\State Assign label $0$ to each vertex of $V$
\For{$i=n$ \textbf{to} $1$ \textbf{step} $-1$}
\State Pick an unnumbered vertex $v$ with highest label
\State $\fkt{\sigma}{i}=v$
\For{each unnumbered vertex $w\in\nb{v}$}
\State add $1$ to $\fkt{label}{w}$
\EndFor
\EndFor
\State \Return $\sigma$
\end{algorithmic}
\end{algorithm}
\begin{theorem}[Tarjan, Yannakakis. 1984 \cite{tarjan1984simple}]\label{thm5.19}
The algorithm \textbf{MCS} returns a PES $\sigma$ if and only if the input graph $G$ is chordal. 
\textbf{MCS} has a running time of $\Ord{\abs{V}+\abs{E}}$.
\end{theorem}
Both algorithms, \textbf{LexBFS} and \textbf{MCS} return an ordering $\sigma$ of the vertices of 
the input graph $G$. This ordering $\sigma$ is a perfect elimination scheme if and only if $G$ is 
chordal, but neither \textbf{LexBFS}, nor \textbf{MCS} actually check for their vertices to be 
simplicial in their corresponding induced subgraph. So neither of our two algorithms actually checks 
if $G$ is chordal. To do this, the simplest way would be to check the previously computed ordering 
$\sigma$ for being an actual PES of $G$. For this purpose another algorithm was introduced by 
the authors of the two algorithms above.
\begin{algorithm}[H]
\caption{\textbf{PERFECT}}\label{alg5.3}
\begin{algorithmic}[1]
\Require Graph $G=\lb V,E\rb$, ordering $\sigma$
\Ensure \textbf{true} if $\sigma$ is a PES, \textbf{false} otherwise
\For{all vertices $v\in V$}
\State $\fkt{A}{v}=\emptyset$
\EndFor
\For{$i=1$ \textbf{to} $n$ \textbf{step} $1$}
\State $v\define\fkt{\sigma}{i}$
\State $X\define\condset{x\in\nb{v}}{\fkt{\sigma^{-1}}{v}<\fkt{\sigma^{-1}}{x}}$
\If{$X=\emptyset$}
\State Go to 12.
\EndIf
\State $u\define \fkt{\sigma}{\min\condset{\fkt{\sigma^{-1}}{x}}{x\in X}}$
\State Concatenate $X-\set{u}$ to $\fkt{A}{u}$
\If{$\fkt{A}{v}\setminus\nb{v}\neq\emptyset$}
\algstore{myalg}
\end{algorithmic}
\end{algorithm}
\begin{algorithm}[H]
\addtocounter{algorithm}{-1}
\caption{\textbf{PERFECT} (ctd.)}
\begin{algorithmic}[1]
\algrestore{myalg}
\State\Return \textbf{false}
\EndIf
\EndFor
\State \Return \textbf{true}
\end{algorithmic}
\end{algorithm}
\begin{theorem}[Tarjan, Yannakakis. 1984 \cite{tarjan1984simple}]\label{thm5.20}
The algorithm \textbf{PERFECT} recognizes $\sigma$ correctly as a PES, if it is one, in time 
$\Ord{\abs{V}+\abs{E}}$.
\end{theorem}
Note that \textbf{PERFECT} can easily be extended to insert additional edges into the input graph 
$G$ in order to obtain a chordal supergraph of $G$, which can be used to give an easy to compute 
approximation for the so called {\em minimum fill-in} problem. This problem asks for the minimum 
number of edges that have to be added to a given graph $G$ in order to make this graph chordal. 
The treewidth of a graph modified in such a manner acts as an upper bound of the treewidth of 
$G$ itself. 
\begin{theorem}[Fulkerson, Gross. 1965 \cite{fulkerson1965incidence}]\label{thm5.21}
A chordal graph on $n$ vertices contains at most $n$ maximal cliques, where equality holds if and only if the graph does not contain any edges.
\end{theorem}
\begin{proof}
Let $G$ be a chordal graph and $\sigma$ a PES of $G$. For all vertices $v$ we define the higher order neighborhood of $v$ by
\begin{align*}
X_v\define \condset{x\in\nb{v}}{\fkt{\sigma^{-1}}{v}<\fkt{\sigma^{-1}}{x}}.
\end{align*}
Now let $C$ be a maximal clique and $v=\operatorname{argmin}_{x\in C}\fkt{\sigma^{-1}}{x}$ the vertex with the smallest index in $C$. Hence $C\subseteq X_v\cup\set{v}$.\\
With $X_v\cup\set{v}$ being a clique and $C$ maximal we get $C=X_v\cup\set{v}$, hence for each vertex there is at most one maximal clique $X_v\cup\set{v}$ and if $G$ does not contain any edge $X_v\cup\set{v}=\set{v}$ holds and the isolated vertices are the maximal cliques.\\
Otherwise w.l.o.g. $\set{v_n}$ does not form a maximal clique. We can chose $v_n$ arbitrarily and so we can chose it with $\nb{v_n}\neq\emptyset$. Then $v_n$ never is the vertex with smallest index in a maximal clique.
\end{proof}
So all maximal cliques of a chordal graph $G=\lb V,E\rb$ are contained in $\condset{X_v}{v\in V}$, which can be found by simply iterating over the vertices of $G$ in the order of the corresponding PES $\sigma$. This way it is easy to find a maximum clique in $G$.\\
What about a stable set, which would correspond to an induced matching in a chordal line graph square?\\
To find $\fkt{\alpha}{G}$ we define the following sequence of vertices:
\begin{enumerate}[i)]

\item $y_1=\fkt{\sigma}{1}$ and

\item $y_i=\fkt{\sigma}{\min\condset{j\leq n}{\fkt{\sigma}{j}\notin\bigcup_{k=1}^{i-1}X_{y_k}}}$

\end{enumerate}
until there is no vertex left. Hence in the end we find some $t\leq n$ with
\begin{align*}
\bigcup_{k=1}^{t}\lb\set{y_k}\cup X_{y_k}\rb=V.
\end{align*}
\begin{theorem}[Gavril. 1972 \cite{gavril1972algorithms}]\label{thm5.22}
The set $\set{y_1,\dots,y_t}$ is a maximum stable set.
\end{theorem}
\begin{proof}
With $y_iy_j\in E$ for $i<j\leq t$ $y_j\in X_{y_i}$ would hold, hence $\set{y_1,\dots,y_t}$ is a stable set and we obtain $\fkt{\alpha}{G}\geq t$.\\
Now suppose $t<\fkt{\alpha}{G}$. With $\bigcup_{k=1}^{t}\lb\set{y_k}\cup X_{y_k}\rb$ covering the entire vertex set of $G$, in an optimal stable set of $G$ two vertices $v,w$ must exist with $v,w\in X_{y_i}\cup\set{y_i}$ for some $i\leq t$. But $X_{y_i}\cup\set{y_i}$ is a clique and therefore $vw\in E$, contradicting $v$ and $w$ contained together in a stable set.
\end{proof}
Note that with this procedure we also found an optimal clique cover of $G$ using exactly $\fkt{\alpha}{G}$ cliques.\\
This finally leads us to an optimal and efficient coloring algorithm for chordal graphs, which was introduced by Gavril in 1972 (see \cite{gavril1972algorithms}) and can be used to compute the maximal cliques of $G$ as well.
\begin{algorithm}[H]
\caption{\textbf{COLORS}}\label{alg5.6}
\begin{algorithmic}[1]
\Require A graph $G=\lb V,E\rb$ and a PES $\sigma$
\Ensure The chromatic number $\fkt{\chi}{G}$
\State $\chi\define 1$
\For{all vertices $v\in V$}
\State $\fkt{S}{v}\define 0$
\EndFor
\For{$i=1$ \textbf{to} $n-1$ \textbf{step} $1$}
\State $v\define\fkt{\sigma}{i}$
\If{$\nb{v}=\emptyset$}
\State Go to 5.
\EndIf
\State$X\define\condset{x\in\nb{v}}{\fkt{\sigma^{-1}}{v}<\fkt{\sigma^{-1}}{x}}$
\If{$X=\emptyset$}
\State Go to 5.
\EndIf
\State $u\define\fkt{\sigma}{\max\condset{\fkt{\sigma^{-1}}{x}}{x\in X}}$
\State $\fkt{S}{u}\define\max\set{\fkt{S}{u},\abs{X}-1}$
\If{$\fkt{S}{v}<\abs{X}$}
\algstore{myalg}
\end{algorithmic}
\end{algorithm}
\begin{algorithm}[H]
\addtocounter{algorithm}{-1}
\caption{\textbf{COLORS} (ctd.)}
\begin{algorithmic}[1]
\algrestore{myalg}
\State $\chi\define\max\set{\chi,\abs{X}+1}$
\EndIf
\EndFor
\State\Return $\chi$
\end{algorithmic}
\end{algorithm}
\begin{theorem}[Gavril. 1972 \cite{gavril1972algorithms}]\label{thm5.23}
The Algorithm \textbf{COLORS} computes an optimal coloring of a chordal graph in time 
$\Ord{\abs{V}}$.
\end{theorem}
Now, as we are able to compute maximum cliques, stable sets and optimal colorings on chordal graphs, we need to apply those techniques to graph powers. For this we first need an efficient way to compute the $k$th power of a given graph.\\
First we introduce an algorithm to compute the $k$-strong neighborhood of a given vertex. This can also be done by using the established distance algorithms like Dijkstra's Algorithm. 
\begin{algorithm}[H]
\caption{Breadth Neighbor Search (\textbf{BNS})}\label{alg5.7}
\begin{algorithmic}[1]
\Require A connected graph $G=\lb V,E\rb$, a vertex $v\in V$ and an integer $k$ 
\Ensure The $k$-strong neighborhood $\knb{k}{v}$ of $v$
\If{$k\geq\abs{V}-1$}
\State\Return $V$
\EndIf
\State $N\define\nb{v}$
\State $Q'\define \emptyset$
\State $Q\define N$
\State $V\define V\setminus\lb N\cup\set{v}\rb$
\For{$i=1$ \textbf{to} $k-1$ \textbf{step} $1$}
\For{$u\in Q$}
\State $N\define N\cup\nb{u}$
\State $Q'\define\nb{u}\cap V$
\State $V\define V\setminus\nb{u}$
\EndFor
\State $Q\define Q'$
\State $Q'\define\emptyset$
\EndFor
\State\Return $N$
\end{algorithmic}
\end{algorithm}
\begin{lemma}
The algorithm \textbf{BNS} computes the $k$-strong neighborhood of a given vertex in time 
$\Ord{kn}\subseteq\Ord{n^2}$.
\end{lemma}
\begin{proof}
Let $G=\lb V,E\rb$ be a graph with $\abs{V}=n$. Each vertex appears at most once in the set $Q$, the loop in 8. takes exactly $k$ steps and the inner loop in 9. is finished in $\abs{Q}\leq n-1$ steps. So in total the two loops take time $\Ord{kn}$, with $k\leq n-2$ we obtain a total running time of $\Ord{n^2}$.\\
Now for the correctness of \textbf{BNS}. No vertex $u$ in $N$ can yield $\distg{G}{u}{v}\geq k+1$, 
since every vertex in $N$ is either a direct neighbor of $v$, or adjacent to another vertex $u'\in N$ 
with $\distg{G}{v}{u'}=i\leq k-1$ that has been considered by the algorithm.\\
So suppose there is some vertex $u\in V\setminus{N\cup\set{v}}$ with $\distg{G}{v}{u}\leq k$. then a path of length at most $k$ from $v$ to $u$ must exist and the neighbor of $u$ on this path has a distance of at most $k-1$ to $v$. Hence it must have appeared in the set $Q$ once, with that $u$ itself must be contained in $N$. 
\end{proof}
Please note that for the $2$-strong neighborhood of a vertex \textbf{BNS} just takes 
$\Ord{2n}=\Ord{n}$ time.\\
Now in order to compute the $k$-th power of a graph we simply compute the $k$-strong neighborhood for each vertex, which can be realized in time $\Ord{n^3}$, or $\Ord{n^2}$ if we are just looking for the square of a graph. Then we simply iterate over all vertices and their $k$-strong neighborhood, joining each vertex in this neighborhood to their currently considered $k$-strong neighbor. Again this takes $\Ord{n^2}$ time. Hence $G^k$ can be computed in time $\Ord{n^3}$ and for $k$ constant, e.g. $2$, we obtain the power in $\Ord{n^2}$.\\
Thus we obtain the following corollary.
\begin{corollary}\label{cor5.5}
For a graph $G$ and an integer $k\in\N$, the chordality of $G^k$ can be decided in time $\Ord{n^3}$.
\end{corollary} 
To take our investigation of the algorithmic properties of chordal graphs one step further, we will now have a deeper look into the relation of chordal graphs and tree decompositions, on which the next section will heavily rely. The existence of a clique related tree decomposition for chordal graphs is given by the characterization {\em iv)} of Theorem \autoref{thm5.15}, since the tree $T$ poses a tree decomposition of width $\fkt{\omega}{G}-1$ for a chordal graph $G$. Such trees are also known as so called clique trees and have been thoroughly studied in the past.\\
In order to find such a tree, we define the very helpful structure: the clique graph of a chordal graph, introduced in 1995 by Galinier, Habib and Paul (see \cite{GalinierHabibPaul1995cliquegraph}).
\begin{definition}[Clique Graph]
Let $G=\lb V,E\rb$ be a chordal graph. The {\em clique graph} of $G$, denoted by $\fkt{C}{G}=\lb V_c,E_c,\mu\rb$, with $\mu\colon E_c\rightarrow\N$, is defined as follows:
\begin{enumerate}[i)]

\item The vertex set $V_c$ is the set of maximal cliques of $G$.

\item The edge $C_iC_j$ belongs to $E_c$ if and only if $C_i\cap C_j$ is a minimal $a,b$-separator of $G$ for all $a\in C_i\setminus C_j$ and $b\in C_j\setminus C_i$.

\item The edges $C_iC_j\in E_c$ are weighted by the cardinality of the corresponding minimal separator $S_{ij}$: $\fkt{\mu}{C_iC_j}=\abs{S_{ij}}$.
\end{enumerate}
\end{definition}

\begin{definition}[Clique Tree]
Let $G=\lb V,E\rb$ be a chordal graph. A {\em clique tree} of $G$ is a tree $T_C=\lb C,F\rb$, such that $C$ is the set of all maximal cliques of $G$ and for each vertex $x\in V$ the set of maximal cliques containing $x$ induces a subtree of $T_C$.
\end{definition}

It is easy to see, that such a clique tree of a chordal graph $G$ yields an optimal tree decomposition of $G$, as we stated before. So if we are able to compute a clique tree, we also find a tree decomposition with minimal width. For this we need another basic concept of graph theory: the spanning tree.

\begin{definition}[Spanning Tree]
Let $G=\lb V,E\rb$ be a graph and $T=\lb V,F\rb$ a tree on the same vertex set. The tree $T$ is called a {\em spanning tree} of $G$ if $F\subseteq E$.\\
If some weight function $\mu\colon E\rightarrow\N$ is given together with $G$, a {\em maximum spanning tree} of $G$ is a spanning tree of $G$ with maximum weight of its edges. 
\end{definition}

\begin{lemma}[Rose. 1970 \cite{rose1970triangulated}]\label{lemma5.8}
Let $G=\lb V,E\rb$ be a chordal graph, each vertex $v\in V$ either is simplicial or belongs to a minimal separator.
\end{lemma}

\begin{theorem}[Galinier, Habib, Paul. 1995 \cite{GalinierHabibPaul1995cliquegraph}]\label{thm5.24}
Let $G=\lb V,E\rb$ be a chordal graph and $\fkt{C}{G}$ its clique graph. Let $T=\lb V_c,F\rb$ be a spanning tree of $\fkt{C}{G}$. Then $T$ is a clique tree of $G$ if and only if $T$ is a maximum spanning tree of $\fkt{C}{G}$.
\end{theorem}

\begin{proof}
Since $T$ and $\fkt{C}{G}$ have the same vertex set, it suffices to show that the set of maximal cliques containing a vertex $v\in V$ induces a subtree of $T$ if and only if $T$ is a maximum spanning tree of $\fkt{C}{G}$.\\
\\
So assume that $T$ is not a maximum spanning tree of $\fkt{C}{G}$. Then there exists a pair of maximal cliques, $C_i$ and $C_j$, adjacent in $\fkt{C}{G}$ and not in $T$ such that $S_{ij}$ strictly contains a minimal separator of the unique path $P_{ij}$ between $C_i$ and $C_j$ in $T$. Let $S_{kl}$ be such a minimal separator of $P_{ij}$.\\
Let $v\in S_{ij}\setminus S_{kl}$. So $v$ is not contained in at least one of the two maximal cliques $C_k$ and $C_l$. Hence, the set of maximal cliques containing the vertex $v$ does not induce a subtree of $T$ and so $T$ is not a clique tree of $G$.\\
\\
Now suppose $T$ to be a maximum spanning tree of $\fkt{C}{G}$. Let $C_i$ and $C_j$ be two non-adjacent maximal cliques of $T$ containing the vertex $v\in V$ and assume that $v$ is not contained in any maximal cliques along the unique path $P_{ij}$ between $C_i$ and $C_j$ in $T$.\\
Consider the subgraph $G'\subseteq G$ induced by the maximal cliques of $P_{ij}$. Obviously $G'$ is again a chordal graph.\\
Lemma \autoref{lemma5.8} yields that any vertex of $G'$ either is simplicial or belongs to a minimal separator.\\
Simplicial vertices belong to exactly one maximal clique. Therefore there exists a minimal separator $S$ in $G'$ with $v\in S$, since $v$ is contained in the two cliques $C_i$ and $C_j$ and in no other. In addition, due to this fact, $S=S_i\cap S_j$ must hold. Hence the edge $c_ic_j\in E_c$ induces a cycle together with the maximal cliques of $P_{ij}$.\\
Now let $x\in C_i\setminus S$ and $y\in C_j\setminus S$ be two vertices, then $S$ is a minimal $x,y$-separator. If no minimal separator on $P_{ij}$ is strictly included in $S$, we can easily extract a path in $G'$ between $x$ and $y$ which is not cut by $S$, since no such separator can contain $v$, a contradiction.\\
So at least one separator $S'$ on $P_{ij}$ must yield $S'\subsetneq S$ and so $T$ is not a maximal spanning tree.
\end{proof}

Actually the clique graph $\fkt{C}{G}$, which is a subgraph of the clique intersection graph of the chordal graph $G$ is optimal, in some sense, in its property of clique trees as maximum spanning trees.

\begin{theorem}[Galinier, Habib, Paul. 1995 \cite{GalinierHabibPaul1995cliquegraph}]\label{thm5.25}
Let $\fkt{C}{G}$ be the clique graph of the chordal graph $G$. Then each edge of $\fkt{C}{G}$ belongs to at least one clique tree.	
\end{theorem}

In the proof of this theorem it becomes clear, that an edge contained in the clique intersection graph of $G$ but not in $\fkt{C}{G}$ never belongs to a clique tree. Since every edge of a clique tree corresponds to a minimal separator, clique graphs are the minimal structures containing all clique trees.\\
In the following we will use the structure of clique graphs to proof that our already known 
algorithms \textbf{MCS} and \textbf{BFS} can be seen as special cases of a simple algorithm 
computing a maximum spanning tree on $\fkt{C}{G}$ and therefore themselves are algorithms that 
compute a clique tree of a chordal graph $G$. This again is due to the enormous algorithmic 
potential of PESs.

\begin{lemma}[Galinier, Habib, Paul. 1995 \cite{GalinierHabibPaul1995cliquegraph}]\label{lemma5.9}
Let $G=\lb V,E\rb$ be a chordal graph. In an execution of \textbf{MCS} or \textbf{BFS} on $G$, 
maximal cliques are visited consecutively.	
\end{lemma}

\begin{proof}
Let $\sigma$ be the PES computed by \textbf{MCS} and let $N=\set{v_i,\dots,v_n}$ be the set of 
numbered vertices at some step. Then $v_i$ is simplicial in $\induz{G}{\set{v_i,\dots,v_n}}$, and so 
$v_i$ belongs to a unique maximal clique as we have seen before. Let us prove the $v_{i-1}$ 
belongs to a new maximal clique if and only if  $\fkt{label}{v_{i-1}}\leq\fkt{label}{v_i}$, where $label$ 
is the labeling function used in \textbf{MCS}.\\
\\
First assume that $v_{i-1}$ and $v_i$ belong to the same maximal clique. Since $x_i$ is simplicial in $\induz{G}{\set{v_i,\dots,v_n}}$, all the vertices in $\fkt{L}{v_i}=\condset{v_j\in V}{v_j\in\nb{v_i}~\text{and}~j>i}$ belong to this maximal clique. Hence $\fkt{label}{v_{i-1}}=\fkt{label}{v_i}+1$.\\
\\
Now suppose $v_{i-1}$ to belong to a new maximal clique. Since $v_i$ was a vertex with the highest label over all unnumbered vertices when it was chosen, $\fkt{label}{v_i}\geq\abs{\fkt{L}{v_{i-1}}\setminus\set{v_i}}$. But there exists at least one vertex of the maximal clique containing $v_i$ in $\induz{G}{\set{v_i,\dots,v_n}}$ which does not increase the label of $v_{i-1}$, otherwise $x_{i-1}$ does not belong to a new maximal clique. Hence $\fkt{label}{v_{i-1}}\leq\fkt{label}{v_i}$.\\
\\
If we define the trace of $\sigma$ as the sequence of label levels of the vertices when the $y$ are 
numbered, each maximal clique is represented by an increasing sequence. We can conclude that 
the maximal cliques are visited consecutively. A similar argument holds for \textbf{BFS}. 	
\end{proof}
\begin{algorithm}[H]
	\caption{Spanning Clique Graph Tree (\textbf{SCGT}) (see 
	\cite{GalinierHabibPaul1995cliquegraph})}\label{alg5.8}
	\begin{algorithmic}[1]
		\Require A clique graph $\fkt{C}{G}$ 
		\Ensure A maximum spanning tree of $\fkt{C}{G}$
		\State Choose a maximal clique $C_1$
		\For{$i=2$ \textbf{to} n \textbf{step} $1$}
		\State Choose a maximally (w.r.t. inclusion) labeled edge adjacent to $C_1,\dots,C_{i-1}$ to connect the new clique $C_i$
		\EndFor
		\State\Return The constructed tree
	\end{algorithmic}
\end{algorithm}
\begin{lemma}[Galinier, Habib, Paul. 1995 \cite{GalinierHabibPaul1995cliquegraph}]\label{lemma5.10}
Let $G$ be a chordal graph. The algorithm \textbf{SCGT} computes a maximum spanning tree of 
the clique graph $\fkt{C}{G}$. 	
\end{lemma}
\begin{lemma}[Galinier, Habib, Paul. 1995 \cite{GalinierHabibPaul1995cliquegraph}]\label{lemma5.11}
\textbf{MCS} and \textbf{BFS} compute maximum spanning trees of $\fkt{C}{G}$.	
\end{lemma}
\begin{proof}
It suffices to show that \textbf{MCS} and \textbf{BFS} are special cases of \textbf{SCGT}.\\
Let $C_1\dots,C_{k-1}$ be the visited cliques and $v_{i-1},\dots,v_n$ be the numbered vertices. 
When $v_i$ is numbered, a new clique $C_k$ is visited. Since $v_i$ is simplicial in 
$\induz{G}{\set{v_i,\dots,v_n}}$, its neighborhood is complete and strictly included in a clique $C_j$, 
$1\geq j\geq k-1$. Therefore the label of all edges connecting $C_k$ to $C_1,\dots,C_{k-1}$ must 
be included in the label of $\lb C_i,C_j\rb$ in the clique graph $\fkt{C}{G}$. In order to connect 
$C_k$ the maximum edge is chosen. This is what \textbf{SCGT} does, since a maximum edge has 
necessarily a maximal label. Similarily for \textbf{BFS} the lexicographic search ensures also a 
maximal labeled edge to be chosen at each step.
\end{proof}
So with the help of the recognition algorithms for chordal graphs we are easily able to compute an 
optimal tree decomposition of a given chordal graph $G$. Galinier, Habib and Paul also introduced 
a version of \textbf{MCS} that is slightly modified to compute a clique tree of the input graph.\\
With this algorithm and the corresponding proof of its correctness we will conclude this section and go back to chordal graph powers in the next one.
\begin{lemma}[Galinier, Habib, Paul. 1995 \cite{GalinierHabibPaul1995cliquegraph}]\label{lemma5.12}
Let $G$ be a chordal graph. Let $T$ be a tree whose nodes are the maximal cliques of $G$ and such that an edge between the maximal cliques $C_j$ and $C_k$, with $k<j$, exists if and only if the last vertex $v_j$ with $C_J\subseteq\nb{v_j}\cup\set{v_j}$ is contained in $C_k$.	
\end{lemma}
\begin{algorithm}[H]
	\caption{Maximum Cardinality Search and Clique Tree (\textbf{MCCT}) (see 
	\cite{GalinierHabibPaul1995cliquegraph})}\label{alg5.9}
	\begin{algorithmic}[1]
	\Require A graph $G=\lb V,E\rb$
	\Ensure If the input graph is chordal: a PES $\sigma$ and an associated clique tree $T=\lb I,F\rb$ where $I$ is the set of maximal cliques
	\For{$v\in V$}
	\State $\fkt{X}{v}\define\emptyset$
	\State $\fkt{last}{v}\define\emptyset$
	\State $\fkt{C}{v}\define0$
	\EndFor
	\State $previousmark\define-1$
	\State $j\define0$
	\For{$i=n$ \textbf{to} $1$ \textbf{step} $-1$}
	\State Choose a vertex $v$ not yet numbered such that $\abs{\fkt{X}{v}}$ is maximal
	\If{$\abs{\fkt{X}{v}}\leq previousmark$}
	\State $j=j+1$
	\State Create the maximal clique $C_j=\fkt{X}{v}\cup\set{v}$
	\State Join $C_j$ and $C_{C\lb last\lb v\rb\rb}$ in $T$
	\Else
	\State $C_j\define C_j\cup\set{v}$
	\EndIf
	\For{$w\in\nb{v}$, $w$ not yet numbered}
	\State $\fkt{L}{w}\define\fkt{X}{w}\cup\set{v}$
	\State $\fkt{last}{w}\define v$
	\EndFor
	\State $previousmark\define\abs{\fkt{X}{v}}$
	\State $v$ is numbered by $i$
	\State $\fkt{C}{v}\define j$
	\EndFor
	\State\Return The PES $\sigma$, the set $I=\set{C_0,\dots,C_k}$ and the edges created between the $C_j$
	\end{algorithmic}
\end{algorithm}
\begin{theorem}[Galinier, Habib, Paul. 1995 \cite{GalinierHabibPaul1995cliquegraph}]\label{thm5.26}
The algorithm \textbf{MCCT} computes a PES and its associated clique tree in time 
$\Ord{\abs{V}+\abs{E}}$ if and only if the input graph $G=\lb V,E\rb$ is chordal. 	
\end{theorem}
\begin{proof}
If \textbf{MCCT} computes a PES and a clique tree, then the input graph $G$ is trivially chordal. 
\textbf{MCCT} is an extended version of \textbf{MCS}, hence if $G$ is chordal, the computed 
elimination scheme is perfect. By Lemma \autoref{lemma5.11}, the step loop at 8. builds a clique 
tree of $G$ is chordal.\\
Now we investigate the size of a clique tree- First of all, when a vertex is labeled by \textbf{MCS}, 
this label corresponds to an edge. Hence the size of all label sets is $\Ord{m}$, the number of 
edges in $G$. Since there are at most $n=\abs{V}$ increasing sequences, $T$ contains at most 
$n$ vertices (see Theorem \autoref{thm5.21}) and so $\Ord{n}$ edges.\\
Let us examine the size of the set vertices, or bags, in $T$. Since the vertices of minimal 
separators in chordal graphs belong to several maximal cliques, the size of the vertex sets is bigger 
than $n$. Let $\fkt{X}{v}$ be the set of labels of $v$, where $v$ is the first vertex of an increasing 
sequence $L_j$. The bag associated to $v$ by \textbf{MCCT} is $\fkt{X}{v}\cup L_j$. But 
$\fkt{X}{v}$ corresponds to the edges between vertices of $\fkt{X}{v}$ and $v$. Since every 
duplicated vertex, each vertex appearing in multiple bags, belongs to the label set of the first vertex 
of an increasing sequence, each of them can be associated an edge. Therefore the sum of the 
cardinality of the bags of $T$ is smaller than $n+m$, and the space complexity is $\Ord{n+m}$.\\
All the operations of the loops at 8. and 17., as well as 10. can be done in constant time. Hence this 
algorithm has the same time complexity as \textbf{MCS}: $\Ord{n+m}$. 
\end{proof}

\subsection{Chordal Powers and Complexity}
As we have already seen in previous chapters, for any graph $G$ there is at least one integer $k\leq\abs{\V{G}}$ for which $G^k$ becomes a chordal graph. So for any $K\in\N$ there is a class of graphs for which the $k$-strong coloring problem can be solved in polynomial time.\\
One could conjecture that these classes form a strictly increasing chain in terms of subsets, but this seems unlikely since the powers of chordal graphs themselves are not necessarily chordal anymore. In fact for some powers of chordal graphs the coloring problem becomes $\mathcal{NP}$-hard again.\\
Another observation we made was that certain parameters of a chordal power of a graph can be used as upper bounds on the parameters of the original graph itself, which might even yield some nice algorithmic results. So the following holds for any graph $G$ and its power $G^k$ and any lower power $1\leq i<k$:
\begin{align*}
\tw{G^i}\leq\tw{G^k}.
\end{align*}
If $G^k$ is chordal we are able to actually compute a tree decomposition of $G^k$ in polynomial 
time using the algorithm \textbf{MCCT}, which also gives us the maximal cliques of the chordal 
power $G^k$. For cliques again we obtained a similar bound:
\begin{align*}
\fkt{\omega}{G^i}\leq\fkt{\omega}{G^k}.
\end{align*}
And of course, since we assume $G^k$ to be chordal we gain the equality $\tw{G^k}=\fkt{\omega}{G^k}$. In terms of parameterized complexity this means if we have an upper bound on the size of a maximum clique in a certain chordal power $G^k$ of a graph $G$, every problem that can be parameterized by the maximum size of a clique or the treewidth of $G$ can be solved on $G$ in polynomial time.\\
In order to obtain better control over the appearance of chordal powers we will limit ourselves to the smallest integer for which a graph becomes chordal.
\vspace{-2mm}
\begin{definition}[Power of Chordality]
Let $G$ be a graph. The smallest integer $k\in N$ for which $G^{k}$ becomes chordal is called the {\em power of chordality} of $G$ and we write $k=\fkt{k_0}{G}$.	
\end{definition}
\vspace{-2mm}
For any graph a trivial upper bound on the power of chordality is given by the length of a longest shortest path in $G$, the so called diameter, since no two vertices of $G$, that are part of the same component, can be further apart from each other, hence $G$ to the power of its diameter becomes a graph whose components are complete graphs.
\begin{definition}[Diameter]
Let $G$ be a connected graph. The {\em diameter} of $G$ is the length of a longest shortest path, or the greatest possible distance between two vertices of $G$.
\vspace{-1mm}
\begin{align*}
\fkt{\operatorname{dm}}{G}=\max_{x,y\in\V{G}}\distg{G}{x}{y}.
\end{align*}
\vspace{-1mm}
If $G$ is not connected, the diameter of $G$ is the maximum over all diameters of the components of $G$.
\end{definition}
\vspace{-2mm}
\begin{lemma}\label{lemma5.13}
Let $G$ be a graph, then $\fkt{k_0}{G}\leq\fkt{\operatorname{dm}}{G}$.
\end{lemma}
With Duchet's Theorem (Theorem \autoref{thm3.12}) we can conclude the chordality of $G^i$ with $i-\fkt{k_0}{G}\equiv 0\lb\!\!\!\mod 2\rb$ and $i\leq\fkt{\operatorname{dm}}{G}$, or, since they consist exclusively of complete components, for $i\geq \fkt{\operatorname{dm}}{G}$.\\
\\
Now, before we start to investigate the actual algorithmic power of the power of chordality, how hard is it to find $\fkt{k_0}{G}$ for a given graph $G$?\\
This question is easy to answer. We can formulate the following simple algorithm.
\begin{algorithm}[H]
	\caption{Power of Chordality (\textbf{POC})}\label{alg5.10}
	\begin{algorithmic}[1]
		\Require A graph $G=\lb V,E\rb$
		\Ensure The power of chordality $\fkt{k_0}{G}$ and the chordal power $G^{k_0\lb G\rb}$
		\For{$k=1$ \textbf{to} $n-1$ \textbf{step} $1$}
		\State Compute $G^k$
		\If{$G^k$ is chordal}
		\State\Return the power of chordality $k$ and $G^k$
		\EndIf
		\EndFor
	\end{algorithmic}
\end{algorithm}
\begin{theorem}\label{thm5.27}
The algorithm computes the power of chordality $\fkt{k_0}{G}$ of the input graph $G$ in time $\Ord{n^4}$.
\end{theorem}
\begin{proof}
Corollary \autoref{cor5.5} yields that the recognition of a chordal power can be realized in time $\Ord{n^3}$ where $n$ is the number of vertices of $G$, so it would suffice to just iterate over all powers of $G$ until the first chordal power appears. Since $\fkt{\operatorname{dm}}{G}\leq\abs{\V{G}}-1$ this will terminate after $\Ord{n}$ trials.
\end{proof}

\subsection{Cliques in Powers of Graphs}

Some basic observations regarding the growth of certain parameters, such as the chromatic number, can be made for increasing powers, which could be very helpful in order to give a good approximation of the according parameter on non-chordal powers.\\
So are the clique number and the chromatic number monotonous increasing in the number of edges of a graph. For increasing graph powers the following relation holds: $\E{G^k}\subseteq\E{G^{k+1}}$ and with this we obtain the following:
\begin{align*}
&\fkt{\chi}{G^k}\leq\fkt{\chi}{G^{k+1}}\leq\fkt{\chi}{G^{k+2}}\\
\text{and}~&\fkt{\omega}{G^k}\leq\fkt{\omega}{G^{k+1}}\leq\fkt{\omega}{G^{k+2}}.\\
\end{align*}
For $G$ connected and $G^{k+1}$ not complete the above inequalities are strict. In a graph $G$ there must exist a pair of vertices $x,y\in\V{G}$ for any distance $1\leq i\leq\fkt{\operatorname{dm}}{G}$ with $\distg{G}{x}{y}=i$.\\
These observations are equivalent to the basic properties for the $k$-strong clique number and chromatic number mentioned in Chapter 2.\\
This raises the question if there can be a $k<\fkt{\operatorname{dm}}{G}$ with $G^k$ exclusively consisting of complete components.
\begin{theorem}\label{thm5.28}
Let $G$ be a graph. For any $k<\fkt{\operatorname{dm}}{G}$ there is a component of $G^k$ which is not complete.	
\end{theorem}
\begin{proof}
It suffices to assume $G$ to be connected, since in any non-connected graph there is at least one component realizing its diameter. So let $G$ be a connected graph with diameter $\fkt{\operatorname{dm}}{G}=k$. Let $x,y\in\V{G}$ be two vertices of $G$ realizing the diameter, so $\distg{G}{x}{y}=k$, at least one pair of such vertices must exist, otherwise $k$ would not be the diameter.\\
Now, for any $k'<k$ the vertices $x$ and $y$ cannot be adjacent in $G^{k'}$ and thus the graph $G^{k'}$ is complete if and only if $\fkt{\operatorname{dm}}{G}\leq k'$.
\end{proof}
By revisiting Chapter $3$ we can give a characterization of graphs with a complete square similar to the characterization of graphs with chordal squares.
\begin{theorem}\label{thm5.29}
Let $G$ be a connected graph, the following properties are equivalent:
\begin{enumerate}[i)]
	\item $G^2$ is complete
	
	\item $\fkt{\operatorname{dm}}{G}\leq 2$
	
	\item For any induced path of length $3$, $P_3\subseteq G$, with endpoints $x,y\in\V{G}$ there exists an additional vertex $v\in\V{G}\setminus\V{P_3}$ with $x,y\in\nb{v}$.
\end{enumerate}	
\end{theorem}
\begin{proof}
With Theorem \autoref{thm5.28} the equivalence of $i)$ and $ii)$ is already proven. So we now will show the equivalence of $ii)$ and $iii)$.\\
\\
First assume $\fkt{\operatorname{dm}}{G}\leq 2$. Then, if $G$ contains an induced path of length $3$, $P_3$, with endpoints $x,y\in P_3$ we get $\distg{G}{x}{y}\leq 2$. Since $P_3$ is induced $x$ and $y$ cannot be adjacent and so $\distg{G}{x}{y}=2$. Hence a vertex $v\in\V{G}$ must exist with $x,y\in\nb{v}$ and again, since $P_3$ is induced, $v$ cannot be contained in $\V{P_3}$.\\
\\
So suppose such a vertex $v$ exists for every induced path of length $3$ and $\fkt{\operatorname{dm}}{G}=k\geq 3$. Then there must exist a pair of vertices $x,y\in\V{G}$ with $\distg{G}{x}{y}=3$ and thus two vertices $u,w$ must exist that form a path of length $3$ together with $x$ and $y$ which is induced. By assumption $x$ and $y$ must be adjacent to a common vertex $v$ and therefore $\distg{G}{x}{y}\leq 2$, a contradiction.
\end{proof}
For induced cycles of length $\geq6$ property $iii)$ is equivalent to the condition of a withered flower and so we obtain the following corollary by applying similar methods.
\begin{corollary}\label{cor5.6}
Let $G$ be graph. If $G$ contains an induced cycle of length $n\geq6$, $G^2$ is not complete.	
\end{corollary}
Now we return to the power of chordality. Let $G$ be a graph with $\fkt{k_0}{G}=k$, then, by Duchet's Theorem, $G^{k+2}$ is chordal too. Therefore the parameters $\fkt{\chi}{G^k}$ and $\fkt{\chi}{G^{k+2}}$ can be computed in polynomial time. What about $\fkt{\chi}{G^{k+1}}$?\\
With $\fkt{\chi}{G^k}$ and $\fkt{\chi}{G^{k+2}}$ we already found a solid lower bound and an upper bound as well as a tree decomposition of $G^{k+2}$ which poses as a, not necessarily optimal, tree decomposition of $G^{k+1}$. Is it possible to use all this information to compute the chromatic number of $G^{k+1}$, or at least find a decent approximation of it?\\
For graphs with $\fkt{k_0}{G}=1$, i.e. chordal graphs, Agnarsson, Greenlaw and Halld{\'o}rsson found a somewhat disappointing answer to the second question.
\begin{theorem}[ Agnarsson, Greenlaw, Halld{\'o}rsson.2000 \cite{agnarsson2000powers}]\label{thm5.30}
Coloring even powers of chordal graphs is $\mathcal{NP}$-hard to approximate within $n^{\frac{1}{2}-\epsilon}$ for any $\epsilon>0$.	
\end{theorem}
So in general it is $\mathcal{NP}$-hard to approximate the chromatic number of $G^{k_0\lb G\rb+1}$ with a factor better than $\sqrt{n}$ where $n$ is the number of vertices of $G$. In addition, Theorem \autoref{thm5.11} showed us that treewidth is not exactly the best parameter for computing the chromatic number. So we will restrict ourselves to the computation of the maximal cliques of a non-chordal power of $G$ in order to find a better lower bound on its chromatic number using the following relation.
\begin{lemma}\label{lemma5.14}
Let $G$ be a connected graph and $k=\fkt{k_0}{G}+2j$ with $j\in\N$ and $k\leq\fkt{\operatorname{dm}}{G}$ and $G^{k-1}$ not chordal. The following inequalities hold.
\begin{align*}
\fkt{\omega}{G^{k-2}}=\fkt{\chi}{G^{k-2}}<\fkt{\omega}{G^{k-1}}\leq\fkt{\chi}{G^{k-1}}\leq\fkt{\omega}{G^k}=\fkt{\chi}{G^k}
\end{align*}	
\end{lemma}
Note that with $G$ being connected the inequality between $\fkt{\omega}{G^{k-2}}$ and $\fkt{\omega}{G^{k-1}}$ is strict, so the parameter at least increases by $1$ and thus the clique number of $G^{k-1}$ actually is an improved lower bound. So how to compute it?\\
For chordal graphs we can find a PES such that all maximal cliques of the graph are visited consecutively. And, if the chordal graph we are investigating is a power $G^k$ of the graph $G$ we are interested in, a maximal clique of $G$ is contained in at least one maximal clique of $G^k$.\\
Let $\sigma=\left[ v_1,\dots,v_n\right]$ be a PES of $G^k$, then $v_1$ is contained in exactly one maximal $k$-strong clique, namely $X_1\define\knb{k}{v_1}\cup\set{v_1}$ and so, in order to find all maximal cliques of $G$ that contain $v_1$ we just have to search in $X_1$. The same holds for any $v_i$ and the corresponding graph $\induz{G^k}{\set{v_1,\dots,v_n}}$. For any $v_i$ there is exactly one maximal clique of $G^k$ that has to be searched for maximal cliques in $G$ containing $v_i$, in order to find, together with the former found cliques for $v_1,\dots,v_{i-1}$, all maximal cliques in $G$ containing $v_i$.\\
This not only gives us a $\FPT$ algorithm for the maximum clique problem with parameter $\fkt{\omega}{G^k}=\kcl{k}{G}$, but also a nice upper bound on the number of maximal cliques of non chordal powers of $G$. For this we will need another famous result: Sperner's Theorem.
\begin{theorem}[Sperner. 1928 \cite{sperner1928satz}]\label{thm5.31}
Let $N$ be a set with $\abs{N}=n$ and $X$ a family of subsets of $N$ such that
\begin{align*}
U,V\in X\Rightarrow U\subsetneq V~\text{and}~V\subsetneq U,
\end{align*}
then
\begin{align*}
\abs{X}\leq\binom{n}{\abr{\frac{n}{2}}}.
\end{align*}	
\end{theorem}
\begin{lemma}
Let $G$ be a graph with $\abs{\V{G}}=n$, $k\in\N$ such that $G^k$ is chordal, $\kcl{k}{G}=\omega_k$ and let $k'<k$, then $G^{k'}$ has at most $n\binom{\omega_k}{\abr{\frac{\omega_k}{2}}}$ maximal cliques.
\end{lemma}
\begin{proof}
Let $\sigma=\left[ v_1,\dots,v_n\right]$ be a PES of $G^k$, then for any $i\in\set{1,\dots,n}$ all maximal cliques containing $v_i$, but none of the vertices $v_1,\dots,v_{i-1}$ of $G^{k'}$, are contained in the set $X_i\define\lb\knb{k}{v_i}\cap\set{v_{i+1},\dots,v_n}\rb\cup\set{v_i}$, which forms a clique in $G^k$.\\
So in order to find all those maximal cliques of $G^{k'}$ we have to check at most $2^{\abs{X_i}}\leq 2^{\omega_k}$ subsets of $X_i$. Since no maximal clique can be contained in another one Sperner's Theorem yields the existence of at most $\binom{\omega_k}{\abr{\frac{\omega_k}{2}}}$ such maximal cliques in $X_i$ and since we have to search $X_i$ for all $i\in\set{1,\dots,n}$ we obtain a maximum of $n\binom{\omega_k}{\abr{\frac{\omega_k}{2}}}$ such cliques.
\end{proof}
The proposed procedure induces a parameterized algorithm for the maximum $k'$-strong clique problem with parameter $\kcl{k}{G}$, where $G^k$ is a chordal power of $G$ and $k'\leq k$. Such an algorithm with $G$ and $k'$ as the input would have a running time of
\begin{align*}
\Ord{n^4+n\,2^{\omega_k}},
\end{align*}
where the $n^4$ corresponds to the time it takes to compute $G^{k'}$, $G^{k}$ and the verification of the chordality of $G^k$.

Duchet's Theorem implies that there are chordal graphs $G$ with some $k\in\N$ and $G^k$ not being chordal.\\
Is there any relation between the structure of maximal Cliques of some non chordal power of a chordal graph and it's ordinary cliques?\\
To further investigate this question we start by some very general properties of stronger cliques in graphs.

\begin{lemma}\label{lemma3.10}
Let $G$ be a graph and $C\subseteq\V{G}$ a maximal clique in $G$ and $X_1,\dots X_q$ being all maximal cliques in $G$ with $C\neq X_i$ and $C\cap X_i\neq\emptyset$ for all $i\in\set{1,\dots,q}$. Then $C\cup X_i$ is a $2$-strong clique in $G$ and $X\define C\cup \bigcup_{i=1}^qX_i$ a $3$-strong clique.\\
In addition $\induz{G}{C\cup X_i}^2$ and $\induz{G}{X}^3$ are complete.
\end{lemma}

\begin{proof}
Let $i\in\set{1,\dots,q}$, $v\in C$ and $x\in X_i$. With $C\cap X_i\neq\emptyset$ there exists some $w\in C\cap X_i$ with either $v=w$ or $vw\in\E{G}$, hence $\distg{G}{v}{w}\leq 1$ and so for arbitrary pairs of vertices $c\in C$ and $x\in X$ it holds $\distg{G}{c}{x}\leq 2$. Therefore $C\cup X_i$ is a $2$-strong clique. For each of those pairs $c$ and $x$ at least one path of length $\leq 2$ uses a vertex in $C\cap X_i$ thus $\distg{\induz{G}{C\cup X_i}}{c}{x}\leq 2$ and $\induz{G}{C\cup X_i}^2$ is  complete.\\
Adding some additional $X_j$ with $i\neq j$, we gain another vertex $u\in C\cap X_j$ we get $uw\in\E{G}$ with $w$ being chosen as above. Thus $\distg{G}{x}{y}\leq 3$ holds for all $x\in X_i$ and $y\in X_j$ and so $\distg{G}{a}{b}\leq 3$ for all pairs of vertices $a,b\in X=C\cup\bigcup_{i=1}^qX_i$, i.e. for all such pairs a path of at most length $3$ is completely contained in $\induz{G}{X}$ hence $X$ is a $3$-strong clique in $G$ and $\induz{G}{X}^3$ is complete.
\end{proof} 

\begin{remark}\label{cor3.7}
Let $G$ be a graph and $\mathscr{C}$ an intersecting family of maximal cliques in $G$, then $X\define\bigcup_{C\in\mathscr{C}}C$ is a $2$-strong clique in $G$.
\end{remark}

For $2$-strong cliques in chordal graphs Corollary \autoref{cor3.7} becomes an exact description. But first we need some helpful results.

\begin{lemma}\label{lemma3.11}
Let $G$ be a chordal graph and $C$ and $W$ two disjoint cliques in $G$ such that $C\cup W$ is a $2$-strong clique in $G$. Then $C$ or $W$ is not maximal.
\end{lemma} 

\begin{proof}
We begin by showing that it is sufficient to reduce the proof to the case that for all $c\in C$ there exists a $w\in W$ with $w\in\nb{c}$.\\
Suppose there is some $c\in C$ with $\nb{c}\cap W=\emptyset$, then there is a $u_{cw}\in\V{G}$ with $cu_{cw}$ and $u_{cw}w$ being edges in $G$ for every $w\in W$. Furthermore suppose the set $U$, containing all these $u_{cw}$ does not induce a clique, i.e. $u_{cw_1}\neq u_{cw_2}$ are not adjacent. There are two cases to be investigated.\\
Case 1: $w_1=w_2$\\
Then $w_1u_{cw_1}cu_{cw_2}$ is a circle of length $4$ and has to contain a chord. The vertex $c$ cannot be adjacent to $w_1$ and thus $u_{cw_1}u_{cw_2}$ is left as the only possible chord..\\
Case 2: $w_1\neq w_2$\\
Now $w_2w_1u_{cw_1}cu_{cw_2}$ is a circle of length $5$ with only two possible chords, namely $w_2u_{cw_1}$ and $w_1u_{cw_2}$. The existence of one or both of this chords would result in a circle of length $4$ similar to the first case and therefore $u_{cw_1}$ and $u_{cw_2}$ must be adjacent.\\
Thus $U$ is a clique with $\distg{G}{u}{w}\leq 2$ for all $u\in U$ and $w\in W$ and therefore by Lemma \autoref{lemma3.10} $U\cup W$ is a $2$-strong clique and every $w\in W$ is adjacent to a $u\in U$.\\
Therefore it suffices to show that the case in which every $c\in C$ is adjacent to a $w\in W$ results in $C$ not being maximal.\\
Claim: For all $i\in\set{1,\dots,\abs{C}}$ for every subset $C_i\subseteq C$ with $\abs{C_i}=i$ exists a $w\in W$ with $C_i\subseteq\nb{w}$. The case $i=1$ obviously complies. So consider $C_i\subseteq C$ with $\abs{C_i}=i$, for $C_{i-1}\subseteq C_i$ with $\abs{C_{i-1}}=i-1$ and $C_i\setminus C_{i-1}=\set{c}$. By induction some $w_{i-1}\in W$ exists with $C_{i-1}\subseteq\nb{w_{i-1}}$, furthermore there exists a $w_i\in W$ with $w_i\in\nb{c}$.\\
If $w_{i-1}=w_i$ we are already done, so suppose $c$ is not adjacent to $w_{i-1}$. Let $C_{i-1}=\set{c_1,\dots,c_{i-1}}$, for all $1\leq j\leq i-1$ the edges $c_jw_{i-1}$, $w_{i-1}w_i$, $w_ic$ and $cc_j$ exist and furthermore the edge $cw_{i-1}$ does not. Thus the only possible chord in those circles is the edge $c_jw_i$. Hence $C_i\subseteq\nb{w_i}$ and we are done.   
\end{proof}

\begin{lemma}\label{lemma3.12}
Let $G$ be a chordal graph and $X\subseteq\V{G}$ a maximal $2$-strong clique in $G$, then $\induz{G}{X}^2$ is complete.
\end{lemma}

\begin{proof}
Suppose $\induz{G}{X}^2$ is not complete, then two nonadjacent vertices $x,y\in X$ exist with $\distg{\induz{G}{X}}{x}{y}\geq 3$ and $\distg{G}{x}{y}\leq 2$. Hence a vertex $v\in\V{G}\setminus X$ exists such that $xvy$ is a path of length $2$ in $G$. Let $z\in X\setminus\set{x,y}$ so $\distg{G}{z}{x}\leq 2$ and $\distg{G}{z}{y}\leq 2$. The following three cases can occur and are partially divided into subcases.
\begin{enumerate}
\item[] \begin{enumerate}
		\item[Case 1] $z\in\nb{x}\cap\nb{y}$
		
		\item[Case 2] $z\in\nb{x}$ and $yz\notin\E{G}$ or $z\in\nb{y}$ and $xz\notin\E{G}$
				
		\item[Case 3] $xz, yz\notin\E{G}$

		\end{enumerate}
\end{enumerate}
Case 1: $z\in\nb{x}\cap\nb{y}$\\
This case can not occur because $xzy$ would be a path of length $2$ between $x$ and $y$ in $X$.\\

Case 2: $z\in\nb{x}$ and $yz\notin\E{G}$ or $z\in\nb{y}$ and $xz\notin\E{G}$\\
W.l.o.g. we consider $z\in\nb{x}$ and $yz\notin\E{G}$, then a $u\in\V{G}\setminus\set{x,y}$ exists such that $yuz$ is a path of length $2$. Three subcases have to be considered as follows.\\
Subcase 2.1: $u=v$\\
Then $z$ is adjacent to $u$ and we are done.\\
Subcase 2.2: $u\neq v$ and $u\notin X$\\
Now the edges $xy$ and $yz$ must not exist, $xu$ does not exist as well, otherwise we would be in Subcase 2.1 again. Two possible chords remain for the resulting cycle of length $5$ are $vu$ and $vz$, both must exist and again we get $\distg{G}{v}{z}\leq 2$.\\
Subcase 2.3: $u\in X$\\
Again we obtain a $5$-cycle with the necessary chords $vu$ and $vz$ and case 2 is closed.\\

Case 3: $xz, yz\notin\E{G}$\\
Now $u\in\nb{x}\cap\nb{z}$ and $w\in\nb{y}\cap\nb{z}$ need to exist with $u\neq w$. For $u$ and $w$ it holds $\distg{G}{v}{u}\leq 2$ and $\distg{G}{v}{w}\leq 2$. Again three subcases are to be considered.\\
Subcase 3.1: w.l.o.g. $u=v$\\
As we have seen in Subcase 2.1 $z$ is now adjacent to $v$ and we are done.\\
Subcase 3.2: $u$ or $w$ or $u$ and $w$ are not in $X$.\\
The edges $yu$ and $xw$ would result in Subcase 3.1 and therefore do not exist here, furthermore $xz$ and $yz$ cannot exist either, otherwise we would be in Case 2. Hence the edges $vu$, $vw$ and $vz$ exist and $z$ is adjacent to $v$.\\
Subcase 3.3: $u,w\in X$\\
Again all chords are forbidden except for $vu$, $vw$ and $vz$.\\

Hence for all vertices $x\in X$ we obtain $\distg{G}{x}{v}\leq 2$ and therefore $X\cup\set{v}$ is a $2$-strong clique, contradicting the maximality of $X$. Thus $\induz{G}{X}^2$ is complete. 
\end{proof}

\begin{theorem}\label{thm3.21}
Let $G$ be a chordal graph and $X\subseteq\V{G}$. Then $X$ is a maximal $2$-strong clique in $G$ if and only if the family of maximal cliques $C_X$ in $\induz{G}{X}$ is a maximal intersecting family of cliques in $G$.
\end{theorem}

\begin{proof}
Suppose there exist $C_1, C_2\in C_X$ with $C_1\cap C_2=\emptyset$. By Lemma \autoref{lemma3.12} $\induz{G}{X}^2$ is complete and thus $C_1\cup C_2$ is a $2$-strong clique in $\induz{G}{X}$. Lemma \autoref{lemma3.11} gives the existence of a vertex $v\in X$ that would extend either $C_1$ or $C_2$ to an even bigger clique in $\induz{G}{X}$, contradicting their maximality. Hence all pairs of cliques in $C_X$ have a non-empty intersection.\\
Now suppose $C_X$ is not maximal, then there is a maximal Clique $C$ in $G$ with $C\setminus X\neq\emptyset$ and $C'\cap C\neq\emptyset$ for all $C'\in C_X$. Thus $C\cup C'$ is a $2$-strong clique by Lemma \autoref{lemma3.10} and there exists some $c\in C\setminus C$ with $\distg{G}{c}{x}\leq2$ for all $x\in X$ and therefore $X$ is not maximal.\\
Now for the other direction we have $\bigcup_{C\in C_X}C=X$ and with $C_1\cap C_2\neq\emptyset$ we get $\distg{G}{c_1}{c_2}\leq 2$ for all $c_1\in C_1$ and $c_2\in C_2$. Hence $X$ is a $2$-strong clique. Suppose $X$ is not maximal, then there exists some $c\in\V{G}\setminus X$ with $\distg{G}{c}{x}\leq 2$ for all $x\in X$ and thus $\set{c}\cup C'$ is a $2$-strong clique for all $C'\in C_X$. Lemma \autoref{lemma3.11} gives us, regarding the maximality of $C'$, some clique $C$ that contains $c$ and that can be chosen maximal such that $C\cap C'\neq\emptyset$, contradicting the maximality of $C_X$.
\end{proof}

As a last result in this chapter we will extend Lemma \autoref{lemma3.11} to additive steps in powers of graphs and thus giving some more structure to general powers of chordal graphs.

\begin{lemma}\label{lemma3.13}
Let $G$ be a graph and $k\in\N$ such that $G^k$ is chordal, furthermore let $C$ and $Q$ be maximal $k$-strong cliques in $G$ with $C\cup Q$ is a $\lb k+1\rb$-strong clique in $G$, then for all $c\in C$ there is a vertex $q\in Q$ adjacent to $c$.
\end{lemma}

\begin{proof}
If there is a vertex $c\in C$ with $\distg{G}{c}{q}\geq 2$ for all $q$ there is a set $U\subseteq\nb{c}$ such that for every $q\in Q$ exists a $u\in U\cap\knb{k}{q}$.\\
Similarly to the proof of Lemma \autoref{lemma3.11} we claim that for every $Q_i\subseteq Q$ with $1\leq i=\abs{Q_i}\leq\abs{Q}$ a $u\in U$ exists with $Q_i\subseteq\knb{k}{u}$. Obviously the case $i=1$ holds, so consider $Q_i\subseteq Q$ with $\abs{Q_i}=i$. Furthermore consider $Q_{i-1}\subseteq Q_i$ with $\abs{Q_{i-1}}=i-1$ and $Q_i\setminus Q_{i-1}=\set{q_i}$. By assumption there is a $u_{i-1}\in U$ with $Q_{i-1}\subseteq\knb{k}{u_{i-1}}$ and there is a $u_i\in U$ with $q_i\in\knb{k}{u_i}$.\\
If $u_i=u_{i-1}$ there is nothing to do, so suppose $u_i\neq u_{i-1}$ and therefore $\distg{G}{q_i}{u_{i-1}}\geq k+1$. For all $1\leq j\leq i-1$ we have $\distg{G}{q_j}{u_{i-1}}\leq k$, $\distg{G}{u_{i-1}}{u_i}\leq k$, $\distg{G}{u_i}{q_i}\leq k$ and $\distg{G}{q_i}{q_j}\leq k$, hence $\lb q_iq_ju_{i-1}u_i\rb$ is a cycle of length $4$ in $G^k$. The chord $q_iu_{i-1}$ does not exist by assumption and so we get $\distg{G}{q_j}{u_i}\leq k$. Hence $Q_i\subseteq\knb{k}{u_i}$.\\
Now we have a $u\in U$ with $u\notin Q$ and therefore $\distg{G}{u}{q}\leq k$ for all $q\in Q$ contradicting the maximality of $Q$.
\end{proof}

With these observation on the structure of cliques in both chordal and non chordal powers of graphs, implying some ideas of their algorithmic use, especially for finding maximum cliques in general graphs, we close this last section of the thesis and move on to some further thoughts on this tpoic in order to conclude this work.

\chapter{Conclusion}


In this thesis we have investigated one of many possibilities to generalize the ordinary coloring 
problem. Other, but related problems like list-coloring, $\fkt\textbf{mathcal{L}}{p,q}$-labeling or the 
even more general channel assignment problem were not even taken into account. Other 
problems, like acyclic colorings or star colorings just have occurred in small side notes. This alone 
shows the sheer variety of possible problems that may occur in this field.\\
And for any single one of those the general structural problems that are responsible for their immense complexity are unclear. In many cases we can find very basic structures that behave similar to cliques for the ordinary vertex coloring problem. Remember our $k$-strong cliques or $k$-strong anti-matchings. Especially the cases of acyclic colorings and star coloring showed us that there even are possible deeper connections between, at first, different approaches on the concept of colorings. So especially extremal question like the one for a conceptual analog to perfection seem to be universally applicable.\\
The basic problems for such questions seem to be the translation of the terminology of one problem to another. But finding such universal similarities is essential for the complete understanding of the true nature of this wide family of combinatorial problems.\\
For the ordinary coloring problem a conjecture was formulated by Hadwiger which is said to be the deepest and most profound structural question regarding graph coloring, the so called Hadwiger Conjecture, which states that any graph $G$ with chromatic number $\fkt{\chi}{k}$ possess a complete minor on $k$ vertices.\\
This said, is it possible to find analog and maybe even more general conjectures - or even theorems - for at least some of our generalized coloring problems?\\
Another important question is the one of the approach. Many researchers, as we did ourselves in this thesis, restrict themselves to smaller graph classes with richer structures that therefore yield more promising result. But to obtain a more general understanding of what coloring, and its countless derivatives, truly is, it is the authors belief that we need to broaden our horizon and, as said before, look for more general structural similarities.\\
\\
With this we slightly change our perspective and take a look on graph powers. The concept of graph powers has been studied mainly for its application on distance coloring, or, in the broad majority, for its application on the coloring of graph squares. Some, very few, attempts have already been made to find some use of this concept in other structural graph problems like Hamiltonian graphs. Without a doubt many other connections of graph powers to maybe even surprising problems are yet to be discovered. And this is not even limited to simple graphs. In the literature some connections to hypergraph theory have already been drawn and especially the question for chordal or even perfect graph powers yields strong connections to the theory of hypergraphs that are generalizations of the concept of bipartite graphs like hyper trees or normal hypergraphs.\\
In this context the question for perfect graph powers, and perfect line graph powers has to be asked again. In this thesis we already made some small steps into this topic by finding some structural hints on perfect graph squares. As we have seen the underlying structure of perfection in graph powers quickly becomes extremely complex but it is the author's belief that an investigation into this topic yields strong results, not only for structural but also for algorithmic graph theory.\\
Which brings us to the last chapter of this thesis, the algorithmic application. We introduced the power of chordality as an easy to compute parameter for graphs that seems to yield strong algorithmic potential. It not only holds an approximation of the minimum fill-in problem, treewidth and $k$-strong colorings but also brings the possibility to find orderings of the vertices of a graph with some structural meaning. Especially the point of orderings raised some additional questions as we encountered some further differentiation between chordal powers for the first time. It appears that chordal graph powers, or graphs with chordal powers, can be classified as either weak or strong in terms of the translation of a $k$-strong PES to the original graph $G$.\\
Especially with the study of perfect graph powers, such concepts might even prove to be helpful tools in the search for feasible inequalities that almost always occurs if one choses the approach in terms of linear integer programming on optimization problems on graphs. Especially sparse networks and related design problems might profit from the structure hidden in chordal or perfect powers and the knowledge they hold on their base graphs.\\
We have seen that with the clique number of the power of chordality we found another promising parameter, at least for those problems that already can be parameterized by treewidth. But our toolbox is much bigger. The power of chordality and higher chordal powers that necessarily exist due to Duchet's Theorem yield upper and lower bounds on $k$-strong clique numbers, colorings and maybe even other structural properties of a graph and, as mentioned before, provide vertex orderings. By combining all these information $\omega_{k_0}$ might even turn out to be a fitting parameter for the coloring problem, as we have seen vertex cover already is.\\
At last we would like to mention even more advanced concepts of graph decomposition that are somewhat similar to the tree decomposition. Some of them have come up just recently and often have proven to be suitable for problems that would resist any approach in terms of a tree decomposition. Such parameters are the connected treewidth, the treedepth, the cliquewidth or the branchwidth.\\
As one can see in this conclusion the field of structural chromatic graph theory and its adjacent fields still holds a great amount of unanswered and intriguing questions that just wait for someone to pick them up and once more start to discuss them.\\
Graph theory, maybe because of its strong relation to computer science, is one of the fastest growing fields of modern mathematics and the more it broadens the more fascinating it becomes and the more beautiful flowers of mathematics can be found in it.

\newpage
\bibliographystyle{amsalpha}
\bibliography{bibliography}

\newpage
\includepdf[pages=-]{Eidesstatt.pdf} 

\end{document}